\UseRawInputEncoding
\documentclass[a4paper, 10.5pt]{article}
\usepackage{amsmath}
\usepackage{amssymb}
\usepackage{amsthm}
\usepackage{mathtools}
\usepackage[dvipdfmx]{hyperref}
\usepackage[dvipdfmx]{graphicx}
\usepackage{comment}
\usepackage{cases}
\usepackage{float}
\usepackage{color}
\usepackage[top=26truemm, bottom=26truemm, left=28truemm, right=28truemm]{geometry}
\usepackage{authblk}
\usepackage{empheq}

\allowdisplaybreaks[1]

\newcommand{\spt}{\mathrm{spt}\,}
\newcommand{\divergence}{\mathrm{div}}

\newcommand{\dist}{\mathrm{dist}}

\newcommand{\boldv}{\mathbf{v}}
\newcommand{\boldH}{\mathbf{H}}

\newcommand{\boldN}{\mathbf{N}}
\newcommand{\boldnu}{\mathbf{\nu}}
\newcommand{\boldgam}{\mathbf{\gamma}}
\newcommand{\boldg}{\mathbf{g}}
\newcommand{\boldten}{\mathbf{\otimes}}
\newcommand{\boldA}{\mathbf{A}}
\newcommand{\boldB}{\mathbf{B}}

\let\OLDthebibliography\thebibliography
\renewcommand\thebibliography[1]{
	\OLDthebibliography{#1}
	\setlength{\parskip}{1pt}
	\setlength{\itemsep}{1pt plus 0.3ex}
}

\makeatletter

\@addtoreset{equation}{section}
\makeatother

\makeatletter
\def\blfootnote{\gdef\@thefnmark{}\@footnotetext}
\makeatother

\theoremstyle{mystyle}
\newtheorem{theorem}{Theorem}[section]
\newtheorem*{theorem*}{Theorem}
\newtheorem{lemma}[theorem]{Lemma}
\newtheorem{proposition}[theorem]{Proposition}

\theoremstyle{definition}
\newtheorem{definition}[theorem]{Definition}
\theoremstyle{remark}
\newtheorem{remark}[theorem]{Remark}

\title{A Varifold Formulation of Mean Curvature Flow with Dirichlet or Dynamic Boundary Conditions}

\author[1]{Yoshikazu Giga}
\author[1, 2]{Fumihiko Onoue}
\author[3]{Keisuke Takasao}

\affil[1]{Graduate School of Mathematical Sciences, The University of Tokyo, Komaba 3-8-1 Meguro Tokyo 153-8914, Japan. E-mail: (Y. Giga)\,{\tt labgiga@ms.u-tokyo.ac.jp} and (F. Onoue)\,{\tt fumihiko@ms.u-tokyo.ac.jp}}
\affil[2]{Scuola Normale Superiore, Piazza dei Cavalieri 7, Pisa 56126, Italy}
\affil[3]{Department of Mathematics/Hakubi Center, Kyoto University, Kitashirakawa-Oiwakecho Sakyo Kyoto 606-8502,  Japan. E-mail: {\tt k.takasao@math.kyoto-u.ac.jp}}

\date{}

\begin{document}
\blfootnote{\textup{2000} \textit{Mathematics Subject Classification}: 35A15, 35K20, 49Q20, 35D30}

\maketitle　　　　　　　　　　　　　　

\begin{abstract}
We consider the sharp interface limit of the Allen-Cahn equation with Dirichlet or dynamic boundary conditions and give a varifold characterization of its limit which is formally a mean curvature flow with Dirichlet or dynamic boundary conditions. In order to show the existence of the limit, we apply the phase field method under the vanishing on the boundary and the boundedness of the discrepancy measure. For this purpose, we extend the usual Brakke flow under these boundary conditions by the first variations for varifolds on the boundary.
\end{abstract}

\section{Introduction}\label{introduction}
Let $\Omega\subset \mathbb{R}^n$ be an open and bounded set with a smooth boundary $\partial \Omega$. For a parameter $\sigma \in (0,\,\infty)$, let $\{M^{\sigma}_t\}_{t\geq 0}$ be a family of hypersurfaces in $\Omega$ such that $\partial M^{\sigma}_t \subset \partial \Omega$ and $\partial M^{\sigma}_t$ is oriented. Let $\boldH^{\sigma}$ denote the mean curvature vector of a hypersurface $M^{\sigma}_t$ and $\boldN^{\sigma}_b$ denote the unit normal vector of $ \partial M^{\sigma}_t$ on $\partial \Omega$. We consider one of the generalized solutions to mean curvature flow with the following boundary condition:
\begin{align}\label{1.1.2}
\begin{cases}
\textbf{v}^{\sigma}=\textbf{H}^{\sigma}&\quad {\rm on}\ M^{\sigma}_t,\, t>0,\\
\displaystyle 
{\boldv}^{\sigma}_b=\frac{\sigma}{\tan\theta}\,\boldN^{\sigma}_b&\quad {\rm on }\ \partial M^{\sigma}_t,\, t>0, 
\end{cases}
\end{align}
where $\boldv^{\sigma}$ and ${\boldv}^{\sigma}_b$ are the velocity vector of $M_t^{\sigma}$ and $\partial M^{\sigma}_t$, respectively and $\theta$ is the contact angle formed by $M_t^{\sigma}$ and $\boldN^{\sigma}_b$. The motivation to investigate the $\sigma$-parametrized boundary condition given in \eqref{1.1.2} derives from the formal observation that one can study the three boundary conditions at the same time: Dirichlet, dynamic, and Neumann boundary conditions by choosing $\sigma$ as 0, a finite and positive number, and $\infty$, respectively. The dynamic boundary condition of various types has been studied from another point of view such as the theory of viscosity solutions or the semilinear elliptic problems (see, for instance, \cite{AnGu}, \cite{EGG}, \cite{GiHa}, \cite{FiIsKa01}, \cite{FiIsKa02}, \cite{FiIsKaLa01}, or \cite{FiIsKaLa02}). Moreover, the motion by mean curvature of the graph or level-set with Dirichlet boundary conditions has been also investigated in \cite{LT, SZ, Priwitzer}. 

Our goals in this paper are the following two observations: the first goal is that we consider the singular limit of the Allen-Cahn equation by applying the phase field method under the assumptions that the discrepancy measure vanishes on the boundary $\partial \Omega$ and the discrepancy function (see \eqref{discrepancyFunc} for the definition) is uniformly bounded from above in $\Omega$. The second one is that we formulate a Brakke flow with Dirichlet or dynamic boundary conditions, which can be regarded as a generalization of the motion in \eqref{1.1.2}. As we mentioned, formally speaking, we have that the boundary condition in \eqref{1.1.2} corresponds to Dirichlet, dynamic and right-angle Neumann boundary conditions as $\sigma = 0$, $\sigma$ is finite and positive, or $\sigma=\infty$, respectively. Motivated by this formal argument, we try to consider the formulation of a Brakke flow with Dirichlet or dynamic boundary conditions in the following way: first, we will formulate a generalized solution to the mean curvature flow described by \eqref{1.1.2} in Brakke's sense. Secondly, we will take the limit of $\sigma$ to 0 or finite positive to obtain the definition of Dirichlet or dynamic boundary conditions, respectively in some weak sense. Note that this argument is not rigorous but just a formal argument and we will state the rigorous analysis for such limiting procedure in Section \ref{existence}.

Here we briefly recall the mean curvature flow of closed hypersurfaces. We say that a family of hypersurfaces $\{M_t\}_{t\geq 0}$ in $\mathbb{R}^{n}$ moves by its mean curvature if the following equation holds:
\begin{equation}\label{1.1.0}
\boldv(\cdot,\,t)=\boldH(\cdot,\,t),\quad \text{on $M_t,\,t>0$},
\end{equation}
where $\boldv(\cdot,\,t)$ is the velocity vector of a hypersurface $M_t$ and $\boldH(\cdot,\,t)$ is the mean curvature vector of $M_t$. In this case, the hypersurface $M_t$ evolves to minimize its area. The notion of mean curvature flow was proposed by Mullins \cite{Mullins} to describe the motion of grain boundaries. Generally, it is known that singularities such as cusps or vanishing may occur in finite time and the motion of $M_t$ after the appearance of singularities cannot be analysed in the classical sense. However, from the point of revealing the phenomena in the nature, we would like to study the motions of evolving surfaces even though some singularity occurs.

Several generalized flows by mean curvature have been introduced so far. One is the level set flow proposed by Ohta, Jasnow and Kawasaki \cite{OJK} or Osher and Sethian \cite{OS}; the latter introduced the level-set equation to study the motion by mean curvature numerically. Later, Chen, Giga, and Goto \cite{CCG} and, independently, Evans and Spruck \cite{EvSp} rigorously introduced the generalized solutions to mean curvature flow in the viscosity sense and proved the existence and uniqueness of the viscosity solutions. We refer to the book \cite{Giga} for this approach.

Another approach is the Brakke flow that Brakke in \cite{Brakke} proposed by using geometric measure theory, especially the theory of varifolds. This approach made it possible to deal with the motion of hypersurfaces with a variety of singularities such as triple junctions. He proved the global-in-time existence of a Brakke flow in $\mathbb{R}^n$ with an approximation scheme and compactness-type theorems on varifolds if a general integral varifold defined on $\mathbb{R}^n$ is given as an initial data. One problem on Brakke's results in \cite{Brakke} was that the construction of the approximation scheme and the proof of existence theorem he obtained does not preclude the possibility of a trivial flow, for instance, the one that has a sudden loss of its mass for all time except an initial time. This problem remained open for a long time. Fortunately, Kim and Tonegawa \cite{KiTo} recently have succeeded proving, for the first time, a global-in-time existence theorem of the nontrivial mean curvature flow of grain boundaries by reformulated and modified approximation scheme. Moreover, Stuvard and Tonegawa in \cite{StTo} studied the existence of the non-trivial Brakke flow with fixed boundary conditions.

As a different view of a Brakke flow, by applying the phase field method via Allen-Cahn equations, Ilmanen \cite{Ilmanen01} considered the nontrivial global-in-time solutions to a Brakke flow without boundaries. The phase field method is the method that we make an approximation of a hypersurface by the transition layer with a small width of an order $\varepsilon>0$ and characterize it by considering the singular limit ($\varepsilon \downarrow 0$) of the transition layer. With the same method, Mizuno and Tonegawa in \cite{MiTo} firstly formulated the mean curvature flow with right-angle Neumann boundary conditions in the sense of Brakke. They consider the singular limit of the Allen-Cahn equations with Neumann boundary conditions when the concerned domain is strictly convex and bounded. Later, Kagaya \cite{Kagaya} extended their results to a non-convex bounded domain. However, as far as we know, a weak formulation of a Brakke flow with other boundary conditions has not been considered. In connection with the phase field method for the boundary-value problems, the authors in \cite{KKR} studied the singular limit of the Allen-Cahn equations with right-angle Neumann boundary conditions on the convex domain and rigorously proved that the separating front moves by its mean curvature in the sense of viscosity solutions, not Brakke's sense. Motivated by these works, we aim to consider the singular limit of the Allen-Cahn equations and formulate a Brakke flow with Dirichlet or dynamic boundary conditions. To attain this goal, as our first attempt, we study the singular limit under an assumption that the discrepancy measure vanishes on the closure of the domain and characterize its limit.

For the problem how to define a Brakke flow with Dirichlet or dynamic boundary conditions, we need to obtain at least the following two sufficient conditions: \textcircled{\scriptsize 1} \textbf{Brakke's inequality describing the motion} and \textcircled{\scriptsize 2} \textbf{boundary conditions in some sense corresponding to Dirichlet or dynamic boundary conditions.} 

\textcircled{\scriptsize 1} \textbf{Brakke's inequality}. First of all, we recall that Brakke's inequality, which Brakke introduced in \cite{Brakke}, can be regarded as a weak formulation of the motion of evolving surfaces (see, for instance, \cite{Tonegawa03}) and this inequality is motivated by the transport equation for a family of smooth hypersurfaces with boundaries as we show in the following; let $\{M_t\}_{t\in [0,\,\infty)}$ be a family of smooth hypersurfaces on $\Omega$ with a smooth boundary $\partial M_t \subset \partial \Omega$. If $M_t$ evolves by its mean curvature vector $\boldH$, then, calculating the quantity that corresponds to the time derivative of the surface area, we can have 
\begin{equation}\label{1.1.3.1}
\frac{d}{dt}\int_{M_t}\phi \,d\mathcal{H}^{n-1} = \int_{M_t}\left(-\phi |\boldH|^2+ \nabla \phi \cdot \boldH +\partial_t \phi \right)\, d\mathcal{H}^{n-1}+ \int_{\partial M_t} \phi\, {\boldv}_b \cdot \textbf{$\gamma$} \,d\mathcal{H}^{n-2}
\end{equation}
for any test function $\phi$ with $\phi\geq 0$ and $t>0$, where $\boldv_b$ is the velocity vector of $\partial M_t$ on $\partial \Omega$ and $\boldgam$ is the unit co-normal vector of $\partial M_t$ (see  Figure \ref{fig.2}). It is seen by a simple calculation that we may not need to assume any boundary conditions on $M_t$ to derive \eqref{1.1.3.1} and thus we should emphasize that, in order to consider the surface evolution with boundary conditions, Brakke's inequality is not sufficient for the formulations. Note that Brakke considered in \cite{Brakke} the case that $\partial M_t=\emptyset$ and hence Brakke's original inequality was introduced from \eqref{1.1.3.1} without the last term.  

As an analogy of this identity, by considering the singular limit of the Allen-Cahn equations \eqref{1.1.1}, we will define the following inequality as a Brakke's inequality for a Brakke flow with Dirichlet or dynamic conditions (see Section \ref{formulation} for precise Brakke's inequality); let $\{V_t\}_{t\geq 0}$, $\alpha$, and $\tilde{\boldv}_b$ be a family of $(n-1)$-varifolds on $\overline{\Omega} \subset \mathbb{R}^n$, a non-zero Radon measure on $\partial \Omega \times [0,\,\infty)$, and a vector-valued function on $\partial \Omega \times [0,\,\infty)$, respectively. Then we will define a Brakke's inequality for the triplet $(V_t,\,\alpha,\,\tilde{\boldv}_b)$ by the following inequality:
\begin{equation}\label{1.1.4}
\int_{\overline{\Omega}}\phi\,d\|V_t\|\Big|_{t=t_1}^{t_2}\leq \int_{t_1}^{t_2}\int_{\overline{\Omega}}(-\phi |\widetilde{\boldH}_V|^2+\nabla \phi \cdot \widetilde{\boldH}_V+\partial_t \phi)\,d\|V_t\|\,dt- \frac{1}{\sigma}\int_{\partial \Omega\times [t_1,\,t_2]} \phi |\tilde{\boldv}_b|^2\, d\alpha,
\end{equation}
for any non-negative test function $\phi$ with some conditions, where $\|V_t\|$ is the mass measure of $V_t$, and $\widetilde{\boldH}_V$ is the modified generalized mean curvature vector (see Definition \ref{def.4.1} or \ref{def.4.2} in Section \ref{formulation} for more detail). Note that the inequality \eqref{1.1.4} is for the case of dynamic boundary conditions (in the case when $\sigma$ is finite and positive). For the case of Dirichlet boundary conditions ($\sigma\to0$), if we consider the singular limit of Allen-Cahn equations under some assumptions, especially the one \eqref{3.1.9.2} (see Subsection \ref{exis.assmp.diri} of Section \ref{existence}), then we may have that the second term of the right-hand side in \eqref{1.1.4} vanishes as $\sigma\to0$. In this analogy, we may see that the mass $\|V_t\|$ of $V_t$ and the measure $\alpha$, roughly speaking, correspond to the measure $\mathcal{H}^{n-1}\lfloor_{M_t}$ and the measure  $(\sin\theta)\mathcal{H}^{n-2}\lfloor_{\partial M_t}\otimes\mathcal{L}^1_t$, respectively, where $\otimes$ is the product of measures which is exactly defined in Definition \ref{def.prod} of Section \ref{prelim} and $\mathcal{L}^1_t$ is the 1D Lebesgue measure on $\mathbb{R}$. In addition, $\tilde{\boldv}_b$ can be seen as the velocity vector of $\partial M_t$ on $\partial \Omega$. As we mentioned before, only to have the Brakke's inequality \eqref{1.1.4} is not sufficient for us to give the notion of a Brakke flow with Dirichlet or dynamic boundary conditions. As for the final comment of this paragraph, we refer to the work of Kasai and Tonegawa \cite{KaTo}. They proved the local-in-time regularity results for the varifold solutions in $\mathbb{R}^n$ satisfying the inequality similar to \eqref{1.1.4} and, hence it makes sense to consider \eqref{1.1.4} as a weak formulation of mean curvature flow. 

\textcircled{\scriptsize 2} \textbf{Boundary conditions corresponding to Dirichlet or dynamic boundary conditions}. Secondly, we need to determine boundary motions of varifolds in such a way that these motions represent Dirichlet or dynamic boundary conditions. To do this, we first define the following two linear functionals: a \textit{boundary functional} $\mathcal{S}_{\alpha,\,{\boldv}_b}$ on $(C_c(\partial \Omega \times [0,\,\infty)))^n$ for a Radon measure $\alpha$ on $\partial\Omega\times[0,\,\infty)$ and a vector-valued function ${\boldv}_b\in(L^2(\alpha))^n$. Roughly speaking, the total variations of these functionals are regarded as the $L^2$-norm of ${\boldv}_b$ with respect to $\alpha$ (see Definition \ref{def3.1.2} in Section \ref{prelim}). Then, as the boundary condition for varifolds, we, roughly speaking, define the absolute continuities by using the total variations of $\mathcal{S}_{\alpha,\,\boldv_b}$, the mass measures and the total variation measures for varifolds as follows: 
\begin{itemize}
	\item (Dirichlet boundary condition)
	\begin{equation}\label{1.1.5.1}
	\left\| \mathcal{S}_{\alpha,\,{\boldv}_b}\right\| \ll \|V_t\|\boldten\mathcal{L}^1_t\quad \text{on $\partial \Omega\times [0,\,\infty)$}.
	\end{equation}
	\item (Dynamic boundary condition for $\sigma \in (0,\,\infty)$)
	\begin{equation}\label{1.1.5.2}
	\left\|\int_{0}^{\infty}\delta V_t\lfloor_{\partial \Omega}^T\,dt+\mathcal{S}_{\alpha,\,{\boldv}_b}\right\| \ll \|V_t\|\boldten\mathcal{L}^1_t\quad \text{on $\partial \Omega\times [0,\,\infty)$}.
	\end{equation}
\end{itemize}
Here $\ll$ means the absolute continuity for measures and $\delta V_t\lfloor^T_{\partial \Omega}$ is the tangential first variation restricted to $\partial \Omega$ (see Definition \ref{def.3.3} for more detail). In addition to the absolute continuity \eqref{1.1.5.1}, we require one more condition for the definition of Brakke flow only \textbf{in the case of Dirichlet boundary conditions}. Precisely, the measure $\alpha$ and the function $\boldv_b$ satisfy the condition that there exists a sequence of solutions to Allen-Cahn equations with proper boundary conditions and the convergence
\begin{equation}
	\lim_{j\to\infty}\int_{\partial \Omega \times [0,\,\infty)} \boldv_b^{j} \cdot \boldg \,d\alpha^{j} = -\mathcal{S}_{\alpha,\,\boldv_b}(\boldg)
\end{equation}
where the measures $\{\alpha^{j}\}_{j\in\mathbb{N}}$ and the functions $\{\boldv_b^{j}\}_{j\in\mathbb{N}}$ are defined properly by the solutions of the Allen-Cahn equations (see what we call ``Approximation property on the boundary" in Section \ref{formudiri} for the precise definition). The approximation property only in the case of Dirichlet boundary conditions may be necessary in our formulation because without this approximation property, it may be possible that the classical mean curvature flow with the right-angle Neumann boundary conditions can be also our Brakke flow with Dirichlet boundary conditions (see Remark \ref{remarkDefBrakkeDiri} for more detail). Let us remark that the conditions \eqref{1.1.5.1} and \eqref{1.1.5.2} are natural as a result of considering the limit of $\sigma$-parametrized boundary condition \eqref{1.1.2} by taking $\sigma\downarrow0$ or $\sigma\to1$. However, since the assumption that we call ``Uniform upper bound on the boundary $\partial\Omega$" (see Section \ref{existence} for more detail) seems to be strong, we can actually obtain a stronger result than \eqref{1.1.5.1} (see Theorem \ref{thm3.4} in Section \ref{existence}). So far, we are not able to eliminate or relax this assumption because of the technical issue in deriving a priori estimates. Moreover, we are able to construct a counterexample of curves moving along the motion along \eqref{1.1.2} shown in Remark \ref{examplecurvatureflow} of Section \ref{existence}. This counterexample implies that without this assumption we might fail to obtain a singular limit of the Allen-Cahn equations.

Now let us give a formal explanation of these boundary conditions \eqref{1.1.5.1} and \eqref{1.1.5.2} in the case that hypersurfaces $M_t$ satisfying our definitions are sufficiently smooth (see Remark \ref{rem.2.1} and \ref{rem.3} in Section \ref{formulation}). When the hypersurfaces $\{M_t\}_{t}$ moving by their mean curvatures are smooth, we may regard the mass measure of a varifold associated with $M_t$, the Radon measure $\alpha$, and the vector field $\boldv$ as the following quantities:
\begin{equation}\label{1.1.5.3}
\|V_t\|\thickapprox \sigma_0 \mathcal{H}^{n-1}\lfloor_{M_t},\quad\alpha\thickapprox \sigma_0 (\sin\theta)\mathcal{H}^{n-2}\lfloor_{\partial M_t}\boldten\mathcal{L}^1_t,\quad {\boldv}_b\thickapprox(\text{the velocity of $\partial M_t$ on $\partial \Omega$}),
\end{equation}
where $\sigma_0 \coloneqq \int_{-1}^{1}\sqrt{2W(s)}\,ds$ (see also Remark \ref{rem.1} for this constant $\sigma_0$) and $W$ is a given function defined later (see also \eqref{1.1.1}). In the case of Dirichlet boundary conditions, if we assume that $\|V_t\|(\partial \Omega)=0$ for all $t>0$ which means that the geometric interior of $M_t$, that is, $M_t\setminus\partial M_t$ does not exist on $\partial \Omega$ for all the time, then we obtain, from \eqref{1.1.5.1},  $\|\mathcal{S}_{\alpha,\,{\boldv}_b}\|(\partial \Omega\times[0,\,\infty))=0$. Moreover, if we assume that the contact angle $\theta$ is not identically equal to zero, then we have that $\sin\theta$ is not identically equal to zero on $\partial \Omega$ for $t>0$. Since we may regard the total variation of the functional $\mathcal{S}_{\alpha,\,{\boldv}_b}$ as the $L^2$-norm of ${\boldv}_b$ with respect to the measure $\alpha = (\sin\theta)\mathcal{H}^{n-2}\lfloor_{\partial M_t}\boldten\mathcal{L}^1_t$, we have that ${\boldv}_b=0$ in $L^2$ on $\partial \Omega \times [0,\,\infty)$. Hence it is natural to consider \eqref{1.1.5.1} as Dirichlet boundary conditions (see Remark \ref{rem.2.1} in Section \ref{formulation} for more detail). Similarly, in the case of dynamic boundary conditions, if we again assume that $\|V_t\|(\partial \Omega)=0$ for all $t>0$, then we obtain, from \eqref{1.1.5.2}, $\|\int_{0}^{\infty}\delta V_t\lfloor_{\partial \Omega}^T \,dt + \mathcal{S}_{\alpha,\,{\boldv}_b}\|(\partial \Omega \times [0,\,\infty))=0$. Here, from the analogy between the classical and the measure theoretic first variation, which are explained later, we may regard the total variation of $\delta V_t\lfloor_{\partial \Omega}^T$ as the $L^2$-norm of $\boldgam^T$ with respect to $\mathcal{H}^{n-2}\lfloor_{\partial M_t}\boldten\mathcal{L}^1_t$, where $\boldgam$ is the outer unit conormal of $\partial M_t$ and $\boldgam^T$ is the tangential projection of $\boldgam$ onto $\partial \Omega$. Therefore, from \eqref{1.1.5.3} and $\|\int_{0}^{\infty}\delta V_t\lfloor_{\partial \Omega}^T \,dt + \mathcal{S}_{\alpha,\,{\boldv}_b}\|\equiv 0$, we obtain $(\sin\theta){\boldv}_b+\boldgam^T=0$ in $(L^2(\alpha))^n$ on $\partial\Omega$ and thus we conclude that ${\boldv}_b\cdot \boldN_b$ is equal to $(\tan\theta)^{-1}$ on $\partial\Omega$, where $\boldN_b$ is the outer unit normal vector of $\partial M_t$ on $\partial \Omega$. Hence, it is reasonable to consider the condition \eqref{1.1.5.2} as dynamic boundary conditions (see Remark \ref{rem.3} in Section \ref{formulation} for more detail).

These ideas of the formulation of Brakke flows are considered as one generalization of the results obtained by Mizuno and Tonegawa in \cite{MiTo}. They proved that the associated varifold with the limit measure of $\mu^{\varepsilon}_t$, which is defined later, and its first variation satisfies a proper absolute continuity on $\overline{\Omega}$, as a result of the singular limit of the Allen-Cahn equations with right-angle Neumann boundary conditions. Moreover, it is shown in their paper that the absolute continuity represents right-angle Neumann boundary conditions in a weak sense. The key idea is the analogy between the first variation for a hypersurface $M$ and a varifold $V$ with a locally bounded first variation $\delta V$ in the following; let $\{\Psi^g_t\}$ be the one-parameter group of diffeomorphism generated by the vector fields $\boldg\in (C^{\infty}_c(\overline{\Omega}))^{n}$. Suppose that $V$ has the locally bounded first variation. Then, from the definitions of the first variations, we have 
\begin{align}
\frac{d}{dt}\mathcal{H}^{n-1}(M_t)\Big|_{t=0}&
=\int_{M} \boldg\cdot (-\boldH)\,d\mathcal{H}^{n-1} + \int_{\partial M} \boldg\cdot \boldgam \,d\mathcal{H}^{n-2}, \label{1.1.4.1}\\
\delta V(\boldg)&= \int_{\overline{\Omega}}\boldg\cdot \eta\,\frac{d\|\delta V\|_{ac}}{d\|V\|} \,d\|V\|+ \int_{\overline{\Omega}}\boldg\cdot \eta\,d\|\delta V\|_{sing},\label{1.1.4.2}
\end{align}
where $M_t\coloneqq\Psi^{\boldg}_t(M)$, and $\boldgam$ is the unit co-normal vector of $\partial M$. Here $\eta$ is a $\|\delta V\|$-measurable vector-valued function such that $|\eta|=1$ $\|\delta V\|$-a.e. in $\overline{\Omega}$, and $\|\delta V\|_{ac}$ and $\|\delta V\|_{sing}$ are the absolute continuous and singular part of the measure $\|\delta V\|$ with respect to $\|V\|$, respectively. The existence of these quantity is derived by Riesz representation and Radon-Nikodym theorem. We refer to Remark \ref{rem.0} for more detail. 

To show the existence of the singular limit for our Brakke flow with Dirichlet or dynamic boundary conditions, we will apply the phase field method as we mentioned before. In the phase field method, the singular limit of the following Allen-Cahn equations corresponding to the equations \eqref{1.1.2} should be studied:
\begin{align}\label{1.1.1}
\begin{cases}
\partial_t u^{\varepsilon,\,\sigma} = \Delta u^{\varepsilon,\,\sigma}- \varepsilon^{-2}W'(u^{\varepsilon,\,\sigma})&\quad \text{in $\Omega \times (0,\,\infty)$}, \\
\partial_t u^{\varepsilon,\,\sigma}+\sigma \nabla u^{\varepsilon,\,\sigma} \cdot \boldnu=0& \quad \text{on  $\partial \Omega \times (0,\,\infty)$},\\
u^{\varepsilon,\,\sigma}(\cdot,\,0)=u^{\varepsilon,\,\sigma}_0(\cdot)&\quad \text{in $\Omega$},
\end{cases}
\end{align}
where $\varepsilon \in (0,\,1)$, $\sigma \in (0,\,\infty)$, $\nu$ is the outer unit normal vector and $W(s)\coloneqq\frac{1}{2}(1-s^2)^2$ is the double-well potential. The Allen-Cahn equation without boundary conditions in \eqref{1.1.1} was first proposed by Allen and Cahn \cite{AlCa} in order to study the phase separation in alloys. They introduced the free energy functional
\begin{equation}\label{1.1.3}
E[u^{\varepsilon}(\cdot,\,t)]= \int_\Omega \left(\frac{\varepsilon|\nabla u^{\varepsilon}(x,\,t)|^2}{2} + \frac{W(u^{\varepsilon}(x,\,t))}{\varepsilon}\right) \ dx
\end{equation}
for an order parameter $u^{\varepsilon}$. The Allen-Cahn equation is a $L^2$-gradient flow of the energy functional \eqref{1.1.3} and, by considering this equation, Allen and Cahn also formally established the mean curvature flow \eqref{1.1.0} as the correct limiting law of motion for antiphase boundaries. Later, their analysis was justified rigorously by, for instance, Bronsard and Kohn \cite{BrKo}. They proved that the solution of the Allen-Cahn equation converges to a piecewise constant function whose surfaces of discontinuities move along \eqref{1.1.0}. With these formal and rigorous analyses and by setting the Radon measure $\mu^{\varepsilon}_t$ as
\begin{equation}\label{1.1.5}
d\mu^{\varepsilon}_t\coloneqq\left(\frac{\varepsilon|\nabla u^{\varepsilon}(\cdot,\,t)|^2}{2} + \frac{W(u^{\varepsilon}(\cdot,\,t))}{\varepsilon}\right)\,dx,
\end{equation} 
one may expect that the measure $\mu^{\varepsilon}_t$ behaves like surface measures of moving phase boundaries under the finiteness assumption for $E[u^{\varepsilon}(\cdot,\,t)]$ for sufficiently small $\varepsilon>0$. For our problems, we consider the limit measure of $\mu^{\varepsilon,\,\sigma}_t$ defined by the solution $u^{\varepsilon,\,\sigma}$ to the equation \eqref{1.1.1}, which is a slight modification of $\mu^{\varepsilon}_t$. Then, by using the limiting measure of $\mu^{\varepsilon,\,\sigma}_t$, we characterize the motion by mean curvature in \eqref{1.1.2} in the sense of Brakke.

One of the interesting observations on the Allen-Cahn equations \eqref{1.1.1} is that the boundary condition in \eqref{1.1.2} may be obtained by considering the asymptotic analysis of the boundary condition in \eqref{1.1.1} as $\varepsilon \to 0$. From the asymptotic analysis, we may have that, if $\varepsilon$ is sufficiently close to 0, the following approximations hold:
\begin{equation}\label{1.1.3.0}
\frac{-\partial_t u^{\varepsilon,\,\sigma}}{|\nabla_{\partial \Omega} u^{\varepsilon,\,\sigma}|} \frac{\nabla_{\partial \Omega} u^{\varepsilon,\,\sigma}}{|\nabla_{\partial \Omega} u^{\varepsilon,\,\sigma}|} \thickapprox {\boldv}^{\sigma}_b,\quad \frac{|\nabla_{\partial \Omega} u^{\varepsilon,\,\sigma}|}{|\nabla u^{\varepsilon,\,\sigma}|} \thickapprox \sin\theta^{\sigma},\quad \frac{\nabla u^{\varepsilon,\,\sigma}}{|\nabla u^{\varepsilon,\,\sigma}|} \cdot \boldnu \thickapprox \cos\theta^{\sigma}\quad \text{on $\partial \Omega$}, 
\end{equation}
where ${\boldv}^{\sigma}_b$ is the velocity vector of $\partial M^{\sigma}_t$ on $\partial \Omega$ and $\theta^{\sigma}$ is the contact angle formed by $M_t^{\sigma}$ and $\partial \Omega$ (see Remark \ref{rem.1} and Figure \ref{fig.4} in Section \ref{existence}). From these approximations, it may hold that the boundary condition of \eqref{1.1.2} is obtained by taking the limit ($\varepsilon \to 0$) in the boundary condition of \eqref{1.1.1}. Another one is that, as we mentioned before, the boundary condition in \eqref{1.1.2} may be regarded formally as Dirichlet and dynamic boundary conditions when $\sigma\to 0$ and $\sigma>0$ is finite, respectively and as in the boundary condition in \eqref{1.1.1}. Thus, when we consider a Brakke flow with Dirichlet or dynamic boundary conditions, it is natural to consider the singular limit ($\varepsilon\to 0$) of \eqref{1.1.1} first and then take the limit of the parameter $\sigma$ which goes to 0 or positive finite $\sigma'$, respectively. However, for technical reasons, we need to take the limit of both $\varepsilon$ and $\sigma$ simultaneously to characterize the limit in the case of Dirichlet boundary conditions. Moreover, for the purpose of simplifying arguments, the parameter $\sigma$ is fixed with 1 when we consider the formulation of a Brakke flow with dynamic boundary conditions. 

In the above situation, we intend to characterize the limit of the Allen-Cahn equations and to obtain a Brakke flow with Dirichlet or dynamic boundary conditions which satisfies Brakke's inequality as in \eqref{1.1.4} and the condition as in \eqref{1.1.5.1} or \eqref{1.1.5.2}. One of the features to study the characterization is that we, for the first time, introduce a proper Radon measure $\alpha^{\varepsilon,\,\sigma}$ and a proper vector-valued function $\boldv^{\varepsilon,\,\sigma}_b(x,\,t)$ on $\partial \Omega \times [0,\,\infty)$ for any solutions $u^{\varepsilon,\,\sigma}$ to the Allen-Cahn equations \eqref{1.1.1}, any $\varepsilon>0$ and $\sigma>0$. Moreover, we newly define a proper linear functional $\mathcal{S}_{\alpha,\,{\boldv}_b}$ defined on $\partial \Omega \times [0,\,\infty)$ for a Radon measure $\alpha$ which is the proper limit of $\alpha^{\varepsilon,\,\sigma}$ and a vector-valued function ${\boldv}_b$ as we mentioned in the above. Roughly speaking, $\alpha^{\varepsilon,\,\sigma}$ and $\boldv^{\varepsilon,\,\sigma}_b$ approximates the product measure of the weighted area measure of $\partial M_t$ on $\partial \Omega$ and the Lebesgue measure on $[0,\,\infty)$ and the velocity vector of $\partial M_t$ on $\partial \Omega$, respectively. Those quantities make it possible to control the boundary terms of integrals which do not appear in the case of right-angle Neumann boundary conditions and then obtain the proper singular limits. Another feature is that we apply the convergence theorem for measure-function pairs which is introduced by Hutchinson \cite{Hutchinson} in order to consider the convergence of $\alpha^{\varepsilon,\,\sigma}$ and $\boldv^{\varepsilon,\,\sigma}_b$ and to show that the limits satisfy the definition of our Brakke flow with Dirichlet or dynamic boundary conditions. Note that it is necessary to show that the limit measure $\alpha$ is not identically equal to zero to obtain our Brakke flow as the singular limit. Fortunately, we may prove the local positivity of $\alpha$ in the case that the boundary of $\Omega$ is connected and we impose some assumption on $u^{\varepsilon,\,\sigma}$ (see Lemma \ref{thm3.2.1} or Lemma \ref{thm3.7.1}).

Generally there are two main difficulties in our problems in order to attain desirable singular limits for the Allen-Cahn equations with boundary conditions. One is to show the vanishing of the discrepancy measure parametrized by $\varepsilon$ up to the boundary of the domain $\Omega$ as $\varepsilon \to 0$. We define the \textit{discrepancy function} $\xi^{\varepsilon,\,\sigma}$ associated with the solutions $u^{\varepsilon,\,\sigma}$ to \eqref{1.1.1} as
\begin{equation}\label{discrepancyFunc}
	\xi^{\varepsilon,\,\sigma}(x,\,t) \coloneqq \frac{\varepsilon|\nabla u^{\varepsilon,\,\sigma}|^2(x,t)}{2} - \frac{W(u^{\varepsilon,\,\sigma})(x,t)}{\varepsilon}
\end{equation}
for any $(x,\,t)\in \Omega \times (0,\,\infty)$ and set the \textit{discrepancy measure}, denoted by $\xi^{\varepsilon,\,\sigma}_t$, as $\xi^{\varepsilon,\,\sigma}_t \coloneqq \xi^{\varepsilon,\,\sigma}(x,\,t) \mathcal{L}^n(x)$ for each $t>0$ where $\mathcal{L}^n$ is the $n$-dimensional Lebesgue measure. The vanishing in $\Omega$, not up to its boundary, was proved by several authors. Precisely they showed that, if the support of the limit measure of $\mu^{\varepsilon,\,\sigma}_t$ does not exist on the boundary of $\Omega$, then it follows that the limit measure of $\xi^{\varepsilon,\,\sigma}_t$ with respect to $\varepsilon$ is identically equal to zero in $\Omega$. For instance, Ilmanen \cite{Ilmanen01} proved the vanishing of the discrepancy measure in the case that $\Omega=\mathbb{R}^n$ by constructing, what we call, \textit{the monotonicity formula} for a measure $\mu^{\varepsilon}_t$ with respect to $t$. In the context of De Giorgi conjecture, R\"{o}ger and Sch\"{a}tzle \cite{RoSc} proved the vanishing of the discrepancy measure for the elliptic Allen-Cahn equations if $\Omega$ is an open subset of $\mathbb{R}^2$ or $\mathbb{R}^3$. Similarly, Tonegawa \cite{Tonegawa01} also proved the vanishing of the discrepancy measure for the elliptic Allen-Cahn equations in the context of van der Waals-Cahn-Hilliard theory of phase transitions. Moreover, in the higher dimensions, Takasao and Tonegawa \cite{TaTo} also constructed a localized version of the monotonicity formula and showed the vanishing of the discrepancy measure in the context of the mean curvature flow with transport term. In the case of right-angle Neumann boundary conditions, Mizuno and Tonegawa \cite{MiTo} and Kagaya \cite{Kagaya} proved the vanishing of the discrepancy measure. To do this, they constructed the monotonicity formula by using the reflection arguments. This monotonicity formula is based on the one that Ilmanen showed in \cite{Ilmanen01}. However, the vanishing up to the boundary remains to be solved in the case of other boundary conditions or even the monotonicity formula up to the boundary is not known in the case of other boundary conditions. The other difficulty is, in the case of Dirichlet boundary conditions, to obtain the uniform upper bound of the Dirichlet energy of $u^{\varepsilon,\,\sigma}$ along $\boldnu$ on $\partial \Omega$ in $\varepsilon$ and $\sigma$. One possible problem is that the Dirichlet energy of $u^{\varepsilon,\,\sigma}$ on the boundary $\partial \Omega$ can blow up as $\sigma \to 0$ due to the form of the boundary condition in \eqref{1.1.1}. Thus, after we consider the limit ($\varepsilon \to 0$), the boundary condition in \eqref{1.1.2} may not approximate Dirichlet boundary condition as $\sigma\to 0$. To see this, we succeed in constructing a family of the curves moving along the motion \eqref{1.1.2} and this construction indicates that the boundary condition in \eqref{1.1.2} may not approximate Dirichlet boundary condition if $\sigma$ converges to 0 (see Subsection \ref{exis.assmp.diri} of Section \ref{existence}). Because of this example, it is reasonable to put some assumption on the Dirichlet energy of $u^{\varepsilon,\,\sigma}$ along to $\boldnu$ and this assumption may prevent the occurrence of irregular motions.

Concerning the first difficulty, we are unfortunately unable to prove the vanishing of the discrepancy measure up to the boundary. However, one progress that we made in this paper is that, under some conditions such as the convexity of the domain and some estimate for the Dirichlet energy of $u^{\varepsilon,\,\sigma}$ along the unit normal vector of the boundary, we can show the boundedness of the discrepancy measure associated with the equation \eqref{1.1.1} (see Proposition \ref{appendAProp2} in Appendix A for the proof). This property was also proved by Mizuno and Tonegawa in \cite{MiTo}, who dealt with the case of Neumann boundary conditions. In addition to this, by using this boundedness property and employing the methods shown in \cite{MiTo}, we can prove the vanishing of the discrepancy measure only in the interior of the domain (see Proposition \ref{vanishingProp} in Appendix B for more detail). In the case of $\Omega=\mathbb{R}^n$, Ilmanen in \cite{Ilmanen01} proved the non-positivity result by using the maximal principle for the quotient of the two terms in the energy density. Regarding to the second difficulty, we still don't know the way to avoid assuming the uniform upper bound of the Dirichlet energy of $u^{\varepsilon,\,\sigma}$ on $\partial \Omega$. Hence, in this paper, we assume these two properties to obtain the main results. We emphasize that, in the case of dynamic boundary conditions, we may obtain the main results without assuming the uniform upper bound of the Dirichlet energy of $u^{\varepsilon,\,\sigma}$ on $\partial \Omega$. Thus this upper-bound assumption is necessary only in the case of Dirichlet boundary conditions. 

The organization of this paper is as follows; in Section \ref{prelim}, we give several notations and definitions of varifolds to make a formulation of our Brakke flow and derive its approximation results. Moreover, we introduce two important linear functionals; the one is defined on $(C_c(\partial \Omega \times [0,\,\infty)))^n$ and the other is defined on $(C_c(\Omega\times[0,\,\infty)))^n$. We also introduce two important Radon measures for the solutions of the Allen-Cahn equations \eqref{1.1.1}; the one is defined on $\partial \Omega \times [0,\,\infty)$ and the other is defined on $\Omega\times[0,\,\infty)$. These quantities play an important role to consider our problems especially when we formulate the boundary conditions for varifolds. We also describe the intuitive geometric meaning of this measure. In Section \ref{formulation}, we state the formulation of a Brakke flow with Dirichlet or dynamic boundary conditions and then we give the motivation of these formulations. In Section \ref{existence}, we state a sequence of the main lemmas and theorem, that is, the results of the singular limit of the Allen-Cahn equations and its characterization. Before we mention the main results, we give several assumptions and an important hypothesis in each case. Besides, in the case of Dirichlet boundary conditions, we also show the example which implies that the motion in \eqref{1.1.2} may not necessarily approximate the motion of mean curvature flow with Dirichlet boundary conditions as  $\sigma \to 0$. In Section \ref{apriori} and Section \ref{chara.lim}, we prove that the singular limit of the Allen-Cahn equations actually satisfies the definition of our Brakke flow with Dirichlet or dynamic boundary conditions. In Section 5, we derive a priori estimates and then, in Section 6, we calculate the first variation of the associated varifolds with $\mu^{\varepsilon,\,\sigma}$ and finally, consider the limit of the varifolds. In Section \ref{appendix}, we prove the uniform estimate of both a solution $u^{\varepsilon,\,\sigma}$ for the Allen-Cahn equations \eqref{1.1.1} and the discrepancy measure $\xi^{\varepsilon,,\sigma}_t$ under several assumptions. Moreover, we also prove the vanishing of the discrepancy measure $\xi^{\varepsilon,,\sigma}_t$ only in the interior of the domain under suitable assumptions. At the last, we recall Poincar\'e inequality on hypersurfaces.

\section{Preliminaries}\label{prelim}
We recall several definitions and notations related to varifolds and geometric measure theory to fix the notations. See for instance \cite{Allard} and \cite{Simon} for more detail.

Let $X \subset \mathbb{R}^n$ be an open or compact subset. Let $\mathbb{G}(n\,,k)$ be the set of $n$-dimensional subspaces of $\mathbb{R}^n$ equipped with the metric $d(S,\,T)\coloneqq\|S-T\|_{*}$ where let $\| \cdot\|_{*}$ denote the operator norm on the space of linear endomorphism of $\mathbb{R}^n$. We set $G_{k}(X)\coloneqq X\times \mathbb{G}(n\,,k)$ for $n,\,k\in\mathbb{N}$ with $n>k\geq 1$. For any $S\in \mathbb{G}(n\,,k)$, we can identify $S$ with the corresponding orthogonal projection of $\mathbb{R}^n$ onto $S$ and its matrix representation.

We define $\boldA:\boldB$ for $(n\times n)$-matrices $\boldA=(A_{ij})$ and $\boldB=(B_{ij})$ by 
\begin{equation}\label{2.1.0}
\boldA:\boldB\coloneqq \sum_{i,\,j=1}^{n} A_{ij}B_{ij}.
\end{equation}
Now we define the support of a measure $\mu$ on $X$ by
\begin{equation}\label{2.1.1}
\spt \mu\coloneqq\{x\in X \, | \, \mu(B_r(x))>0 \quad{\rm for \; all} \; r>0 \},
\end{equation}
where $B_r(x)$ is an open ball of a centre $x$ with a radius $r$. One may easily show that the set defined by the right-hand-side of \eqref{2.1.1} is a closed subset of $X$. 

In the following, we state several definitions of function spaces we use in the present paper. Let $m\in \mathbb{N}$ with $m\geq 1$ and $p\in [1,\infty]$ and let $\mu$ be a measure on $X$. First we say that $f$ belongs to $(L^p(\mu,\,X))^m$ for $p\in[1,\,\infty)$ if $f$ is defined $\mu$-a.e. on $X$ with the values on $\mathbb{R}^m$ and 
\begin{equation}\label{2.1.2}
\|f\|_{L^p(\mu,\,X)}\coloneqq\left(\int_{X} |f|^p\, d\mu \right)^{\frac{1}{p}} <\infty
\end{equation} 
holds. Moreover, we say that $f$ belongs to $(L^{\infty}(\mu,\,X))^m$ if 
\begin{equation}\label{2.1.2.1}
\|f\|_{L^{\infty}(\mu,\,X)}\coloneqq\inf\{\lambda\in \mathbb{R} \mid |f|\leq\lambda\quad \text{$\mu$-a.e. on $X$}\} <\infty.
\end{equation}
In particular, we write $L^p(\mu,\,X)$ when $m=1$.

Secondly, we say that $f$ belongs to $(C^k(X))^m$ for any $k\in \mathbb{N}\cup \{\infty\}$ and $m\in \mathbb{N}$ if $f$ is a $C^k$-function defined on $X$ taking the values on $\mathbb{R}^m$. Finally we say $f\in (C^k_c(X))^m$ if $f\in (C^k(X))^m$ with a compact support in $X$.
\begin{definition}[Convergence of Radon measures]
	Let $\{\mu_k\}_{k\in\mathbb{N}}$ be a family of Radon measures on $X$. We say that $\{\mu_k\}_{k\in \mathbb{N}} $ converges to a Radon measure $\mu$ on $X$ as Radon measures if and only if
	\begin{equation}\label{2.1.3}
	\lim_{k \to \infty} \int_{X} \phi \,d\mu_k=\int_{X} \phi \,d\mu
	\end{equation}  
	holds for all $\phi \in C_c(X)$. We write $\mu_k \rightharpoonup \mu$ as $k \to \infty$ if $\{\mu_k\}_{k\in \mathbb{N}}$  converges to $\mu$ in Radon measure and also often write $\mu_k(\phi) \to \mu(\phi)$ as $k \to \infty$ for any $\phi \in C_c(X)$, where we set $\mu(\phi)$ by 
	\begin{equation}\label{2.1.3.1}
	\int_{X}\phi\,d\mu.
	\end{equation}
	for any Radon measure $\mu$ on $X$ and any function $\phi\in C_c(X)$. 
\end{definition}

\begin{definition}[Product of measures]\label{def.prod}
	Let $\{\mu_t\}_{t\in[0,\,\infty)}$, $f(x,\,t)$, and $\mathcal{L}^1_t$ be a family of Radon measures on $X$ parametrized by $t\in[0,\,\infty)$, a function on $X\times[0,\,\infty)$ such that $f(\cdot,\,t)$ is $\mu_t$-integrable on $X$ for a.e. $t$, and the one-dimensional Lebesgue measure on $\mathbb{R}$, respectively. Then, throughout this paper, we define $\mu_t\otimes\mathcal{L}^1_t$ by
	\begin{equation}\label{2.1.4.1}
	(f\,\mu_t\otimes\mathcal{L}^1_t)(A\times[a,\,b))\coloneqq\int_{a}^{b}
	\!\!\int_{A}f(x,\,t)\,d\mu_t(x)\,dt,
	\end{equation}
	for any $A\subset X$ and any $0\leq a< b\leq\infty$.	
\end{definition}

\begin{definition}[$k$-rectifiable and $k$-integral Radon measure]
	We say that a Radon measure $\mu$ on $X$ is $k$-rectifiable if there exists $\mathcal{H}^k$-measurable countably $k$-rectifiable set $M\subset X$ and a locally $\mathcal{H}^k$-integrable positive function $\theta$ defined on $M$ such that
	\begin{equation}\label{2.1.5}
	\mu(A)= \int_{M\cap A} \theta \, d\mathcal{H}^{k}
	\end{equation}
	for any Borel set $A \subset X$.
	Moreover $\mu$ is $k$-integral if $\theta$ takes the values on $\mathbb{Z}_{>0}$ $\mathcal{H}^{k}$-a.e. in $M$.
\end{definition}

\begin{definition}[General $k$-varifold]
	A general $k$-varifold $V$ in $X$ is a Radon measure on $G_{k}(X)$. Let $\mathbb{V}_{k}(X)$ denote the set of all $k$-varifolds.
\end{definition}

\begin{definition} [Rectifiable $k$-varifold]
	Let $V\in \mathbb{V}_{k}(X)$. We say that $V$ is a rectifiable $k$-varifold if there exist a $\mathcal{H}^{k}$-measurable countably $k$-rectifiable set $M\subset X$ and a locally $\mathcal{H}^{k}$-integrable function $\theta$ defined on $M$ such that  
	\begin{equation}\label{2.1.6}
	V(\phi)\coloneqq \int_{X}\phi(x,\,T_x M)\theta(x) \,d\mathcal{H}^{k}\lfloor_{M},
	\end{equation}
	for any $\phi\in C_c(G_k(X))$, where $T_x M$ is the approximate tangent plane of $M$ at $x$ which exists $\mathcal{H}^k$-a.e. on $M$. The existence of the approximate tangent plane is due to the rectifiability of $M$.
\end{definition}

\begin{definition} [Mass measure]
	For any $V\in \mathbb{V}_{k}(X)$, we define the mass measure $\|V\|$ of $V$  on $X$ as the push-forward of $V$ by the projection $\pi: G_{k}(X)\to X$. In particular, if $V$ is a rectifiable $k$-varifold, its mass measure is expressed by $\|V\|= \theta \mathcal{H}^k\lfloor_M$.
\end{definition}

\begin{remark}
	The rectfiable $k$-varifold is uniquely determined by its mass measure through the identity \eqref{2.1.6}. Due to this, we say that a $k$-rectifiable varifold $V$ associated with a $k$-rectifiable Radon measure $\mu$ is a varifold such that the mass measure of $V$ is equal to $\mu$. 
\end{remark}

\begin{definition} [First variation of a varifold]
	For $V\in \mathbb{V}_{k}(X)$, we define the first variation $\delta V$ of $V$ by
	\begin{equation}\label{2.1.7}
	\delta V(\boldg)\coloneqq \int_{G_{k}(X)} \divergence_{S} \boldg(x) \,dV(x,\,S)
	\end{equation} 
	for any $\boldg\in (C^1_c(X))^n$, where $\divergence_S \boldg(x)$ is defined by 
	\begin{equation}\label{2.1.8}
	\divergence_S \boldg(x) \coloneqq \sum_{j=1}^n (\textbf{S}(\nabla g_j(x))\cdot e_j)= \sum_{i,j=1}^{n}S_{ij} \partial_{x_i} g_j(x)= \nabla g(x): \textbf{S},
	\end{equation}
	for any $\boldg\in (C^1_c(X))^n$ and any $S\in \mathbb{G}(n,\,k)$ and $\{e_j\}_{j=1}^n$ is the canonical basis of $\mathbb{R}^{n}$.
\end{definition}
\begin{remark}\label{rem.0}
	If a varifold $V$ has a locally bounded first variation, then we may extend the linear functional $\delta V$ into a locally bounded linear functional on $(C_c(X))^n$. Thus, from Riesz representation theorem, we have that the total variation $\|\delta V\|$ is a Radon measure on $X$ and there exists a $\|\delta V\|$-measurable function $\eta: X\to \mathbb{R}^n$ such that  $|\eta|=1$ $\|\delta V\|$-a.e. in $X$ and 
	\begin{equation}\label{2.1.8.1}
	\delta V(\boldg)= \int_{X}\boldg\cdot \eta \,d\|\delta V\|
	\end{equation} 
	for every $\boldg\in (C_c(X))^n$. Then, from Lebesgue decomposition theorem, we may decompose $\|\delta V\|$ into the absolutely continuous part $\|\delta V\|_{ac}$ and the singular part $\|\delta V\|_{sing}$ with respect to $\|V\|$. Therefore, from Radon-Nikodym theorem, we obtain
	\begin{equation}\label{2.1.9}
	\delta V(\boldg)= \int_{X}\boldg\cdot \eta\,\frac{d\|\delta V\|_{ac}}{d\|V\|}\,d\|V\|+ \int_{X} \boldg\cdot \eta \,d\|\delta V\|_{sing}
	\end{equation}
	for any $\boldg\in (C_c(X))^n$, where $\frac{d\|\delta V\|_{ac}}{d\|V\|}$ is the Radon-Nikodym derivative. If we set $\boldH_V\coloneqq- \, \frac{d\|\delta V\|_{ac}}{d\|V\|}\eta$, $\boldH_V$ is called the \textit{generalized mean curvature vector} of $V$. This definition is the analogy of the classical version: if $M \subset \mathbb{R}^n$ is a $k$-dimensional smooth embedded manifold, then, from the divergence theorem and the first variation of $M$ with a vector field $\boldg$, we have 
	\begin{equation}\label{2.1.10}
	\frac{d}{d\varepsilon}\mathcal{H}^{k}(\Psi^{\boldg}_{\varepsilon}(M))\Big|_{\varepsilon=0}=\int_{M}\divergence_{M}\boldg\,d\mathcal{H}^k= -\int_{M} \boldg\cdot \boldH\,d\mathcal{H}^{k}+ \int_{\partial M} \boldg\cdot \boldgam\,d\mathcal{H}^{k-1}
	\end{equation}
	for all $\boldg\in (C^{\infty}_c(X))^n$, where $\boldH$ is the mean curvature vector of $M$, $\boldgam$ is the outer unit normal vector of $\partial M$, tangential to $M$. Here a map $\Psi^{\boldg}_{\varepsilon}:M\to \mathbb{R}^n$ is defined by  $\Psi^{\boldg}_{\varepsilon}(x)\coloneqq x+\varepsilon \boldg(x)$ for all $x\in M$ and $\varepsilon \in (-1,\,1)$.
\end{remark}
\begin{definition}[First variation of a varifold with time integral]
	Let $\{V_t\}_{t\in [0,\,\infty)}\subset\mathbb{V}_{k}(X)$ be a family of varifolds. We define $(\int_{0}^{\infty}\delta V_t\,dt)(\boldg)$ for every $\boldg\in (C^1_c(X\times[0,\,\infty)))^n$ by
	\begin{equation}\label{2.1.10.1}
	\left(\int_{0}^{\infty}\delta V_t\,dt\right)(\boldg)\coloneqq \int_{0}^{\infty} \delta V_t(\boldg)\,dt.
	\end{equation}
\end{definition}

\begin{remark}
	Since the first variation of a varifold with time integral is a linear functional, we can also define \textit{the generalized mean curvature vector} in the following way: first, we assume that there exists a constant $C>0$ such that
	\begin{equation}\label{2.2.1}
	\left|\int_{0}^{\infty}\!\!\delta V_t\,dt(\boldg)\right|\leq C\,\sup_{X\times[0,\,\infty)}|\boldg|,
	\end{equation}
	for any $\boldg\in (C^1_c(X\times[0,\,\infty)))^n$. Then, we may extend the domain of $\int_{0}^{\infty}\!\!\delta V_t\,dt$ into $(C_c(X\times[0,\,\infty)))^n$. Thus, from Riesz representation theorem, we have that the total variation of $\int_{0}^{\infty}\!\!\delta V_t\,dt$ is a Radon measure on $X\times[0,\,\infty)$ and there exists a $\|\int_{0}^{\infty}\!\!\delta V_t\,dt\|$-measurable function $\boldsymbol{\eta}$ with $|\eta|=1$ $\|\int_{0}^{\infty}\!\!\delta V_t\,dt\|$-a.e. in $X\times[0,\,\infty)$ such that
	\begin{equation}\label{2.2.2}
	\int_{0}^{\infty}\!\!\delta V_t\,dt(\boldg)=\int_{X\times[0,\,\infty)}\boldg\cdot\boldsymbol{\eta}\,d\left\|\int_{0}^{\infty}\!\!\delta V_t\,dt\right\|,
	\end{equation}
	for any $\boldg\in (C_c(X\times[0,\,\infty)))^n$. By decomposing  $\|\int_{0}^{\infty}\!\!\delta V_t\,dt\|$ into the absolute and singular part with respect to $\|V_t\|\otimes\mathcal{L}^1_t$ and applying the Radon-Nikodym theorem, we can also define \textit{the generalized mean curvature vector in space-time} $\boldH_V$, as we did in Remark \ref{rem.0}, by
	\begin{equation}\label{2.2.3}
	\boldH_V\coloneqq-\frac{d\|\int_{0}^{\infty}\!\!\delta V_t\,dt\|_{ac}}{d(\|V_t\|\otimes\mathcal{L}^1_t)}\boldsymbol{\eta},
	\end{equation}
	where $\frac{d\|\int_{0}^{\infty}\!\!\delta V_t\,dt\|_{ac}}{d(\|V_t\|\otimes\mathcal{L}^1_t)}$ is a Radon-Nikodym derivative. Note that, in this paper, we will use the notation $\boldH_V$ in the sense of \textit{the generalized mean curvature vector in space-time}, which we define in this remark.
\end{remark}

In the following, we assume that $\Omega$ is an open subset of $\mathbb{R}^n$ with smooth boundary $\partial \Omega$.

\begin{definition}\label{def.3.3}
	Let $V \in \mathbb{V}_{n-1}(\overline{\Omega})$ be a varifold with a locally bounded first variation on $\overline{\Omega}$. We define the first variation of varifold tangential to $\partial \Omega$, denoting $\delta V \lfloor_{\partial \Omega}^T$, by
	\begin{equation}\label{3.2.5}
	\delta V\lfloor^{T}_{\partial \Omega}(\boldg)\coloneqq\delta V \lfloor_{\partial \Omega}(\boldg-(\boldg\cdot \boldnu)\boldnu)
	\end{equation}
	for any $\boldg\in (C^1(\partial \Omega))^n$, where $\boldnu$ is the outer unit normal vector of $\partial \Omega$.
\end{definition}

Now we define two linear functionals, which we name \textit{boundary functional} and \textit{interior functional} defined on $(C_c(\partial \Omega \times [0,\,\infty)))^n$ and $(C_c(\Omega\times[0,\,\infty)))^n$, respectively. The boundary functional is one of the keys to do the weak formulation of Dirichlet or dynamic boundary conditions of a Brakke flow and to prove the existence of the singular limits of the Allen-Cahn equations \eqref{1.1.1}. On the other hand, the interior functional is one of the keys to do the weak formulation of only Dirichlet boundary conditions.

\begin{definition}[Boundary functional]\label{def3.1.2}
	Let $\omega$ be a Radon measure on $\partial \Omega \times [0,\infty)$ and $\textbf{p}$ be in $(L^2(\omega,\,\partial \Omega \times [0,\infty)))^n$. Then we define the boundary linear functional $\mathcal{S}_{\omega,\,\textbf{p}}$ by
	\begin{equation}\label{3.1.15}
	\mathcal{S}_{\omega,\,\textbf{p}}(\boldg)\coloneqq \int_{\partial \Omega \times [0,\infty)} \boldg\cdot \textbf{p}\, d\omega,
	\end{equation}
	for any $\boldg\in (C_c(\partial \Omega \times [0,\infty)))^n$. Let $\|\mathcal{S}_{\omega,\,\textbf{p}}\|$ denote the total variation of $\mathcal{S}_{\omega,\,\textbf{p}}$.
\end{definition}

\begin{remark}\label{rem3.1.2}
	From its definition, the total variation $\|\mathcal{S}_{\omega,\,\textbf{p}}\|$ is in fact a Radon measure on $\partial \Omega \times [0,\,\infty)$. Indeed, let $K \subset \partial \Omega \times [0,\,\infty)$ be a compact set and we take any $\boldg \in (C_c(\partial \Omega \times [0,\,\infty)))^n$ such that $\spt \boldg \subset K$ holds. Then, by using Cauchy-Schwarz inequality, we have $|\mathcal{S}_{\omega,\,\textbf{p}}(\boldg)| \leq \|\textbf{p}\|_{L^2(\omega,\,\partial \Omega\times [0,\,\infty))} (\omega(K))^{\frac{1}{2}} \,\|\boldg\|_{L^\infty} <\infty $. This estimate allows us to apply Riesz representation theorem to $\mathcal{S}_{\omega,\,\textbf{p}}$ and then we obtain the conclusion. 
\end{remark}

Now we define \textit{the weighted boundary area measure} defined on $\partial \Omega$ for a solution $u^{\varepsilon,\,\sigma}$ of Allen-Cahn equations. This measure plays an important role when we formulate a Brakke flow with Dirichlet or dynamic boundary condition and prove its existence theorem. The reason we name the measure ``weighted boundary area" is stated in Remark \ref{rem.1} right after the definition.
\begin{definition}[Weighted boundary area measures]\mbox{} \label{def.2.3}
	Let $u^{\varepsilon,\,\sigma}$ be a solution to the equation \eqref{1.1.1}. Then, for all $\varepsilon>0$, $\sigma>0$ and all $t>0$, we define a weighted boundary area measure $\alpha^{\varepsilon,\,\sigma}_t$ on $\partial \Omega$ by 
	\begin{equation}\label{2.1.13}
	\alpha^{\varepsilon,\,\sigma}_t\coloneqq \varepsilon |\nabla_{\partial \Omega} u^{\varepsilon,\,\sigma}(\cdot,\,t)|^2\, \mathcal{H}^{n-1}\lfloor_{\partial \Omega},
	\end{equation}
	where $\mathcal{H}^{n-1}$ is the $(n-1)$-dimensional Hausdorff measure and $\nabla_{\partial \Omega}$ is a tangential gradient on $\partial \Omega$. Moreover, we set $\alpha^{\varepsilon,\,\sigma}\coloneqq \alpha^{\varepsilon,\,\sigma}_t \otimes \mathcal{L}^1_t$.
\end{definition}
\begin{remark}\label{rem.1}
	We briefly give the geometric interpretation of the measure $\alpha^{\varepsilon,\,\sigma}_t$. For simplicity, we do not write the index $\sigma$ in the following. Let $\{M_t\}_{t\geq0}$ be a family of smooth  hypersurfaces on $\overline{\Omega} \subset \mathbb{R}^{n}$ with a boundary $\partial M_t \subset \partial \Omega$ and we assume that $M_t$ moves by mean curvature under certain boundary condition. Let us recall that the one-dimensional stationary solution to Allen-Cahn equation
	\begin{equation}\label{1dAllenCahneq}
		(q^{\varepsilon})''(s) - \varepsilon^{-2} W'(q^{\varepsilon}(s)) = 0
	\end{equation}
	with $q^{\varepsilon}(\pm \infty) = \pm 1$, $q'>0$, and $q^{\varepsilon}(0)=0$ has the form of $q^{\varepsilon}(s) = \tanh(\varepsilon^{-1}s)$ and by simple computations we obtain $2^{-1}\varepsilon |(q^{\varepsilon})^{\prime}(s)|^2 = \varepsilon^{-1} W(q^{\varepsilon}(s))$. Note that for any function $q: \mathbb{R} \to \mathbb{R}$ with $q(\pm \infty) = \pm 1$ and $q'>0$, we have
	\begin{align}
		E^{\varepsilon}[q] &\coloneqq \int_{\mathbb{R}}\left( \frac{\varepsilon(q'(s))^2}{2} + \frac{W(q(s))}{\varepsilon} \right)\,ds \nonumber\\
		&= \int_{\mathbb{R}}\frac{1}{2} \left(\sqrt{\varepsilon}|q'(s)| - \frac{\sqrt{2W(q(s))}}{\sqrt{\varepsilon}} \right)^2 \,ds + \int_{\mathbb{R}}\sqrt{2W(q(s))}|q'(s)|\,ds \nonumber\\
		&= \int_{\mathbb{R}}\frac{1}{2} \left(\sqrt{\varepsilon}|q'(s)| - \frac{\sqrt{2W(q(s))}}{\sqrt{\varepsilon}} \right)^2 \,ds + \int_{-1}^{1} \sqrt{2W(\tilde{s})}\,d\tilde{s}
	\end{align}
	where we used the change of variables $\tilde{s} = q(s)$ in the last equality. Therefore, $q^{\varepsilon}$ is in fact a minimizer of $E^{\varepsilon}$ subject to its boundary conditions and the minimum value is equal to $\sigma_0 \coloneqq \int_{-1}^{1}\sqrt{2W(s)}\,ds$. Thus we may have the approximation $\varepsilon |(q^{\varepsilon})^{\prime}(s)|^2 \,\mathcal{L}^1(s) \thickapprox \sigma_0\mathcal{H}^{0}(\{0\})$ if $\varepsilon > 0$ is sufficiently small. Moreover, we can see that the solution $u^{\varepsilon}$ in the general dimensions has the asymptotic form $u^{\varepsilon} \thickapprox q^{\varepsilon}(r^{\varepsilon})$, where $r^{\varepsilon}$ is the signed distance function to the front. Thus, we may expect that  $2^{-1}\varepsilon |\nabla u^{\varepsilon}|^2 \thickapprox \varepsilon^{-1}W(u^{\varepsilon})$ and, moreover, $\varepsilon |\nabla u^{\varepsilon}|^2 \,\mathcal{L}^n \thickapprox \sigma_0\mathcal{H}^{n-1}\lfloor_{M_t}$ if $\varepsilon> 0$ is sufficiently small. Then the measure $\mu^{\varepsilon}_t$ may be regarded as $\mathcal{H}^{n-1}\lfloor_{M_t}$ up to constants. To see this, we refer to, for example, the formal analysis done by Rubinstein, Sternberg, and Keller \cite{RSK} or the rigorous analysis done by Soner in \cite{Soner}. They also show that, as $\varepsilon \to 0$, $\Omega$ separates into two regions $P_t$ and $N_t$ where $u^{\varepsilon}\thickapprox+1$ and $u^{\varepsilon}\thickapprox -1$ respectively and the separating front corresponds to $M_t$. This means that, for sufficiently small $\varepsilon>0$, we may consider the hypersurface $M_t$ as the zero level set of $u^{\varepsilon}(\cdot,\,t)$. Therefore, in the phase field method, evolving surfaces $\{M_t\}_{t\geq 0}$ are approximated by the thin transition layers of an order $\varepsilon$. As an analogue of this, we also have that the transition layer on $\partial \Omega$ which approximates  $\partial M_t$ is supposed to have the width of an order $\varepsilon(\sin\theta)^{-1}$ (see Figure \ref{fig.1}).
	\begin{figure}[H]
		\begin{center}
			\includegraphics[keepaspectratio,scale=0.50,angle=0]{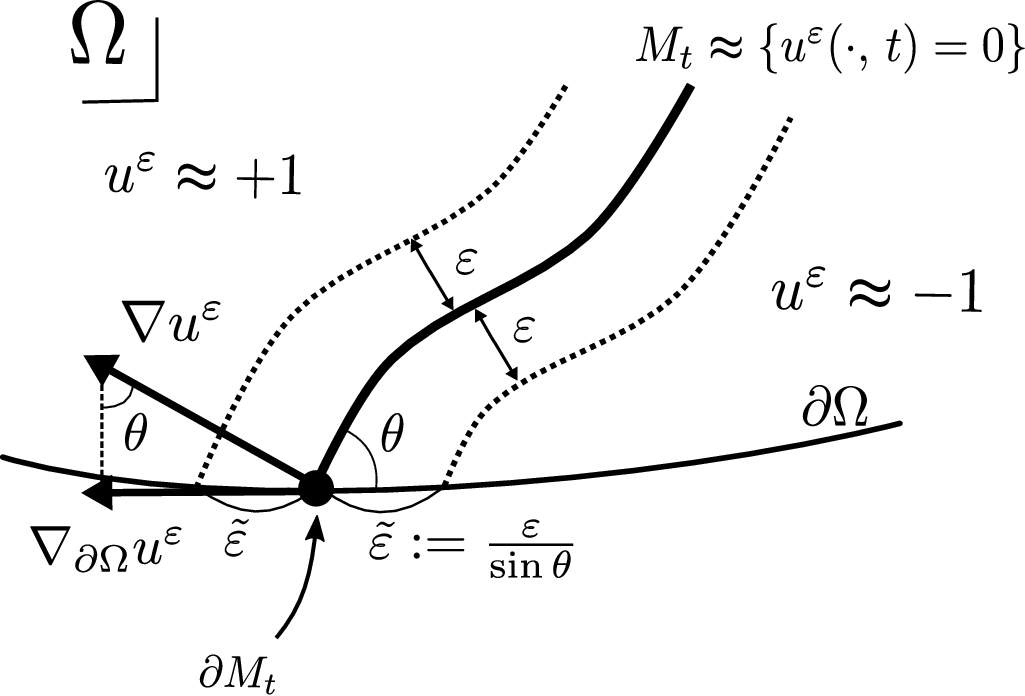}
			\caption{Approximated moving surfaces on $\partial \Omega$ with an order $\tilde{\varepsilon}$}
			\label{fig.1}
		\end{center}
	\end{figure}
	As an analogy of the asymptotic analysis for the measure $\mu^{\varepsilon}_t$, it is reasonable to expect that the measure $\tilde{\varepsilon} |\nabla_{\partial \Omega}u^{\varepsilon}|^2\, \mathcal{H}^{n-1}$ approximates the area measure of $\partial M_t$ on $\partial \Omega$, that is, $\mathcal{H}^{n-2}\lfloor_{\partial M_t\cap \partial \Omega}$, where we set $\tilde{\varepsilon}\coloneqq\varepsilon(\sin\theta)^{-1}$. Moreover, from Figure \ref{fig.1}, we may compute as follows:
	\begin{equation}\label{2.1.13.1}
	\sin \theta(x,\,t)= \frac{|\nabla_{\partial \Omega} u^{\varepsilon}(x,\,t)|}{|\nabla u^{\varepsilon}(x,\,t)|}.
	\end{equation} 
	Therefore, we obtain, as $\varepsilon \to 0$,
	\begin{equation}
	\alpha^{\varepsilon}_t = \varepsilon|\nabla_{\partial \Omega} u^{\varepsilon}|^2 \mathcal{H}^{n-1}\lfloor_{\partial \Omega} = (\sin\theta)\tilde{\varepsilon}|\nabla_{\partial \Omega} u^{\varepsilon}|^2\,\mathcal{H}^{n-1}\lfloor_{\partial \Omega} \thickapprox \sigma_0\,(\sin\theta) \mathcal{H}^{n-2}\lfloor_{\partial M_t \cap \partial \Omega}. \label{2.1.14}
	\end{equation}
\end{remark}

\section{Formulation of Brakke flow}\label{formulation}
In this section, we will give the definition of a Brakke flow with Dirichlet or dynamic boundary conditions in each subsection. Before stating the formulation of a Brakke flow, we will give the assumptions on the initial hypersurface $M_0\subset\overline{\Omega}$ in the case of both Dirichlet and dynamic boundary conditions. Note that this initial condition allows us to have a variety of singularities on the initial hypersurface and to consider a wide range of mean curvature flow such as the flow of grain boundaries. The idea is based on the idea by Ilmanen (see \cite{Ilmanen01}).
\subsection{Initial hypersurface}\label{initialdata}
We choose the initial hypersurface $M_0$ in the following manner; let $E_0$ be an open set in $\mathbb{R}^n$ with $E_0\cap (\mathbb{R}^{n}\setminus \overline{\Omega}) \neq \emptyset$ and $E_0\cap\Omega\neq\emptyset$. Defining $M_0\coloneqq\partial E_0 \cap \Omega(\neq\emptyset)$ with $\partial M_0=\partial E_0\cap \partial \Omega$, we assume that the density bound of $M_0$ and the pair $(E_0,\,M_0)$ can be approximated by smooth family of pairs $\{(E^l_0,\,M_0^l)\}_{l\in\mathbb{N}}$ (see \cite{Ilmanen01}). More precisely, we assume that
\begin{enumerate}
	\item There exists a constant $C>0$ such that, for any $R>0$ and $x\in\overline{\Omega}$,
	\begin{equation}\label{3.1.1.0}
	\frac{\mathcal{H}^{n-1}(M_0\cap B_R(x))}{\omega_{n-1}R^{n-1}}\leq C,
	\end{equation}
	where $\omega_{n-1}$ is the area of $(n-1)$-dimensional sphere.
	\item  There exists a family of pairs $\{(E^l_0,\,M_0^l)\}_{l\in\mathbb{N}}$ such that $E_0^l$ is open, $M_0^l\coloneqq\partial E_0^l$ is a smooth hypersurface and the convergences
	\begin{align}
	\chi_{E^l_0} \xrightarrow[l\to\infty]{} \chi_{E_0}&\quad\text{in $BV(\Omega)$}, \label{3.1.1.1} \\
	\mathcal{H}^{n-1}\lfloor_{M_0^l} \xrightarrow[l\to\infty]{} \mathcal{H}^{n-1}\lfloor_{M_0}&\quad \text{as Radon measures} \label{3.1.1.2}
	\end{align}
	hold.
\end{enumerate}

In the following, we state the formulation of a Brakke flow with Dirichlet or dynamic boundary conditions starting from the initial hypersurface $M_0$ given in Subsection \ref{initialdata}.

\subsection{Dirichlet boundary condition}\label{formudiri}
We now state the definition of a Brakke flow with Dirichlet boundary conditions and then explain our motivation to introduce such a flow. Specifically, we state the reason why our definition is reasonable for a weak formulation of Dirichlet boundary conditions.

Recall that we try to consider a weak solution of the following mean curvature flow with Dirichlet boundary conditions in the sense of Brakke:
\begin{align}
\begin{cases}
\boldv(\cdot,\,t)=\boldH(\cdot,\,t)&\quad \text{on $M_t,\,t>0$}, \\
{\boldv}_b(\cdot,\,t)=0&\quad \text{on $\partial M_t,\,t>0$},
\end{cases}
\end{align}
where ${\boldv}_b$ is the velocity vector of the boundary of a hypersurface $M_t$ on $\partial \Omega$.
\begin{definition}\label{def.4.1}
	Let $\{V_t\}_{t\geq 0}$ be a family of varifolds on $\overline{\Omega}$ with locally bounded first variations and be $(n-1)$-rectifiable for a.e. $t\geq 0$. Let $\alpha\not\equiv0$ and ${\boldv}_b$ be a Radon measure on $\partial \Omega \times [0,\,\infty)$ and a function in $(L^2(\alpha))^n$, respectively. In addition, we take a sequence of solutions $\{u^{\varepsilon,\,\sigma}\}_{\varepsilon,\,\sigma>0}$ to the Allen-Cahn equations \eqref{1.1.1}. Then we say that the quartet $(\{V_t\}_{t\geq 0},\,\alpha,\,{\boldv}_b,\,\{u^{\varepsilon,\,\sigma}\}_{\varepsilon,\,\sigma>0})$ moves along \textit{Brakke flow with Dirichlet boundary conditions} starting from $V_0$ with $\|V_0\|=\sigma_0\mathcal{H}^{n-1}\lfloor_{M_0}$, where $M_0$ is as in Subsection \ref{initialdata}, if the following conditions hold:
	\begin{enumerate}
		\item \textbf{Approximation property on the boundary}. If we define measures $\alpha_t^{\varepsilon,\,\sigma}$ and vector-valued functions $\boldv_b^{\varepsilon,\,\sigma}$ for any $\varepsilon,\,\sigma \in (0,\,1)$ as
		\begin{equation}
			\alpha_t^{\varepsilon,\,\sigma} \coloneqq \varepsilon|\nabla_{\partial \Omega} u^{\varepsilon,\,\sigma}|^2\mathcal{H}^{n-1}\lfloor_{\partial \Omega}, \quad \boldv_b^{\varepsilon,\,\sigma} \coloneqq - \frac{\partial_t u^{\varepsilon,\,\sigma}}{|\nabla_{\partial \Omega} u^{\varepsilon,\,\sigma}|}\frac{\nabla_{\partial \Omega} u^{\varepsilon,\,\sigma}}{|\nabla_{\partial \Omega} u^{\varepsilon,\,\sigma}|},
		\end{equation}
		there exists a subsequence $\{(\varepsilon_j,\sigma_j)\}_{j\in\mathbb{N}}$ such that, for any $\boldg \in (C_c(\partial \Omega \times [0,\,\infty)))^n$,
		\begin{equation}
			\lim_{j \to \infty}\int_{\partial \Omega \times [0,\,\infty)} \boldv_b^{\varepsilon_j,\,\sigma_j} \cdot \boldg \,d\alpha_t^{\varepsilon_j,\,\sigma_j}\,dt = - \mathcal{S}_{\alpha,\,\boldv_b}(\boldg) 
		\end{equation}
		where $\mathcal{S}_{\alpha,\,{\boldv}_b}$ is as in Definition \ref{def3.1.2} (see also Remark \ref{remarkDefBrakkeDiri} below).
		\item \textbf{Absolute continuity with Dirichlet boundary conditions}. There exists the generalized mean curvature vector $\boldH_V^{\Omega}$ in $\Omega\times[0,\,\infty)$ such that $\delta V_t\lfloor_{\Omega}=-\boldH_V^{\Omega}(\cdot,\,t)\|V_t\|$ holds in $\Omega$ for a.e. $t\in[0,\,\infty)$. In addition, the following absolute continuity is valid:
		\begin{equation}\label{4.1.1.1}
		\left\|\mathcal{S}_{\alpha,\,{\boldv}_b}+ \int_{0}^{\infty}\delta V_t\lfloor_{\Omega}\,dt\right\| \ll \|V_t\|\otimes\mathcal{L}^1_t\quad \text{on $\overline{\Omega} \times [0,\,\infty)$}, 
		\end{equation}
		where $\tilde{\mu}\coloneqq\|V_t\|\otimes\mathcal{L}^1_t$. 
		\item \textbf{Modified generalized mean curvature vector}. There exists a vector field $\widetilde{\boldH}_V$ such that
		\begin{equation}\label{4.1.1.4}
		\mathcal{S}_{\alpha,\,{\boldv}_b} + \int_{0}^{\infty}\delta V_t\lfloor_{\Omega}\,dt =-\widetilde{\boldH}_V\|V_t\|\otimes\mathcal{L}^1_t\quad\text{on $\overline{\Omega}\times[0,\,\infty)$},\quad\widetilde{\boldH}_V\lfloor_{\Omega}=\boldH_V^{\Omega}\quad\text{in $\Omega\times[0,\,\infty)$,}
		\end{equation}
		and $\widetilde{\boldH}_V$ belongs to $(L^2(\|V_t\|\otimes \mathcal{L}^1_t,\,\overline{\Omega}\times [0,\,\infty)))^n$. We call $\widetilde{\boldH}_V$ the \textit{modified generalized mean curvature vector} throughout this paper.
		\item \textbf{Brakke's inequality}. The inequality
		\begin{equation}\label{4.1.1.5}
		\int_{\overline{\Omega}} \phi \,d\|V_t\| \Big|_{t=t_1}^{t_2} \leq \int_{t_1}^{t_2}\int_{\overline{\Omega}} (-\phi|\widetilde{\boldH}_V|^2+(\nabla \phi \cdot \widetilde{\boldH}_V)+ \partial_t \phi)\,d\|V_t\|\,dt
		\end{equation}
		holds for any $\phi \in C^1(\overline{\Omega}\times [0,\infty))$ such that $\phi\geq 0$ and $\phi(\cdot,\,t)\in C^1_c(\overline{\Omega})$, and for any $0<t_1\leq t_2<\infty$.
	\end{enumerate}
\end{definition}

\begin{remark}\label{remarkDefBrakkeDiri}
	Compared to the definition of Brakke flow with dynamic boundary conditions given in the next subsection (see Section \ref{formudyna}) or Neumann boundary conditions given by Mizuno and Tonegawa in \cite{MiTo}, the definition of Brakke flow with Dirichlet boundary conditions requires us to have a sequence of solutions to Allen-Cahn equations which, roughly speaking, approximates the measure $\alpha$ and the function $\boldv_b$. Without adding the sequence to our definition of Brakke flow, we might have the possibility that the classical mean curvature flow with Neumann boundary conditions can be also our Brakke flow with Dirichlet boundary conditions. Indeed, let $\Omega$ be a upper half-space in $\mathbb{R}^2$ and $\{M_t\}_{t\geq0}$ denote a family of curves described by $\{(x_1(t),\,x_2(t)) \in \mathbb{R}^2 \mid x_2(t) \geq 0, \, x_1^2(t)+x_2^2(t)= r^2(t)\}$ where $r:[0,\,1/2) \to [0,\,\infty)$ is a smooth function with $r(0)=1$. Let $a(t) \coloneqq (r(t),\,0)$ be one of the points at which $M_t$ contacts $\partial \Omega$. Then, we let the curves flow by their curvatures and the curves will shrink to the origin with the contact angle between $M_t$ and $\partial \Omega$ always equal to $\pi/2$. Since the curvature of $M_t$ is equal to $-r(t)^{-1}$, we easily obtain $r(t) = \sqrt{1-2t}$ and thus the norm of the velocity vector $\boldv_b(t)$ of $M_t$ at $\pm a(t)$ is equal to $r'(t) = -r^{-1}(t) = -(1-2t)^{-\frac{1}{2}} \not\equiv 0$. From this, we have the classical curvature flow $\{M_t\}_{t\geq0}$ with Neumann boundary conditions where the velocity of $M_t$ at $\{\pm a(t)\}$ is not zero. However, when we set the triplet $(\{V_t\}_t,\,\alpha,\,\boldv_b)$ as $(\{M_t\}_t,\,\mathcal{H}^0\lfloor_{\{\pm a(t)\}},\,0)$ from the above, this triplet satisfies three conditions from 2 to 4 in Definition \ref{def.4.1}. This implies that the classical mean curvature flow with Neumann boundary conditions could be also our Brakke flow with Dirichlet boundary conditions without ``\textbf{Approximation property on the boundary}".  
\end{remark}

\begin{remark}\label{rem.2}
	We first remark that the existence of the modified generalized mean curvature vector $\widetilde{\boldH}_V$ can be obtained from \eqref{4.1.1.1}, however, the important thing is that $\widetilde{\boldH}_V$ needs to belong to $(L^2(\|V_t\|\otimes\mathcal{L}^1_t,\,\overline{\Omega}\times[0,\,\infty))) ^n$ and it coincides with the generalized mean curvature vector $\boldH_V^{\Omega}$ in the interior of the domain. Secondly, the absolute continuity \eqref{4.1.1.1} is a stronger condition than \eqref{1.1.5.1}. Thus, in order to see how the boundary conditions look like, it is sufficient to consider the condition \eqref{4.1.1.1} whose domain is restricted to $\partial\Omega\times[0,\,\infty)$. 
\end{remark}
\begin{remark}\label{rem.2.1}
	Now we show that the absolutely continuity in \eqref{4.1.1.1} corresponds to a formulation of Dirichlet boundary conditions in measure theoretic sense if we focus on $\partial\Omega\times[0,\,\infty)$. Indeed, let $M_t$ be a smooth hypersurface corresponding to a varifold $V_t$ for all $t\geq0$. Assume that $M_t$ moves along the motion describing in \eqref{4.1.1.1} and \eqref{4.1.1.5}. Moreover, we impose the following assumptions:
	\begin{description}
		\item[(A1)] $\|V_t\|=\mathcal{H}^{n-1}\lfloor_{M_t}$ on $\overline{\Omega}$ and  $\alpha=(\sin \theta)(\mathcal{H}^{n-2}\lfloor_{\partial M_t} \otimes \mathcal{L}^1_t )$ on $\partial \Omega \times [0,\,\infty)$.
		\item[(A2)] $\partial M_t \subset \partial \Omega$ for all $t\geq0$.
		\item[(A3)] $\mathcal{H}^{n-1}(\partial \Omega \cap M_t)=0$ for all $t\geq0$, meaning that the geometric interior of $M_t$ is not on $\partial \Omega$ for all $t\geq0$.
	\end{description}
	Here $\theta \in [0,\,\pi]$ in (A1) is the contact angle formed by the unit normal vector $N_b$ of $\partial M_t$ on $\partial \Omega$ and $-\boldgam$, where $\boldgam$ is the outer unit normal vector of $\partial M_t$ on $\partial \Omega$ and is tangential to $M_t$ and thus we have $\cos\theta= -\boldgam\cdot N_b$ (see Figure \ref{fig.2}). The assumption (A1) comes from Remark \ref{rem.1} saying that the measure $\alpha^{\varepsilon,\,\sigma}$ defined in Definition \ref{def.2.3} should be considered as $(\sin \theta)(\mathcal{H}^{n-2}\lfloor_{\partial M_t} \otimes \mathcal{L}^1_t )$ as $\varepsilon \to 0$.
	\begin{figure}[H]
		\begin{center}
			\includegraphics[keepaspectratio,scale=0.50,angle=0]{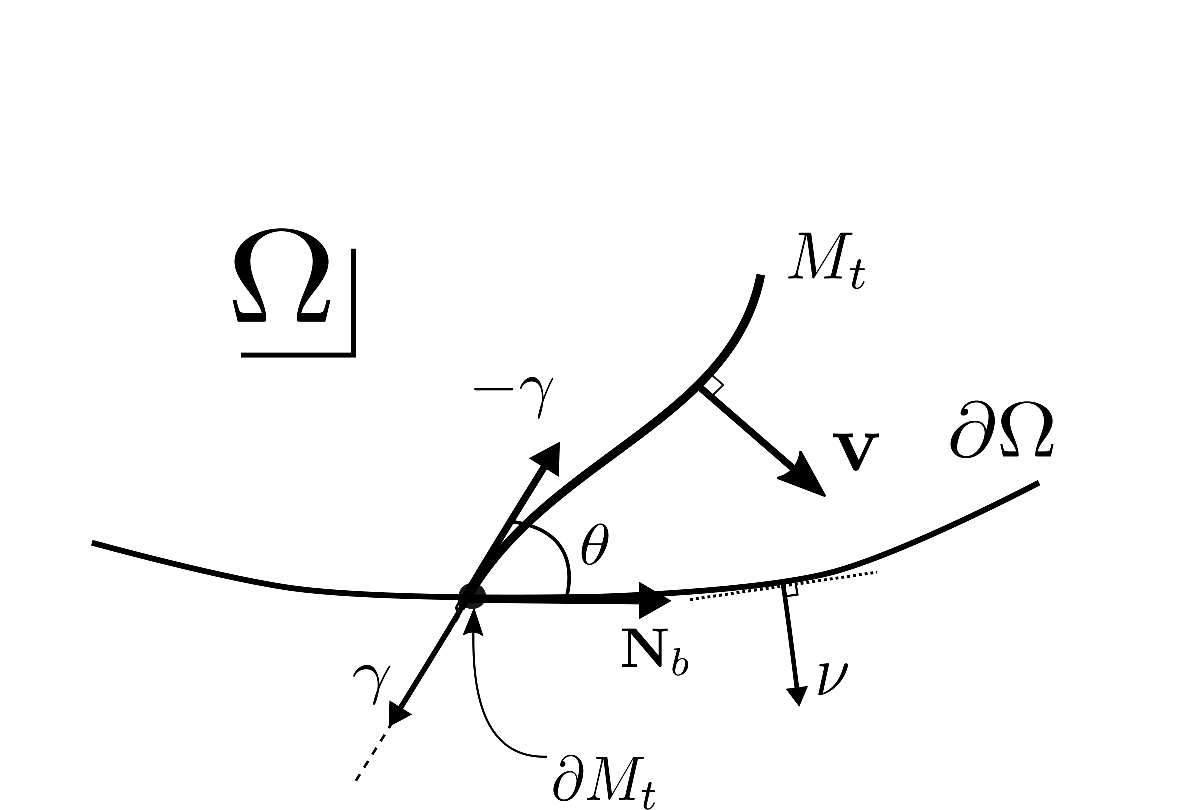}
			\caption{Definition of $\theta$}
			\label{fig.2}
		\end{center}
	\end{figure}
	Since (A1) and (A3) imply $(\|V_t\|\otimes \mathcal{L}^1_t)(\partial \Omega \times [0,\,\infty))=0$, we have from the absolute continuity \eqref{4.1.1.1}
	\begin{equation}\label{3.1.20}
	\left\|\mathcal{S}_{\alpha,\,{\boldv}_b}+\int_{0}^{\infty}\delta V_t\lfloor_{\Omega}\,dt\right\|(\partial \Omega\times [0,\,\infty))=0.
	\end{equation}
	Thus, for any $\boldg\in (C_c(\partial \Omega\times [0,\infty)))^n$ with $|\boldg|\leq1$, from the fact that $\spt\|\int_{0}^{\infty}\delta V_t\lfloor_{\Omega}\,dt\|\subset\Omega\times[0,\,\infty)$, we have
	\begin{align}
	0 = \mathcal{S}_{\alpha,\,{\boldv}_b}(\boldg) = \int_{0}^{\infty}\int_ {\partial \Omega} \boldg\cdot {\boldv}_b\, d\alpha= \int_{0}^{\infty}\int_{\partial M_t} \boldg\cdot {\boldv}_b \,\sin\theta\,d\mathcal{H}^{n-2}\,dt. \label{3.1.21}
	\end{align}
	Moreover, if $\theta(x,\,t)$ is not identically equal to zero and $\pi$, then $\sin\theta(x,\,t)$ is not identically equal to zero on $\partial \Omega$ for $t \in [0,\,\infty)$. Hence we obtain ${\boldv}_b = 0$ on $\partial \Omega$ for a.e. $t\in[0,\,\infty)$. From the proof of Lemma \ref{thm3.2} (see also the next section \ref{mainresults.diri}), we may regard ${\boldv}_b$ as the generalized velocity vector of $\partial M_t$ on $\partial \Omega$ and thus we can say that this implies Dirichlet boundary conditions.
\end{remark}

\subsection{Dynamic boundary condition}\label{formudyna}
We now state the definition of a Brakke flow with dynamic boundary conditions and then we show the reason why our boundary condition is reasonable for the weak formulation of dynamic boundary conditions. In the sequel, we assume that $\sigma \in (0,\,\infty)$.

Recall that we try to consider a weak solution of the following mean curvature flow with dynamic boundary conditions in the sense of Brakke:
\begin{align}\label{3.2.0.0}
\begin{cases}
\displaystyle
\boldv(\cdot,\,t)=\boldH(\cdot,\,t)&\quad \text{on $M_t,\,t>0$},\\
\displaystyle {\boldv}_b(\cdot,\,t)=\sigma(\tan\theta(\cdot,\,t))^{-1}\boldN_b(\cdot,\,t)&\quad \text{on $\partial M_t,\,t>0$},
\end{cases}
\end{align}
where $\boldN_b$ is the unit normal vector of $\partial M_t$ on $\partial \Omega$, and $\theta$ is the contact angle formed by a hypersurface $M_t$ and $\boldN_b$ on $\partial \Omega$.
\begin{definition}\label{def.4.2}
	Let $\{V_t\}_{t\geq 0}$ be a family of varifolds on $\overline{\Omega}$ with locally bounded first variations and be $(n-1)$-rectifiable for a.e. $t\geq 0$. Let $\alpha\not\equiv0$ and ${\boldv}_b$ be a Radon measure on $\partial \Omega \times [0,\,\infty)$ and a function in $(L^2(\alpha))^n$, respectively. Then we say that the triplet $(\{V_t\}_{t\geq 0},\,\alpha,\,{\boldv}_b)$ moves along a \textit{Brakke flow with dynamic boundary conditions} starting from $V_0$ with $\|V_0\| = \sigma_0 \mathcal{H}^{n-1}\lfloor_{M_0}$, where $M_0$ is as in Subsection \ref{initialdata}, if the following conditions hold:
	\begin{enumerate}
		\item \textbf{Absolute continuity with dynamic boundary condition}. The following absolute continuity holds.
		\begin{equation}\label{4.1.1.6}
		\left\|\int_{0}^{\infty}\delta V_t\lfloor_{\Omega}\,dt+\int_{0}^{\infty}\delta V_t\lfloor_{\partial \Omega}^T\,dt+  \sigma^{-1}\mathcal{S}_{\alpha,\,{\boldv}_b}\right\| \ll \|V_t\|\otimes\mathcal{L}^1_t\quad \text{on $\overline{\Omega}\times [0,\,\infty)$},
		\end{equation}
		where $\delta V_t\lfloor_{\partial\Omega}^T$ and $\mathcal{S}_{\alpha,\,{\boldv}_b}$ are as in Definition \ref{def.3.3} and Definition \ref{def3.1.2}, respectively.
		\item \textbf{Modified generalized mean curvature vector}. There exists a vector field $\widetilde{\boldH}_V$ such that
		\begin{equation}\label{4.1.1.9}
		\int_{0}^{\infty}\delta V_t\lfloor_{\Omega}\,dt+\int_{0}^{\infty}\delta V_t\lfloor_{\partial \Omega}^T\,dt+ \sigma^{-1} \mathcal{S}_{\alpha,\,{\boldv}_b}=-\widetilde{\boldH}_V\|V_t\|\otimes\mathcal{L}^1_t\quad\text{on $\overline{\Omega}\times[0,\,\infty)$}
		\end{equation}
		holds and $\widetilde{\boldH}_V$ belongs to $(L^2(\|V_t\|\otimes \mathcal{L}^1_t,\,\overline{\Omega}\times [0,\,\infty)))^n$. We call $\widetilde{\boldH}_V$ the \textit{modified generalized mean curvature vector} throughout this paper.
		\item \textbf{Brakke's inequality}. The inequality
		\begin{align}\label{4.1.1.10}
		\int_{\overline{\Omega}} \phi\, d\|V_t\| \Big|_{t=t_1}^{t_2}  \leq \int_{t_1} ^{t_2} \int_{\overline{\Omega}} \left(-\phi |\widetilde{\boldH}_V|^2 + (\nabla \phi \cdot \widetilde{\boldH}_V ) + \partial_t \phi \right) \, d\|V_t\|\,dt-\frac{1}{\sigma}\int_{\partial \Omega \times [t_1,\,t_2]} \phi |{\boldv}_b|^2 \, d\alpha 
		\end{align}
		holds for any $\phi \in C^1(\overline{\Omega}\times [0,\infty))$ such that $\phi\geq 0$ and $\nabla \phi(\cdot,\,t) \cdot \boldnu=0$ on $\partial \Omega$ and $\phi(\cdot,\,t)\in C^1_c(\overline{\Omega})$, and for any $0<t_1\leq t_2 <\infty$.
	\end{enumerate}
\end{definition}

\begin{remark}\label{rem.3.0}
	First of all, we remark that the existence of the modified generalized mean curvature vector $\widetilde{\boldH}_V$ can be obtained from \eqref{4.1.1.6}, however, the important thing is that $\widetilde{\boldH}_V$ needs to belong to $(L^2(\|V_t\|\otimes\mathcal{L}^1_t,\,\overline{\Omega}\times[0,\,\infty)))^n$. Secondly, we can also define, from \eqref{4.1.1.6}, the generalized mean curvature vector $\boldH_V^{\Omega}$ defined in Remark \ref{rem.0} such that 
	\begin{equation}\label{4.1.2.1}
	\delta V_t\lfloor_{\Omega}=-\boldH_V^{\Omega}(\cdot,\,t)\,\|V_t\|\quad\text{in $\Omega$ for a.e. $t\in[0,\,\infty)$.}
	\end{equation}
	In addition, if we restrict the domain of the modified generalized mean curvature vector $\widetilde{\boldH}_V$ in \eqref{4.1.1.9} into $\Omega\times[0,\,\infty)$ (denoted by $\widetilde{\boldH}_V\lfloor_{\Omega}$), then we can show that $\widetilde{\boldH}_V\lfloor_{\Omega}$ coincides with $\boldH_V^{\Omega}$ in $\Omega\times[0,\,\infty)$. Thirdly, as we also mentioned in the case of Dirichlet boundary conditions, the absolutely continuity \eqref{4.1.1.6} is a stronger condition than \eqref{1.1.5.2}. Thus, in order to see how the boundary conditions look like, it is sufficient to consider the condition \eqref{4.1.1.6} whose domain is restricted to $\partial\Omega\times[0,\,\infty)$.
\end{remark}

\begin{remark}\label{rem.3}
	Now we show that the absolute continuity \eqref{4.1.1.6} corresponds to a formulation of dynamic boundary conditions in the measure theoretic sense if we focus on $\partial\Omega\times[0,\,\infty)$. Indeed, let $M_t$ be a smooth hypersurface corresponding to a varifold $V_t$ for all $t\geq0$ and evolve by the motion described in \eqref{4.1.1.6} and \eqref{4.1.1.10}. Note that we do not write the index $\sigma$ for simplicity. Furthermore, we impose the same assumptions (A1), (A2) and (A3) in Remark \ref{rem.2.1} (see also Figure \ref{fig.2}). Since (A1) and (A3) implies $\|V_t\|\otimes\mathcal{L}^1_t(\partial \Omega \times[0,\,\infty))=0$, from the absolute continuity \eqref{4.1.1.6}, we have
	\begin{equation}\label{3.2.12}
	\left\| \int_{0}^{\infty}\delta V_t\lfloor_{\partial \Omega}^T\,dt+\int_{0}^{\infty}\delta V_t\lfloor_{\Omega}\,dt+ \sigma^{-1}\mathcal{S}_{\alpha,\,{\boldv}_b} \right\|(\partial \Omega \times [0,\,\infty))=0.
	\end{equation}
	Then, for any $\boldg\in (C_c(\partial \Omega\times [0,\infty)))^n$ with $|\boldg|\leq1$, we have, from the divergence theorem, (A1), (A3) and \eqref{3.2.12},
	\begin{align}
	0 &= \int_{0}^{\infty}\delta V_t \lfloor_{\partial \Omega}^T(\boldg)\,dt + \mathcal{S}_{\alpha,\,{\boldv}_b}(\boldg) \nonumber \\
	&=\int_{0}^{\infty}\int_{\partial M_t\cap \partial \Omega}(\boldg-(\boldg\cdot \boldnu)\boldnu)\cdot \boldgam \, d\mathcal{H}^{n-2}\,dt+ \sigma^{-1}\int_{0}^{\infty}\int_ {\partial \Omega} \boldg\cdot {\boldv}_b\, d\alpha, \label{3.2.14}
	\end{align}
	where $\boldnu$ is the outer unit normal vector of $\partial \Omega$. From the relation $\cos\theta=-\boldgam\cdot N_b$, we may deduce $-\boldgam^{T}\cdot N_b= -(\boldgam-(\boldgam\cdot\boldnu)\boldnu) \cdot N_b=-\boldgam \cdot N_b=\cos \theta$, where $\boldgam^T$ is the orthogonal projection of $\boldgam$ onto $\partial \Omega$. Hence, from $(A1)$ and $(A2)$, we obtain
	\begin{align}
	0 &= \int_{0}^{\infty}\int_{\partial M_t} \boldg\cdot (\boldgam-(\boldgam \cdot \boldnu)\boldnu)\, d\mathcal{H}^{n-2}\,dt + \sigma^{-1} \int_{0}^{\infty}\int_{\partial M_t} \boldg\cdot {\boldv}_b(\sin \theta)\,d\mathcal{H}^{n-2}\,dt \nonumber\\
	&= \int_{0}^{\infty}\int_{\partial M_t} \boldg\cdot (\boldgam^T+ \sigma^{-1}{\boldv}_b(\sin \theta)) \, d\mathcal{H}^{n-2}\,dt \label{3.2.15}
	\end{align}
	for all $\boldg\in (C^1_c(\overline{\Omega}\times [0,\infty)))^n$. This implies that $\boldgam^T+ \sigma^{-1}{\boldv}_b(\sin \theta)= 0$ on $\partial M_t$ for all $t\geq0$. Here it holds that $\sin\theta$ is not equal to 0, otherwise we have $\gamma^T=0$ and this implies $\theta=\frac{\pi}{2}$ which contradicts $\sin\theta=0$. Therefore we obtain
	\begin{equation}\label{3.2.16}
	{\boldv}_b \cdot \boldN_b = \left(- \frac{\boldgam^T}{\sigma^{-1}\sin \theta} \right) \cdot \boldN_b = \sigma\,\frac{\cos \theta}{\sin \theta}= \frac{\sigma}{\tan \theta}
	\end{equation}
	on $\partial M_t$ and we can say that this implies dynamic boundary condition.
\end{remark}

\begin{remark}\label{rem.4}
	Now we briefly refer to the formulation of right-angle Neumann boundary conditions, comparing to our formulation. We recall that the boundary condition of Allen-Cahn equations \eqref{1.1.1} corresponds to dynamic and right-angle Neumann boundary conditions in the case that $\sigma$ is in $(0,\,\infty)$ and $\sigma=\infty$, respectively. In Subsection \ref{mainresults.dyna}, we give the conditions for a Brakke flow with dynamic boundary conditions as the absolute continuity \eqref{4.1.1.6} and Brakke's inequality \eqref{4.1.1.10} with a parameter $\sigma$. Hence, if we formally substitute $\sigma=\infty$ for \eqref{4.1.1.6} and \eqref{4.1.1.10}, then we have
	\begin{equation}\label{4.1.1.11}
	\left\|\int_{0}^{\infty}\delta V_t\lfloor_{\partial \Omega}^T\,dt+ \int_{0}^{\infty}\delta V_t\lfloor_{\Omega}\,dt\right\| \ll \|V_t\|\otimes\mathcal{L}^1_t\quad \text{on $\overline{\Omega}\times [0,\,\infty)$},
	\end{equation}
	and 
	\begin{equation}\label{4.1.1.12}
	\int_{\overline{\Omega}} \phi\, d\|V_t\| \Big|_{t=t_1}^{t_2} \leq \int_{t_1} ^{t_2} \int_{\overline{\Omega}} \left(-\phi |\widetilde{\boldH}_V|^2 + (\nabla \phi \cdot \widetilde{\boldH}_V ) + \partial_t \phi \right) \, d\|V_t\|\,dt
	\end{equation}
	for all $\phi \in C^1(\overline{\Omega}\times [0,\,\infty))$ with some conditions. Actually, these results essentially corresponds to a slightly weaker version of the solutions studied by Mizuno and Tonegawa in \cite{MiTo}. We refer to \cite{MiTo} for more detail.
\end{remark}

\section{Existence results of sharp interface limit}\label{existence}
Now we state the results of the sharp interface limits of our Brakke flow with Dirichlet or dynamic boundary conditions which we defined in the previous section. We emphasize that we applied the phase field method to show the approximation of our Brakke flow although one may obtain the similar results with ours by applying another method. In each subsection, we first state several assumptions and then we give a sequence of main lemmas and the main theorem.

\subsection{Dirichlet boundary condition}\label{exisdiri}
All the conditions described in ``General assumptions" in the following are mainly based on Ilmanen's work in \cite{Ilmanen01}. Remark that we may be able to weaken these assumptions since, in the case of $\Omega=\mathbb{R}^n$, Soner \cite{Soner} later removed the technical assumptions imposed by Ilmanen in \cite{Ilmanen01}.

\subsubsection{Assumptions, hypothesis and example}\label{exis.assmp.diri}
In this subsection, we will state three important assumptions in our study which consists of ``\textit{General assumptions}", ``\textit{Vanishing hypothesis for the discrepancy measure}", and ``\textit{Uniform uppr bound for the solution of Allen-Cahn equations}". Moreover, we will state one example which gives us the validity to impose the assumption ``Uniform upper bound on the boundary $\partial \Omega$".

\begin{itemize}
	\item \textbf{General assumptions.}\label{generalassmp}\\
	\quad Suppose that $n\geq2$, $\Omega$ is a bounded domain with smooth boundaries and we define the potential function $W: \mathbb{R} \to \mathbb{R}$ by $W(s)=\frac{1}{2}(1-s^2)^2$. Note that $W$ is said to be a double-well potential and we may also apply the generalized $W$, which is defined in, for instance, \cite{MiTo} or \cite{Kagaya}.
	
	\quad Next we give the assumptions on the initial data of the solutions of the Allen-Cahn equations \eqref{1.1.1}. We assume that there exists a subsequence $\{u^{\varepsilon_l,\,\sigma_l}_0\}_{l\in\mathbb{N}}$ such that
	\begin{align}
	u^{\varepsilon_l,\,\sigma_l}_0 \xrightarrow[l \rightarrow \infty]{}  2\chi_{E_0\cap \Omega}-1 & \quad \text{in  $BV(\Omega)$},\label{3.1.1}\\
	\mu^{\varepsilon_l,\,\sigma_l}_0 \xrightarrow[l \rightarrow \infty]{} \sigma_0\,\mathcal{H}^{n-1}\lfloor_{\partial E_0} & \quad  \text{as Radon measures},\label{3.1.2}
	\end{align}
	where $E_0$ is as in Subsection \ref{initialdata}, $\sigma_0\coloneqq\int_{-1}^{1}\sqrt{2W(u)}\,du = 4/3$, and $\mu^{\varepsilon,\,\sigma}_0$ is defined by
	\begin{equation}\label{3.1.2.0}
	\mu^{\varepsilon,\,\sigma}_0\coloneqq \left(\frac{\varepsilon|\nabla u^{\varepsilon,\,\sigma}_0|^2}{2}+\frac{W(u^{\varepsilon,\,\sigma}_0)}{\varepsilon}\right)\,\mathcal{L}^n.
	\end{equation}
	Now we suppose that the initial data $\{u^{\varepsilon,\,\sigma}_0\}_{\varepsilon,\,\sigma>0}$ satisfy
	\begin{equation}\label{3.1.7}
	\sup_{\Omega}|u^{\varepsilon,\,\sigma}_0| \leq 1,
	\end{equation}
	for any $\varepsilon,\,\sigma>0$ and there exists $D>0$ such that 
	\begin{equation}\label{3.1.6}
	\sup_{\varepsilon>0,\,\sigma>0} E^{\varepsilon,\,\sigma}[u^{\varepsilon,\,\sigma}_0]= \sup_{\varepsilon>0,\,\sigma>0} \mu^{\varepsilon,\,\sigma}_0(\Omega) \leq D,
	\end{equation}
	and this indicates that the surface area of the initial hypersurface cannot blow up as $\varepsilon \downarrow 0$ or $\sigma \downarrow 0$. In addition, from \eqref{3.1.7}, we can show that
	\begin{equation}\label{3.1.8}
	\sup_{\Omega\times [0,\,\infty)}|u^{\varepsilon,\,\sigma}| \leq 1.
	\end{equation}
	for any $\varepsilon,\,\sigma>0$ (see Appendix A in Section \ref{appendix} for the proof).
	
	\quad Next we observe the solutions to the equations \eqref{1.1.1} with $\sigma \in (0,\,1)$. Since our aim in the present paper is to study a formulation of the singular limit of the Allen-Cahn equations, we assume that the desired regularity of the global-in-time solution to \eqref{1.1.1} is obtained. Regarding the existence and regularity of these kinds of equations, we refer to, for instance, a paper by Escher \cite{Escher}. Hence, in this paper, we may assume that a solution $u^{\varepsilon,\,\sigma}$ to \eqref{1.1.1} exists, is not a constant function and it holds that
	\begin{equation}
		u^{\varepsilon,\,\sigma} \in L^2_{loc}\left([0,\,\infty);\,H^{2}(\Omega)\right)\cap L^2_{loc}\left([0,\,\infty);\,H^{1}(\partial \Omega,\,\mathcal{H}^{n-1})\right)
	\end{equation}
	and
	\begin{equation}
		\partial_t u^{\varepsilon,\,\sigma} \in L^2_{loc}\left([0,\,\infty);\,L^2(\Omega)\right)\cap L^2_{loc}\left([0,\,\infty);\,L^2(\partial \Omega,\,\mathcal{H}^{n-1})\right).
	\end{equation}
	As far as we know on the existence of the classical solutions, Guidetti in \cite{Guidetti01} proved the locally-in-time existence and uniqueness of the classical solutions to the parabolic equations under a compatibility condition. The equations the author studied is considered as generalized equations of \eqref{1.1.1} (see also \cite{Guidetti02}). 
	
	\item \textbf{Vanishing hypothesis of the discrepancy measure.}\label{hypothesis}
	
	\quad Let $\xi^{\varepsilon,\,\sigma}(x,\,t)$ be the discrepancy function defined in \eqref{discrepancyFunc} for any $(x,\,t) \in \Omega \times (0,\,\infty)$. Recall that the measure $\xi^{\varepsilon,\,\sigma}_t = \xi^{\varepsilon,\,\sigma}(x,\,t)\mathcal{L}^n(x)$ for each $t>0$ is called the \textit{discrepancy measure}. Then, in our study, we assume the following two conditions: 
	\begin{itemize}
		\item there exists a subsequence  $\{(\varepsilon_j,\,\sigma_j)\}_{j\in\mathbb{N}}$ with $\lim_{j \to \infty}\varepsilon_j = \lim_{j\to\infty}\sigma_j = 0$ such that
		\begin{equation}\label{vanishingDisDirichlet}
			|\xi^{\varepsilon_j,\,\sigma_j}_t| \xrightarrow[j\to\infty]{} 0 \quad \text{on } \overline{\Omega} \quad \text{as Radonn measures}
		\end{equation}
		for a.e. $t\in[0,\,\infty)$.
		\item there exists a constant $c_0>0$ such that
		\begin{equation}\label{boundednessDisDirichlet}
			\sup_{x\in\Omega}\xi^{\varepsilon,\,\sigma}(x,\,t) \leq c_0
		\end{equation}
		for any $\varepsilon>0$, $\sigma>0$, and $t>0$.
	\end{itemize}
	Throughout this paper, we call the first assumption the \textit{vanishing} of the discrepancy measure.
	
	\quad Let us remark the second assumption. The boundedness of the discrepancy measure will be applied only to obtain the integrality of the limit measure of $\{\sigma_0^{-1}\mu^{\varepsilon,\,\sigma}_t\}_{\varepsilon>0}$ in the interior of the domain (see Proposition \ref{prop.5.6} in Seciton \ref{chara.lim} for more detail). However we emphasize that, as it is shown in \cite{Ilmanen01, MiTo}, we can actually prove the boundedness of the discrepancy measure $\xi^{\varepsilon,\,\sigma}_t$ under more conditions than we assume in this section. Namely, we obtain that if $\Omega$ is convex, some gradient estimate on the boundary $\partial \Omega$ (see \eqref{assumptionOnBoundary} in Proposition \ref{appendAProp2} for this estimate), and the boundedness of the discrepancy measure at the initial time, then we have \eqref{boundednessDisDirichlet} for some constant $c_0>0$. For the proof, we refer to Proposition \ref{appendAProp2} in Section \ref{appendix}.
	
	\quad Let us give more comments on the vanishing of the discrepancy measure. In the case of $\Omega=\mathbb{R}^n$, Ilmanen proved in \cite{Ilmanen01} the vanishing of the discrepancy measure in $\Omega$ by showing the non-positivity of the discrepancy measure and constructing the monotonicity formula for the measure $\mu^{\varepsilon}_t$ in \eqref{1.1.5}. Moreover, in the case that $n=2,\,3$ and $\Omega$ is an open subset, R\"oger and Sch\"atzle proved in \cite{RoSc} the vanishing of the discrepancy measure in the interior $\Omega$ via the estimates of non-negative part of the discrepancy measure for elliptic Allen-Cahn equations in the context of De Giorgi conjecture. In higher dimensions, we might be able to show the vanishing of the discrepancy measure in the interior $\Omega$ if we apply the monotonicity formula studied by, for instance, Ilmanen in \cite{Ilmanen01} and Takasao and Tonegawa in \cite[Proposition 4.1]{TaTo}. We should remark that, if we employ the monotonicity formula studied in \cite{Ilmanen01} or \cite{TaTo}, we need to have that formula localized by using proper cut-off functions, in order to consider the effects only from the interior $\Omega$. On the other hand, to obtain the vanishing of the discrepancy measure up to the boundary $\partial\Omega$, we may need to construct the monotonicity formula for $\mu^{\varepsilon,\,\sigma}_t$, taking into account the effects from the boundary $\partial\Omega$. Indeed, in the case of right-angle Neumann boundary conditions, Mizuno and Tonegawa in \cite{MiTo} proved the vanishing of the discrepancy measure up to the boundary by constructing the monotonicity formula when $\Omega$ is strictly convex. They employ the reflection argument with respect to the boundary to obtain that formula. After that, Kagaya in \cite{Kagaya} extended the result by Mizuno and Tonegawa \cite{MiTo} into the case that $\Omega$ is not necessarily convex.
	
	\quad Finally, we remark that, even in our case, it is possible to show the vanishing of the discrepancy measure only in the interior $\Omega$, namely, we can show that, for given $T>0$, there exist subsequences $\{\varepsilon_j\}_{j}$ and $\{\sigma_j\}_{j}$ such that
	\begin{equation}
		|\xi^{\varepsilon_j,\,\sigma_j}_t|\otimes \mathcal{L}^1_t \xrightarrow[j\to\infty]{} 0 \quad \text{ in } \Omega \times (0,\,T) \quad \text{as Radon measures}.
	\end{equation}
	One could prove this statement by using \eqref{boundednessDisDirichlet} and combining the argument in \cite{Ilmanen01} and \cite{MiTo} with proper cut-off functions. Indeed, for instance, the authors in \cite{MiTo} constructed the monotonicity formula up to the boundary, which is a key estimate to prove the vanishing of discrepancy measures, by using the reflection argument with respect to the boundary. However, if we use proper cut-off functions whose supports lie in the interior, then there is no need to deal with the effects from the boundary. Thus we could employ the same argument shown as in \cite{MiTo} to prove the vanishing of the discrepancy measure only in the interior of the domain.
	
	\item \textbf{Uniform upper bound on the boundary $\partial\Omega$.}\label{uniformupperbound}
	
	\quad We next assume the local-in-time uniform upper bound of the Dirichlet energy of the Allen-Cahn equations \eqref{1.1.1} on $\partial \Omega$ along $\boldnu$. More precisely, we suppose that, for any $0\leq t_1<t_2<\infty$, there exists a constant $C_0(t_1,\,t_2)>0$ such that it holds that
	\begin{equation}\label{3.1.9.2}
	\sup_{\varepsilon,\,\sigma>0} \int_{t_1}^{t_2}\int_{\partial \Omega} \frac{\varepsilon}{2} \left(\frac{\partial u^{\varepsilon,\,\sigma}}{\partial \boldnu}\right)^2 \,d\mathcal{H}^{n-1}\,dt \leq C_0(t_1,\,t_2).
	\end{equation}
	Note that this assumption could not be removed due to our construction of the following example in Remark \ref{examplecurvatureflow} in the sense of classical curvature flow.
	
	\textbf{Construction of curves moving by motion descibed in \eqref{1.1.2}.}
	\begin{remark}\label{examplecurvatureflow}
	We now give a reason why it is reasonable to assume the uniform upper bound on $\partial\Omega$ for $u^{\varepsilon,\,\sigma}$ to show the existence of the singular limit in the case of Dirichlet boundary conditions and study a Brakke flow which we defined in Section \ref{formulation}.
	Generally speaking, it is necessary for us to make it clear which class of solutions to some equation is proper as the definition of the solutions. In our case, we have to clarify what kind of the solutions are suitable for the ones to a Brakke flow with Dirichlet boundary conditions. To see this, we consider the following example; first of all, we assume that $\Omega$ is a half space in $\mathbb{R}^2$. For each $\sigma>0$ sufficiently small, we take an initial curve $M^{\sigma}_0$ as a part of the $(1,\,-\sigma^{-1})$-centered circle with radius $R\coloneqq\sqrt{1+\sigma^{-2}}$ and we set two points $x_0(0)$ and $x_1(0)$ on $\partial \Omega$ as $x_0(0)\coloneqq0$ and $x_1(0)\coloneqq2$ (see Figure \ref{fig.3}).
	\begin{figure}[H]
		\begin{center}
			\includegraphics[keepaspectratio, scale=0.65,angle=0]{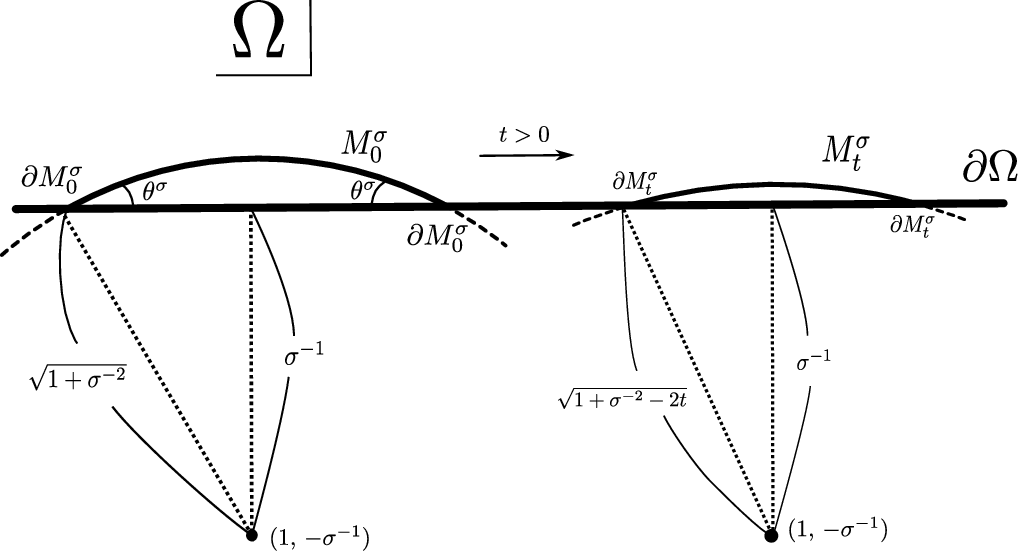}
		\end{center}
		\caption{An initial curve and its motion by its curvature}
		\label{fig.3}
	\end{figure}
	In this setting, we may easily show that the curvature of a curve $M^{\sigma}_t$ is $r^{-1}(t)$ where $r(t)\coloneqq\sqrt{R^2-2t}$ is the radius of a circle, a part of which is $M^{\sigma}_t$. Then we see that a family of curves $\{M^{\sigma}_t\}_{t\geq 0}$ moves by their curvatures and moreover, we can easily show that the boundaries $\{\partial M_t\}_{t\geq 0}$ evolve under dynamical boundary conditions. Indeed, we can calculate the explicit forms of the velocities of $\partial M^{\sigma}_t\coloneqq\{x^{\sigma}_0(t),\,x^{\sigma}_1(t)\}$ as follows: since $M^{\sigma}_t$ is a part of the circle defined by $\{(x,\,y)\mid(x-1)^2+(y+\sigma^{-1})^2=r(t)^2\}$ for each $t$, we can have the explicit forms of $x^{\sigma}_0(t)$ and $x^{\sigma}_1(t)$ by $-\sqrt{r(t)^2-\sigma^{-2}}+1$ and $\sqrt{r(t)^2-\sigma^{-2}}+1$, respectively. Hence we have that, for example, the velocity ${v}^{\sigma}_b$ of $\partial M_t$ at $x^{\sigma}_0(t)$ is 
	\begin{align}
	{v}^{\sigma}_b(x^{\sigma}_0,\,t)&=(x^{\sigma}_0)'(t)= \frac{1}{\sqrt{r^2(t)-\sigma^{-2}}}= \frac{1}{\sqrt{1-2t}}\quad \text{on $\partial \Omega$ for $0\leq t<\frac{1}{2}$}\label{3.1.9.2.1}\\
	{v}^{\sigma}_b(x^{\sigma}_1,\,t)&=-(x^{\sigma}_1)'(t)=-\frac{-1}{\sqrt{r^2(t)-\sigma^{-2}}}=\frac{1}{\sqrt{1-2t}}\quad \text{on $\partial \Omega$ for $0\leq t<\frac{1}{2}$},\label{3.1.9.2.3}
	\end{align}
	and 
	\begin{equation}
	\frac{\sigma}{\tan\theta^{\sigma}}= \sigma \,\frac{\sigma^{-1}}{\sqrt{r^2(t)-\sigma^{-2}}}= \frac{1}{\sqrt{r^2(t)-\sigma^{-2}}}=\frac{1}{\sqrt{1-2t}}\quad \text{on $\partial \Omega \cap \partial M_t$}, \label{3.1.9.2.2}
	\end{equation}
	Thus, we obtain the equality that ${v}^{\sigma}_b=\sigma(\tan\theta)^{-1}$ on $\partial \Omega$ and this is exactly the same boundary condition as \eqref{1.1.2}. From \eqref{3.1.9.2.1} and \eqref{3.1.9.2.3}, we see that the velocities $(x^{\sigma}_0)'(t)$ and $(x^{\sigma}_1)'(t)$ on $\partial M^{\sigma}_t$ are independent of $\sigma$. Hence, if the boundary condition of \eqref{1.1.2} corresponds to Dirichlet boundary conditions as $\sigma \to 0$, then the velocity of $\partial M_t^{\sigma}$ should converge to 0 as $\sigma \to 0$, however, this contradicts the fact that the velocity of $\partial M_t^{\sigma}$ is independent of $\sigma$. Therefore, this implies that, in the above example, the curvature flow with Dirichlet boundary conditions cannot be characterized by the one with boundary conditions \eqref{1.1.2} as $\sigma\to0$. Indeed, the curve $M^0_t$ is actually the interval $[0,\,2]$ for all $t\geq 0$, which does not move for all the time. \\
	\quad Now we focus on how we can interpret this example on the level of the phase field method. To see this, we focus on the Dirichlet energy, one of the characteristic quantities in the phase field method. In the phase field method, a curve $M_t^{\sigma}$ moving by its curvature can be expected to be the zero level set of $u^{\varepsilon,\,\sigma}$ satisfying the equation \eqref{1.1.1} and we may have that $\Omega$ is separated into the region that $u^{\varepsilon,\,\sigma}$ is almost +1 and the region that $u^{\varepsilon,\,\sigma}$ is almost -1 (see Figure \ref{fig.4}).
	\begin{figure}[H]
		\begin{center}
			\includegraphics[keepaspectratio,scale=0.50]{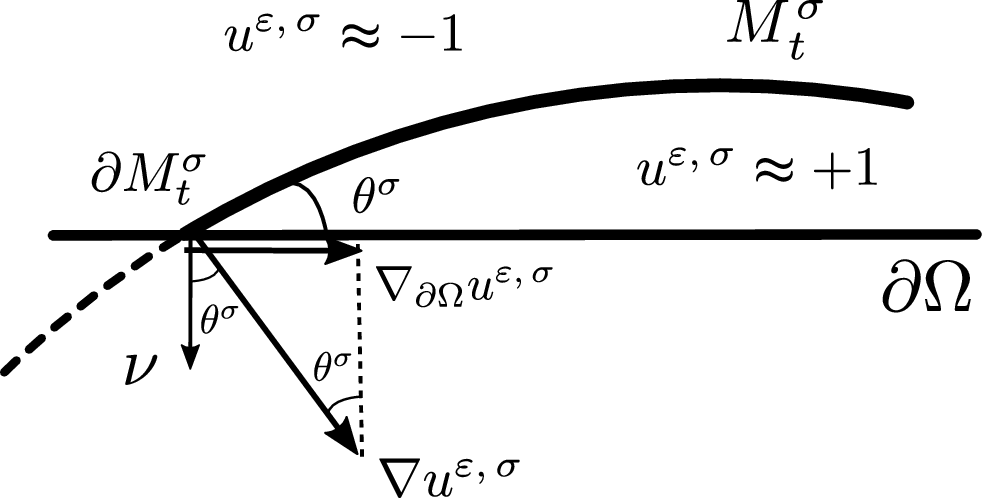}
			\caption{Interpretation in the phase field method}
			\label{fig.4}
		\end{center}
	\end{figure}
	Thus we may calculate the Dirichlet energy on $\partial \Omega$ along $\boldnu$ as follows:
	\begin{align}
	\int_{\partial\Omega}\varepsilon |\nabla u^{\varepsilon,\,\sigma}\cdot\boldnu|^2\,d\mathcal{H}^{1}&= \int_{\partial\Omega} \varepsilon |\nabla u^{\varepsilon,\,\sigma}|^2 \,(\cos\theta^{\sigma})^2\,d\mathcal{H}^{1}\nonumber\\
	&= \int_{\partial\Omega}(\cos\theta^{\sigma})^2\frac{|\nabla u^{\varepsilon,\,\sigma}|^2}{|\nabla_{\partial \Omega}u^{\varepsilon,\,\sigma}|^2} \,\varepsilon|\nabla_{\partial \Omega}u^{\varepsilon,\,\sigma}|^2\,d\mathcal{H}^{1}\nonumber\\
	&= \int_{\partial\Omega} \frac{(\cos\theta^{\sigma})^2}{\sin\theta^{\sigma}}\, \frac{\varepsilon}{\sin\theta^{\sigma}}|\nabla_{\partial \Omega}u^{\varepsilon,\,\sigma}|^2\,d\mathcal{H}^{1},\label{eq.5.1} 
	\end{align}
		for each $0\leq t<\frac{1}{2}$. Recalling the phase field method and \eqref{2.1.14}, we may have the following approximation;
		\begin{equation}\label{eq.5.2}
		\int_{\partial\Omega}\frac{\varepsilon}{\sin\theta^{\sigma}}|\nabla_{\partial \Omega}u^{\varepsilon,\,\sigma}|^2\,d\mathcal{H}^{1}\thickapprox \int_{\partial\Omega}d\mathcal{H}^{0}\lfloor_{\partial M_t^{\sigma}}=(\text{the number of elements of $\partial M_t^{\sigma}$})=2
		\end{equation}
		as $\varepsilon \to 0$. Since we have  $\sin\theta^{\sigma}=(\sqrt{1-2t+\sigma^{-2}})^{-1}\sqrt{1-2t}$ and $\cos\theta^{\sigma} = (\sqrt{1-2t+\sigma^{-2}})^{-1} \sigma^{-1}$, it holds that 
		\begin{equation}\label{eq.5.3}
		\frac{(\cos\theta^{\sigma})^2}{\sin\theta^{\sigma}}=\frac{\sigma^{-1}}{\sqrt{1-2t}}\frac{\sigma^{-1}}{\sqrt{1-2t+\sigma^{-2}}}\geq \frac{\sigma^{-1}}{\sqrt{2}}.
		\end{equation}
		Therefore, from \eqref{eq.5.2} and \eqref{eq.5.3}, we may obtain the following estimate:
		\begin{align}
		\int_{\partial\Omega}\varepsilon |\nabla u^{\varepsilon,\,\sigma}\cdot\boldnu|^2\,d\mathcal{H}^{1}&\geq \frac{1}{\sqrt{2}\sigma}  \int_{\partial\Omega}\frac{\varepsilon}{\sin\theta^{\sigma}}|\nabla_{\partial \Omega}u^{\varepsilon,\,\sigma}|^2\,d\mathcal{H}^{1}\nonumber\\
		&\thickapprox\frac{1}{\sqrt{2}\sigma}\mathcal{H}^0(\partial M_t^{\sigma})=\frac{\sqrt{2}}{\sigma},
		\end{align}
		for each $t\in[0,\,\frac{1}{2})$. This implies that the Dirichlet energy on $\partial\Omega$ along $\nu$ goes to infinity as $\sigma$ goes to zero for each $t$. Therefore, we may conclude that it is reasonable to assume \eqref{3.1.9.2} or this kind of the energy estimate for the Dirichlet energy in order to consider our definition of a Brakke flow.
	\end{remark}
\end{itemize}

\subsubsection{Main theorem and important lemmas}\label{mainresults.diri}
Now we state a sequence of one definition, several lemmas and the main theorem.
\begin{lemma}\label{thm3.1}
	Suppose that ``General assumptions", ``Vanishing hypothesis of the discrepancy measure", and ``Uniform upper bound on the boundary $\partial\Omega$" in Subsection \ref{generalassmp} hold and $\{\varepsilon_i\}_{i\in\mathbb{N}} \subset (0,\,1)$ and $\{\sigma_j\}_{j\in\mathbb{N}} \subset (0,\,1)$ are two families of parameters with $\varepsilon_i \to 0$ and $\sigma_j \to 0$. Let $\{u^{\varepsilon_i,\,\sigma_j}\}_{i,\,j\in\mathbb{N}}$ be the solutions of \eqref{1.1.1} and $\mu^{\varepsilon_i,\,\sigma_j}_t$ be as in \eqref{1.1.5}. Then there exist a subsequence $\{\mu^{\varepsilon'_j,\,\sigma'_j}_t\}_{j\in\mathbb{N}}$ and a family of Radon measures $\{\mu_t\}_{t\geq 0}$ on $\overline{\Omega}$ such that for all $t\geq 0$, $\mu^{\varepsilon'_j,\,\sigma'_j}_t \rightharpoonup \mu_t$ as $j \to \infty$ on $\overline{\Omega}$. Moreover, for a.e. $t\geq 0$ and for all $j\in\mathbb{N}$, $\mu_t$ are $(n-1)$-rectifiable on $\overline{\Omega}$. 
\end{lemma}

\begin{remark}
	As we mention in the case of Dirichlet boundary conditions, the integrality of the Radon measure only in the interior $\Omega$ follows from the interior argument of Tonegawa  \cite{Tonegawa02} or Takasao and Tonegawa \cite{TaTo} by using the local monotonicity formula. Thus we may expect that $\sigma^{-1}_0\mu_t$ is a $(n-1)$-integral Radon measure in $\Omega$ for a.e. $t\geq 0$, where we recall that $\sigma_0$ is defined as $\sigma_0 \coloneqq \int_{-1}^{1}\sqrt{2W(u)}\,du = 4/3$.
\end{remark}

\begin{lemma}\label{thm3.2}
	Suppose that ``General assumptions" and ``Uniform upper bound on the boundary $\partial\Omega$" in Subsection \ref{exis.assmp.diri} hold and  $\{\varepsilon'_i\}_{i\in\mathbb{N}}$ and $\{\sigma'_j\}_{j\in\mathbb{N}}$ are as in Lemma \ref{thm3.1}. Let $\{\alpha^{\varepsilon'_i,\,\sigma'_j}\}_{i,\,j\in\mathbb{N}}$ be as in Definition \ref{def.2.3}. Then there exist a subsequence $\{\alpha^{\varepsilon'_j,\,\sigma'_j}\}_{j\in\mathbb{N}}$ (labelled with the same index) and a Radon measure $\alpha$ on $\partial \Omega \times [0,\infty)$ such that $\alpha^{\varepsilon'_j,\,\sigma'_j} \rightharpoonup \alpha$ as $j \to \infty$ on $\partial \Omega \times [0,\infty)$. In addition, setting a vector-valued function $\boldv^{\varepsilon'_j,\,\sigma'_j}_b:\partial\Omega\rightarrow \mathbb{R}^n$ by
	\begin{equation}
	\boldv^{\varepsilon'_j,\,\sigma'_j}_b \coloneqq \begin{cases}
	\displaystyle -\frac{\partial_t u^{\varepsilon'_j,\,\sigma'_j}}{|\nabla_{\partial \Omega} u^{\varepsilon'_j,\,\sigma'_j}|}\frac{\nabla_{\partial \Omega} u^{\varepsilon'_j,\,\sigma'_j}}{|\nabla_{\partial \Omega} u^{\varepsilon'_j,\,\sigma'_j}|}&\quad \text{if $|\nabla_{\partial \Omega} u^{\varepsilon'_j,\,\sigma'_j}| \neq 0$}, \\
	0 &\quad  \text{otherwise}, \label{5.2.3}
	\end{cases} 
	\end{equation}
	then $\boldv_b^{\varepsilon'_j,\,\sigma'_j}\xrightarrow[j\to\infty]{}0$ in the sense that
	\begin{equation}\label{3.1.10}
	\lim_{j \to \infty} \int_{\partial\Omega\times[0,\,\infty)} \boldg\cdot\boldv^{\varepsilon'_j,\,\sigma'_j}_b\, d\alpha^{\varepsilon'_j,\,\sigma'_j}=0
	\end{equation}
	for all $\boldg\in (C_c(\partial \Omega \times [0,\,\infty)))^n$.
\end{lemma}

\begin{remark}\label{rem.thm3.2}
	In Lemma \ref{thm3.2}, the convergence of $\boldv_b^{\varepsilon'_j,\,\sigma'_j}$ means that there exists $\boldv_b\in(L^2(\alpha,\,\partial\Omega\times[0,\,\infty)))^n$ such that
	\begin{equation}\label{3.1.11.1}
	\lim_{j \to \infty} \int_{\partial\Omega\times[0,\,\infty)} \boldg\cdot\boldv^{\varepsilon'_j,\,\sigma'_j}_b\, d\alpha^{\varepsilon'_j,\,\sigma'_j}=-\int_{\partial \Omega \times [0,\infty)} \boldg\cdot {\boldv}_b \, d\alpha
	\end{equation}
	for any $\boldg\in (C_c(\partial \Omega \times [0,\,\infty)))^n$, and $\boldv_b=0$ in $(L^2(\alpha))^n$.
\end{remark}

\begin{lemma}\label{thm3.2.1}
	Suppose that ``General assumptions" and ``Uniform upper bound on the boundary $\partial\Omega$" in Subsection \ref{exis.assmp.diri} hold and the space-dimension $n$ is larger than 2. Let a subsequence $\{\alpha^{\varepsilon'_j,\,\sigma'_j}\}_{j\in\mathbb{N}}$ and $\alpha$ be as in Lemma \ref{thm3.2}. Moreover, setting
	\begin{equation}\label{3.1.11.2.0}
	w^{j}(x,\,t)\coloneqq\Phi\circ u^{\varepsilon'_j,\,\sigma'_j}\coloneqq\int_{0}^{u^{\varepsilon'_j,\,\sigma'_j}(x,\,t)
	}\sqrt{2W(s)}\,ds= u^{\varepsilon'_j,\,\sigma'_j}(x,\,t)-\frac{1}{3}(u^{\varepsilon'_j,\,\sigma'_j}(x,\,t))^3
	\end{equation}
	for any $j\in\mathbb{N}$ and $(x,\,t)\in\overline{\Omega}\times[0,\,\infty)$, we assume that there exist $0\leq t_1<t_2<\infty$ and a non-empty connected component $\Gamma_1$ of $\partial\Omega$ such that 
	\begin{equation}\label{3.1.11.2}
	\liminf_{j\to\infty}\int_{t_1}^{t_2}\!\left|\int_{\Gamma_1}w^{j}\,d\mathcal{H}^{n-1}\right|\,dt <\frac{2}{3}\mathcal{H}^{n-1}(\Gamma_1)\,(t_2-t_1)<\infty.
	\end{equation}
	holds. Then there exists a positive constant $0<\tilde{C}(t_1,\,t_2)<\infty$ such that $\alpha(\Gamma_1\times[t_1,\,t_2])\geq \tilde{C}(t_1,\,t_2)$, which means that $\alpha$ is not identically zero.
\end{lemma}

\begin{remark}\label{uniformPosiBoundaryMeasDiri}
	In Lemma \ref{thm3.2.1}, the assumption \eqref{3.1.11.2} means that there always exists a phase boundary on $\partial \Omega$, that is, the boundary $\partial M_t$ of the hypersurface $M_t$ always lies on some part of $\partial \Omega$; otherwise we may derive the fact that $u^{j}$ is equal to either $+1$ or $-1$ almost everywhere on $\partial \Omega$ as $j\to\infty$ and thus this implies, from the definition of $w^{j}$, 
	\begin{equation}\label{3.1.11.3}
	\lim_{i\to \infty}\int_{t_1}^{t_2}\!\left|\int_{\Gamma}w^{j}\,d\mathcal{H}^{n-1}\right|\,dt=\frac{2}{3}\mathcal{H}^{n-1}(\Gamma)\,(t_2-t_1)
	\end{equation}
	for any time $0<t_1<t_2<\infty$ and any connected component $\Gamma\subset\partial\Omega$, which contradicts \eqref{3.1.11.2}.
\end{remark}

Due to Lemma \ref{thm3.1}, we may define the unique rectifiable varifolds as follows:

\begin{definition}\label{def.3.1}
	We define a rectifiable varifold $V_t$ associated with $\mu_t$ as follows: for $t\geq 0$ where $\mu_t$ is $(n-1)$-rectifiable on $\overline{\Omega}$,
	\begin{equation}\label{3.1.12}
	V_t(\phi)\coloneqq \int_{\overline{\Omega}} \phi(x,\,T_x\mu_t)\,d\mu_t(x)
	\end{equation}
	for every $\phi\in C_c(G_{n-1}(\overline{\Omega}))$. For $t\geq 0$ where $\mu_t$ is not $(n-1)$-rectifiable, we define $V_t$ by $V_t(\phi)\coloneqq \int_{\overline{\Omega}}\phi(x,\,\mathbb{R}^{n-1}\times \{0\})\,d\mu_t(x)$ for every $\phi\in C_c(G_{n-1}(\overline{\Omega}))$. Here $T_x\mu_t$ is an approximate tangent space at $x$, which exists for $\mu_t$-a.e. $x\in \overline{\Omega}$ because of the rectifiability of $\mu_t$.
\end{definition}

\begin{lemma}\label{thm3.3}
	Suppose that ``General assumptions", ``Vanishing hypothesis for the discrepancy measure", and ``Uniform upper bound on the boundary $\partial\Omega$" in Subsection \ref{exis.assmp.diri} hold and let $\{V_t\}_{t\geq 0}$ be as in Definition \ref{def.3.1}. Then the following properties hold:
	\begin{equation}\label{3.1.13}
	\|\delta V_t\|(\overline{\Omega})<\infty,\qquad \int_{0}^{T} \|\delta V_t\|(\overline{\Omega})\,dt<\infty
	\end{equation} 
	for a.e. $t\geq 0$ and all $T>0$.
\end{lemma}

Now we state the main theorem of this subsection, that is, the approximation result of our Brakke flow with Dirichlet boundary conditions, which we defined in the previous section.

\begin{theorem}\label{thm3.4}
	Suppose that ``General assumptions", ``Vanishing hypothesis for the discrepancy measure", and ``Uniform upper bound on the boundary $\partial\Omega$" in Subsection \ref{exis.assmp.diri} hold. Let $V_t$, $\alpha$, and ${\boldv}_b$ be the quantities in Lemma \ref{thm3.1} and \ref{thm3.3}, Lemma \ref{thm3.2}, and Remark \ref{rem.thm3.2} (see also Definition \ref{def.3.1}), which are obtained from the singular limits of the Allen-Cahn equations \eqref{1.1.1} by taking $\varepsilon,\,\sigma\to0$. Then the quartet $(\{V_t\}_{t\geq 0},\,\alpha,\,{\boldv}_b,\,\{u^{\varepsilon,\,\sigma}\}_{\varepsilon,\,\sigma\in(0,\,1)})$ is a Brakke flow with Dirichlet boundary conditions in Definition \ref{def.4.1} with $\|V_0\|=\sigma_0\,\mathcal{H}^{n-1}\lfloor_{M_0}$ where $M_0$ is as in Subsection \ref{exis.assmp.diri}. In addition, we have the estimate that
	\begin{equation}\label{3.1.18}
	\int_{0}^{\infty} \int_{\overline{\Omega}}|\widetilde{\boldH}_V|^2\,d\|V_t\|\,dt \leq D,
	\end{equation}
	where $\widetilde{\boldH}_V$ is the modified generalized mean curvature vector and $D$ is the constant given in \eqref{3.1.6}.
	
	Moreover, the assumption ``Uniform upper bound on the boundary $\partial\Omega$" in Subsection \ref{exis.assmp.diri} actually leads us to obtain the stronger result that the total variation measure $\|\mathcal{S}_{\alpha,\,\boldv_b}\|$ is identically equal to zero on $\partial\Omega\times[0,\,\infty)$ (this is equivalent to the claim that $\boldv_b=0$ in $(L^2(\alpha))^n$ on $\partial\Omega\times[0,\,\infty)$ as in Lemma \ref{thm3.2}).
\end{theorem}

\subsection{Dynamic boundary condition}\label{exisdyna}
In this subsection, we will first give an assumption named ``General assumptions" and a working hypothesis named ``Vanishing hypothesis for the discrepancy measure"  and then will state a sequence of the lemmas and the main theorem of the sharp interface limits of Allen-Cahn equations \eqref{1.1.1}. One feature of the results in the case of dynamic boundary conditions is that we do not need to assume the uniform upper bound which we state in the assumption in the case of Dirichlet boundary conditions and it is described in Subsection \ref{exis.assmp.diri}.
\subsubsection{Assumptions and hypothesis}\label{exis.assmp.dyna}
\begin{itemize}
	\item \textbf{General assumptions.}\label{generalassmp2}
	
	\quad We impose the general assumptions essentially same as those in Subsection \ref{generalassmp}. Note that, in the case of dynamic boundary conditions, the parameter $\sigma\in(0,\,\infty)$ is given and fixed. 
	\item \textbf{Vanishing hypothesis for the discrepancy measure.}\label{hypothesis2}
	
	\quad Let $\sigma>0$ be a given constant and $\xi^{\varepsilon,\,\sigma}(x,\,t)$ be the discrepancy function given in \eqref{discrepancyFunc} for any $(x,\,t) \in \Omega \times (0,\,\infty)$. Recall that the measure $\xi^{\varepsilon,\,\sigma}_t = \xi^{\varepsilon,\,\sigma}(x,\,t)\mathcal{L}^n(x)$ is called the \textit{discrepancy measure}. Then, as is the case of Dirichlet boundary condition, we also assume the following two conditions: 
	\begin{itemize}
		\item there exists a subsequence  $\{\varepsilon_j\}_{j\in\mathbb{N}}$ with $\lim_{j \to \infty}\varepsilon_j = 0$ such that
		\begin{equation}\label{assmpVaniDisDynamic}
			|\xi^{\varepsilon_j,\,\sigma}_t| \xrightarrow[j\to\infty]{} 0 \quad \text{on } \overline{\Omega} \quad \text{as Radonn measures}
		\end{equation}
		for a.e. $t\in[0,\,\infty)$.
		\item there exists a constant $c_0>0$ such that
		\begin{equation}\label{boundednessDisDynamic}
			\sup_{x\in\Omega}\xi^{\varepsilon,\,\sigma}(x,\,t) \leq c_0
		\end{equation}
		for any $\varepsilon>0$, $\sigma>0$, and $t>0$.
	\end{itemize}
 	We emphasize that, as it is stated in the case of Dirichlet boundary conditions (see Subsection \ref{exis.assmp.diri} for more detail), we can actually show the boundedness of the discrepancy measure $\xi^{\varepsilon,\,\sigma}_t$ in $\Omega$ uniformly in $t > 0$ and $\varepsilon,\,\sigma>0$ under some assumptions. Namely, if $\Omega$ is convex, some gradient estimate on the boundary $\partial \Omega$ (see \eqref{assumptionOnBoundary} in Proposition \ref{appendAProp2} for this estimate), and the boundedness of the discrepancy measure at the initial time, then we have \eqref{boundednessDisDynamic} for some constant $c_0>0$ (see Proposition \ref{appendAProp2} in Appendix A). 
 	
 	\quad Finally, we should mention that, as it is stated in the case of Dirichlet boundary conditions (see Subsection \ref{exis.assmp.diri}), we can show the vanishing of the discrepancy measure in $\Omega$ under some assumptions. Precisely, we can show that, for given $T>0$, there exists a subsequence $\{\varepsilon_j\}_{j\in\mathbb{N}}$ with $\lim_{j\to\infty}\varepsilon_j = 0$ such that
 	\begin{equation}
 		|\xi^{\varepsilon_j,\,\sigma}_t| \otimes \mathcal{L}^1_t \xrightarrow[j\to \infty]{} 0 \quad \text{ in }\Omega \times (0,\,T) \quad \text{as Radon measures}.
 	\end{equation}
 	This can be done by applying the methods investigated in, for instance, \cite{Ilmanen01} and \cite{MiTo} (see Proposition \ref{vanishingProp} in Appendix B of Section \ref{appendix} for more detail).
\end{itemize}

\subsubsection{Main theorem and important lemmas}\label{mainresults.dyna}
Now we state a sequence of several lemmas and the main theorem in the case of dynamic boundary conditions. We remark that, although we can show the proofs of the statements with the parameter $\sigma$ given as $0< \sigma <\infty$, only in Lemma \ref{thm3.7.1}, we need to restrict ourselves to consider $\sigma$ as the parameter more than or equal to 1. This is because of some technical issue, however, one of our main goals is to investigate how different the formulations in Brakke sense between dynamic and right-angle Neumann boundary conditions (the case of $\sigma = \infty$) are. Thus this restriction does not seem to be essential here. 

First of all, we give two important lemmas, which describe the convergence of the measures $\mu^{\varepsilon,\,\sigma}_t$ and $\alpha^{\varepsilon,\,\sigma}$. These lemmas are essentially the same results as Lemma \ref{thm3.1} and \ref{thm3.2} in Subsection \ref{mainresults.diri}.
\begin{lemma}\label{thm.3.6}
	Suppose that ``General assumptions" and ``Vanishing hypothesis of the discrepancy measure" in Subsection \ref{generalassmp2} holds and $\{\varepsilon_j\}_{j\in\mathbb{N}} \subset (0,\,1)$ is a family of parameters with $\varepsilon_j \to 0$ as $j\to\infty$. Let $\{u^{\varepsilon_j,\,\sigma}\}_{j\in\mathbb{N}}$ be a family of the solutions of \eqref{1.1.1} and $\mu^{\varepsilon_j,\,\sigma}_t$ be as in \eqref{1.1.5}. Then there exist a subsequence $\{\mu^{\varepsilon'_j,\,\sigma}_t\}_{j\in\mathbb{N}}$ and a family of Radon measures $\{\mu^{\sigma}_t\}_{t\geq0}$ on $\overline{\Omega}$ such that for all $t\geq 0$, $\mu^{\varepsilon'_j,\,\sigma}_t \rightharpoonup \mu^{\sigma}_t$ as $j\to\infty$ on $\overline{\Omega}$. Moreover, for a.e. $t\geq 0$, $\sigma_0^{-1}]\mu^{\sigma}_t$ is $(n-1)$-rectifiable on $\overline{\Omega}$. 
\end{lemma}

\begin{remark}
	The integrality of the Radon measure $\mu^{\sigma}_t$ only in the interior $\Omega$ follows from the interior argument of Tonegawa \cite{Tonegawa02} or Takasao and Tonegawa \cite{TaTo} by using the local monotonicity formula. Thus, we may expect that $\sigma^{-1}_0\mu^{\sigma}_t$ is a $(n-1)$-integral Radon measure in $\Omega$ for a.e. $t\geq 0$, where $\sigma_0\coloneqq\int_{-1}^{1}\sqrt{2W(u)}\,du$.
\end{remark}

\begin{lemma}\label{thm3.7}
	Suppose that ``General assumptions" in Subsection \ref{generalassmp2} holds and $\{\varepsilon_j\}_{j\in\mathbb{N}}\subset(0,\,1)$ is as in Lemma \ref{thm.3.6}. Let $\{\alpha^{\varepsilon_j,\,\sigma}\}_{\varepsilon>0}$ be as in \eqref{2.1.13}. Then there exist a subsequence $\{\alpha^{\varepsilon'_j,\,\sigma}\}_{j\in\mathbb{N}}$ and a Radon measure $\alpha^{\sigma}$ on $\partial \Omega \times [0,\infty)$ such that $\alpha^{\varepsilon'_j,\,\sigma} \rightharpoonup \alpha^{\sigma}$ as $j\to\infty$ on $\partial \Omega \times [0,\infty)$. In addition, there exists a function ${\boldv}^{\sigma}_b \in (L^2(\alpha^{\sigma},\,\partial \Omega \times [0,\infty)))^n$ such that
	\begin{equation}\label{3.2.1}
	\lim_{j\to\infty} \int_{0}^{\infty}\int_{\partial \Omega} \boldg\cdot\boldv^{\varepsilon'_j,\,\sigma}_b\, d\alpha^{\varepsilon'_j,\,\sigma}=- \int\int_{\partial \Omega \times [0,\infty)} \boldg\cdot {\boldv}^{\sigma}_b \, d\alpha^{\sigma}
	\end{equation}
	for all $g \in (C_c(\partial \Omega \times [0,\infty)))^n$ where $\boldv^{\varepsilon'_j,\,\sigma}_b$ is as in \eqref{5.2.3}, and such that
	\begin{equation}\label{3.2.2}
	\int\int_{\partial \Omega \times[0,\,\infty)} |{\boldv}^{\sigma}_b|^2\,d\alpha^{\sigma} \leq \liminf_{j\to\infty} \int_{0}^{\infty}\int_{\partial \Omega}\varepsilon'_j (\partial_t u^{\varepsilon'_j,\,\sigma})^2 \,d\mathcal{H}^{n-1}\,dt.
	\end{equation}
\end{lemma}

\begin{lemma}\label{thm3.7.1}
	Suppose that ``General assumptions" in Subsection \ref{exis.assmp.diri} holds and the space-dimension $n$ is larger than 2. Let $\sigma$ be in $[1,\,\infty)$ and let a subsequence $\{\alpha^{\varepsilon'_j,\,\sigma}\}_{j\in\mathbb{N}}$ and $\alpha$ be as in Lemma \ref{thm3.7}. We set
	\begin{equation}\label{3.2.11.1}
	w^{j,\,\sigma}(x,\,t)\coloneqq\Phi\circ u^{\varepsilon'_j,\,\sigma}\coloneqq\int_{0}^{u^{\varepsilon'_j,\,\sigma}(x,\,t)
	}\sqrt{2W(s)}\,ds= u^{\varepsilon'_j,\,\sigma}(x,\,t)-\frac{1}{3}(u^{\varepsilon'_j,\,\sigma}(x,\,t))^3
	\end{equation}
	and $w^{j,\,\sigma}_0(x)\coloneqq w^{j,\,\sigma}(x,\,0)$ for any $j\in\mathbb{N}$ and $(x,\,t)\in\overline{\Omega}\times[0,\,\infty)$. Moreover, we assume the following two assumptions on the initial data; 
	\begin{enumerate}
		\item There exists a non-empty connected component $\Gamma_2
		$ of $\partial\Omega$ such that 
		\begin{equation}\label{3.2.11.2}
		\liminf_{j\in\mathbb{N}}\left|\int_{\Gamma_2}w^{j,\,\sigma}_0\,d\mathcal{H}^{n-1}\right| <\frac{2}{3}\mathcal{H}^{n-1}(\Gamma_2)
		\end{equation}
		holds. 
		\item $\lim_{\gamma_0\downarrow0}\,\sup_{j\in\mathbb{N}}\mu^{j}_0(\Omega\cap\{x\mid\dist(x,\,\partial \Omega)<\gamma_0\})=0$. 
	\end{enumerate}
	Then, there exists a time $s\in(0,\,\infty)$ such that, for any $0<t_1<t_2<s$, there exists a positive constant $0<\tilde{C}(t_1,\,t_2)<\infty$ such that $\alpha^{\sigma}(\Gamma_2\times[t_1,\,t_2])\geq \tilde{C}(t_1,\,t_2)$, which means that $\alpha$ is not identically zero.
\end{lemma}

\begin{remark}
	Note that only in this lemma, we assume that $\sigma$ is larger than or equal to 1 due to the technical issue. In the other claims, we basically assume that $\sigma$ is positive constant.
\end{remark}

\begin{remark}	
	The explanation of the first assumption in Lemma \ref{thm3.7.1} is also shown in Remark \ref{uniformPosiBoundaryMeasDiri} in the case of Dirichlet boundary conditions.
	
	Regarding to the second assumption of the initial data, it can be interpreted that the geometric interior of $M_0$ does not exist on the boundary $\partial \Omega$. If this is not true, we have that there exists a constant $\delta>0$ such that, for any $\gamma>0$, $\mathcal{H}^{n-1}\lfloor_{M_0}(\partial\Omega\cap \{x\mid\dist(x,\,\partial \Omega)<\gamma\})\geq\delta$. From the convergence of $\mu^j_0$ to $\mathcal{H}^{n-1}\lfloor_{M_0}$, up to constants, as $j\to\infty$, we obtain the approximation
	\begin{align}\label{3.1.11.4}
	\mu^{j}_0(\Omega\cap\{x\mid\dist(x,\,\partial \Omega)<\gamma\})&\thickapprox \mu_0(\overline{\Omega}\cap\{x\mid\dist(x,\,\partial \Omega)<\gamma\})\nonumber\\
	&\geq \mathcal{H}^{n-1}\lfloor_{M_0}(\partial\Omega\cap \{x\mid\dist(x,\,\partial \Omega)<\gamma\}) \geq \delta
	\end{align}
	for $j\in\mathbb{N}$ large enough.
\end{remark}

Due to Lemma \ref{thm.3.6}, we may define the unique rectifiable varifolds as follows:

\begin{lemma}\label{thm3.8}
	Suppose that ``General assumptions" and ``Vanishing hypothesis for the discrepancy measure" in Subsection \ref{exis.assmp.dyna} hold and let $\{V^{\sigma}_t\}_{t\geq 0}$ be the associated varifold with $\mu^{\sigma}_t$ (see the definition stated in Definition \ref{def.3.1}). Then the following properties hold:
	\begin{equation}\label{3.2.4}
	\|\delta V^{\sigma}_t\|(\overline{\Omega})<\infty, \qquad \int_{0}^{T}  \|\delta V^{\sigma}_t\|(\overline{\Omega})\,dt<\infty
	\end{equation} 
	for a.e. $t\geq 0$ and all $T\geq 0$ respectively.
\end{lemma}

Now we state the main theorem of this subsection, that is, the approximation result of our Brakke flow with dynamic boundary conditions which we defined in the previous section.
\begin{theorem}\label{thm3.9}
	Suppose that ``General assumptions" and ``Vanishing hypothesis for the discrepancy measure" in Subsection \ref{exis.assmp.dyna} hold. Let $V^{\sigma}_t$, $\alpha^{\sigma}$, and ${\boldv}^{\sigma}_b$ be the quantities in Lemma \ref{thm.3.6}, \ref{thm3.7}, and \ref{thm3.8}, which are obtained from the singular limits of the Allen-Cahn equations \eqref{1.1.1} by taking $\varepsilon\to0$. Then the triplet $(\{V^{\sigma}_t\}_{t \geq 0},\,\alpha^{\sigma},\,{\boldv}^{\sigma}_b)$ is a Brakke flow with dynamic boundary conditions in Definition \ref{def.4.2} with $\|V^{\sigma}_0\|=\sigma_0\,\mathcal{H}^{n-1}\lfloor_{M_0}$ where $M_0$ is as in Subsection \ref{exis.assmp.dyna}. Moreover, we have the estimate
	\begin{equation}\label{3.2.10}
	\int_{0}^{\infty} \int_{\overline{\Omega}}|\widetilde{\boldH}^{\sigma}_V|^2\,d\|V^{\sigma}_t\|\,dt \leq D,
	\end{equation}
	where $\widetilde{\boldH}^{\sigma}_V$ is the modified generalized mean curvature vector and $D$ is the constant given in \eqref{3.1.6}. 
\end{theorem}

\section{A priori estimates for Allen-Cahn equations}\label{apriori}
In this section, we derive a priori estimate of Allen-Cahn equations \eqref{1.1.1} in the case that $\sigma$ is in positive and finite. This estimate is important to consider the characterization of the singular limit in the case of both Dirichlet and dynamic boundary conditions.

\begin{proposition}\label{prop.4.1}
	It holds that
	\begin{align}\label{4.1.1}
	\sup_{\varepsilon>0,\,\sigma>0} \left( E^{\varepsilon,\,\sigma}[u^{\varepsilon,\,\sigma}(\cdot \,,T)]+ \int_{0}^{T} \left(\int_{\Omega} \varepsilon (\partial_t u^{\varepsilon,\,\sigma})^2\,dx + \int_{\partial \Omega} \frac{\varepsilon}{\sigma} (\partial_t u^{\varepsilon,\,\sigma})^2\, d\mathcal{H}^{n-1} \right) \,dt \right) \leq D,
	\end{align}
	for all $T>0$. Here $D$ is as in \eqref{3.1.6}. Moreover we have
	\begin{equation}\label{4.1.2}
	\sup_{\varepsilon>0,\,\sigma>0} \mu^{\varepsilon,\,\sigma}_t(\Omega)\leq D,
	\end{equation}
	for all $t\geq 0$.
\end{proposition}
\begin{proof}
	By integration by parts, we can calculate in the following manner.
	\begin{align}
	\frac{d}{dt}E^{\varepsilon,\,\sigma}[u^{\varepsilon,\,\sigma}] &=\int _{\Omega}\left(-\varepsilon \Delta u^{\varepsilon,\,\sigma}  +\frac{W'(u^{\varepsilon,\,\sigma})}{\varepsilon} \right) \partial_t u^{\varepsilon,\,\sigma}\, dx+ \int _{\partial \Omega} \varepsilon \frac{\partial u^{\varepsilon,\,\sigma}}{\partial \boldnu} \partial_t u^{\varepsilon,\,\sigma} \, d\mathcal{H}^{n-1} \nonumber\\
	&=- \int _{\Omega}\varepsilon (\partial_t u^{\varepsilon,\,\sigma})^2\,dx- \int _{\partial \Omega} \frac{\varepsilon}{\sigma} (\partial_t u^{\varepsilon,\,\sigma})^2 \, d\mathcal{H}^{n-1}. \label{4.1.3}
	\end{align}
	Thus, for any $T> 0$, $\varepsilon>0$ and $\sigma>0$, we have
	\begin{align}
	&E^{\varepsilon,\,\sigma}[u^{\varepsilon,\,\sigma}(\cdot,\,T)]+\int_{0}^{T} \left(\int_{\Omega} \varepsilon (\partial_t u^{\varepsilon,\,\sigma}))^2\,dx +  \int_{\partial \Omega} \frac{\varepsilon}{\sigma} (\partial_t u^{\varepsilon,\,\sigma})^2\, d\mathcal{H}^{n-1} \right) \,dt \nonumber\\ 
	&= E^{\varepsilon,\,\sigma}[u^{\varepsilon,\,\sigma}_0] \leq D. \label{4.1.4}
	\end{align}
	Therefore \eqref{4.1.1} follows by taking supremum with respect to $\varepsilon>0$ and $\sigma>0$. From \eqref{3.1.6}, \eqref{4.1.2} also easily follows. 
\end{proof}

\subsection{The case $\sigma \in (0,\,1)$}\label{apr.diri}
In this subsection, we show the energy estimates of Allen-Cahn equations \eqref{1.1.1} on the boundary $\partial \Omega$ in the case $\sigma\in(0,\,1)$. This estimate plays an important role in considering the singular limit and formulate a Brakke flow especially with Dirichlet boundary conditions. Note that we only have the energy estimate in an integration form with respect to time $t>0$ so far. Note that we assume ``General assumptions" and ``Uniform upper bound on the boundary $\partial\Omega$" in Subsection \ref{exisdiri} in this case.

\begin{proposition}\label{prop.4.2}                        
	There exists $C_1=C_1(n,\partial \Omega,\,D)>0$ such that 
	\begin{equation}\label{4.1.5}
	\sup_{\varepsilon>0,\,\sigma\in(0,\,1)} \int_{t_1}^{t_2}\, \int_{\partial \Omega}  \left(\frac{\varepsilon |\nabla_{\partial \Omega} u^{\varepsilon,\,\sigma}| ^2}{2} +\frac{W (u^{\varepsilon,\,\sigma} )}{\varepsilon}\right) \, d\mathcal{H}^{n-1} dt \leq C_1(t_2-t_1+1)+C_0,
	\end{equation} 
	for any $0\leq t_1\leq t_2 <\infty$, where $C_0=C_0(t_1,\,t_2)>0$ is as in \eqref{3.1.9.2}. 
\end{proposition}
\begin{proof}
	For any $\phi \in C^2(\overline{\Omega})$ and by using integration by part and denoting $f^{\varepsilon,\,\sigma} \coloneqq -\varepsilon\Delta u^{\varepsilon,\,\sigma}+ \frac{W'(u^{\varepsilon,\,\sigma})}{\varepsilon}$, we may obtain 
	\begin{align}\label{4.1.6}
	\frac{d}{dt}\left(\int_{\Omega} \phi \,d\mu^{\varepsilon,\,\sigma}_t \right) &= \int _\Omega \phi\left(-\varepsilon \Delta u^{\varepsilon,\,\sigma} + \frac{W'(u^{\varepsilon,\,\sigma})}{\varepsilon} \right) \partial_t u^{\varepsilon,\,\sigma} \, dx -\int _{\Omega} \varepsilon (\nabla \phi \cdot \nabla u^{\varepsilon,\,\sigma}) \partial_t u^{\varepsilon,\,\sigma} \, dx \nonumber\\
	&\quad \qquad \quad + \int_{\partial \Omega} \varepsilon \phi \frac{\partial u^{\varepsilon,\,\sigma}}{\partial \boldnu} \partial_t u^{\varepsilon,\,\sigma} \, d\mathcal{H}^{n-1} \nonumber\\
	&= -\frac{1}{\varepsilon}\int_{\Omega} \phi (f^{\varepsilon,\,\sigma})^2 \, dx + \int_\Omega f^{\varepsilon,\,\sigma}(\nabla \phi \cdot \nabla u^{\varepsilon,\,\sigma}) \, dx \nonumber\\
	&\quad \qquad \quad- \int_{\partial \Omega} \frac{\varepsilon}{\sigma} \phi (\partial_t u^{\varepsilon,\,\sigma})^2 \, d\mathcal{H}^{n-1}.
	\end{align} 
	By integration by parts again, 
	\begin{align}
	\int_\Omega f^{\varepsilon,\,\sigma}(\nabla \phi \cdot \nabla u^{\varepsilon,\,\sigma}) \, dx &= - \int_{\Omega} \Delta \phi\, d\mu^{\varepsilon,\,\sigma}_t + \int _\Omega \varepsilon (\nabla u^{\varepsilon,\,\sigma} \otimes \nabla u^{\varepsilon,\,\sigma} : \nabla ^2 \phi) \, dx \nonumber\\
	&\qquad -\int_{\partial \Omega} \varepsilon (\nabla \phi \cdot \nabla u^{\varepsilon,\,\sigma}) \frac{\partial u^{\varepsilon,\,\sigma}}{\partial \boldnu} \, d\mathcal{H}^{n-1}\nonumber\\
	&\quad \qquad +\int _{\partial \Omega} \frac{\partial \phi}{\partial \boldnu} \left( \frac{\varepsilon |\nabla u^{\varepsilon,\,\sigma}| ^2}{2} +\frac{W (u^{\varepsilon,\,\sigma} )}{\varepsilon}\right)  \, d\mathcal{H}^{n-1}.\label{4.1.7} 
	\end{align}
	Therefore we can compute as follows.
	\begin{align}
	\frac{d}{dt}\left(\int_{\Omega} \phi \,d\mu^{\varepsilon,\,\sigma}_t \right) =& - \frac{1}{\varepsilon}\int_{\Omega} \phi (f^{\varepsilon,\,\sigma})^2\, dx -\int_{\Omega} \Delta \phi\, d\mu^{\varepsilon,\,\sigma}_t+\int_{\Omega}  \varepsilon (\nabla u^{\varepsilon,\,\sigma} \otimes \nabla u^{\varepsilon,\,\sigma} : \nabla ^2 \phi) \, dx \nonumber\\
	&\quad -\int_{\partial \Omega} \varepsilon (\nabla \phi \cdot \nabla u^{\varepsilon,\,\sigma}) \frac{\partial u^{\varepsilon,\,\sigma}}{\partial \boldnu} \, d\mathcal{H}^{n-1}- \int_{\partial \Omega} \varepsilon\,\sigma \phi \left( \frac{\partial u^{\varepsilon,\,\sigma}}{\partial \boldnu}\right)^2\,d\mathcal{H}^{n-1} \nonumber\\
	&\quad \qquad +\int _{\partial \Omega} \frac{\partial \phi}{\partial \boldnu} \left( \frac{\varepsilon |\nabla u^{\varepsilon,\,\sigma}| ^2}{2} +\frac{W (u^{\varepsilon,\,\sigma} )}{\varepsilon}\right)  \, d\mathcal{H}^{n-1}. \label{4.1.8}
	\end{align}
	Specifically, we can choose $\phi= d+1 $, where $d$ is a signed distance function from the boundary $\partial \Omega$ which is positive in the domain $\Omega$. However, since $\Omega$ is general open domain with smooth boundary, $d$ is smooth only in some open neighborhood of the boundary $\partial \Omega$. Thus we have to extend $d$ smoothly into $\mathbb{R}^n$ such that $|d|$ and $|\nabla^2 d|$ are uniformly bounded in $\Omega$. This extension can be done by a simple argument.\\
	Thus, by using  $\nabla \phi = -\boldnu$ and $\phi=1$ on $\partial \Omega$ and the fact $\sigma\in(0,\,1)$, we have 
	\begin{align}\label{4.1.9}
	\frac{d}{dt}\left(\int_{\Omega} \phi \,d\mu^{\varepsilon,\,\sigma}_t \right) &\leq  -\int_{\Omega} \Delta \phi\, d\mu^{\varepsilon,\,\sigma}_t+\int_{\Omega\cap \{|\nabla u^{\varepsilon,\,\sigma}|\neq 0\}} (a^{\varepsilon,\,\sigma} \otimes a^{\varepsilon,\,\sigma} : \nabla ^2 \phi) \varepsilon |\nabla u^{\varepsilon,\,\sigma}|^2 \, dx \nonumber\\
	&\quad + \int_{\partial \Omega} \frac{\varepsilon}{2} \left(\frac{\partial u^{\varepsilon,\,\sigma}}{\partial \boldnu}\right)^2\,d\mathcal{H}^{n-1} -\int _{\partial \Omega} \left( \frac{\varepsilon |\nabla_{\partial \Omega} u^{\varepsilon,\,\sigma}| ^2}{2} +\frac{W (u^{\varepsilon,\,\sigma} )}{\varepsilon}\right)  \, d\mathcal{H}^{n-1},
	\end{align}
	where $a^{\varepsilon,\,\sigma}$ is defined by $\frac{\nabla u^{\varepsilon,\,\sigma}}{|\nabla u^{\varepsilon,\,\sigma}|}$. Note that, in \eqref{4.1.9}, we used the fact that
	\begin{equation}\label{4.1.9.1}
	|\nabla u^{\varepsilon,\,\sigma}|^2= |\nabla_{\partial \Omega} u^{\varepsilon,\,\sigma}|^2+ \left(\frac{\partial u^{\varepsilon,\,\sigma}}{\partial \boldnu}\right)^2, \quad {\rm on}\ \overline{\Omega}\times [0,\,\infty).
	\end{equation}
	Recalling the estimate \eqref{4.1.1} and the assumption \eqref{3.1.9.2}, and integrating both members of the inequality \eqref{4.1.9} from time $t_1$ to $t_2$ , we obtain 
	\begin{align}
	\int_{\Omega}\phi\,d\mu^{\varepsilon,\,\sigma}_{t_2}-\int_{\Omega}\phi\,d\mu^{\varepsilon,\,\sigma}_{t_1} &\leq \sup_{\Omega}|\nabla^2\phi|\,\int_{t_1}^{t_2}\mu^{\varepsilon,\,\sigma}_t(\Omega)\,dt+ 2\sup_{\Omega}|\nabla^2\phi|\,\int_{t_1}^{t_2}\mu^{\varepsilon,\,\sigma}_t(\Omega\cap\{|\nabla u^{\varepsilon,\,\sigma}|\neq 0\})\,dt \nonumber\\
	&\qquad +\int_{t_1}^{t_2}\!\!\int_{\partial\Omega}\frac{\varepsilon}{2}\left(\frac{\partial u^{\varepsilon,\,\sigma}}{\partial \nu}\right)^2\,d\mathcal{H}^{n-1}\,dt \nonumber\\
	& \qquad \quad - \int_{t_1}^{t_2}\int_{\partial\Omega}\left( \frac{\varepsilon |\nabla_{\partial \Omega} u^{\varepsilon,\,\sigma}| ^2}{2} +\frac{W (u^{\varepsilon,\,\sigma} )}{\varepsilon}\right)  \, d\mathcal{H}^{n-1}\,dt \nonumber\\
	&\leq 3D\,\sup_{\Omega}|\nabla^2\phi|\,(t_2-t_1)+C_0(t_1,\,t_2) \nonumber\\
	&\qquad  -\int_{t_1}^{t_2}\!\!\int_{\partial\Omega}\left( \frac{\varepsilon |\nabla_{\partial \Omega} u^{\varepsilon,\,\sigma}| ^2}{2}+\frac{W (u^{\varepsilon,\,\sigma} )}{\varepsilon}\right)  \, d\mathcal{H}^{n-1}.
	\end{align}
	Since $\phi$ is bounded in $C^2$-norm and \eqref{4.1.2}, we have 
	\begin{equation}\label{4.1.10}
	\int_{t_1}^{t_2} \int _{\partial \Omega} \left(\frac{\varepsilon |\nabla_{\partial \Omega} u^{\varepsilon,\,\sigma}|^2}{2}+\frac{W(u^{\varepsilon,\,\sigma})}{\varepsilon}\right) \, d\mathcal{H}^{n-1}\,dt \leq 3D\,\|\phi \|_{C^2(\overline{\Omega})}(t_2-t_1+1)+C_0(t_1,\,t_2) < +\infty.
	\end{equation}
	Therefore \eqref{4.1.5} follows by taking the supremum with respect to $\varepsilon>0$ and $\sigma>0$ in \eqref{4.1.10}.
\end{proof}

\subsection{The case $\sigma \in [1,\,\infty)$}\label{apri.dyna}
In this subsection, we show the energy estimate of Allen-Cahn equations on the boundary $\partial \Omega$ in the case $\sigma\in[1,\,\infty)$. This estimate plays an important role in considering the singular limit and formulate a Brakke flow with dynamic boundary conditions in the case $\sigma\in[1,\,\infty)$. Note that, as same as the case $\sigma\in(0,\,1)$, we only have the energy estimate in an integration form with respect to time $t>0$. Note that we only assume ``General assumptions" in Subsection  \ref{exisdyna} in this case.

\begin{proposition}\label{prop.4.4}                        
	There exists $C_1=C_1(n,\partial \Omega,D)>0$ such that 
	\begin{equation}\label{4.2.5}
	\sup_{\varepsilon>0,\,\sigma\in [1,\,\infty)} \int_{t_1}^{t_2}\, \int_{\partial \Omega}  \left(\frac{\varepsilon |\nabla u^{\varepsilon,\,\sigma}| ^2}{2} +\frac{W (u^{\varepsilon,\,\sigma} )}{\varepsilon}\right) \, d\mathcal{H}^{n-1} dt \leq C_1(t_2-t_1+1),
	\end{equation} 
	for any $0\leq t_1\leq t_2 <\infty$.
\end{proposition}
\begin{proof}
	For any $\phi \in C^2(\overline{\Omega})$ and by applying the same argument in the proof of Proposition \ref{prop.4.2}, we may obtain  the identity \eqref{4.1.8}. Here, as we stated in Proposition \ref{prop.4.2}, we choose $\phi= d+1 $, where $d$ is a signed distance function from the boundary $\partial \Omega$ which is positive in the domain $\Omega$. Note that this $d$ is also smoothly extended into the function whose domain is $\mathbb{R}^n$. Thus, by using  the properties of the signed distance function, we have 
	\begin{align}\label{4.2.9}
	\frac{d}{dt}\left(\int_{\Omega} \phi \,d\mu^{\varepsilon,\,\sigma}_t \right) \leq&  -\int_{\Omega} \Delta \phi\, d\mu^{\varepsilon,\,\sigma}_t+\int_{\Omega\cap \{|\nabla u^{\varepsilon,\,\sigma}|\neq 0\}} (a^{\varepsilon,\,\sigma} \otimes a^{\varepsilon,\,\sigma} : \nabla ^2 \phi) \varepsilon |\nabla u^{\varepsilon,\,\sigma}|^2 \, dx \nonumber\\
	&\quad + \int_{\partial \Omega} \frac{\varepsilon}{\sigma^2} (\partial_t u^{\varepsilon,\,\sigma})^2\,d\mathcal{H}^{n-1}- \int_{\partial \Omega} \frac{\varepsilon}{\sigma} (\partial_t u^{\varepsilon,\,\sigma})^2\,d\mathcal{H}^{n-1}\nonumber\\
	&\quad \qquad -\int _{\partial \Omega} \left( \frac{\varepsilon |\nabla u^{\varepsilon,\,\sigma}| ^2}{2} +\frac{W (u^{\varepsilon,\,\sigma} )}{\varepsilon}\right)  \, d\mathcal{H}^{n-1},
	\end{align}
	where $a^{\varepsilon,\,\sigma}$ is defined by $\frac{\nabla u^{\varepsilon,\,\sigma}}{|\nabla u^{\varepsilon,\,\sigma}|}$.  Hence, by integrating both members of the inequality \eqref{4.2.9} from time $t_1$ to $t_2$ and using the estimate \eqref{4.1.1} and the fact that $\sigma$ is in $[1,\,\infty)$, we have
	\begin{align}
	\int_{\Omega}\phi\,d\mu^{\varepsilon,\,\sigma}_{t_2}-\int_{\Omega}\phi\,d\mu^{\varepsilon,\,\sigma}_{t_1} &\leq \sup_{\Omega}|\nabla^2\phi|\,\int_{t_1}^{t_2}\mu^{\varepsilon,\,\sigma}_t(\Omega)\,dt+ 2\sup_{\Omega}|\nabla^2\phi|\,\int_{t_1}^{t_2}\mu^{\varepsilon,\,\sigma}_t(\Omega\cap\{|\nabla u^{\varepsilon,\,\sigma}|\neq 0\})\,dt \nonumber\\
	&\quad \qquad \quad \qquad -\int_{t_1}^{t_2}\int_{\partial\Omega}\left( \frac{\varepsilon |\nabla u^{\varepsilon,\,\sigma}| ^2}{2} +\frac{W (u^{\varepsilon,\,\sigma} )}{\varepsilon}\right)  \, d\mathcal{H}^{n-1}\,dt\nonumber\\
	&\leq 3D\,\sup_{\Omega}|\nabla^2\phi|\,(t_2-t_1)-\int_{t_1}^{t_2}\!\!\int_{\partial\Omega}\left( \frac{\varepsilon |\nabla u^{\varepsilon,\,\sigma}| ^2}{2} +\frac{W (u^{\varepsilon,\,\sigma} )}{\varepsilon}\right)  \, d\mathcal{H}^{n-1}.\label{4.2.10.1}
	\end{align} 
	Thus, recalling the choice of $\phi$, we may obtain
	\begin{equation}\label{4.2.10}
	\int_{t_1}^{t_2} \int _{\partial \Omega} \left(\frac{\varepsilon |\nabla u^{\varepsilon,\,\sigma}|^2}{2}+\frac{W(u^{\varepsilon,\,\sigma})}{\varepsilon}\right) \, d\mathcal{H}^{n-1}\,dt \leq 3D\,\|\phi \|_{C^2(\overline{\Omega})}(t_2-t_1+1)< +\infty.
	\end{equation}
	Therefore \eqref{4.2.5} follows by taking the supremum with respect to $\varepsilon>0$ in \eqref{4.2.10}.
\end{proof}

\section{Characterization of the limits}\label{chara.lim}
In this section, we will show the proofs of a sequence of the main results in each case of Dirichlet or dynamic boundary conditions.
\subsection{Dirichlet boundary condition}\label{chara.diri}
In this section, we prove a sequence of the main results which we stated in Section \ref{mainresults.diri}. We note that the positive constants $C_1$ and $D$ are as in Proposition \ref{prop.4.1} and Proposition \ref{prop.4.2}.
\subsubsection{Convergence of the measures $\{\mu^{\varepsilon,\,\sigma}_t\}_{\varepsilon,\,\sigma>0,\,t\geq 0}$ (Dirichlet boundary conditions)}\label{chara.1}
First of all, in order to prove the convergence of $\{\mu^{\varepsilon,\,\sigma}_t\}_{\varepsilon>0,\,\sigma>0}$ for all $t\geq0$, we derive an estimate on the change of the diffuse surface area measures in time. The main idea in the following proof comes from Mugnai and R\"{o}ger \cite{MuRo}. In the following, we set
\begin{equation}\label{5.1.0}
\mu^{\varepsilon,\,\sigma}_t(\phi)\coloneqq \int_{\Omega} \phi\,d\mu^{\varepsilon,\,\sigma}_t
\end{equation}
for all $\phi \in C_c(\overline{\Omega})$.

Note that we assume that ``General assumptions" in Subsection \ref{exis.assmp.diri} holds through this subsection.
\begin{lemma}\label{lem.5.1}
	Let $T>0$ be arbitrary. Then we have, for all $\phi \in C^1_c(\overline{\Omega})$,
	\begin{equation}\label{5.1.1}
	\sup_{\varepsilon,\,\sigma>0}\int_{0}^{T} \left|\frac{d}{dt} \mu^{\varepsilon,\,\sigma}_t(\phi) \right| \,dt \leq D\left(T+\frac{3}{2}\right)\|\phi\|_{C^1(\overline{\Omega})} <\infty.
	\end{equation}
\end{lemma}
\begin{proof}
	Let $\phi$ be in $C^1_c(\overline{\Omega})$ and $T>0$ be any time. From the calculation in the proof of Proposition \ref{prop.4.2}, we have
	\begin{align}
	\left|\frac{d}{dt} \mu^{\varepsilon,\,\sigma}_t(\phi) \right| &= \left|-\int_{\Omega} \varepsilon \phi (\partial_t u^{\varepsilon,\,\sigma})^2\,dx- \int_{\partial \Omega} \frac{\varepsilon}{\sigma} \phi (\partial_t u^{\varepsilon,\,\sigma})^2\, d\mathcal{H}^{n-1} - \int_{\Omega} \varepsilon \partial_t u^{\varepsilon,\,\sigma} \nabla u^{\varepsilon,\,\sigma}\cdot \nabla \phi \, dx\right| \nonumber\\
	&\leq \|\phi\|_{C^1(\overline{\Omega})} \left(\int_{\Omega} \varepsilon  (\partial_t u^{\varepsilon,\,\sigma})^2\,dx +  \int_{\partial \Omega} \frac{\varepsilon}{\sigma} \phi (\partial_t u^{\varepsilon,\,\sigma})^2\, d\mathcal{H}^{n-1}\right) \nonumber\\
	&\qquad + \frac{1}{2} \|\phi\|_{C^1(\overline{\Omega})} \int_{\Omega} \left(\varepsilon(\partial_t u^{\varepsilon,\,\sigma})^2+\varepsilon|\nabla u^{\varepsilon,\,\sigma}|^2\right)\,dx. \label{5.1.2}
	\end{align}
	Thus by integrating $0$ to $T$ and using the estimates \eqref{4.1.1}, we obtain the conclusion.
\end{proof}

\begin{proof}[Proof of a part of Lemma \ref{thm3.1}]
	Let $T>0$ be fixed. Choose the countable family $\{\phi_k\}_{k\in \mathbb{N}}$ of $C^1_c(\overline{\Omega})$ which is dense in $C_c(\overline{\Omega})$. Since we have $\sup_{i\in\mathbb{N}} \|\mu^{\varepsilon_i,\,\sigma_i}_t(\phi_k)\|_{BV(0,\,T)} <\infty$ from \eqref{4.1.2} and Lemma \ref{lem.5.1}, we may apply the compactness for BV functions $\{\mu^{\varepsilon_i,\,\sigma_i}_t(\phi_k)\}_{i}$ and thus, by the diagonal argument, there exist a subsequence independent of $k\in\mathbb{N}$ (labelled with the same index) and a family of BV functions $\{f_k\}_{k\in \mathbb{N}}$ such that
	\begin{align}
	\mu^{\varepsilon_i,\,\sigma_i}_t(\phi_k) &\xrightarrow[i\to\infty]{} f_k(t)\quad \text{for a.e. $t\in (0,\,T)$},\label{5.1.3} \\
	\left(\frac{d}{dt} \mu^{\varepsilon_i,\,\sigma_i}_t(\phi_k)\right)\mathcal{L}^1 &\xrightarrow[i\to\infty]{} Df_k\quad \text{as Radon measures on }(0,\,T) \label{5.1.4} 
	\end{align}
	Note that $Df_k$ is a signed measure on $(0,\,T)$ and a distributional derivative of $f_k$. Generally, the set of discontinuous points for functions of bounded variations is at most countable and thus we can choose a countable set $S$ such that, for all $k\in \mathbb{N}$, $f_k$ is continuous on $(0,\,T)\setminus S$.
	
	Next we claim that \eqref{5.1.3} holds on $(0,\,T)\setminus S$. To see this, we take an arbitrary $t\in (0,\,T)\setminus S$ and choose a sequence $\{t_l\}_{l\in \mathbb{N}}$ such that $t_l \to t$ and \eqref{5.1.3} holds for all $t_l$. Defining $[a,\,b]\subset\mathbb{R}$ by the closed interval between $a$ and $b$, then we have, by \eqref{5.1.4}
	\begin{equation}\label{5.1.6}
	\lim_{l\to \infty}\lim_{i\to\infty} \left(\frac{d}{dt} \mu^{\varepsilon_i,\,\sigma_i}_t(\phi_k) \mathcal{L}^1\right)([t_l,\,t]) = \lim_{l\to \infty}Df_k([t_l,\,t])= 0,
	\end{equation}
	for all $k\in\mathbb{N}$. Moreover, we have
	\begin{align}
	|f_k(t)- \mu^{\varepsilon_i,\,\sigma_i}_t(\phi_k)| &\leq |f_k(t)- f_k(t_l)|+|{f_k}(t_l)-\mu^{\varepsilon_i,\,\sigma_i}_{t_l}(\phi_k)|+|\mu^{\varepsilon_i,\,\sigma_i}_{t_l}(\phi_k)-\mu^{\varepsilon_i,\,\sigma_i}_t(\phi_k)| \nonumber\\
	&\leq |f_k(t)- f_k(t_l)|+|f_k(t_l)-\mu^{\varepsilon_i,\,\sigma_i}_{t_l}(\phi_k)|+\left|\left(\frac{d}{dt}\mu^{\varepsilon_i,\,\sigma_i}_t(\phi_k)\mathcal{L}^1\right)([t_l,\,t])\right|. \label{5.1.7}	
	\end{align}
	Then we first take $i\to\infty$ and, after that, take $l\to \infty$ to 
	conclude that \eqref{5.1.3} holds for all $t\in (0,\,T)\setminus S$.
	
	Now let $t$ be in $(0,\,T)\setminus S$. Since we have the estimate \eqref{4.1.1}, there exist a subsequence of $\{\varepsilon_i\}_{i\in\mathbb{N}}$ (labelled with the same index) and a Radon measure $\mu_t$ such that $\mu^{\varepsilon_i,\,\sigma_i}_t \rightharpoonup \mu_t$ as $i\to\infty$ on $\overline{\Omega}$. Hence we deduce that $\mu_t(\phi_k)=f_k(t)$ for all $k\in \mathbb{N}$. Since a family $\{\phi_k\}_{k\in \mathbb{N}} \subset C^1_c(\overline{\Omega})$ is dense in $C_c(\overline{\Omega})$, it holds that, for all $\phi \in C_c(\overline{\Omega})$ and all $t\in (0,\,T)\setminus S$,
	\begin{equation}\label{5.1.8}
	\mu^{\varepsilon_i,\,\sigma_i}_t(\phi) \to \mu_t(\phi) \quad \text{on $\overline{\Omega}$}.
	\end{equation}
	After taking another subsequence (labelled with the same index), we can also ensure that 
	\begin{equation}\label{5.1.9}
	\mu^{\varepsilon_i,\,\sigma_i}_0 \rightharpoonup \mu_0 \quad\text{as Radon measures on }\overline{\Omega}.
	\end{equation}
	Therefore we can deduce that there exist a subsequence $\{\mu^{\varepsilon_i,\,\sigma_i}_t\}_{i\in \mathbb{N}}$ and a Radon measure $\mu_t$ on $\overline{\Omega}$ such that, for all $t\in [0,\,T)\setminus S$, we have $\mu^{\varepsilon_i,\,\sigma_i}_t \rightharpoonup \mu_t$ as $i\to \infty$ on $\overline{\Omega}$. Note that the subsequence we obtain is same as the one selected in \eqref{5.1.3} and \eqref{5.1.4} because we can easily see that any subsequence of the sequence $\{\mu^{\varepsilon_i,\,\sigma_i}_t\}_{i}$ converges to the same limit $\mu_t$ for $t\in[0,\,T]\setminus S$. Since the set $S$ is countable, we may apply the further diagonal argument and then we can choose a further subsequence(labelled with the same index) such that  $\mu^{\varepsilon_i,\,\sigma_i}_t$ converges to some Radon measure $\mu_t$ as $i\to\infty$ on $\overline{\Omega}$ for all $t\in [0,\,T)$.
	
	Finally, we may conclude that there exist a subsequence and a Radon measure $\mu_t$ such that $\mu^{\varepsilon_i,\,\sigma_i}_t \rightharpoonup \mu_t$ on $\overline{\Omega}$ for all $t \in [0,\infty)$ from the fact $[0,\infty)= \bigcup_{n=1}^{\infty}[n-1,\,n)$ and by using the diagonal argument.
\end{proof}

\subsubsection{Convergence of the measures $\{\alpha^{\varepsilon,\,\sigma}\}_{\varepsilon>0,\,\sigma>0}$ and proof of Lemma \ref{thm3.2} and Lemma \ref{thm3.2.1} (Dirichlet boundary conditions)}\label{chara.2}
In this subsection, we show the existence of the convergent subsequence of a family of  $\{\alpha^{\varepsilon,\,\sigma}\}_{\varepsilon,\,\sigma>0}$ and we also prove Lemma \ref{thm3.2} and \ref{thm3.2.1}. Note that we assume that ``General assumptions" and ``Uniform upper bound on the boundary $\partial\Omega$" hold in this subsection.
\begin{proof}[Proof of Lemma \ref{thm3.2}]
	Let $\{\varepsilon'_j\}_{j\in\mathbb{N}}$ and $\{\sigma'_j\}_{j\in\mathbb{N}}$ be subsequences such that Lemma \ref{thm3.1} holds. First of all, we show that there exist a subsequence (labelled with the same index)  $\{\alpha^{\varepsilon'_j,\,\sigma'_j}\}_{j\in\mathbb{N}}$ and a Radon measure $\alpha$ on $\partial \Omega \times [0,\infty)$ such that $\alpha^{\varepsilon'_j,\,\sigma'_j} \rightharpoonup \alpha$ as $j\to\infty$ on $\partial \Omega \times [0,\infty)$.
	
	Let $T>0$ be a positive constant. Then, from Proposition \ref{prop.4.2}, we have for all $j\in\mathbb{N}$
	\begin{align}
	\alpha^{\varepsilon'_j,\,\sigma'_j}(\partial \Omega \times [0,\,T])&= \int_{0}^{T}\int _{\partial \Omega}\varepsilon'_j |\nabla_{\partial \Omega} u^{\varepsilon'_j,\,\sigma'_j}|^2\, d\mathcal{H}^{n-1}\,dt \nonumber\\
	&\leq 2\int_{0}^{T}\int _{\partial \Omega} \left(\frac{\varepsilon'_j |\nabla_{\partial \Omega} u^{\varepsilon'_j,\,\sigma'_j}|^2}{2} +\frac{W (u^{\varepsilon'_j,\,\sigma'_j} )}{\varepsilon'_j}\right) \, d\mathcal{H}^{n-1}\,dt \nonumber\\
	&\leq 2C_1\,T+C_0(T). \label{5.2.1}
	\end{align} 
	Thus we have the locally uniform boundedness, that is,  $\sup_{j\in\mathbb{N}}\alpha^{\varepsilon'_j,\,\sigma'_j}(K) <+\infty$ for any compact subset $K\subset \partial \Omega \times [0,\infty)$. Since this shows that we can apply the compactness theorem for Radon measures, we may conclude that there exist a subsequence (labelled with the same index) $\{\alpha^{\varepsilon'_j,\,\sigma'_j}\}_{j\in\mathbb{N}}$ and a Radon measure $\alpha$ on $\partial \Omega \times [0,\,\infty)$ such that $\alpha^{\varepsilon'_j,\,\sigma'_j} \rightharpoonup \alpha$ as $j\to \infty$. This completes the proof of the first claim.
	
	Secondly, we will prove the claims \eqref{3.1.10} and that $\boldv_b=0$ in $(L^2(\alpha))^n$ as in Lemma \ref{thm3.2}. We take any $j\in\mathbb{N}$. From \eqref{5.2.3}, \eqref{4.1.1}, and the fact that $\sigma'_j\in(0,\,1)$, we deduce that 
	\begin{equation}\label{5.2.4}
	\int_{0}^{\infty}\int_{\partial \Omega} |\boldv^{\varepsilon'_j,\,\sigma'_j}_b|^2\, d\alpha^{\varepsilon'_j,\,\sigma'_j} \leq \int_{0}^{\infty}\int_{\partial \Omega} \varepsilon'_j (\partial_t u^{\varepsilon'_j,\,\sigma'_j})^2 \,d\mathcal{H}^{n-1}\,dt \leq \sigma'_j\,D.
	\end{equation} 
	Therefore ($\alpha^{\varepsilon'_j,\,\sigma'_j}$, $\boldv^{\varepsilon'_j,\,\sigma'_j}_b$) is a measure-function pair which satisfies the $L^2$-uniform boundedness with respect to $j\in\mathbb{N}$. Since we have the convergence such that $\alpha^{\varepsilon'_j,\,\sigma'_j} \rightharpoonup \alpha$ as $j\to\infty$ on $\partial \Omega \times [0,\infty)$ and we can apply the theorem \cite[Theorem 4.4.2.]{Hutchinson}, we may conclude that there exist a subsequence (labelled with the same index) and a function ${\boldv}_b\in (L^2(\alpha,\,\partial \Omega \times [0,\infty)))^n$ such that
	\begin{align}\label{5.2.5}
	\lim_{j\to \infty} \int_{0}^{\infty}\int_{\partial \Omega} \boldg\cdot (\varepsilon'_j \partial_t u^{\varepsilon'_j,\,\sigma'_j} \nabla_{\partial \Omega} u^{\varepsilon'_j,\,\sigma'_j}) \, d\mathcal{H}^{n-1}\, dt &= -\lim_{j\to\infty} \int_{\partial \Omega \times [0,\infty)} \boldg\cdot \boldv^{\varepsilon'_j,\,\sigma'_j}_b \,d\alpha^{\varepsilon'_j,\,\sigma'_j}\nonumber\\
	&= - \int_{\partial \Omega \times [0,\infty)}\boldg\cdot {\boldv}_b\, d\alpha
	\end{align}  
	for all $\boldg\in (C_c(\partial \Omega \times [0,\infty)))^n$ and moreover, 
	\begin{equation}\label{5.2.6}
	\int_{\partial \Omega \times [0,\,\infty)} |{\boldv}_b|^2\, d\alpha \leq \liminf_{j\to\infty}\int_{0}^{\infty}\int_{\partial \Omega} \varepsilon'_j (\partial_t u^{\varepsilon'_j,\,\sigma'_j})^2 \,d\mathcal{H}^{n-1}\,dt
	\end{equation}
	holds. Here, from \eqref{4.1.1}, we have
	\begin{align}
	\limsup_{j\to\infty}\int_{0}^{\infty}\!\!\int_{\partial \Omega} \varepsilon'_j (\partial_t u^{\varepsilon'_j,\,\sigma'_j})^2 \,d\mathcal{H}^{n-1}\,dt&=\limsup_{j\to\infty}\left(\sigma'_j\,\int_{0}^{\infty}\!\!\int_{\partial\Omega}\frac{\varepsilon'_j}{\sigma'_j}\,\left(\partial_t u^{\varepsilon'_j,\,\sigma'_j}\right)^2\,d\mathcal{H}^{n-1}\,dt\right)\nonumber\\
	&\leq (\lim_{j\to\infty}\sigma'_j)\,D<\infty\nonumber\\
	&=0.\label{5.2.6.1}
	\end{align}
	Therefore, from \eqref{5.2.6} and \eqref{5.2.6.1}, we obtain that $\boldv_b=0$ in $(L^2(\alpha))^n$ on $\partial\Omega\times[0,\,\infty)$ and this completes the proof of Lemma \ref{thm3.2} and also the proof of the claim in Remark \ref{rem.thm3.2}.
	
	The only claim remaining to be proved is on the rectifiability of the limit measure $\mu_t$ and this can be done in Proposition \ref{prop.5.6}.
\end{proof}

Next we prove Lemma \ref{thm3.2.1} in the following.

\begin{proof}[Proof of Lemma \ref{thm3.2.1}]
	Since $C^\infty$ is dense in $W^{1,\,2}$ with respect to $W^{1,\,2}$-topology, it is sufficient to show the claim when $u^{\varepsilon'_j,\,\sigma'_j}$ is smooth on $\partial \Omega \times (0,\,T)$ for some $T>0$. From the assumption in Theorem \ref{thm3.2.1}, we can choose $t_1,\,t_2\in[0,\,T]$ and $\Gamma_1$ such that $0\leq t_1<t_2\leq T$ and $\Gamma_1$ is a non-empty connected component of $\partial\Omega$ and then we fix these quantities. First of all, from the estimate \eqref{3.1.8}, \eqref{4.1.1}, \eqref{4.1.5} and the definition of $w^{j}$, we have the following two inequalities;
	\begin{align}
	\sup_{j\in\mathbb{N}}\int_{t_1}^{t_2}\!\!\int_{\Gamma_1}|w^{j}|\,d\mathcal{H}^{n-1}\,dt & \leq \mathcal{H}^{n-1}(\Gamma_1)\,(t_2-t_1)<\infty, \label{5.2.7}\\
	\sup_{j\in\mathbb{N}}\int_{t_1}^{t_2}\!\!\int_{\Gamma_1}\left(|\nabla_{\partial \Omega}w^{j}|+|\partial_t w^{j}|\right)\,d\mathcal{H}^{n-1}\,dt& \leq \tilde{C}_1(t_1,\,t_2,\,D) <\infty, \label{5.2.8}
	\end{align}
	where $\tilde{C}_1(t_1,\,t_2,\,D)$ is some positive constant depending only on $\Omega,\,t_1,\,t_2$, and $D$. Here $D>0$ is as in Propositon \ref{prop.4.1}. Indeed, we can show  \eqref{5.2.8} in the following way; for any $j\in\mathbb{N}$ and $0<t_1<t_2< T$, we have, from Cauchy-Schwarz inequality,
	\begin{align}
	\int_{t_1}^{t_2}\!\!\int_{\Gamma_1}\left(|\nabla_{\partial \Omega}w^{j}|+|\partial_t w^{j}|\right)\,d\mathcal{H}^{n-1}\,dt& = \int_{t_1}^{t_2}\!\!\int_{\Gamma_1}\sqrt{2W(u^{\varepsilon'_j,\,\sigma'_j})}\left(|\nabla_{\partial \Omega} u^{\varepsilon'_j,\,\sigma'_j}|+ |\partial_t u^{\varepsilon'_j,\,\sigma'_j}|\right)\,d\mathcal{H}^{n-1}\,dt \nonumber\\
	&\leq 2\int_{t_1}^{t_2}\!\!\int_{\Gamma_1}\left(\frac{\varepsilon_j|\nabla_{\partial \Omega} u^{\varepsilon'_j,\,\sigma'_j}|^2}{2}+\frac{W(u^{\varepsilon'_j,\,\sigma'_j})}{\varepsilon_j}\right)\,d\mathcal{H}^{n-1}\,dt\nonumber\\
	&\quad \qquad \quad + \frac{1}{2}\int_{t_1}^{t_2}\!\!\int_{\Gamma_1}\varepsilon_j(\partial_t u^{\varepsilon'_j,\,\sigma'_j})^2\,d\mathcal{H}^{n-1}\,dt \nonumber\\
	&\leq 2C_1(t_2-t_1)+2C_1+\frac{1}{2}D, \label{5.2.9}
	\end{align} 
	where $C_1$ is as in Proposition \ref{prop.4.4}. Hence,  $\{w^{j}\}_{j\in\mathbb{N}}$ is a bounded sequence in $BV(\Gamma_1 \cap [t_1,\,t_2])$. 
	
	Since $\partial \Omega$ is a smooth $(n-1)$-dimensional embedded manifold in $\mathbb{R}^n$, we can choose one atlas $(V_{a},\,\varphi_a)$ of $\partial\Omega$ at $p\in \Gamma_1$ such that $\varphi_a: V_a\to \varphi_a(V_a)\subset\mathbb{R}^{n-1}$ is a smooth diffeomorphism. Replacing $\Gamma_1$ with one of the connected components of $\Gamma_1\cap V_a$, if necessary, we may assume that $\Gamma_1\subset V_a$ holds. Then, we can show that $\{w^j(\varphi^{-1}_a)\}_{j\in\mathbb{N}}$ is also a bounded sequence in $BV(\varphi_a(\Gamma_1)\times[t_1,\,t_2])$. Thus, we may apply the compactness theorem for $BV$ functions to $\{w^{j}(\varphi_a^{-1})\}_{j\in\mathbb{N}}$ and then we have that there exist a subsequence $\{w^{j_i}\}_{i\in\mathbb{N}}$ and $\tilde{w}^{\infty}\in BV(\varphi_a(\Gamma_1)\times[t_1,\,t_2])$ such that $w^{j_i}(\varphi^{-1}_a) \rightarrow \tilde{w}^{\infty}$ in $L^1$-topology as $i\to\infty$. Defining $w^{\infty}$ by $\tilde{w}^{\infty}(\varphi_a)$ and taking another subsequence (labelled with the same index), we have that $w^{j_i}(\varphi^{-1}_a) \rightarrow w^{\infty}(\varphi^{-1})$ $(\mathcal{L}^{n-1}\otimes\mathcal{L}^1)$-a.e. in $\varphi_a(\Gamma_1)\times[t_1,\,t_2]$ as $i\to\infty$. Then, setting $u^{j_i}\coloneqq u^{\varepsilon'_{j_i},\,\sigma'_{j_i}}$ and $u^{\infty}\coloneqq\Phi^{-1}\circ w^{\infty}$, we have that $u^{j_i}(\varphi^{-1}_a)\rightarrow u^{\infty}(\varphi^{-1}_a)$ $(\mathcal{L}^{n-1}\otimes\mathcal{L}^1)$-a.e. in $\varphi_a(\Gamma_1)\times[t_1,\,t_2]$. These imply that $w^{j_i}\to w^{\infty}$ and $u^{j_i} \to u^{\infty}$  $(\mathcal{H}^{n-1}\otimes\mathcal{L}^1_t)$-a.e. on $\Gamma_1\times [t_1,\,t_2]$. 
	
	Moreover, from a priori estimate \eqref{4.1.5} and the dominated convergence theorem, we have
	\begin{equation}\label{5.2.10}
	0\leq \int_{t_1}^{t_2}\!\!\!\!\int_{\Gamma_1}W(u^{\infty})\,d\mathcal{H}^{n-1}\,dt \xleftarrow[i\to\infty]{}  \int_{t_1}^{t_2}\!\!\!\!\int_{\Gamma_1}W(u^{j_i})\,d\mathcal{H}^{n-1}\,dt \leq \varepsilon_{j_i} \tilde{C}(t_1,\,t_2) \xrightarrow[i\to\infty]{} 0,
	\end{equation}
	and thus we conclude that $u^{\infty}=\pm 1$ $(\mathcal{H}^{n-1}\otimes\mathcal{L}^1_t)$-a.e. on $\Gamma_1\times [t_1,\,t_2]$ and this yields that $w^{\infty}=\pm \frac{2}{3}$ $(\mathcal{H}^{n-1}\otimes\mathcal{L}^1_t)$-a.e. on $\Gamma_1\times [t_1,\,t_2]$.
	
	Now, since $\Gamma_1 \neq \emptyset$ is bounded and connected, from Poincar\'e-Wirtinger inequality on manifolds (see Lemma \ref{app.lem} in Appendix B of Section \ref{appendixPoincareInequality}), it may follow that there exists a constant $C>0$ depending only on $n$ such that
	\begin{equation}\label{5.2.11}
	\|w^{j_i}(t)-w^{j_i}_{\Gamma_1}(t)\|_{L^{1}(\Gamma_1)} \leq C \|\nabla_{\partial \Omega} w^{j_i}(t)\|_{L^1(\Gamma_1)} 
	\end{equation}
	for any $i\in\mathbb{N}$ and any $t\in[t_1,\,t_2]$, where we set
	\begin{equation}\label{5.2.12}
	w^{j_i}_{\Gamma_1}(t)\coloneqq \frac{1}{\mathcal{H}^{n-1}(\Gamma_1)} \int_{\Gamma_1}w^{j_i}(t)\,d\mathcal{H}^{n-1}.
	\end{equation}
	From Cauchy-Schwarz inequality and \eqref{5.2.11}, we may obtain the following calculation;
	\begin{align}
	\int_{t_1}^{t_2}\!\!\!\!\int_{\Gamma_1}|w^{j_i}-w^{j_i}_{\Gamma_1}|\,d\mathcal{H}^{n-1}\,dt &\leq  C\int_{t_1}^{t_2}\!\!\!\!\int_{\Gamma_1} |\nabla_{\partial \Omega}w^{j_i}|\,d\mathcal{H}^{n-1}\,dt \nonumber\\
	&\leq \delta\,C\int_{t_1}^{t_2}\!\!\!\!\int_{\Gamma_1}\frac{W(u^{\varepsilon'_j,\,\sigma'_j})}{\varepsilon_j}\,d\mathcal{H}^{n-1}\,dt+ \frac{C}{2\delta} \alpha^{\varepsilon'_{j_i},\,\sigma'_{j_i}}(\Gamma_1\times[t_1,\,t_2]) \nonumber\\
	&\leq \delta\,C\,C_1(t_2-t_1+1)+ \frac{C}{2\delta} \alpha^{\varepsilon'_{j_i},\,\sigma'_{j_i}}(\Gamma_1\times[t_1,\,t_2]), \label{5.2.13}
	\end{align}
	where $\delta>0$ independent of $i\in\mathbb{N}$ will be chosen later. Hence, from \eqref{5.2.13} and the triangle inequality, we obtain
	\begin{equation}\label{5.2.14}
	\int_{t_1}^{t_2}\!\!\!\!\int_{\Gamma_1}|w^{j_i}|\,d\mathcal{H}^{n-1}\,dt- \int_{t_1}^{t_2}\!\!\left|\int_{\Gamma_1} w^{j_i}\,d\mathcal{H}^{n-1}\right|\,dt \leq \delta\,C^{\prime}+ \frac{C}{2\delta} \alpha^{\varepsilon'_{j_i},\,\sigma'_{j_i}}(\Gamma_1\times[t_1,\,t_2]),
	\end{equation}
	where we put $C^{\prime}\coloneqq C\,C_1(t_2-t_1+1)$. Here, from the assumption \eqref{3.1.11.2}, we can choose the constant $c_0>0$ such that
	\begin{equation}\label{5.2.15}
	\frac{2}{3}\mathcal{H}^{n-1}(\Gamma_1)\,(t_2-t_1)-\liminf_{i\to \infty} \int_{t_1}^{t_2}\!\!\left|\int_{\Gamma_1} w^{j_i}\,d\mathcal{H}^{n-1}\right|\,dt>c_0>0 .
	\end{equation}
	Therefore, from Lemma \ref{thm3.2} and \eqref{5.2.15}, taking the limit ($i\to \infty$) in \eqref{5.2.14}, we have
	\begin{align}
	\delta\,C^{\prime}+ \frac{C}{2\delta} \alpha(\Gamma_1\times[t_1,\,t_2]) &\geq  \delta\,C^{\prime}+ \frac{C}{2\delta} \limsup_{i\to\infty}\alpha^{\varepsilon'_{j_i},\,\sigma'_{j_i}}(\Gamma_1\times[t_1,\,t_2]) \nonumber\\
	&\geq \frac{2}{3}\mathcal{H}^{n-1}(\Gamma_1)\,(t_2-t_1)- \liminf_{i\to \infty} \int_{t_1}^{t_2}\!\!\left|\int_{\Gamma_1} w^{j_i}\,d\mathcal{H}^{n-1}\right|\,dt \nonumber\\
	&> c_0>0. \label{5.2.16}
	\end{align} 
	Taking $\delta$ such that $0<\delta<\frac{c_0}{2C^{\prime}}$, we may conclude that the limit measure $\alpha$ is positive on $\Gamma_1\times[t_1,\,t_2]$. 
\end{proof}

\subsubsection{First variations of associated varifolds and proof of Lemma \ref{thm3.3} (Dirichlet boundary conditions)}\label{chara.3}
In the previous subsection, we have already proved that there exists a convergent subsequence $\{\mu^{\varepsilon'_i,\,\sigma'_j}_t\}_{i,\,j\in\mathbb{N}}$ for all $t\geq 0$. Then, in this subsection, we mainly discuss the first variation of the associated varifold with $\mu^{\varepsilon'_i,\,\sigma'_j}_t$ and the proof of Lemma \ref{thm3.3}. Note that the first variation of a varifold plays a very important role to prove Lemma \ref{thm3.3}. Note that, through this subsection, we assume that ``General assumptions", ``Vanishing hypothesis for the discrepancy measure", and ``Uniform upper bound on the boundary $\partial\Omega$" in Subsection \ref{exis.assmp.diri} hold.

First of all, we associate a varifold with each $\mu^{\varepsilon'}_t$ as follows.
\begin{definition}\label{def.5.3}
	Let $\{u^{\varepsilon,\,\sigma}\}_{\varepsilon,\,\sigma>0}$ be the solutions to the equations \eqref{1.1.1} and $\mu^{\varepsilon
		,\,\sigma}_t$ be as in \eqref{1.1.5}. Then for $\psi \in C_c(G_{n-1}(\overline{\Omega}))$ and any $t\geq 0$, define
	\begin{equation}\label{5.3.1}
	V^{\varepsilon,\,\sigma}_t(\psi)\coloneqq \int_{\Omega \cap \{|\nabla u^{\varepsilon,\,\sigma}(\cdot,\,t)| \neq 0\}} \psi(x,\,\textbf{I}-\textbf{a}^{\varepsilon,\,\sigma} \otimes \textbf{a}^{\varepsilon,\,\sigma})\, d\mu^{\varepsilon,\,\sigma}_t(x),
	\end{equation}
	where $\textbf{a}^{\varepsilon,\,\sigma}\coloneqq\frac{\nabla u^{\varepsilon,\,\sigma}}{|\nabla u^{\varepsilon,\,\sigma
		}|}$.
\end{definition}
Note that from the definition, we can obtain  $\|V^{\varepsilon,\,\sigma}_t\|=\mu^{\varepsilon,\,\sigma}_t \lfloor_{\{|\nabla u^{\varepsilon,\,\sigma}(\cdot,\,t)|\neq 0\}}$, hence, by the definition of the first variation, we may derive a formula for the first variation of $V^{\varepsilon,\,\sigma}_t$ up to the boundary.
\begin{lemma}\label{lem.5.4}
	Let $\{u^{\varepsilon,\,\sigma}\}_{\varepsilon,\,\sigma>0}$ and $\mu^{\varepsilon,\,\sigma}_t$ be as in Definition \ref{def.5.3}. Then, for any $\varepsilon>0$, $\sigma>0$, $t\geq 0$ and all $\boldg\in (C^1_c(\overline{\Omega}))^n$, we have
	\begin{align}\label{5.3.2}
	\delta V^{\varepsilon,\,\sigma}_t(\boldg) &= \int _{\Omega} (\boldg\cdot \nabla u^{\varepsilon,\,\sigma}) \left(\varepsilon \Delta u^{\varepsilon,\,\sigma} -\frac{W'(u^{\varepsilon,\,\sigma})}{\varepsilon}\right) \, dx
	+  \int_{\Omega\cap\{ |\nabla u^{\varepsilon,\,\sigma}|\not=0 \}} \nabla \boldg : (\textbf{a}^{\varepsilon,\,\sigma}\otimes \textbf{a}^{\varepsilon,\,\sigma}) \,d\xi_t^{\varepsilon,\,\sigma} \nonumber\\
	&\quad + \int _{\partial \Omega} (\boldg\cdot \boldnu) \left( \frac{\varepsilon |\nabla u^{\varepsilon,\,\sigma}|^2}{2}+\frac{W(u^{\varepsilon,\,\sigma})}{\varepsilon}\right) \, d\mathcal{H}^{n-1}
	- \int _{\partial \Omega} \varepsilon (\boldg\cdot \nabla u^{\varepsilon,\,\sigma}) \frac{\partial u^{\varepsilon,\,\sigma}}{\partial \boldnu} \, d\mathcal{H}^{n-1} \nonumber\\
	&\qquad - \int_{\Omega\cap\{ |\nabla u^{\varepsilon,\,\sigma}|=0 \}} \nabla \boldg : \textbf{I} \frac{W(u^{\varepsilon,,\sigma})}{\varepsilon} \, dx.
	\end{align}	
\end{lemma}
\begin{proof}
	From the definition of the first variation, we have
	\begin{equation}\label{5.3.3}
	\delta V^{\varepsilon,\,\sigma}_t (\boldg)= \int_{\Omega \cap \{|\nabla u^{\varepsilon,\,\sigma}(\cdot,\,t)|\neq 0 \}} \nabla \boldg(x):(\textbf{I}-\textbf{a}^{\varepsilon,\,\sigma}\otimes \textbf{a}^{\varepsilon,\,\sigma})\,d\mu^{\varepsilon,\,\sigma}_t.
	\end{equation}  
	Using integration by part, we have
	\begin{align}\label{5.3.4}
	\int_{\Omega} (\nabla \boldg:\textbf{I})\frac{\varepsilon|\nabla u^{\varepsilon,\,\sigma}|^2}{2} \,dx &= \int_{\partial \Omega} \left(\boldg\cdot \boldnu \frac{\varepsilon|\nabla u^{\varepsilon,\,\sigma}|^2}{2} -\varepsilon (\boldg\cdot \nabla u^{\varepsilon,\,\sigma}) \frac{\partial u^{\varepsilon,\,\sigma}}{\partial \boldnu}\right)\, d\mathcal{H}^{n-1} \nonumber \\ 
	&\quad + \int_{\Omega \cap \{|\nabla u^{\varepsilon,\,\sigma}(\cdot,\,t)|\neq 0\}} \nabla \boldg : (\textbf{a}^{\varepsilon,\,\sigma}\otimes \textbf{a}^{\varepsilon,\,\sigma})\varepsilon |\nabla u^{\varepsilon,\,\sigma}|^2 \,dx \nonumber\\
	&\quad \qquad +\int_{\Omega} (\boldg\cdot \nabla u^{\varepsilon,\,\sigma})\varepsilon \Delta u^{\varepsilon,\,\sigma}\,dx.
	\end{align}
	Similary by using integration by part, we get
	\begin{align}\label{5.3.5}
	\int_{\Omega \cap \{|\nabla u^{\varepsilon,\,\sigma}|\neq 0\}} \nabla \boldg : \textbf{I} \frac{W(u^{\varepsilon,\,\sigma})}{\varepsilon}\,dx &= \int_{\partial \Omega} (\boldg\cdot \boldnu)\frac{W(u^{\varepsilon,\,\sigma})}{\varepsilon}\,d\mathcal{H}^{n-1} - \int_{\Omega} (\boldg\cdot \nabla u^{\varepsilon,\,\sigma})\frac{W'(u^{\varepsilon,\,\sigma})}{\varepsilon}\, dx \nonumber\\
	&\quad -\int_{\Omega \cap \{|\nabla u^{\varepsilon,\,\sigma}|=0\}} \nabla \boldg\cdot I \frac{W(u^{\varepsilon,\,\sigma})}{\varepsilon}\,dx.
	\end{align}
	By substituting \eqref{5.3.4} and \eqref{5.3.5} into \eqref{5.3.3} and recalling the definition of $\xi^{\varepsilon,\,\sigma}_t$ which is equal to $\varepsilon|\nabla u^{\varepsilon,\,\sigma}|^2\,\mathcal{L}^n-\mu^{\varepsilon,\,\sigma}_t$ where $\mathcal{L}^n$ is the $n$-dimensional Lebesgue measure, we obtain \eqref{5.3.2}. 
\end{proof}

\begin{remark}\label{rem.5.5.0}
	By recalling again the definition of $\xi^{\varepsilon,\,\sigma}_t$, we have 
	\begin{equation}\label{5.3.5.1}
	\frac{W(u^{\varepsilon,\,\sigma})}{\varepsilon}\,\mathcal{L}^n=\frac{\varepsilon|\nabla u^{\varepsilon,\,\sigma}|^2}{2}\,\mathcal{L}^n-\xi^{\varepsilon,\,\sigma}_t.
	\end{equation}
	Thus, we can rewrite the last term in the left-hand side in \eqref{5.3.2} as follows;
	\begin{equation}\label{5.3.5.2}
	-\int_{\Omega \cap \{|\nabla u^{\varepsilon,\,\sigma}|=0\}} \nabla \boldg\cdot I \frac{W(u^{\varepsilon,\,\sigma})}{\varepsilon}\,dx=\int_{\Omega \cap \{|\nabla u^{\varepsilon,\,\sigma}|=0\}} \divergence \boldg\,d\xi^{\varepsilon,\,\sigma}_t.
	\end{equation}
\end{remark}

\begin{proposition}\label{prop.5.5}
	Let $\{\varepsilon'_j\}_{j\in\mathbb{N}}$ and $\{\sigma'_j\}_{j\in\mathbb{N}}$ be such that Lemma \ref{thm3.1} and \ref{thm3.2} hold and let $\{u^{\varepsilon'_j,\,\sigma'_j}\}_{j\in\mathbb{N}}$ be the solutions to the equations \eqref{1.1.1}. Then, defining $c(t)$ by 
	\begin{align}\label{5.3.6}
	c(t)\coloneqq\liminf_{j\to\infty} & \left( \int_\Omega |\nabla u ^{\varepsilon'_j,\,\sigma'_j}|\, \left|\varepsilon'_j \Delta u^{\varepsilon'_j,\,\sigma'_j} -\frac{W'(u^{\varepsilon'_j,\,\sigma'_j})}{\varepsilon'_j} \right| \, dx \right.+ \int _{\partial \Omega}  \frac{\varepsilon'_j |\nabla u^{\varepsilon'_j,\,\sigma'_j}|^2}{2}+\frac{W(u^{\varepsilon'_j,\,\sigma'_j})}{\varepsilon'}   \, d\mathcal{H}^{n-1} \nonumber\\
	&\quad \qquad \quad \qquad \left. +\int_{\partial \Omega} \varepsilon'_j |\nabla u^{\varepsilon'_j,\,\sigma'_j}| \left|\frac{\partial u^{\varepsilon'_j,\,\sigma'_j}}{\partial \boldnu}\right|\,d\mathcal{H}^{n-1} \right),
	\end{align}
	we have $c\in L^1_{loc}([0,\,\infty))$ and $c(t)<\infty$ for a.e. $t\in[0,\,\infty)$.
\end{proposition}
\begin{proof}
	It is sufficient to show that \eqref{5.3.6} holds for a.e. $t\in [0,\,T]$ for every $T>0$ because we have $\bigcup_{l=1}^{\infty}[0,\,l]=[0,\,\infty)$. Let $T>0$ be arbitrary. For simplicity, let $(\varepsilon,\,\sigma)$ denote the parameters $(\varepsilon'_j,\,\sigma'_j)$ in this proof. We set 
	\begin{align}
	I^{\varepsilon,\,\sigma}_1 &\coloneqq \int_\Omega |\nabla u ^{\varepsilon,\,\sigma}|\, \left|\varepsilon \Delta u^{\varepsilon,\,\sigma} -\frac{W'(u^{\varepsilon,\,\sigma})}{\varepsilon}\right| \, dx \label{5.3.7}\\
	I^{\varepsilon,\,\sigma}_2 &\coloneqq \int _{\partial \Omega} \left( \frac{\varepsilon |\nabla u^{\varepsilon,\,\sigma}|^2}{2}+\frac{W(u^{\varepsilon,\,\sigma})}{\varepsilon}  \right)  \, d\mathcal{H}^{n-1} \label{5.3.8}\\
	I^{\varepsilon,\,\sigma}_3 &\coloneqq \int_{\partial \Omega}  \varepsilon |\nabla u^{\varepsilon,\,\sigma}|\left|\frac{\partial u^{\varepsilon,\,\sigma}}{\partial \boldnu}\right|\,d\mathcal{H}^{n-1}. \label{5.3.9}
	\end{align}
	From \eqref{4.1.1} and \eqref{4.1.2} and by using Cauchy-Schwarz inequality, we have
	\begin{align}
	\left(\int_{0}^{T} I^{\varepsilon,\,\sigma}_1\,dt\right)^2 &\leq \left(\int_{0}^{T}\int_{\Omega} \varepsilon |\nabla u^{\varepsilon,\,\sigma}|^2 \,dx\,dt \right) \left(\int_{0}^{T} \int_{\Omega} \varepsilon\left(\Delta u^{\varepsilon,\,\sigma} -\frac{W'(u^{\varepsilon,\,\sigma})}{\varepsilon}\right)^2 \,dx\,dt \right) \nonumber\\
	&\leq 2D\,\int_{0}^{T}\mu^{\varepsilon,\,\sigma}_t(\Omega)\,dt \leq  2D^2\,T, \label{5.3.10}
	\end{align} 
	and, from the assumption of uniform upper bound \eqref{3.1.9.2} and \eqref{4.1.5},
	\begin{equation}\label{5.3.11}
	\int_{0}^{T} I^{\varepsilon,\,\sigma}_2\,dt\leq C_1T+C_0(T),
	\end{equation}
	and finally, from \eqref{3.1.9.2} and \eqref{4.1.5},
	\begin{align}
	\left(\int_{0}^{T}I^{\varepsilon,\,\sigma}_3\,dt\right)^2 &\leq \int_{0}^{T}\int_{\partial \Omega} \varepsilon |\nabla u^{\varepsilon,\,\sigma}|^2\,d\mathcal{H}^{n-1}\,dt \leq 2C_1T+2C_0(T) \label{5.3.12}
	\end{align}
	for any $\varepsilon,\,\sigma>0$. Then by using Fatou's lemma we have
	\begin{equation}
	\int_{0}^{T} c(t)\,dt \leq \liminf_{\varepsilon,\,\sigma \to 0} \int_{0}^{T} (I^{\varepsilon,\,\sigma}_1+I^{\varepsilon,\,\sigma}_2+I^{\varepsilon,\,\sigma}_3)\,dt \leq C_2(T,\,D)<\infty, \label{5.3.13}
	\end{equation}
	where $C_2(T,\,D)\coloneqq\sqrt{2}D\,\sqrt{T}+C_1\,T+C_0(T)+\sqrt{2}\sqrt{C_1\,T+C_0(T)}$. This shows that $c\in L^1_{loc}([0,\,\infty))$ and thus $c(t)<\infty$ holds for a.e. $t\in[0,\,T]$. This completes the proof. 
\end{proof}
Next we will show that $\mu_t$ is actually $(n-1)$-rectifiable measure on $\overline{\Omega}$ for a.e. $t\geq 0$ and a proper subsequence of the associated varifolds $\{V^{\varepsilon,\,\sigma}_t\}_{\varepsilon,\,\sigma>0}$ converges uniquely to the varifold $V_t$ associated with $\mu_t$.
\begin{proposition}\label{prop.5.6}
	For a.e. $t\geq 0$, $\mu_t$ is $(n-1)$-rectifiable on $\overline{\Omega}$ and any convergent subsequence $\{V^{\varepsilon''_j,\,\sigma'_j}\}_{j\in\mathbb{N}}$ of $\{V^{\varepsilon'_j,\,\sigma'_j}_t\}_{j\in\mathbb{N}}$, where $\{\varepsilon'_j\}_{j\in\mathbb{N}}$ and $\{\sigma'_j\}_{j\in\mathbb{N}}$ are such that Lemma \ref{thm3.1} and \ref{thm3.2} hold, converges to the unique $(n-1)$-rectifiable varifold $V_t$ associated with $\mu_t$. Moreover, we have
	\begin{equation}\label{5.3.14}
	\|\delta V_t\|(\overline{\Omega})<\infty,\quad
	\int_{0}^{T}\|\delta V_t\|(\overline{\Omega})\,dt<\infty
	\end{equation}
	for a.e. $t\geq0$ and any $T>0$ respectively.
\end{proposition}
\begin{proof}
	Recalling the assumption in Subsection \ref{exis.assmp.diri}, that is, the vanishing of the discrepancy measure $\xi^{\varepsilon'_j,\,\sigma'_j}_t$ up to the boundary $\partial\Omega$, we now have that, for all $\phi \in C_c(\overline{\Omega})$,
	\begin{equation}\label{5.3.16}
	\lim_{j \to \infty}\int_{\Omega} \phi\,d|\xi^{\varepsilon'_j,\,\sigma'_j}_t|=0
	\end{equation}
	holds for a.e. $t\in[0,\,\infty)$. 
	
	First of all, we show that the limit measure $\sigma_0^{-1}\mu_t$ is $(n-1)$-integral in $\Omega$ for a.e. $t\geq0$. To do this, we refer to the result by Tonegawa in \cite{Tonegawa02}. The author proved the integrality of the measure $\mu_t$ under the essential conditions from (1) to (7) appearing in ``PROOF OF INTEGRALITY" in \cite{Tonegawa02}. Thus, it is sufficient to clarify that our ingredients satisfy the conditions from (1) to (7) written in that paper. 
	
	From the assumptions we impose in the present paper, we already have the conditions. Indeed, we can show, what we call, the monotonicity formula for the measure $\mu^{\varepsilon,\,\sigma}_t$ and its proof is given in Proposition \ref{monotonicityFormula} of Appendix B (see also \cite{Ilmanen01, MiTo}). From the assumption that there exists a constant $c_0>0$ such that
	\begin{equation}
		\sup_{\varepsilon,\,\sigma>0}\xi^{\varepsilon,\,\sigma}_t \leq c_0
	\end{equation}
	for any $t>0$ (see ``Vanishing hypothesis of the discrepancy measure" in Section \ref{exisdiri}),  Then with this monotonicity formula, as well as \cite[Corollary 6.1]{TaTo}, we obtain the estimate
	\begin{equation}
		\mathcal{H}^{n-1}(\spt\mu_t \cap \Omega_m) \leq C(m)\,\liminf_{s\to t}\mu_s(\Omega) <\infty
	\end{equation}
	for some constant $C(m)>0$ and every $m\in\mathbb{N}$. Therefore, we may confirm that all the conditions shown by Tonegawa in \cite{Tonegawa02} are satisfied, we may conclude that, for every $m\in\mathbb{N}$, there exist a $(n-1)$-rectifiable set $M_t\subset \Omega_m$ and a function $\theta_t \in L^1_{loc}(\mathcal{H}^{n-1};\,M_t)$ such that $\theta_t$ takes the values on $\mathbb{N}$ and, for a.e. $t\geq0$,
	\begin{equation}\label{5.3.17}
		\sigma_0^{-1} \mu_t(U)= \int_{M_t\cap U}\theta_t(x)\,d\mathcal{H}^{n-1}(x)
	\end{equation}
	for any measurable set $U \subset \Omega_m$. Moreover \eqref{5.3.17} implies that for a.e. $t\geq 0$ and each $m\in\mathbb{N}$, $\spt\mu_t \cap \Omega_m = \spt\mathcal{H}^{n-1} \lfloor_{M_t}$.
	
	Taking the above arguments into account, we may show that $\mathcal{H}^{n-1}(\spt\mu_t)<\infty$ as follows: From the definition of $\Omega_m$, we have $\Omega_m \subset \Omega_{m+1}$ for each $m\in\mathbb{N}$. Thus from the continuity of Radon measures, we have
	\begin{align}
		\mathcal{H}^{n-1}(\spt\mu_t) &= \mathcal{H}^{n-1}(\spt\mu_t \cap \Omega)+  \mathcal{H}^{n-1}(\spt\mu_t \cap \partial \Omega) \nonumber\\
		&\leq \limsup_{m\to\infty}\mathcal{H}^{n-1}(\spt\mu_t \cap \Omega_m) + \mathcal{H}^{n-1}(\partial \Omega) \label{5.3.18}
	\end{align}
	for any fixed $m\in\mathbb{N}$ and a.e. $t\geq0$. Since $\partial \Omega$ is compact, then it is sufficient to prove that $\limsup_{m\to\infty}\mathcal{H}^{n-1}(\spt\mu_t \cap \Omega_m) < \infty$. From the integrality of $\sigma_0^{-1}\mu_t\lfloor_{ \Omega_m}$ for each $m\in\mathbb{N}$ and \eqref{5.3.17}, we obtain
	\begin{equation}
		\mathcal{H}^{n-1}(\spt\mu_t \cap \Omega_m) = \mathcal{H}^{n-1}(M_t) \leq \int_{M_t}\theta_t(x)\,d\mathcal{H}^{n-1}(x) = \sigma_0^{-1}\mu_t(M_t \cap \Omega_m) \leq \sigma_0^{-1} \mu_t(\Omega) <\infty. \label{5.3.18.1}
	\end{equation}
	Note that, in \eqref{5.3.18.1}, we have used the property of $\theta_t\geq1$ for $\mathcal{H}^{n-1}$-a.e. in $M_t$. Thus, from \eqref{5.3.18}, \eqref{5.3.18.1}, and the arbitrariness of $m\in\mathbb{N}$, we obtain that
	\begin{equation}
		\limsup_{m\to\infty}\mathcal{H}^{n-1}(\spt\mu_t \cap \Omega_m) \leq \sigma_0^{-1} \mu_t(\Omega) <\infty
	\end{equation}
	and thus we conclude that $\mathcal{H}^{n-1}(\spt\mu_t) < \infty$.
	
	In addition, from a standard measure theory (see, for instance, \cite[Theore 3.2, Chapter 1]{Simon}), we have 
	\begin{equation}\label{5.3.19.1}
	\mu_t\left(\left\{x\in\spt\mu_t \mid \limsup_{r\downarrow 0}\frac{\mu_t(B_r(x))}{\omega_{n-1}r^{n-1}}\leq s\right\}\right)\leq 2^{n-1}s\mathcal{H}^{n-1}(\spt\mu_t)<\infty,
	\end{equation}
	for any $0<s<\infty$ and a.e. $t\geq 0$. Letting $s$ go to zero, we have
	\begin{equation}\label{5.3.19.2}
	\mu_t\left(\left\{x\in\spt\mu_t \mid \limsup_{r\downarrow 0}\frac{\mu_t(B_r(x))}{\omega_{n-1}r^{n-1}}=0\right\}\right)=0,
	\end{equation}
	and thus we obatin
	\begin{equation}\label{5.3.19.3}
	\mu_t\equiv\mu_t\lfloor_{\left\{x\mid\limsup_{r\downarrow 0}\frac{\mu_t(B_r(x))}{\omega_{n-1}r^{n-1}}>0\right\}}\quad\text{on $\overline{\Omega}$ for a.e. $t$.}
	\end{equation}
	Now we prove that, for a.e. $t\geq 0$, the total variation $\|\delta V_t\|$ of the first variation of $V_t$ is actually a Radon measure on $\overline{\Omega}$. Note that we only have to show this for any $T>0$ and a.e. $0\leq t \leq T$.
	
	Let $T>0$ be arbitrary and we fix $0\leq t \leq T$ such that \eqref{5.3.6}, \eqref{5.3.16}, and \eqref{5.3.19.3} hold and the boundedness $\mathcal{H}^{n-1}(\spt \mu_t) < \infty$ is valid. Let $\{V^{\varepsilon''_j,\,\sigma'_j}_t\}_{j\in\mathbb{N}}$ be any convergent subsequence of $\{V^{\varepsilon'_j,\,\sigma'_j}_t\}_{j\in\mathbb{N}}$ and let $\tilde{V}_t$ denote its limit. From \eqref{5.3.2}, \eqref{5.3.16} and the varifold convergence of $\{V^{\varepsilon''_j,\,\sigma'_j}_t\}_{j\in\mathbb{N}}$, we have for any $\boldg\in (C^1_c(\overline{\Omega}))^n$
	\begin{align}
	\delta \tilde{V}_t(\boldg) = \lim_{j\to\infty}\delta V^{\varepsilon''_j,\,\sigma'_j}_t(\boldg)
	&=\lim_{j\to\infty} \int_{\Omega} \nabla \boldg(x) : (I-a^{\varepsilon''_j,\,\sigma'_j}\otimes a^{\varepsilon''_j,\,\sigma'_j})\, d\mu^{\varepsilon''_j,\,\sigma'_j}_t(x) \nonumber\\
	&=\lim_{j\to\infty}\left(\int _{\Omega} (\boldg\cdot \nabla u^{\varepsilon''_j\,\sigma'_j}) \left(\varepsilon''_j \Delta u^{\varepsilon''_j,\,\sigma'_j} -\frac{W'(u^{\varepsilon''_j,\,\sigma'_j})}{\varepsilon''_j}\right) \, dx \right.\nonumber\\
	&\quad \qquad +\int _{\partial \Omega} (\boldg\cdot \boldnu) \left( \frac{\varepsilon''_j |\nabla u^{\varepsilon''_j,\,\sigma'_j}|^2}{2}+\frac{W(u^{\varepsilon''_j,\,\sigma'_j})}{\varepsilon''_j}\right) \, d\mathcal{H}^{n-1} \nonumber\\
	&\quad \qquad \quad \left.-\int _{\partial \Omega} \varepsilon''_j (\boldg\cdot \nabla u^{\varepsilon''_j,\,\sigma'_j}) \frac{\partial u^{\varepsilon''_j,\,\sigma'_j}}{\partial \boldnu} \, d\mathcal{H}^{n-1} \right). \label{5.3.20}
	\end{align}
	Then we have
	\begin{equation}\label{5.3.21}
	|\delta \tilde{V}_t(\boldg)| \leq c(t)\,\sup_{\overline{\Omega}}|\boldg| <\infty
	\end{equation}
	for a.e. $t\geq0$, where $c(t)$ is as in \eqref{5.3.6} in Proposition \ref{prop.5.5}. This shows that we can extend the domain of the functional $\delta \tilde{V}_t$ into $(C_c(\overline{\Omega}))^n$ and hence from Riesz representation theorem, we have that the total variation $\|\delta \tilde{V}_t\|$ is a Radon measure on $\overline{\Omega}$. Since we have  $\|V^{\varepsilon''_j,\,\sigma'_j}_t\|=\mu^{\varepsilon''_j,\,\sigma'_j}_t$ and the subsequence $\{\mu^{\varepsilon''_j,\,\sigma'_j}_t\}_{j\in\mathbb{N}}$ converges to $\mu_t$, we have $\|\tilde{V}_t\|=\mu_t$, which is uniquely determined.
	
	From the above argument, we can apply Allard's rectifiability theorem to $\tilde{V}_t$ and conclude that $\tilde{V}_t$ is $(n-1)$-rectifiable varifold on $\overline{\Omega}$. In particular, $\tilde{V}_t$ is determined uniquely by $\|\tilde{V}_t\|=\mu_t$, and therefore, this yields that the uniqueness of the limit $(n-1)$-rectifiable varifold $V_t$ associated with $\mu_t$ (as in Definition \ref{def.3.1}) follows. This also shows that $\mu_t$ is $(n-1)$-rectifiable on $\overline{\Omega}$ for a.e. $0\leq t \leq T$. Furthermore, the calculation in the above shows that the boundedness of the total variation of the first variation of varifolds shown in \eqref{5.3.14} holds for a.e. $t\geq 0$ and $T\geq 0$. Thus we have completed the proof. 
\end{proof}
Considering the all arguments in Lemma \ref{thm3.1}, and Proposition \ref{prop.5.5} and \ref{prop.5.6}, we may conclude  that Lemma \ref{thm3.3} is valid.

\subsubsection{Proof of Theorem \ref{thm3.4} (Dirichlet boundary conditions)}\label{chara.4}
In this section, we will prove Theorem \ref{thm3.4}, that is, the existence of the singular limits of the Allen-Cahn equations described in \eqref{1.1.1}. Before proving Theorem \ref{thm3.4}, as a preparation, we will show three propositions. First of all, we show that the first variation in integral form $\int_{0}^{\infty}\delta V^{\varepsilon'_j,\,\sigma'_j}_t\,dt$ converges to $\int_{0}^{\infty}\delta V_t\,dt$ locally in time as $j\to\infty$, where the subsequence $\{V^{\varepsilon'_j,\,\sigma'_j}_t\}_{j\in\mathbb{N}}$ has the limit varifold $V_t$.

Note that, through this subsection, we assume that ``General assumptions", ``Vanishing hypothesis for the discrepancy measure", and ``Uniform upper bound on the boundary $\partial\Omega$" in Subsection \ref{exis.assmp.diri} hold. Moreover, we only consider the subsequence $\{\varepsilon'_j\}_{j\in\mathbb{N}}$ and $\{\sigma'_j\}_{j\in\mathbb{N}}$ such that Lemma \ref{thm3.1} and \ref{thm3.2} hold.

First of all, we show the following proposition;
\begin{proposition}\label{thm3.5}
	Suppose that ``General assumptions" and ``Vanishing hypothesis for the discrepancy measure" in Subsection \ref{exis.assmp.diri} hold, and $\{\varepsilon'_j\}_{j\in\mathbb{N}}$ and $\{\sigma'_j\}_{j\in\mathbb{N}}$ are such that Lemma \ref{thm3.1} and \ref{thm3.2} hold. Define a Radon measure $\tilde{\mu}^{\varepsilon'_j,\,\sigma'_j}$ on $\Omega\times[0,\,\infty)$ by   
	\begin{equation}\label{3.1.14}
	\tilde{\mu}^{\varepsilon'_j,\,\sigma'_j}(U)\coloneqq\int\!\!\int_{U}\varepsilon'_j|\nabla u^{\varepsilon'_j,\,\sigma'_j}|^2\,dx\,dt\,\left(=\int\!\!\!\!\int_{U}\,d\mu^{\varepsilon'_j,\,\sigma'_j}_t\,dt+\int\!\!\!\!\int_{U}\,d\xi^{\varepsilon'_j,\,\sigma'_j}\,dt\right),
	\end{equation}
	for any open set $U\subset\Omega\times[0,\,\infty)$. Then there exist a subsequence $\{\tilde{\mu}^{\varepsilon'_j,\,\sigma'_j}\}$ (labelled with the same index)  and a function $\boldv\in(L^2(\mu_t\otimes\mathcal{L}^1_t,\,\Omega\times[0,\,\infty)))^n$ such that $\tilde{\mu}^{\varepsilon'_j,\,\sigma'_j}\rightharpoonup \mu_t\otimes\mathcal{L}^1_t$ in $\Omega\times[0,\,\infty)$, where $\mu_t$ is as in Lemma \ref{thm3.1}, and 
	\begin{equation}\label{3.1.14.1}
	 \lim_{j \to \infty}\int_{0}^{\infty}\!\!\int_{\Omega}\boldg\cdot(\varepsilon'_j\,\partial_t u^{\varepsilon'_j,\,\sigma'_j}\,\nabla u^{\varepsilon'_j,\,\sigma'_j})\,dx\,dt=-\int_{0}^{\infty}\int_{\Omega}\boldg\cdot \boldv\,d\mu_t\,dt,
	\end{equation}
	for any $\boldg\in(C_c(\Omega\times[0,\,\infty)))^n$.
\end{proposition}
\begin{proof}
	First of all, we show the convergence of the family of Radon measures $\{\tilde{\mu}^{\varepsilon'_j,\,\sigma'_j}\}_{j\in\mathbb{N}}$ and also show that the limit of $\tilde{\mu}^{\varepsilon'_j,\,\sigma'_j}$ is equal to $\mu_t\otimes\mathcal{L}^1_t$ in $\Omega\times[0,\,\infty)$. The convergence of the measures is readily obtained from its definition and Lemma \ref{thm3.1}. Let $\tilde{\mu}$ denote its limit, that is, we obtain that there exists a subsequence (labelled with the same index) such that  $\tilde{\mu}^{\varepsilon'_j,\,\sigma'_j}\rightharpoonup\tilde{\mu}$ in $\Omega\times[0,\,\infty)$. Regarding to its limit, we should recall the vanishing of the discrepancy measure $\xi^{\varepsilon'_j,\,\sigma'_j}_t$ as $j\to\infty$ on $\overline{\Omega}$ for a.e. $t\geq0$. By the assumption ``Vanishing hypothesis of the discrepancy" in Section \ref{exisdiri}, we have that $\xi^{\varepsilon'_j,\,\sigma'_j}_t\otimes\mathcal{L}^1_t \rightharpoonup 0$ in $\Omega\times[0,\,\infty)$. Moreover, from the definitions, we may easily see that
	\begin{equation}\label{5.3.22}
	\tilde{\mu}^{\varepsilon'_j,\,\sigma'_j}=\mu^{\varepsilon'_j,\,\sigma'_j}_t\otimes\mathcal{L}^1_t +\xi^{\varepsilon'_j,\,\sigma'_j}_t\otimes\mathcal{L}^1_t
	\end{equation}
	holds for any $j\in\mathbb{N}$. Therefore, letting $j\to\infty$, we may conclude that $\tilde{\mu}=\lim_{j \to \infty}\tilde{\mu}^{\varepsilon'_j,\,\sigma'_j}=\mu_t\otimes\mathcal{L}^1_t$ in $\Omega\times[0,\,\infty)$ in the sense of Radon measure.
	
	Next, we prove that there exists a function $\boldv\in(L^2(\tilde{\mu},\,\Omega\times[0,\,\infty)))^n$ such that 
	\begin{equation}\label{5.3.23}
	\lim_{j \to \infty}\int_{0}^{\infty}\!\!\int_{\Omega}\boldg\cdot(\varepsilon'_j\,\partial_t u^{\varepsilon'_j,\,\sigma'_j}\,\nabla u^{\varepsilon'_j,\,\sigma'_j})\,dx\,dt=-\int_{0}^{\infty}\int_{\Omega}\boldg\cdot \boldv\,d\tilde{\mu}\,\left(=\int_{0}^{\infty}\delta V_t\lfloor_{\Omega}\,dt(\boldg)\right),
	\end{equation}
	for any $\boldg\in(C_c(\Omega\times[0,\,\infty)))^n$. To do this, we first define a function $\boldv^{\varepsilon'_j,\,\sigma'_j}$, which can be regarded as the approximation with the normal velocity vector of a seperating front, by
	\begin{align}\label{5.3.24}
	\boldv^{\varepsilon'_j,\,\sigma'_j}=
	\begin{cases}
	\displaystyle
	-\frac{\partial_t u^{\varepsilon'_j,\,\sigma'_j}}{|\nabla u^{\varepsilon'_j,\,\sigma'_j}|}\frac{\nabla u^{\varepsilon'_j,\,\sigma'_j}}{|\nabla u^{\varepsilon'_j,\,\sigma'_j}|}&\quad\text{if $|\nabla u^{\varepsilon'_j,\,\sigma'_j}|\neq 0$.}\\
	0&\quad\text{otherwise}.
	\end{cases}
	\end{align}
	Then, from the definitions, we may deduce that 
	\begin{equation}\label{5.3.25}
	\int_{0}^{\infty}\!\!\int_{\Omega}|\boldv^{\varepsilon'_j,\,\sigma'_j}|^2\,d\tilde{\mu}^{\varepsilon'_j,\,\sigma'_j}=\int_{0}^{\infty}\!\!\int_{\Omega}\varepsilon'_j\,(\partial_t u^{\varepsilon'_j,\,\sigma'_j})^2\,dx\,dt\leq D<\infty. 
	\end{equation}
	This implies that the pair $(\tilde{\mu}^{\varepsilon'_j,\,\sigma'_j},\,\boldv^{\varepsilon'_j,\,\sigma'_j})$ is the one which satisfies the $L^2$-uniform boundedness with respect to $i,\,j\in\mathbb{N}$. Since we have already had the convergence of $\tilde{\mu}^{\varepsilon'_j,\,\sigma'_j}$, we can again apply the theorem \cite[Theorem 4.4.2.]{Hutchinson} and thus we may conclude that there exist a subsequence (labelled with the same index) and a function $\boldv\in(L^2(\tilde{\mu},\,\Omega\times[0,\,\infty)))^n$ such that
	\begin{equation}
		\lim_{j\to\infty}\int_{0}^{\infty}\int_{\Omega}\boldg \cdot \boldv^{\varepsilon'_j,\,\sigma'_j}\,d\tilde{\mu}^{\varepsilon'_j,\,\sigma'_j} = \int_{0}^{\infty}\int_{\Omega}\boldg\cdot \boldv\,d\mu_t\,dt
	\end{equation}
	for any $\boldg\in(C_c(\Omega\times[0,\,\infty)))^n$ and thus, this completes the proof.
\end{proof}

Secondly, we will show that interchanging limit processes and integral signs of the first variation of varifolds is valid.

\begin{proposition}\label{prop.5.7}
	Let $\{V^{\varepsilon'_j,\,\sigma'_j}_t\}_{j\in\mathbb{N}}$ be a family of associated varifolds with $\mu^{\varepsilon'_j,\,\sigma'_j}_t$ satisfying Proposition \ref{prop.5.5} and \ref{prop.5.6} for a.e. $t\geq0$. Then we have 
	\begin{equation}\label{5.4.1}
	\lim_{j\to\infty}\int_{0}^{T}\delta V^{\varepsilon'_j,\,\sigma'_j}_t(\boldg)\,dt =  \int_{0}^{T} \delta V_t(\boldg)\,dt
	\end{equation}
	for all $T>0$ and all $\boldg\in (C^1_c(\overline{\Omega}\times[0,\,\infty)))^n$.
\end{proposition}
\begin{proof}
	Let $T>0$ be arbitrary. From the convergence of $\{\delta V_t^{\varepsilon'_j,\,\sigma'_j}\}_{j\in\mathbb{N}}$, Proposition \ref{prop.5.5} and \eqref{5.3.21} in Proposition \ref{prop.5.6}, we have
	\begin{equation}\label{5.4.3}
	\left|\limsup_{j\in\mathbb{N}} \delta V^{\varepsilon'_j,\,\sigma'_j}_t(\boldg(\cdot,\,t))\right|=|\delta V_t(\boldg(\cdot,\,t))|\leq \sup_{\overline{\Omega}\times[0,\,\infty)}|\boldg|\,c(t)
	\end{equation}
	for any $\boldg\in (C^1_c(\overline{\Omega}\times [0,\,\infty)))^n$ and, from Proposition \ref{prop.5.5}, we have $(\sup_{\overline{\Omega}\times[0,\,\infty)}|\boldg|)\,c\in L^1([0,\,T])$. Therefore, from dominated convergence theorem, we have the equality \eqref{5.4.1}.  
\end{proof}

Next we show the absolute continuities for total variation measures and $L^2$-estimate for the modified generalized mean curvature vector.
\begin{proposition}\label{prop.5.8}
	There exists the generalized mean curvature vector $\boldH_V^{\Omega}(\cdot,\,t)$ in $\Omega$ such that 
	\begin{equation}\label{5.4.5.1}
	\boldH_V^{\Omega}\equiv\boldv\quad\text{in $(L^2(\|V_t\|\otimes\mathcal{L}^1_t,\,\Omega\times[0,\,\infty)))^n$},\quad\delta V_t\lfloor_{\Omega} =-\boldH_V^{\Omega}(\cdot,\,t)\,\|V_t\|\quad\text{in $\Omega$ for a.e. $t\geq 0$}
	\end{equation}
	holds, where $\boldv$ is as in Proposition \ref{thm3.5}. In addition, assume that $\alpha$ and $\boldv_b$ are followed by Lemma \ref{thm3.2}. Let $\mathcal{S}_{\alpha,\,{\boldv}_b}$ be as in Definition \ref{def3.1.2} where $\tilde{\mu}\coloneqq\mu_t\otimes\mathcal{L}^1_t=\|V_t\|\otimes\mathcal{L}^1_t$. Then we have
	\begin{equation}\label{5.4.6}
		\left\|\mathcal{S}_{\alpha,\,{\boldv}_b}+\int_{0}^{\infty}\delta V_t\lfloor_{\Omega}\,dt\right\| \ll \|V_t\|\otimes \mathcal{L}^1_t \quad\text{on $\overline{\Omega} \times [0,\,\infty)$},
	\end{equation}
	and there exists the modified generalized mean curvature vector $\widetilde{\boldH}_V$ (see Definition \ref{def.4.1}) such that $\widetilde{\boldH}_V\lfloor_{\Omega}\equiv\boldH_V^{\Omega}$ in $(L^2(\|V_t\|\otimes\mathcal{L}^1_t,\,\Omega\times[0,\,\infty)))^n$, $\widetilde{\boldH}_V$ belongs to $(L^2(\|V_t\|\otimes \mathcal{L}^1_t,\,\overline{\Omega} \times [0,\infty)))^n$, and we have $\|\widetilde{\boldH}_V\|_{L^2}\leq D^{\frac{1}{2}}$. Besides, because of the assumption ``Uniform upper bound on the boundary $\partial\Omega$" in Subsection \ref{exis.assmp.diri}, we in fact obtain $\|\mathcal{S}_{\alpha,\,\boldv_b}\|\equiv0$ on $\partial\Omega\times[0,\,\infty)$.
	
	Moreover, for any $0<t_1\leq t_2<\infty$, we have
	\begin{equation}\label{5.4.7}
	\int_{t_1}^{t_2}\int_{\overline{\Omega}} \phi |\widetilde{\boldH}_V|^2 \,d\|V_t\|\,dt \leq\liminf_{j\to\infty} \int_{t_1}^{t_2}\int_{\Omega}\varepsilon'_j \phi (\partial_t u^{\varepsilon'_j,\,\sigma'_j})^2 \,dx\,dt
	\end{equation}
	for all $\phi\in C_c(\overline{\Omega}\times [0,\,\infty))$ with $\phi\geq 0$.
\end{proposition}
\begin{proof}
	Take any time $t\geq 0$ such that \eqref{5.3.6} and \eqref{5.3.16} hold. Let $\{V^{\varepsilon'_j,\,\sigma'_j}_t\}_{j\in\mathbb{N}}$ be a subsequence converging to $V_t$. From \eqref{4.1.1}, \eqref{5.3.2} and \eqref{5.3.16}, we have, for any $\boldg\in (C^1_c(\Omega))^n$ and any $t\geq0$ such that Proposition \ref{prop.5.5} holds,
	\begin{align}\label{5.4.8}
	\left|\delta V_t\lfloor_{\Omega}(\boldg)\right|&=\lim_{j\to\infty}\left|\delta V^{\varepsilon'_j,\,\sigma'_j}_t(\boldg)\right| \nonumber\\
	&\leq \liminf_{\varepsilon' \to 0}\left(\int_{\Omega} \varepsilon'_j (\partial_t u^{\varepsilon'_j,\,\sigma'_j})^2\,dx\right)^{\frac{1}{2}} \left(\int_{\Omega} |\boldg|^2\, d\|V_t\|\right)^{\frac{1}{2}} \nonumber\\
	&\leq D^{\frac{1}{2}} \left(\int_{\Omega} |\boldg|^2\, d\|V_t\|\right)^{\frac{1}{2}} <\infty.
	\end{align}
	This shows that $\left\|\delta V_t\lfloor_{\Omega}\right\| \ll \|V_t\|$ in $\Omega$ for a.e. $t\geq0$. Moreover, from Riesz representation theorem and Radon-Nikodym theorem, it holds that, for a.e. $t\geq0$ such that Proposition \ref{prop.5.5} holds, there exists a $(\|V_t\|\otimes\mathcal{L}^1_t)$-integrable vector-valued function $\boldH_V^{ \Omega}(\cdot,\,t)$ such that 
	\begin{equation}\label{5.4.8.2}
	\left(\delta V_t\lfloor_{\Omega}\right)(\boldg)=- \int_{\Omega}\boldH_V^{ \Omega}(\cdot,\,t)\cdot \boldg \,d\|V_t\|
	\end{equation}
	for all $\boldg\in (C_c(\Omega))^n$. Note that, from \eqref{5.4.8} and \eqref{5.4.8.2}, it also holds that $\boldH_V^{\Omega}\in(L^2(\|V_t\|\otimes \mathcal{L}^1_t,\,\Omega\times[0,\,\infty)))^n$. Moreover, from Proposition \ref{thm3.5}, Proposition \ref{prop.5.7}, \eqref{5.3.2}, and \eqref{5.4.8.2}, we may obtain the following calculation:
	\begin{align}
	\int_{0}^{\infty}\!\!\int_{\Omega}\boldg\cdot\boldH_V^{ \Omega}\,d\|V_t\|\,dt &=-\left(\int_{0}^{\infty}\delta V_t\lfloor_{\Omega}\,dt\right)(\boldg) \nonumber\\
	&=-\lim_{j\to\infty}\int_{0}^{\infty}\!\!\int_{\Omega}(\boldg\cdot\nabla u^{\varepsilon'_j,\,\sigma'_j})\,\varepsilon'_j\,(\partial_t u^{\varepsilon'_j,\,\sigma'_j})\,dx\,dt\nonumber\\
	&=\lim_{j\to\infty}\int_{0}^{\infty}\!\!\int_{\Omega}\boldg\cdot\left(-\frac{\partial_t u^{\varepsilon'_j,\,\sigma'_j}}{|\nabla u^{\varepsilon'_j,\,\sigma'_j}|}\frac{\nabla u^{\varepsilon'_j,\,\sigma'_j}}{
		|\nabla u^{\varepsilon'_j,\,\sigma'_j}|}\right)\,\varepsilon'_j|\nabla u^{\varepsilon'_j,\,\sigma'_j}|^2\,dx\,dt\nonumber\\
	&=\lim_{j\to\infty}\int_{0}^{\infty}\!\!\int_{\Omega}\boldg\cdot\boldv^{\varepsilon'_j,\,\sigma'_j}\,d\tilde{\mu}^{\varepsilon'_j,\,\sigma'_j}=\int_{\Omega\times[0,\,\infty)}\boldg\cdot\boldv\,d\tilde{\mu}.\label{5.4.8.2.1}
	\end{align}
	for any $\boldg\in(C^1_c(\Omega\times[0,\,\infty)))^n$. Here we note that the following identity holds; for any $f\in (L^2(\|V_t\|\otimes\mathcal{L}^1_t))^n$,
	\begin{align}\label{5.4.12}
	&\|f\|_{L^2(\|V_t \| \otimes \mathcal{L}^1_t, \,X\times[0,\,\infty))}\nonumber\\
	& = \sup \left\{\int_{0}^{\infty}\!\!\int_{X} f\cdot \boldg \,d\|V_t\|\, dt\ \Big|\ \boldg\in (C^1_c(X\times[0,\,\infty)))^n,\ \|\boldg\|_{L^2(\|V_t\|\otimes \mathcal{L}^1_t)}\leq 1 \right\}
	\end{align}
	holds where $X$ is either $\Omega$ or $\overline{\Omega}$. We refer to Ilmanen \cite{Ilmanen02} for more detail on this identity. Recalling that $\tilde{\mu}=\|V_t\|\otimes\mathcal{L}^1_t$, from \eqref{5.4.8.2.1}, and \eqref{5.4.12}, we have that $\boldH_V^{\Omega}\equiv\boldv$ in $(L^2(\tilde{\mu},\,\Omega\times[0,\,\infty)))^n$. Moreover, from \eqref{5.4.8.2.1}, we actually obtain $\|\int_{0}^{\infty}\delta V_t\lfloor_{\Omega}\,dt\|\equiv\|\int_{0}^{\infty}\delta V_t\lfloor_{\Omega}\,dt\|$ in $\Omega\times[0,\,\infty)$.
	
	Now let $T>0$ be arbitrary number and $\boldg\in (C^1_c(\overline{\Omega}\times[0,\,T)))^n$ be any test function such that $|\boldg|\leq 1$. In order to prove \eqref{5.4.6}, we need to compute the following: from \eqref{5.2.5}, \eqref{5.3.2}, and \eqref{5.3.23}, we have that
	\begin{align}
	\left(\mathcal{S}_{\alpha,\,{\boldv}_b}+\int_{0}^{\infty}\delta V_t\lfloor_{\Omega}\,dt\right)(\boldg) &= \int_{0}^{\infty}\!\!\int_{\partial \Omega } \boldg\cdot {\boldv}_b \,d\alpha - \int_{0}^{\infty}\!\!\int_{\Omega}\boldg\cdot\boldH_V^{\Omega}\,d\|V_t\|\,dt \nonumber\\
	&= -\lim_{j\to\infty}\int_{0}^{\infty}\!\!\int_{\partial \Omega } (\boldg\cdot \nabla_{\partial \Omega} u^{\varepsilon'_j,\,\sigma'_j}) \varepsilon'_j (\partial_t u^{\varepsilon'_j,\,\sigma'_j})\,d\mathcal{H}^{n-1}\,dt\nonumber\\
	&\qquad +\lim_{j \to \infty}\int_{0}^{\infty}\!\!\int_{\Omega}(\boldg\cdot\nabla u^{\varepsilon'_j,\,\sigma'_j})\,\varepsilon'_j(\partial_t u^{\varepsilon'_j,\,\sigma'_j})\,dx\,dt. \label{5.4.8.3}
	\end{align}
	Recalling the boundary condition $\partial_t u^{\varepsilon'_j,\,\sigma'_j}+\sigma'_j \nabla u^{\varepsilon'_j,\,\sigma'_j}\cdot \boldnu=0$, from Cauchy-Schwarz inequality and Lemma \ref{thm3.2}, we can show that
	\begin{align}\label{5.4.8.4}
	\left| -\lim_{j\to\infty} \right. & \left. \int_{0}^{\infty}\!\!\int_{\partial\Omega}(\boldg\cdot \nabla_{\partial \Omega} u^{\varepsilon'_j,\,\sigma'_j}) \varepsilon'_j (\partial_t u^{\varepsilon'_j,\,\sigma'_j})\,d\mathcal{H}^{n-1}\,dt \right| \nonumber\\
	&= \left| \lim_{j\to\infty} \int_{0}^{\infty}\!\!\int_{\partial\Omega}(\boldg\cdot \nabla_{\partial \Omega} u^{\varepsilon'_j,\,\sigma'_j}) \varepsilon'_j\, \sigma'_j\frac{\partial u^{\varepsilon'_j,\,\sigma'_j}}{\partial\nu}\,d\mathcal{H}^{n-1}\,dt \right| \nonumber\\
	&\leq\limsup_{j\to\infty}\sigma'_j\left(\int_{0}^{\infty}\!\!\int_{\partial\Omega}|\boldg|^2\varepsilon'_j|\nabla_{\partial\Omega} u^{\varepsilon'_j,\,\sigma'_j}|^2\,d\mathcal{H}^{n-1}\,dt\right)^{\frac{1}{2}}\nonumber\\
	&\qquad \times \sup_{j\in\mathbb{N}}\left(\int_{0}^{T}\!\!\int_{\partial\Omega}\varepsilon'_j\left(\frac{\partial u^{\varepsilon'_j,\,\sigma'_j}}{\partial \nu}\right)^2\,d\mathcal{H}^{n-1}\,dt\right)^{\frac{1}{2}} \nonumber\\
	&=(\limsup_{j\to \infty}\sigma'_j)\,\left(\int_{0}^{\infty}\!\!\int_{\partial\Omega}|\boldg|^2\,d\alpha\right)^{\frac{1}{2}}\,C_0(t_1,\,t_2) \nonumber\\
	&\leq (\limsup_{j\to \infty}\sigma'_j)\,(\alpha((\partial\Omega)\times[0,\,T]))^{\frac{1}{2}} \, C_0(t_1,\,t_2)<\infty.
	\end{align}
	Since $\alpha$ is a Radon measure on $\partial \Omega \times [0,\,\infty)$, we have that the left-hand side of \eqref{5.4.8.4} is equal to 0. Thus, we obtain, from \eqref{5.4.8.3}, \eqref{5.4.8.4}, and Cauchy-Schwarz inequality,  
	\begin{align}
	\left| \left(\mathcal{S}_{\alpha,\,{\boldv}_b} + \int_{0}^{\infty}\delta V_t\lfloor_{\Omega}\,dt \right)(\boldg) \right|&\leq 0+\liminf_{j\to\infty} \left(\int_{0}^{\infty}\!\!\int_{\Omega} \varepsilon'_j (\partial_t u^{\varepsilon'_j,\,\sigma'_j})^2\,dx\,dt\right)^{\frac{1}{2}}\,
	\left(\int_{0}^{\infty}\!\!\int_{\overline{\Omega}}|\boldg|^2 \,d\|V_t\|\,dt\right)^{\frac{1}{2}}\nonumber\\
	&\leq D^{\frac{1}{2}} \left(\|V_t\|\otimes\mathcal{L}^1_t(\overline{\Omega}\times[0,\,T))\right)^{\frac{1}{2}}. \label{5.4.9}
	\end{align}
	Of course, \eqref{5.4.8.4} also yields that $\mathcal{S}_{\alpha,\,\boldv_b}(\boldg)=0$ for any $\boldg\in (C_c(\partial\Omega\times[0,\,\infty)))^n$, and thus $\|\mathcal{S}_{\alpha,\,\boldv_b}\|\equiv0$ on $\partial\Omega\times[0,\,\infty)$.
	
	Therefore, taking the supremum of the both side of \eqref{5.4.9} with respect to $\boldg$, we obtain the absolute continuity \eqref{5.4.6} from the arbitrariness of $T>0$. Moreover, from \eqref{5.4.6} and Radon-Nikodym theorem, we can show that there exists a $(\|V_t\|\otimes\mathcal{L}^1_t)$-integrable vector-valued function $\widetilde{\boldH}_V$ such that
	\begin{equation}\label{5.4.12.1}
	\left(\mathcal{S}_{\alpha,\,\boldv_b}+\int_{0}^{\infty}\delta V_t\lfloor_{\Omega}\,dt \right)(\boldg)=-\int_{0}^{\infty}\!\!\int_{\overline{\Omega}}\boldg\cdot \widetilde{\boldH}_V\,d\|V_t\|\,dt,
	\end{equation}
	for any $\boldg\in(C_c(\overline{\Omega}\times[0,\,\infty)))^n$.
	
	Next we need to show the three claims, that is, $\widetilde{\boldH}_V\lfloor_{\Omega} \equiv \boldH_V^{\Omega}$ in $(L^2(\|V_t\|\otimes\mathcal{L}^1_t,\,\Omega\times[0,\,\infty)))^n$, $\widetilde{\boldH}_V$ belongs to $(L^2(\|V_t\|\otimes \mathcal{L}^1_t,\,\overline{\Omega} \times [0,\infty)))^n$, and \eqref{5.4.7}. First of all, we show that $\widetilde{\boldH}_V\lfloor_{\Omega}\equiv\boldH_V^{\Omega}$ in $(L^2(\|V_t\|\otimes\mathcal{L}^1_t,\,\Omega\times[0,\,\infty)))^n$. From the absolute continuity \eqref{5.4.6}, the fact that $\boldH_V^{\Omega}$ is equal to $\boldv$ in $\Omega\times[0,\,\infty)$ in $(L^2(\|V_t\|\otimes\mathcal{L}^1_t))^n$, \eqref{5.4.8}, \eqref{5.4.8.2}, and \eqref{5.4.12.1}, we have that, for any $\boldg\in(C_c(\Omega\times[0,\,\infty)))^n$,
	\begin{align}\label{5.4.12.2}
	\int_{0}^{\infty}\!\!\int_{\Omega}\widetilde{\boldH}_V\lfloor_{\Omega}
	\cdot\boldg\,d\|V_t\|\,dt&=-\int_{0}^{\infty}\delta V_t\lfloor_{\Omega}\,dt(\boldg)+\mathcal{S}_{\alpha,\,\boldv_b}(\boldg)\nonumber\\
	&=-\int_{0}^{\infty}\delta V_t\lfloor_{\Omega}\,dt(\boldg)+0\nonumber\\
	&=\lim_{j\to\infty}\int_{0}^{\infty}\!\!\int_{\Omega}(\boldg\cdot\nabla u^{\varepsilon'_j,\,\sigma'_j})\,\varepsilon'_j\,(\partial_t u^{\varepsilon'_j,\,\sigma'_j})\,dx\,dt\nonumber\\
	&=-\int_{0}^{\infty}\!\!\delta V_t\lfloor_{\Omega}\,dt(\boldg)\nonumber\\
	&=\int_{0}^{\infty}\!\!\int_{\Omega}\boldH_V^{\Omega}\cdot\boldg\,d\|V_t\|\,dt.
	\end{align}
	Note that since the support of $\boldg$ is in $\Omega$, we have that $\mathcal{S}_{\alpha,\,\boldv_b}(\boldg)=0$. Thus, from \eqref{5.4.12} and  \eqref{5.4.12.2} and the arbitrariness of $\boldg$, we have that $\widetilde{\boldH}_V\lfloor_{\Omega}\equiv\boldH_V^{\Omega}$ in $(L^2(\|V_t\|\otimes\mathcal{L}^1_t,\,\Omega\times[0,\,\infty)))^n$. 
	
	To prove the second and third claims, we will apply the identity \eqref{5.4.12}. Now take any $\boldg\in (C_c(\overline{\Omega}\times [0,\,\infty)))^n$ with $\|\boldg\|_{L^2(\|V_t\|\otimes \mathcal{L}^1_t,\,\overline{\Omega}\times[0,\,\infty))}\leq 1$. From \eqref{5.4.9} and \eqref{5.4.12.1}, we have the following estimate:
	\begin{align}
	\left|\int_{0}^{\infty}\!\!\int_{\Omega} \widetilde{\boldH}_V\cdot \boldg\,d\|V_t\|\, dt \right| &= \left| -\left(\mathcal{S}_{\alpha,\,\boldv_b}+\int_{0}^{\infty}\delta V_t\lfloor_{\Omega}\,dt\right)(\boldg) \right|\nonumber\\
	&\leq \liminf_{j\to\infty}\left(\int_{0}^{\infty}\!\!\int_{\Omega} \varepsilon'_j (\partial_t u^{\varepsilon'_j,\,\sigma'_j})^2\,dx\,dt\right)^{\frac{1}{2}}
	\,\left(\int_{0}^{\infty}\!\!\int_{\overline{\Omega}}|\boldg|^2\,d\|V_t\|\,dt\right)^{\frac{1}{2}} \nonumber\\
	&\leq D^{\frac{1}{2}} \label{5.4.13}
	\end{align}
	Therefore, carrying out the approximation \eqref{5.4.12}, we have
	\begin{equation}
	\int_{0}^{\infty}\int_{\overline{\Omega}} |\widetilde{\boldH}_V|^2 \,d\|V_t\|\,dt \leq D. \label{5.4.15}
	\end{equation}
	Finally, we prove \eqref{5.4.7}. To do this, we may carry out the approximation argument which is conducted by Ilmanen in \cite{Ilmanen01} and we are able to apply this method to our problem since the associated varifold $V_t$ is rectifiable for a.e. $t\geq0$. Therefore, considering all the above, we may conclude that Proposition \ref{prop.5.8} follows. 
\end{proof}

Finally, we prove the Brakke's inequality \eqref{4.1.1.5} and this completes the proof of Theorem \ref{thm3.4}. First of all, we show that the $(\varepsilon,\,\sigma)$-approximated velocity vector $\boldv^{\varepsilon,\,\sigma}$ converges to the modified generalized mean curvature vector $\widetilde{\boldH}_V$ on $\overline{\Omega}$.
\begin{proposition}\label{prop.5.9}
	Let $\{V^{\varepsilon'_j,\,\sigma'_j}_t\}_{j\in\mathbb{N}}$ be as in Proposition \ref{prop.5.6}, \ref{prop.5.7} and \ref{prop.5.8}. Let $\widetilde{\boldH}_V$ be as in Proposition \ref{prop.5.8}. Then we have
	\begin{equation}\label{5.4.16}
	-\lim_{j\to\infty}\int_{0}^{\infty}\!\!\int_{\Omega} (\boldg\cdot \nabla u^{\varepsilon'_j,\,\sigma'_j})\varepsilon'_j(\partial_t u^{\varepsilon'_j,\,\sigma'_j})\,dx\,dt= \int_{0}^{\infty}\!\!\int_{\overline{\Omega}}\boldg\cdot \widetilde{\boldH}_V\,d\|V_t\|\,dt
	\end{equation}
	for all $\boldg\in(C^1_c(\overline{\Omega}\times [0,\,\infty)))^n$.
\end{proposition}
\begin{proof}
	Let $\boldg$ be in $(C^1_c(\overline{\Omega}\times [0,\,\infty)))^n$. From Proposition \ref{prop.5.6}, we have already known that $V^{\varepsilon'_j,\,\sigma'_j}_t$ converges to $(n-1)$-rectifiable varifold $V_t$ associated with $\mu_t$ for a.e. $t\geq0$ in the sense of varifolds and thus we have $\lim_{j\to\infty}\delta V^{\varepsilon'_j,\,\sigma'_j}_t= \delta V_t$. From \eqref{5.4.6} in Proposition \ref{prop.5.8}, there exists the modified generalized mean curvature vector $\widetilde{\boldH}_V$ such that
	\begin{equation}
	\int_{0}^{\infty}\int_{\overline{\Omega}}\boldg\cdot\widetilde{\boldH}_V \,d\|V_t\|\,dt=-\left(\mathcal{S}_{\alpha,\,{\boldv}_b} + \int_{0}^{\infty}\delta V_t\lfloor_{\Omega}\,dt \right)(\boldg).\label{5.4.16.2}
	\end{equation}
	Thus, from \eqref{5.4.8.3} and \eqref{5.4.8.4}, we have
	\begin{align}
	\int_{0}^{\infty}\int_{\overline{\Omega}} \boldg\cdot\widetilde{\boldH}_V \,d\|V_t\|\,dt &=-\lim_{j\to\infty}\int_{0}^{\infty}\int_{\Omega} (\boldg\cdot \nabla u^{\varepsilon'_j,\,\sigma'_j})\varepsilon'_j(\partial_t u^{\varepsilon'_j,\,\sigma'_j})\,dx\,dt.\label{5.4.16.3}
	\end{align}
	This implies \eqref{5.4.16} and completes the proof of Proposition \ref{prop.5.9}. 
\end{proof}

Finally, considering Proposition \ref{prop.5.7}, \ref{prop.5.8}, and \ref{prop.5.9}, we are prepared to prove Theorem \ref{thm3.4}.

\begin{proof}[Proof of Theorem \ref{thm3.4}]
	Since we have already shown Lemma \ref{thm3.1}, \ref{thm3.2}, \ref{thm3.2.1}, \ref{thm3.3} and \ref{thm3.5}, it is sufficient to prove Brakke's inequality \eqref{4.1.1.5}. From Lemma \ref{thm3.1}, \ref{thm3.2}, \ref{thm3.2.1}, and \ref{thm3.5}, we can take the same subsequence $\varepsilon'_j \to 0$ and $\sigma'_j \to 0$ as $j\to \infty$ such that all the claims in Lemma \ref{thm3.1}, \ref{thm3.2}, \ref{thm3.2.1}, and \ref{thm3.5} hold. Thus, in the following, it is sufficient to consider such a subsequence. Let $\phi$ be in $C^1_c(\overline{\Omega}\times [0,\infty))$ such that $\phi\geq 0$. For any $0\leq t_1<t_2<\infty$, recalling \eqref{4.1.6} in Proposition \ref{prop.4.2} and the notation $f^{\varepsilon'_j,\,\sigma'_j}\coloneqq-\varepsilon'_j \Delta u^{\varepsilon'_j,\,\sigma'_j} +\frac{W'(u^{\varepsilon'_j,\,\sigma'_j})}{\varepsilon'_j}$, we have 
	\begin{align}\label{5.4.17}
	\mu^{\varepsilon'_j,\,\sigma'_j}_t(\phi) \Big|_{t=t_1}^{t_2}&= \int_{t_1}^{t_2}\int_{\Omega} \left(-\frac{1}{\varepsilon'_j}(f^{\varepsilon'_j,\,\sigma'_j})^2 \phi+ f^{\varepsilon'_j,\,\sigma'_j}\nabla \phi\cdot \nabla u^{\varepsilon'_j,\,\sigma'_j}\right) \,dx\,dt \nonumber\\
	& \qquad +\int_{t_1}^{t_2}\!\!\int_{\Omega} \partial_t \phi\,d\mu^{\varepsilon'_j,\,\sigma'_j}_t dt-\int_{t_1}^{t_2}\!\!\int_{\partial\Omega}\varepsilon'_j\,\sigma'_j\,\phi\,\left(\frac{\partial u^{\varepsilon'_j,\,\sigma'_j}}{\partial \nu}\right)^2\,d\mathcal{H}^{n-1}\,dt\nonumber\\
	&\leq\int_{t_1}^{t_2}\!\!\int_{\Omega} \left(-\frac{1}{\varepsilon'_j}(f^{\varepsilon'_j,\,\sigma'_j})^2 \phi + f^{\varepsilon'_j,\,\sigma'_j}\nabla \phi\cdot \nabla u^{\varepsilon'_j,\,\sigma'_j}\right) \,dx\,dt +\int_{t_1}^{t_2}\!\!\int_{\Omega} \partial_t \phi\,d\mu^{\varepsilon'_j,\,\sigma'_j}_t dt.
	\end{align}
	Since we have already proved $\mu^{\varepsilon'_j,\,\sigma'_j}_t \rightharpoonup \mu_t=\|V_t\|$ for all $t\geq 0$ on $\overline{\Omega}$, the left-hand side of \eqref{5.4.17} converges to that of \eqref{4.1.1.5} and so does the third term of the right hand side of \eqref{5.4.17}. Hence, combining \eqref{5.4.7} and \eqref{5.4.16} with \eqref{5.4.17} and taking $j\to\infty$, we may obtain
	\begin{align}
	\int_{\overline{\Omega}} \phi \,d\|V_t\| \Big|_{t=t_1}^{t_2}& \leq -\liminf_{j\to\infty} \left(\int_{t_1}^{t_2}\int_{\Omega}\varepsilon'_j \phi (\partial_t u^{\varepsilon'_j,\,\sigma'_j})^2 \,dx\,dt \right) \nonumber\\
	&\quad \qquad -\liminf_{j\to\infty}\left(\int_{t_1}^{t_2}\int_{\Omega} \nabla \phi \cdot \nabla u^{\varepsilon'_j,\,\sigma'_j} \varepsilon'_j(\partial_t u^{\varepsilon'_j,\,\sigma'_j}) \,dx\,dt\right) \nonumber\\
	&\quad \qquad \quad \qquad +\int_{t_1}^{t_2}\int_{\overline{\Omega}} \partial_t \phi\, d\|V_t\|\,dt \nonumber\\
	&\leq \int_{t_1}^{t_2}\int_{\overline{\Omega}}  \left(-\phi |\widetilde{\boldH}_V|^2 + \nabla \phi \cdot \widetilde{\boldH}_V\right) \,d\|V_t\|\,dt + \int_{t_1}^{t_2}\int_{\overline{\Omega}} \partial_t \phi\, d\|V_t\|\,dt \label{5.4.18}
	\end{align}
	This completes the proof of Brakke's inequality and thus we obtain Theorem \ref{thm3.4}. 
\end{proof}

\subsection{Dynamic boundary condition}\label{chara.dyna}
In this section, we prove a sequence of the main theorems in the case of dynamic boundary conditions, which we stated in Section \ref{mainresults.dyna}. We note that the positive constants $C_1$ and $D$ are as in Proposition \ref{prop.4.1} and Proposition \ref{prop.4.4} and also note that, for simplicity, we only consider the case that the parameter $\sigma$ is equal to 1, which is fixed in the following.
\subsubsection{Convergence of the measures $\{\mu^{\varepsilon,\,\sigma}\}_{\varepsilon>0}$ (dynamic boundary conditions)}\label{chara.6}
The convergence of a sequence of the measures $\{\mu^{\varepsilon,\,\sigma}_t\}_{\varepsilon>0}$ for all $t\geq 0$ can be proved in the similar manner shown in Subsection \ref{chara.1}. Indeed we can show essentially the same lemma as Lemma \ref{lem.5.1} by using the same arguments. Therefore, we do not repeat the proof of the convergence of $\{\mu^{\varepsilon,\,\sigma}_t\}_{\varepsilon>0}$ again here.

\subsubsection{Convergece of the measures $\{\alpha^{\varepsilon,\,\sigma}\}_{\varepsilon>0}$ and proof of Lemma \ref{thm3.7} and Lemma \ref{thm3.7.1} (dynamic boundary conditions)}\label{chara.7}
From Propostion \ref{prop.4.1} and \ref{prop.4.4}, we can also show, in the same manner described in Subsection \ref{chara.2}, the convergence of the measures $\{\alpha^{\varepsilon,\,\sigma}\}_{\varepsilon>0}$ and Lemma \ref{thm3.7}. Moreover, we can also prove Lemma \ref{thm3.7.1} in the case $\sigma\geq1$ in the similar manner to the proof of Lemma \ref{thm3.2.1}. However, as we impose the different asumptions from Lemma \ref{thm3.2.1} on Lemma \ref{thm3.7.1}, we will state the precise proof for clarification and thus show the proof of Lemma \ref{thm3.7.1} in the following. Note that we assume that ``General assumptions" is valid in this subsection.

\begin{proof}[Proof of Lemma \ref{thm3.7.1}]
	Let $\sigma$ be in $[1,\,\infty)$. First of all, we should say that we can apply the proof of Lemma \ref{thm3.2.1} if we have the following claim; there exist $0<s'<\infty$ and $\Gamma_2$ such that $\Gamma_2$ is a non-empty connected component of $\partial\Omega$, and the inequality
	\begin{equation}\label{6.2.0.1}
	\liminf_{j\to\infty}\int_{t_1}^{t_2}\!\left|\int_{\Gamma_2}w^{j}\,d\mathcal{H}^{n-1}\right|\,dt<\frac{2}{3}\mathcal{H}^{n-1}(\Gamma_2)\,(t_2-t_1)
	\end{equation}
	holds for any $0<t_1<t_2<s'$. Thus, it is sufficient to show this claim to prove Lemma \ref{thm3.7.1}.
	
	Now, from the first assumption of the initial data in Lemma \ref{thm3.7.1}, we can choose the constants $c_0>0$ and $\delta_0>0$ such that
	\begin{equation}\label{6.2.0.2}
	\liminf_{j\to\infty}\left|\int_{\Gamma_2}w^{j}_0\,d\mathcal{H}^{n-1}\right|<c_0<c_0+6\delta_0<\frac{2}{3}\mathcal{H}^{n-1}(\Gamma_2). 
	\end{equation}
	Then, from the second assumption, we can also choose the constant $\gamma_0>0$ such that
	\begin{equation}\label{6.2.0.1.1}
	\sup_{j\in\mathbb{N}}\mu^{j}_0(\Omega\cap\{\dist(x,\,\partial \Omega)<\gamma_0\})<\delta_0.
	\end{equation}
	Then, we may take $s_0>0$ such that
	\begin{equation}\label{6.2.0.1.2}
	0<s_0<\frac{\frac{2}{3}\mathcal{H}^{n-1}(\Gamma_2)-c_0-6\delta_0}{\frac{1}{2}D+C_3(\gamma_0)},
	\end{equation} 
	where $C_3(\gamma_0)$ is a positive constant depending only on $\gamma_0$, $\Gamma_2$, and $D$ in \eqref{4.1.2}, which we will choose later.
	
	Next, we will derive the local energy estimate on the boundary which is important to show \eqref{6.2.0.1}. First of all, we choose a bounded open set $U_2$ such that $\Gamma_2\subset U_2$ and $U_2\cap(\partial\Omega\setminus\Gamma_2)=\emptyset$ since $\mathbb{R}^n$ is Hausdorff and $\Gamma_2$ is a connected component of $\partial \Omega$. Moreover, we may choose $\rho>0$ such that $\{x\in\mathbb{R}^n\mid\dist(x,\,\Gamma_2)<\rho\}\subsetneq U_2$. Then, recalling the calculation of the a priori estimate in Proposition \ref{prop.4.2} or \ref{prop.4.4}, and taking the proper test function $\phi_{\gamma_0}\in C^2_c(\overline{\Omega})$ such that $0<\phi_{\gamma_0}\leq2$ in $\Omega$, $\phi_{\gamma_0}=1$ on $\Gamma_2$ and $\spt\phi_{\gamma_0}\subset U_2$, we may obtain, from \eqref{4.1.8} and \eqref{4.2.10.1}, 
	\begin{align}\label{6.2.0.3.0}
	\frac{d}{dt}\left(\int_{\Omega} \phi_{\gamma_0} \,d\mu^{j}_t \right) &\leq-\int_{\Omega} \Delta \phi_{\gamma_0}\, d\mu^{j}_t+\int_{\Omega}  \varepsilon_j (\nabla u^{j,\,\sigma} \otimes \nabla u^{j,\,\sigma} : \nabla ^2 \phi_{\gamma_0}) \, dx \nonumber\\
	&\qquad -\int_{\partial \Omega} \varepsilon_j (\nabla \phi_{\gamma_0} \cdot \nabla u^{j,\,\sigma}) \frac{\partial u^{j,\,\sigma}}{\partial \boldnu} \, d\mathcal{H}^{n-1}- \int_{\partial \Omega} \varepsilon\,\sigma \phi_{\gamma_0} \left( \frac{\partial u^{j,\,\sigma}}{\partial \boldnu}\right)^2\,d\mathcal{H}^{n-1} \nonumber\\
	&\quad \qquad +\int _{\partial \Omega} \frac{\partial \phi_{\gamma_0}}{\partial \boldnu} \left( \frac{\varepsilon_j |\nabla u^{j,\,\sigma}| ^2}{2} +\frac{W (u^{j,\,\sigma} )}{\varepsilon_j}\right)  \, d\mathcal{H}^{n-1}\nonumber\\
	&\leq 3D\sup_{\Omega}|\nabla^2\phi_{\gamma_0}|-\int_{\partial\Omega} \varepsilon_j (\nabla \phi_{\gamma_0} \cdot \nabla u^{j,\,\sigma}) \frac{\partial u^{j,\,\sigma}}{\partial \boldnu} \, d\mathcal{H}^{n-1} \nonumber\\
	&\qquad -\int_{\partial\Omega} \varepsilon_j\,\sigma \phi_{\gamma_0} \left( \frac{\partial u^{j,\,\sigma}}{\partial \boldnu}\right)^2\,d\mathcal{H}^{n-1} \nonumber\\
	&\quad \qquad +\int _{\partial\Omega} \frac{\partial \phi_{\gamma_0}}{\partial \boldnu} \left( \frac{\varepsilon_j |\nabla u^{j,\,\sigma}| ^2}{2} +\frac{W (u^{j,\,\sigma} )}{\varepsilon_j}\right)  \, d\mathcal{H}^{n-1}.
	\end{align}
	The way we construct the proper test function $\phi_{\gamma_0}$ is as follows; let $d_{\partial\Omega}$ be the signed distance function from $\partial \Omega$ which is positive in $\Omega$. Then, because of the smoothness of $\partial \Omega$, we can choose $\rho^{\prime}>0$ such that $d_{\partial\Omega}$ is smooth in $\{|d_{\partial\Omega}|<\rho^{\prime}\}$ and, moreover, setting $\tilde{\rho}\coloneqq\min\{1,\,\rho,\,2^{-1}\rho^{\prime}\}$ and $\tilde{\gamma}_0\coloneqq\min\{\gamma_0,\,\tilde{\rho}\}$, we can extend $d_{\partial\Omega}+1$ into the function $\tilde{d}_{\partial\Omega}$ such that $\tilde{d}_{\partial\Omega}$ is smooth on $\Omega\cup U_{\tilde{\gamma_0}}$ where $U_{\tilde{\gamma_0}}\coloneqq\{x\in\overline{\Omega} \mid |d_{\partial\Omega}|<\tilde{\gamma}_0\}$, $\tilde{d}_{\partial\Omega}$ is equal to $d_{\partial\Omega}+1$ in $\{x\in\overline{\Omega} \mid |d_{\partial\Omega}|<\tilde{\rho}\}$ and $|\tilde{d}_{\partial\Omega}|\leq2$ on $\overline{\Omega}$. Then, setting $\phi_{\gamma_0}\coloneqq\eta_{\tilde{\gamma}_0}\,\tilde{d}_{\partial\Omega}$ where $\eta=\eta_{\tilde{\gamma}_0}$ is the cut-off function such that $\spt\eta\subset U_{\tilde{\gamma}_0}\subsetneq U_2$, $0\leq\eta\leq1$, $\eta\equiv1$ in $U_{\frac{\tilde{\gamma}_0}{2}}\coloneqq\{x\in\overline{\Omega} \mid |d_{\partial\Omega}|<2^{-1}\tilde{\gamma}_0\}\subsetneq U_{\tilde{\gamma}_0}$, and $|\nabla\eta|<\infty$ on $\overline{\Omega}$, then we obtain the required test function satisfying $0<\phi_{\gamma_0}\leq2$ in $\Omega$, $\phi_{\gamma_0}=1$ on $\partial \Omega$, and $\spt\phi_{\gamma_0}\subset U_{\tilde{\gamma}_0}$. Then, from \eqref{6.2.0.3.0}, we can have the following calculation;
	\begin{align}
	\frac{d}{dt}\left(\int_{\Omega} \phi_{\gamma_0} \,d\mu^{j}_t \right) &\leq 3D\sup_{\Omega}|\nabla^2\phi_{\gamma_0}|-\int_{\partial \Omega\cap\spt\eta} \varepsilon_j\,\sigma\,\eta\left( \frac{\partial u^{j,\,\sigma}}{\partial \boldnu}\right)^2\,d\mathcal{H}^{n-1}\nonumber\\
	&\quad-\int_{\partial\Omega}\varepsilon_j(\nabla\eta\cdot\nabla u^{j,\,\sigma})(\nabla u^{j,\,\sigma}\cdot\boldnu)\,d\mathcal{H}^{n-1} \nonumber\\
	&\quad \qquad-\int_{\partial\Omega\cap\spt\eta}\varepsilon_j\,\eta(\nabla\tilde{d}_{\partial\Omega}\cdot \nabla u^{j,\,\sigma})(\nabla u^{j,\,\sigma}\cdot\boldnu)\,d\mathcal{H}^{n-1}\nonumber\\
	&\quad\qquad \quad+\int_{\partial\Omega}(\nabla\eta\cdot\boldnu)\left( \frac{\varepsilon_j |\nabla u^{j,\,\sigma}| ^2}{2} +\frac{W (u^{j,\,\sigma} )}{\varepsilon_j}\right)  \, d\mathcal{H}^{n-1}\nonumber\\
	&\quad\qquad\quad \qquad + \int_{\partial\Omega\cap\spt\eta}\eta(\nabla\tilde{d}_{\partial\Omega}\cdot\boldnu)\left( \frac{\varepsilon_j |\nabla u^{j,\,\sigma}| ^2}{2} +\frac{W (u^{j,\,\sigma} )}{\varepsilon_j}\right) \, d\mathcal{H}^{n-1}\nonumber\\
	&\leq C_3(\gamma_0)+\int_{\partial\Omega\cap U_{\tilde{\gamma}_0}}\varepsilon_j\,\eta\,\left(\frac{\partial u^{j,\,\sigma}}{\partial\nu}\right)^2(1-\sigma)\,d\mathcal{H}^{n-1} \nonumber\\
	&\quad-\int_{\partial\Omega\cap(U_2\setminus U_{\frac{\tilde{\gamma}_0}{2}})}\varepsilon_j(\nabla\eta\cdot\nabla u^{j,\,\sigma})(\nabla u^{j,\,\sigma}\cdot\boldnu)\,d\mathcal{H}^{n-1}\nonumber\\
	&\qquad +\int_{\partial\Omega\cap(U_2\setminus U_{\frac{\tilde{\gamma}_0}{2}})}(\nabla\eta\cdot\boldnu)\left( \frac{\varepsilon_j |\nabla u^{j,\,\sigma}| ^2}{2} +\frac{W (u^{j,\,\sigma} )}{\varepsilon_j}\right)\,d\mathcal{H}^{n-1}\nonumber\\
	&\quad \qquad-\int_{\partial\Omega\cap U_{\tilde{\gamma}_0}}\eta\,\left( \frac{\varepsilon_j |\nabla u^{j,\,\sigma}| ^2}{2} +\frac{W (u^{j,\,\sigma} )}{\varepsilon_j}\right) \, d\mathcal{H}^{n-1},\label{6.2.0.3.1}
	\end{align}
	where $C_3(\gamma_0)\coloneqq 3D\sup_{\Omega}|\nabla^2\phi_{\gamma_0}|<\infty$. Note that, from the definition of $U_{\tilde{\gamma}_0}$, we have $\partial\Omega\cap U_2\setminus U_{\tilde{\gamma}_0}=\emptyset$ and thus $\partial\Omega\cap\spt\eta=\partial\Omega\cap U_{\tilde{\gamma}_0}(=\Gamma_2)$. Therefore, from the choice of $\sigma$, we may obtain
	\begin{align}
	\frac{d}{dt}\left(\int_{\Omega} \phi_{\gamma_0} \,d\mu^{j}_t \right) &\leq C_3(\gamma_0) -\int_{\partial\Omega\cap U_{\frac{\tilde{\gamma}_0}{2}}}\left( \frac{\varepsilon_j |\nabla u^{j,\,\sigma}| ^2}{2} +\frac{W (u^{j,\,\sigma} )}{\varepsilon_j}\right) \, d\mathcal{H}^{n-1}\nonumber\\
	&\leq C_3(\gamma_0)-\int_{\Gamma_2}\left( \frac{\varepsilon_j |\nabla u^{j,\,\sigma}| ^2}{2} +\frac{W (u^{j,\,\sigma} )}{\varepsilon_j}\right) \, d\mathcal{H}^{n-1}.\label{6.2.0.3.3}
	\end{align}
	Therefore, integrating the both sides of \eqref{6.2.0.3.3} over $[0,\,s_0]$, we have that
	\begin{align}\label{6.2.0.3}
	\sup_{\varepsilon_j>0} \int_{0}^{s_0}\!\!\int_{\Gamma_2}  \left(\frac{\varepsilon_j |\nabla u^{j,\,\sigma}| ^2}{2} +\frac{W (u^{j,\,\sigma})}{\varepsilon_j}\right) \, d\mathcal{H}^{n-1} dt &\leq C_3(\gamma_0)\,s_0+2\sup_{j\in\mathbb{N}}\mu^{j}_0(\Omega\cap\spt\phi_{\gamma_0})\nonumber\\
	&\leq C_3(\gamma_0)\,s_0+2\sup_{j\in\mathbb{N}}\mu^{j}_0(\Omega\cap \{\dist(x,\,\partial\Omega)<\gamma_0\})\nonumber\\
	&< C_3(\gamma_0)\,s_0+2\delta_0,
	\end{align}
	Thirdly, we derive another important estimate on the boundary to show \eqref{6.2.0.1}. Taking the distance function $d_{\partial\Omega}$ as above, we may extend $-d_{\partial\Omega}+1$ into a smooth function $\tilde{d}^{-}_{\partial\Omega}$ on $\Omega\cup U_{\tilde{\gamma_0}}$ such that $\tilde{d}^{-}_{\partial\Omega} \equiv-d_{\partial\Omega}+1$ in $U_{\tilde{\gamma_0}}$. Note that, from the definition of $\tilde{\gamma_0}$, we have $\tilde{d}^{-}_{\partial\Omega}\geq0$ in $\Omega\cup U_{\tilde{\gamma_0}}$. Now defining a function $\phi_{\gamma_0}$ by $\eta_{\tilde{\gamma_0}}\tilde{d}^{-}_{\partial\Omega}+1$ where $\eta_{\tilde{\gamma_0}}$ is in the above, we obtain that $\phi_{\gamma_0}\geq 1$ in $\Omega$ and $|\nabla\phi_{\gamma_0}|\leq 1$. Then calculating the time derivative of $\int_{\Omega}\phi_{\gamma_0}\,d\mu^{j}_t(x)$, we have the following:
	\begin{align}
	\frac{d}{dt}\int_{\Omega}\phi_{\gamma_0}\,d\mu^{j}_t(x)&= 
	-\int_{\Omega} \varepsilon_j \phi_{\gamma_0} (\partial_t u^{j,\,\sigma})^2\,dx- \int_{\partial \Omega} \frac{\varepsilon_j}{\sigma} \phi_{\gamma_0} (\partial_t u^{j,\,\sigma})^2\, d\mathcal{H}^{n-1} \nonumber\\
	&\quad \qquad \quad \qquad- \int_{\Omega} \varepsilon_j \partial_t u^{j,\,\sigma} \nabla u^{j,\,\sigma}\cdot \nabla \phi_{\gamma_0} \, dx\nonumber\\
	&\leq -\int_{\Omega}\varepsilon_j\,\phi_{\gamma_0}\left(\partial_tu^{j,\,\sigma}+\frac{\nabla u^{j,\,\sigma}\cdot\nabla\phi_{\gamma_0}}{2\phi_{\gamma_0}} \right)^2\,dx+ \int_{\Omega}\frac{|\nabla \phi_{\gamma_0}|^2}{2\phi_{\gamma_0}}\,\frac{\varepsilon_j|\nabla u^{j,\,\sigma}|^2}{2}\,dx \nonumber\\
	&\quad \qquad \quad \qquad - \int_{\partial\Omega\cap U_{\tilde{\gamma}_0}} \frac{\varepsilon_j}{\sigma}\,\eta_{\tilde{\gamma_0}}\, (\partial_t u^{j,\,\sigma})^2\, d\mathcal{H}^{n-1} \nonumber\\
	&\leq \frac{1}{2}\mu^{j}_t(\Omega)- \int_{\Gamma_2} \frac{\varepsilon_j}{\sigma} (\partial_t u^{j,\,\sigma})^2\, d\mathcal{H}^{n-1}, \label{6.2.0.4}
	\end{align}
	Thus, integrating over $[0,\,s_0]$ in the both sides of \eqref{6.2.0.4}, we have
	\begin{equation}\label{6.2.0.6}
	\int_{0}^{s_0}\!\!\int_{\Gamma_2} \frac{\varepsilon_j}{\sigma} (\partial_t u^{j,\,\sigma})^2\, d\mathcal{H}^{n-1}\,dt\leq \frac{1}{2}D\,s_0+3\sup_{j\in\mathbb{N}}\mu^{j}_0(\Omega\cap \spt\phi_{\gamma_0}) \leq \frac{1}{2}D\,s_0+3\delta_0.
	\end{equation}
	Now, we will calculate the time derivative of $\int_{\Gamma_2}w^{j}\,d\mathcal{H}^{n-1}$ as follows; from Schwarz inequality, we have, for any $t\in[0,\,s_0]$,
	\begin{align}
	\frac{d}{dt}\int_{\Gamma_2}w^{j}\,d\mathcal{H}^{n-1}&= \int_{\Gamma_2}\sqrt{2W(u^{j})}\,(\partial_t u^{j})\,d\mathcal{H}^{n-1} \nonumber\\
	&\leq \left(\int_{\Gamma_2}2\frac{W(u^{j})}{\varepsilon_j}\,d\mathcal{H}^{n-1}\right)^{\frac{1}{2}}\left(\int_{\Gamma_2}\varepsilon_j(\partial_t u^{j})^2\,\mathcal{H}^{n-1} \right)^{\frac{1}{2}}. \label{6.2.0.7}
	\end{align}
	Then, integrating over $[0,\,s]$, where $s\in[0,\,s_0]$ is arbitrary, in the both sides of \eqref{6.2.0.7}, we obtain
	\begin{align}
	\left|\int_{\Gamma_2}w^{j}(s)\,d\mathcal{H}^{n-1}-\int_{\Gamma_2}w^{j}_0\,d\mathcal{H}^{n-1} \right|&\leq \int_{0}^{s_0}\!\!\int_{\Gamma_2}\frac{W(u^{j})}{\varepsilon_j}\,d\mathcal{H}^{n-1}\,dt+\int_{0}^{s_0}\!\!\int_{\Gamma_2}\varepsilon_j(\partial_t u^{j})^2\,d\mathcal{H}^{n-1}\,dt \nonumber\\
	&< \left(\frac{1}{2}D+C_3(\gamma_0)\right)\,s_0+5\delta_0, \label{6.2.0.8}
	\end{align}
	and thus, for any $s\in[0,\,s_0]$,
	\begin{align}
	\left|\int_{\Gamma_2} w^{j}(s)\,\mathcal{H}^{n-1}\right|&< \left(\frac{1}{2}D+C_3(\gamma_0)\right)\,s_0+5\delta_0+\left|\int_{\Gamma_2} w^{j}_0\,\mathcal{H}^{n-1}\right|, \label{6.2.0.8.1}
	\end{align}
	where we used \eqref{6.2.0.3} and \eqref{6.2.0.6} in the above. Therefore, for any $<t_1<t_2<s_0$, integrating over $s\in[t_1,\,t_2]$ in \eqref{6.2.0.8.1} and using the first assumption \eqref{6.2.0.2}, we conclude that
	\begin{align}\label{6.2.0.9}
	\liminf_{j\to\infty}\int_{t_1}^{t_2}\!\!\left|\int_{\Gamma_2} w^{j}(s)\,\mathcal{H}^{n-1}\right|\,ds &\leq \left(\left(\frac{1}{2}D+C_3(\gamma_0)\right)\,s_0+5\delta_0+\liminf_{j\to\infty}\left|\int_{\Gamma_2} w^{j}_0\,\mathcal{H}^{n-1}\right|\right)\,(t_2-t_1) \nonumber\\
	&\leq \left(\left(\frac{1}{2}D+C_3(\gamma_0)\right)\,s_0+5\delta_0+c_0\right)\,(t_2-t_1) \nonumber\\
	&\leq \left(\frac{2}{3}\mathcal{H}^{n-1}(\Gamma_2)-c_0-6\delta_0+5\delta_0+c_0\right)\,(t_2-t_1)\nonumber\\
	&=\left(\frac{2}{3}\mathcal{H}^{n-1}(\Gamma_2)-\delta_0\right)\,(t_2-t_1) < \frac{2}{3}\mathcal{H}^{n-1}(\Gamma_2)\,(t_2-t_1).
	\end{align}
	This completes the proof of \eqref{6.2.0.1}. 
\end{proof}

\subsubsection{First variations of associated varifolds and proof of Lemma \ref{thm3.8} (dynamic boundary conditions)}\label{chara.8}
In Subsection \ref{chara.6}, we have already proved the existence of the convergent subsequence $\{\mu^{\varepsilon'_j,\,\sigma}_t\}_{\varepsilon>0}$ such that it is independent of $t\in[0,\,\infty)$. Then, in this subsection, we mainly discuss the first variation of the varifold associated with $\mu^{\varepsilon'_j,\,\sigma}_t$ and we will show the proof of Lemma \ref{thm3.8} as we discussed the similar topics in Subsection \ref{chara.3}. Note that we assume that ``General assumptions" and ``Vanishing hypothesis for the discrepancy measure" in this subsection are valid.

First of all, we associate a varifold for each $\mu^{\varepsilon'_j,\,\sigma}_t$ as follows:
\begin{definition}\label{def.5.4}
	Let $\{u^{\varepsilon,\,\sigma}\}_{\varepsilon>0}$ be a family of the solutions of the equation \eqref{1.1.1} and $\mu^{\varepsilon,\,\sigma}_t$ be as in \eqref{1.1.5}. Then for $\psi\in C_c(G_{n-1}(\overline{\Omega}))$ and any $t\geq 0$, define
	\begin{equation}\label{6.2.1}
	V^{\varepsilon,\,\sigma}_t(\psi)\coloneqq\int_{\Omega \cap \{|\nabla u^{\varepsilon,\,\sigma}(\cdot,\,t)|\neq 0 \}} \psi(x,\,\textbf{I}-\textbf{a}^{\varepsilon,\,\sigma}\otimes \textbf{a}^{\varepsilon,\,\sigma})\,d\mu^{\varepsilon,\,\sigma}_t(x),
	\end{equation}
	where $\textbf{a}^{\varepsilon,\,\sigma}\coloneqq\frac{\nabla u^{\varepsilon,\,\sigma}}{|\nabla u^{\varepsilon,\,\sigma}|}$.
\end{definition}
From the definition, we may obtain $\|V^{\varepsilon,\,\sigma}_t\|=\mu^{\varepsilon,\,\sigma}_t\lfloor_{\{|\nabla u^{\varepsilon,\,\sigma}(\cdot,\,t)|\neq 0 \}}$, hence, by considering the first variation of $V^{\varepsilon,\,\sigma}_t$, we may derive the same formula as \eqref{5.3.2} in Lemma \ref{lem.5.4}.
\begin{lemma}\label{lem.5.5}
	Let $\{u^{\varepsilon,\,\sigma}\}_{\varepsilon,\,\sigma>0}$ and $\mu^{\varepsilon,\,\sigma}_t$ be as in Definition \ref{def.5.4}. Then, for any $\varepsilon>0$, $\sigma>0$, $t\geq 0$ and all $\boldg\in (C^1_c(\overline{\Omega}))^n$, we have 
	\begin{align}\label{6.2.2.1}
	\delta V^{\varepsilon,\,\sigma}_t(\boldg) &= \int _{\Omega} (\boldg\cdot \nabla u^{\varepsilon,\,\sigma}) \left(\varepsilon \Delta u^{\varepsilon,\,\sigma} -\frac{W'(u^{\varepsilon,\,\sigma})}{\varepsilon}\right) \, dx
	+  \int_{\Omega\cap\{ |\nabla u^{\varepsilon,\,\sigma}|\not=0 \}} \nabla \boldg : (\textbf{a}^{\varepsilon,\,\sigma}\otimes \textbf{a}^{\varepsilon,\,\sigma}) \,d\xi_t^{\varepsilon,\,\sigma} \nonumber\\
	&\quad + \int _{\partial \Omega} (\boldg\cdot \boldnu) \left( \frac{\varepsilon |\nabla u^{\varepsilon,\,\sigma}|^2}{2}+\frac{W(u^{\varepsilon,\,\sigma})}{\varepsilon}\right) \, d\mathcal{H}^{n-1}
	- \int _{\partial \Omega} \varepsilon (\boldg\cdot \nabla u^{\varepsilon,\,\sigma}) \frac{\partial u^{\varepsilon,\,\sigma}}{\partial \boldnu} \, d\mathcal{H}^{n-1} \nonumber\\
	&\qquad - \int_{\Omega\cap\{ |\nabla u^{\varepsilon,\,\sigma}|=0 \}} \nabla \boldg : \textbf{I} \frac{W(u^{\varepsilon,,\sigma})}{\varepsilon} \, dx.
	\end{align}
\end{lemma}
The proof of this lemma is the same as that of Lemma \ref{lem.5.4}, and thus we do not repeat it again.
\begin{proposition}\label{prop.6.1}
	Let $\{\varepsilon'_j\}_{j\in\mathbb{N}}$ be such that Lemma \ref{thm.3.6} and \ref{thm3.7} hold and let $\{u^{\varepsilon'_j,\,\sigma}\}_{j\in\mathbb{N}}$ satisfy the equations \eqref{1.1.1}. Then for a.e. $t\geq 0$, 
	\begin{align}\label{6.2.3}
	\liminf_{j\to\infty} & \left( \int_\Omega |\nabla u ^{\varepsilon'_j,\,\sigma}|\, \left|\varepsilon'_j \Delta u^{\varepsilon'_j,\,\sigma} -\frac{W'(u^{\varepsilon'_j,\,\sigma})}{\varepsilon'_j}\right| \, dx\right.+ \int _{\partial \Omega} \left( \frac{\varepsilon'_j |\nabla u^{\varepsilon'_j,\,\sigma}|^2}{2}+\frac{W(u^{\varepsilon'_j,\,\sigma})}{\varepsilon'}  \right)  \, d\mathcal{H}^{n-1} \nonumber\\
	&\quad \qquad \quad \qquad \left.+\int_{\partial \Omega} \varepsilon'_j |\nabla u^{\varepsilon'_j,\,\sigma}|\left|\frac{\partial u^{\varepsilon'_j,\,\sigma}}{\partial \boldnu}\right|\,d\mathcal{H}^{n-1}\right) <\infty.
	\end{align}
\end{proposition}
We should remark that it is not necessary to impose the "\textit{Uniform upper bound on the boundary}" \eqref{3.1.9.2} in Section \ref{exisdiri} as we do in the case of Dirichlet boundary conditions. In addition, recalling that the parameter $\sigma$ is positive and fixed, the proof of this proposition is basically the same as that of Proposition \ref{prop.5.5}. Hence we do not write the proof here again.

Next we show that the limit measure $\mu^{\sigma}_t$ is actually $(n-1)$-rectifiable on $\overline{\Omega}$ for a.e. $t\geq 0$ and a proper subsequence od the associated varifolds $\{V^{\varepsilon,\,\sigma}_t\}_{\varepsilon>0}$ converges uniquely to the varifold $V_t$ associated with $\mu^{\sigma}_t$ as $\varepsilon \downarrow 0$.
\begin{proposition}\label{prop.6.2}
	For a.e. $t\geq 0$, $\mu^{\sigma}_t$ is $(n-1)$-rectifiable on $\overline{\Omega}$ and any convergent subsequence $\{V^{\varepsilon''_j,\,\sigma}\}_{j\in\mathbb{N}}$ of $\{V^{\varepsilon'_j,\,\sigma}_t\}_{j\in\mathbb{N}}$, where $\{\varepsilon'_j\}_{j\in\mathbb{N}}$ is such that Lemma \ref{thm.3.6} and \ref{thm3.7} hold, converges to the unique $(n-1)$-rectifiable varifold $V_t$ associated with $\mu_t$. Moreover, we have
	\begin{equation}\label{6.2.4}
	\|\delta V^{\sigma}_t\|(\overline{\Omega})<\infty,\quad
	\int_{0}^{T}\|\delta V^{\sigma}_t\|(\overline{\Omega})\,dt<\infty
	\end{equation}
	for a.e. $t\geq0$ and any $T>0$ respectively.
\end{proposition}
We should remark that it is not necessary to impose the "\textit{Uniform upper bound on the boundary}" \eqref{3.1.9.2} in Section \ref{exisdiri} as we do in the case of Dirichlet boundary conditions. In addition, we may prove this proposition in the same manner as we show in the proof of Propostion \ref{prop.5.6}. Thus, we do not write the proof again here.

Finally, considering all the claims shown in Proposition \ref{prop.6.1} and \ref{prop.6.2}, we may conclude that Lemma \ref{thm3.8} is valid.

\subsubsection{Proof of Theorem \ref{thm3.9} (dynamic boundary conditions)}\label{chara.9}
In this section, we will prove Theorem \ref{thm3.9}, that is, the existence of the singular limits of the Allen-Cahn equations described in \eqref{1.1.1} by taking $\varepsilon\to0$ with fixed $\sigma\in[1,\,\infty)$. Before proving Theorem \ref{thm3.9}, as a preparation, we will show three propositions. First of all, we show that the first variation in an integral form $\int_{0}^{\infty}\delta V^{\varepsilon'_j,\,\sigma}_t\,dt$ converges to $\int_{0}^{\infty}\delta V^{\sigma}_t\,dt$ locally in time as $\varepsilon \to 0$, where the subsequence $\{V^{\varepsilon'_j,\,\sigma}_t\}_{j\in\mathbb{N}}$ has the limit varifold $V^{\sigma}_t$.

Note that, through this subsection, we assume that ``Generalized assumptions" and ``Vanishing hypothesis for the discrepancy measure" in Subsection \ref{exis.assmp.dyna} hold. Moreover, we only consider the subsequence $\{\varepsilon'_j\}_{j\in\mathbb{N}}$ such that Lemma \ref{thm.3.6} and \ref{thm3.7} hold.
\begin{proposition}\label{prop.6.3}
	Let $\{V^{\varepsilon'_j,\,\sigma}_t\}_{j\in\mathbb{N}}$ be a family of associated varifolds with $\mu^{\varepsilon'_j,\,\sigma}_t$ satisfying Proposition \ref{prop.6.1} and \ref{prop.6.2}. Then we have 
	\begin{equation}\label{6.2.11}
	\lim_{j\to\infty}\int_{0}^{T}\delta V^{\varepsilon'_j,\,\sigma}_t(\boldg)\,dt =  \int_{0}^{T} \delta V^{\sigma}_t(\boldg)\,dt
	\end{equation}
	for all $T>0$ and all $\boldg\in (C^1(\overline{\Omega}\times[0,\,\infty)))^n$ with $g(\cdot,\,t)\in (C^1_c(\overline{\Omega}))^n$.
\end{proposition}
The proof of this proposition can be done in the same way as Proposition \ref{prop.5.7} and hence we do not write the proof here.

\begin{proposition}\label{prop.6.4}
	Let $\{V_t^{\sigma}\}_{t\geq0}$ be as in Lemma \ref{thm3.8} and Suppose that $\alpha^{\sigma}$ and $\boldv_b^{\sigma}$ are followed from Lemma \ref{thm3.7} and $\delta V^{\sigma}_t\lfloor_{\partial \Omega}^T$ and $\mathcal{S}_{\alpha^{\sigma},\,{\boldv}^{\sigma}_b}$ are as in Definition \ref{def.3.3}. Then we obtain
	\begin{equation}\label{6.2.13}
	\left\|\int_{0}^{\infty}\delta V^{\sigma}_t\lfloor_{\Omega}\,dt+\int_{0}^{\infty} \delta V^{\sigma}_t\lfloor_{\partial \Omega}^T\,dt+ \sigma^{-1}\mathcal{S}_{\alpha^{\sigma},\,\boldv^{\sigma}_b}\right\|\ll \|V^{\sigma}_t\|\otimes \mathcal{L}^1_t\quad \text{on $\overline{\Omega}\times[0,\,\infty)$}.
	\end{equation}
	Thus, we may obtain the existence of the modified generalized mean curvature vector $\widetilde{\boldH}^{\sigma}_V$ (see Definition \ref{def.4.2}) and, moreover, ${\boldH}^{\sigma}_V$ belongs to $(L^2(\|V^{\sigma}_t\|\otimes\mathcal{L}^1_t,\,\overline{\Omega} \times [0,\,\infty)))^n$ and we also obtain \eqref{3.2.10} and the inequality
	\begin{equation}\label{6.2.14}
	\int_{t_1}^{t_2}\int_{\overline{\Omega}}\phi|{\boldH}^{\sigma}_V|^2\,d\|V^{\sigma}_t\|\,dt \leq \liminf_{j\to\infty}\int_{t_1}^{t_2}\int_{\Omega} \varepsilon'_j \phi (\partial_t u^{\varepsilon'_j,\,\sigma})^2\,dx\,dt
	\end{equation}
	for all $\phi \in C_c(\overline{\Omega}\times[0,\,\infty))$ with $\phi \geq 0$ and any $0<t_1\leq t_2<\infty$.
\end{proposition}
\begin{proof}
	First of all, we show the absolute continuity
	\begin{equation}\label{6.2.14.1}
	\left\|\delta V^{\sigma}_t\lfloor_{\Omega}\right\| \ll \|V^{\sigma}_t\|\quad\text{in $\Omega$ for a.e. $t\geq 0$.}
	\end{equation}
	To do this, we take any time $t\geq 0$ such that \eqref{6.2.3} and  the vanishing of the discrepancy measure $\{\xi^{\varepsilon'_j,\,\sigma}_t\}_{j\in \mathbb{N}}$ hold. Let $\{V^{\varepsilon'_j,\,\sigma}_t\}_{j\in\mathbb{N}}$ be a  subsequence converging to $V_t^{\sigma}$. Then, from \eqref{4.1.1}, \eqref{5.3.2} and the vanishing of $\xi^{\varepsilon'_j,\,\sigma}_t$, we have \eqref{5.4.8.2} in Propostition \ref{prop.5.8} for all $\boldg\in (C_c(\Omega))^n$. Thus, by taking the supremum with respect to $\boldg$, we have \eqref{6.2.14.1}. From Riesz theorem and Radon-Nikodym theorem, we have that, for a.e. $t\geq0$, there exists ${\boldH}^{\sigma}_V\lfloor_{ \Omega}(\cdot,\,t)\in (L^1_{loc}(\|V_t\|,\,\Omega))^n$ such that
	\begin{equation}\label{6.2.14.1.2}
	\left(\delta V^{\sigma}_t\lfloor_{\Omega}\right)(\boldg)= -\int_{\Omega}{\boldH}^{\sigma}_V\lfloor_{\Omega}(\cdot,\,t)\cdot \boldg \,d\|V^{\sigma}_t\|
	\end{equation}
	for all $\boldg\in (C_c(\Omega))^n$ and, moreover, we have that ${\boldH}^{\sigma}_V\lfloor_{\Omega}\in (L^2(\|V_t\|\otimes[0,\,\infty),\,\Omega\times[0,\,\infty)))^n$ is valid. 
	
	Now, given arbitrary $\delta>0$, we can take the function $\boldnu^{\delta}\in (C^1(\overline{\Omega}))^n$ such that $\boldnu^{\delta}\lfloor_{\partial \Omega} =\boldnu$, $|\boldnu^{\delta}|\leq 1$ and $\spt\boldnu^{\delta} \subset \Omega_{\delta}\coloneqq\{x \in\overline{\Omega}\mid\dist(x,\,\partial \Omega)<\delta\}$. This function can be simply constructed by using the signed distance function $d$ from $\partial \Omega$ and extending $d$ smoothly onto $\overline{\Omega}$. Then for any $\boldg\in (C^1_c(\overline{\Omega}\times [0,\,T])^n$, setting $\boldg^{\delta}\coloneqq\boldg-(\boldg\cdot \boldnu^{\delta})\boldnu^{\delta}$, we have $\boldg^{\delta}(\cdot,\,t)\cdot \boldnu=0$ on $\partial \Omega$ and $\delta V^{\sigma}_t \lfloor_{\partial \Omega}^T(\boldg)= \delta V^{\sigma}_t \lfloor_{\partial \Omega}(\boldg^{\delta})$. Now let $U\subset\subset\overline{\Omega} \times[0,\,\infty)$ be an open set and $\boldg\in (C^1_c(U))^n$ be any test function such that $|\boldg| \leq 1$. In order to prove \eqref{6.2.13}, we need to compute the following: 
	\begin{align}
	&\left(\int_{0}^{\infty}\delta V^{\sigma}_t \lfloor_{\partial \Omega}^T\,dt + \int_{0}^{\infty}\delta V^{\sigma}_t \lfloor_{\Omega}\,dt + \sigma^{-1}\mathcal{S}_{\alpha^{\sigma},\,{\boldv}^{\sigma}_b}\right)(\boldg) \nonumber\\
	&=\int_{0}^{\infty} \delta V^{\sigma}_t\lfloor_{\partial \Omega} (\boldg^{\delta})\,dt+\int_{0}^{\infty}\delta V^{\sigma}_t\lfloor_{\Omega} (\boldg^{\delta})\,dt+ \int_{0}^{\infty}\delta V^{\sigma}_t \lfloor_{\Omega} (\boldg-\boldg^{\delta})\,dt \nonumber\\
	&\quad \qquad \quad + \sigma^{-1}\int _{\partial \Omega \times[0,\,\infty)} \boldg\cdot {\boldv}^{\sigma}_b \, d\alpha^{\sigma} \nonumber\\
	&=\lim_{j\to\infty}\left(\int_0^{\infty}\delta V^{\varepsilon'_j,\,\sigma}_t (\boldg^{\delta})\,dt - \sigma^{-1} \int_0^{\infty} \int_{\partial \Omega} \boldg^{\delta} \cdot (\varepsilon'_j \partial_t u^{\varepsilon'_j,\,\sigma}\nabla u^{\varepsilon'_j,\,\sigma})\,d\mathcal{H}^{n-1}\,dt \right) \nonumber\\
	&\quad \qquad \quad+\int_0^{\infty} \delta V^{\sigma}_t\lfloor_{\Omega} (\boldg-\boldg^{\delta})\,dt \nonumber\\
	&= \lim_{j\to\infty} \int_0^{\infty}\int_{\Omega} (\boldg^{\delta} \cdot \nabla u^{\varepsilon'_j,\,\sigma} ) \varepsilon'_j(\partial_t u^{\varepsilon'_j,\,\sigma})\,dx\,dt+\int_0^{\infty}\delta V^{\sigma}_t \lfloor _{\Omega} (\boldg-\boldg^{\delta})\,dt. \nonumber\\
	&\leq  D^{\frac{1}{2}} \left(\int_{0}^{\infty}\int_{\overline{\Omega}} |\boldg|^2\, d\|V^{\sigma}_t\|\,dt\right)^{\frac{1}{2}} + \int_0^{\infty} \delta V^{\sigma}_t \lfloor _{\Omega} (\boldg-\boldg^{\delta}) \, dt. \label{6.2.14.2}
	\end{align} 
	Note that, in \eqref{6.2.14.2}, we used Cauchy-Schwarz inequality and \eqref{4.1.1}. From the definitions of $\boldnu^{\delta}$ and $\boldg^{\delta}$, we have
	\begin{align}\label{6.2.14.4}
	\left|\int_{0}^{\infty}\delta V^{\sigma}_t\lfloor_{\Omega}(\boldg-\boldg^{\delta})\,dt\right|&=\left|\int_{0}^{\infty}
	\int_{\Omega} {\boldH}^{\sigma}_V\lfloor_{\Omega}\cdot (\boldg-\boldg^{\delta})\,d\|V_t\|\,dt \right| \nonumber\\
	&\leq \int_{0}^{\infty}\int_{\Omega_\delta\setminus\partial\Omega} \left|{\boldH}^{\sigma}_V\lfloor_{\Omega}\right|\,d\|V_t\|\,dt \to 0
	\end{align}
	as $\delta \to 0$. From dominated convergence theorem, we obtain
	\begin{align}
	\left(\int_{0}^{\infty}\delta V^{\sigma}_t \lfloor^T_{\partial \Omega}\,dt+ \int_{0}^{\infty}\delta V^{\sigma}_t\lfloor_{\Omega}\,dt + \sigma^{-1}\mathcal{S}_{\alpha^{\sigma},\,{\boldv}^{\sigma}_b}\right)(\boldg) &\leq  D^{\frac{1}{2}}\,\left(\int_{0}^{\infty}\int_{\overline{\Omega}} |\boldg|^2\,d\|V^{\sigma}_t\|\,dt\right)^{\frac{1}{2}} \nonumber\\
	&\leq  D^{\frac{1}{2}}\,(\|V^{\sigma}_t\|\otimes \mathcal{L}^1_t(U))^{\frac{1}{2}} \label{6.2.14.5}
	\end{align}
	Therefore, taking the supremum with respect to $\boldg$, we obtain the absolute continuity \eqref{6.2.13} from the arbitrariness of $U \subset \overline{\Omega} \times [0,\,\infty)$ is arbitrary. From this, it follows that there exists a $\|V_t\|\otimes\mathcal{L}^1_t$-integrable vector-valued function $\widetilde{\boldH}_V^{\sigma}$ such that 
	\begin{equation}\label{6.2.14.6}
	\int_{0}^{\infty}\delta V^{\sigma}_t \lfloor^T_{\partial \Omega}\,dt+ \int_{0}^{\infty}\delta V^{\sigma}_t\lfloor_{\Omega}\,dt + \sigma^{-1}\mathcal{S}_{\alpha^{\sigma},\,{\boldv}^{\sigma}_b}=-\widetilde{\boldH}_V^{\sigma}\,\|V_t\|\otimes\mathcal{L}^1_t\quad\text{on $\overline{\Omega}\times[0,\,\infty)$},
	\end{equation}
	
	Next we will show that the Radon-Nikodym derivative $\widetilde{\boldH}^{\sigma}_V$ belongs to $(L^2(\|V^{\sigma}_t\|\otimes \mathcal{L}^1_t,\,\overline{\Omega} \times [0,\infty)))^n$. To prove this, we again use the approximation shownin \eqref{5.4.12} in the proof of Proposition \ref{prop.5.8} as follows: let $\boldg\in (C^1_c(\overline{\Omega}\times [0,\,\infty))^n$ with $\|\boldg\|_{L^2(\|V^{\sigma}_t\|\otimes \mathcal{L}^1_t,\,\overline{\Omega}\times[0,\,\infty))}\leq 1$. Then from \eqref{6.2.14.2} and \eqref{6.2.14.4}, we may compute as follows:
	\begin{align}
	\left|\int_{0}^{\infty}\int_{\overline{\Omega}} \widetilde{\boldH}^{\sigma}_V\cdot \boldg \,d\|V^{\sigma}_t\|\,dt \right| &\leq  \liminf_{j\to\infty} \left(\int_0^{\infty} \int_\Omega \varepsilon'_j (\partial_t u^{\varepsilon'_j,\,\sigma}))^2\,dx\,dt \right)^\frac{1}{2}\,\left(\int_{0}^{\infty}\int_{\overline{\Omega}} |\boldg|^2\, d\|V^{\sigma}_t\|\,dt\right)^{\frac{1}{2}} \nonumber\\
	&\quad \qquad  + \left|\int _0 ^T \delta V^{\sigma}_t \lfloor _{\Omega} (\boldg-\boldg^{\delta}) \, dt \right| \nonumber\\
	&\leq D^{\frac{1}{2}} + \left|\int _0 ^T \delta V^{\sigma}_t \lfloor _{\Omega} (\boldg-\boldg^{\delta})\,dt\right| \nonumber\\
	&\rightarrow D^{\frac{1}{2}}\quad\text{as $\delta\to0$}. \label{6.2.14.8}
	\end{align}
	Here we used the estimate \eqref{4.1.1}. Hence, from \eqref{5.4.12}, we have	
	\begin{equation}\label{6.2.14.9}
	\int_{0}^{\infty}\int_{\overline{\Omega}} |\widetilde{\boldH}^{\sigma}_V|^2\,d\|V^{\sigma}_t\|\,dt \leq D.
	\end{equation}
	Therefore we may conclude that $\widetilde{\boldH}^{\sigma}_V$ belongs to $(L^2(\|V^{\sigma}_t\|\otimes \mathcal{L}^1_t,\,\overline{\Omega} \times [0,\infty)))^n$ and thus we conclude that $\widetilde{\boldH}_V^{\sigma}$ is actually the modified generalized mean curvature vector. 
	
	Finally we need to prove \eqref{6.2.14} for all $\phi \in C_c(\overline{\Omega} \times [0,\,\infty))$ with $\phi\geq 0$. To prove this, we may carry out the approximation argument which is stated by Ilmanen \cite{Ilmanen01}. We can apply it to our problem because the associated varifold $V^{\sigma}_t$ is $(n-1)$-rectifiable on $\overline{\Omega}$ for a.e. $t\geq0$. Therefore, Proposition \ref{prop.6.4} follows.  
\end{proof}

Considering all the claims in Proposition \ref{prop.6.3} and \ref{prop.6.4}, we obtain the absolute continuity \eqref{4.1.1.6} and the estimate \eqref{3.2.10}.

Now we prove Brakke's inequality \eqref{4.1.1.10} and this completes the proof of Theorem \ref{thm3.9}. First of all, we show that the $\varepsilon$-approximated velocity vector converges to the modified generalized mean curvature vector $\widetilde{\boldH}^{\sigma}_V$ up to the boundary.
\begin{proposition}\label{prop.6.5}
	Let $\{V^{\varepsilon'_j,\,\sigma}_t\}_{j\in\mathbb{N}}$ be as in Proposition \ref{prop.6.1} and \ref{prop.6.2} for a.e. $t\geq0$ and let $\widetilde{\boldH}^{\sigma}_V$ be as in Proposition \ref{prop.6.4}. Then we have 
	\begin{equation}\label{6.2.15}
	\lim_{j\to\infty} \int_{0}^{\infty}\int_{\Omega} (\boldg\cdot \nabla u^{\varepsilon'_j,\,\sigma})\varepsilon'_j(\partial_t u^{\varepsilon'_j,\,\sigma})\,dx\,dt= -\int_{0}^{\infty}\int_{\overline{\Omega}}\boldg\cdot \widetilde{\boldH}^{\sigma}_V\,d\|V^{\sigma}_t\|\,dt
	\end{equation}
	for all $\boldg\in(C^1_c(\overline{\Omega}\times [0,\,\infty)))^n$ with $\boldg(\cdot,\,t)\cdot \boldnu=0$ on $\partial \Omega$.
\end{proposition}
\begin{proof}
	Let $\boldg\in(C^1_c(\overline{\Omega}\times [0,\,\infty)))^n$ be such that $\boldg(\cdot,\,t) \cdot \boldnu=0$ on $\partial \Omega$. We have already shown that $V^{\varepsilon'_j,\,\sigma}_t$ converges to $V^{\sigma}_t$ as $j\to\infty$ for a.e. $t\geq0$ and 
	\begin{equation}\label{6.2.16.1}
	\lim_{j\to\infty}\int_{0}^{\infty}\delta V^{\varepsilon'_j,\,\sigma}_t\,dt(\boldg)=\int_{0}^{\infty}\delta V_t\,dt(\boldg)
	\end{equation}
	for any $\boldg\in(C^1_c(\overline{\Omega}\times[0,\,\infty)))^n$. Furthermore, from the choice of $\boldg$, the third term of the left-hand side in \eqref{6.2.2.1} vanishes if we substitute $\boldg$ in \eqref{6.2.2.1}, and we have $\boldg=\boldg-(\boldg\cdot \boldnu)\boldnu$ on $\partial \Omega \times [0,\,\infty)$. Therefore, from \eqref{6.2.14.2}, \eqref{6.2.14.6}, Proposition \ref{prop.6.2}, and \ref{prop.6.4}, we obtain \eqref{6.2.15}. 
\end{proof}
Second, we will show the following inequality.
\begin{proposition}\label{prop.6.6}
	Let ${\boldv}^{\sigma}_b$ and $\alpha^{\sigma}$ be as in Lemma \ref{thm3.7}. Then we have
	\begin{equation}\label{6.2.16}
	\int_{0}^{\infty}\int_{\partial \Omega}\phi |{\boldv}^{\sigma}_b|^2\,d\alpha^{\sigma} \leq \liminf_{j\to\infty} \int_{0}^{\infty}\int_{\partial \Omega} \varepsilon'_j  \phi(\partial_t u^{\varepsilon'_j,\,\sigma})^2 \,d\mathcal{H}^{n-1}\,dt
	\end{equation}
	for all $\phi \in C^1_c(\overline{\Omega} \times [0,\infty))$ with $\phi \geq 0$.
\end{proposition}
\begin{proof}
	Let $\phi$ be in $C^1_c(\overline{\Omega} \times [0,\infty))$ with $\phi \geq 0$. Since $\alpha^{\sigma}$ is a Radon measure on $\partial \Omega \times [0,\infty)$, $(C_c(\partial \Omega \times [0,\infty)))^n$ is dense in $(L^2(\alpha,\,\partial \Omega \times [0,\infty)))^n$ and thus we can choose the sequence $\{\boldg^{\sigma}_m\}_{m\in \mathbb{N}}$ of $(C_c(\partial \Omega \times [0,\infty)))^n$ approximating ${\boldv}^{\sigma}_b$ in $L^2$-norm. Here the sequence $\{\phi \boldg^{\sigma}_m\}_{m\in \mathbb{N}}$ is also a subset of $(C_c(\partial \Omega \times [0,\infty)))^n$ and then, from \eqref{4.1.1}, $\sigma=1$ and Cauchy-Schwarz inequality, we can compute as follows:
	\begin{align}
	\left(-\int_{\partial \Omega\times [0,\,\infty)} \phi \boldg^{\sigma}_m\cdot {\boldv}^{\sigma}_b\, d\alpha^{\sigma} \right)^2 &=  \lim_{j\to\infty} \left( \int_{\partial \Omega \times [0,\,\infty)} \phi \boldg^{\sigma}_m\cdot \boldv^{\varepsilon'_j,\,\sigma}_b\, d\alpha^{\varepsilon'_j,\,\sigma} \right)^2 \nonumber\\
	&\leq \left(\int_{\partial \Omega \times [0,\,\infty)} \phi |\boldg^{\sigma}_m|^2\, d\alpha^{\sigma}\right) \nonumber\\
	&\quad \qquad  \times\left( \liminf_{j\to\infty} \int_{\partial \Omega \times [0,\,\infty)} \varepsilon'_j  \phi (\partial_t u^{\varepsilon'_j,\,\sigma})^2\, d\mathcal{H}^{n-1}\,dt \right) \label{6.2.17}
	\end{align}
	for all $m\in \mathbb{N}$. In \eqref{6.2.17}, we used the convergence of $\{\alpha^{\varepsilon'_j,\,\sigma}\}_{j\in\mathbb{N}}$ in the sense of Radon measures. Then, from the definition of $\{\boldg^{\sigma}_m\}_{m\in \mathbb{N}}$, we can show 
	\begin{align}
	\left(-\int_{\partial \Omega\times [0,\,\infty)} \phi \boldg^{\sigma}_m\cdot {\boldv}^{\sigma}_b\, d\alpha^{\sigma} \right)^2 & \xrightarrow[m\to \infty]{} \left(\int_{\partial \Omega\times [0,\,\infty)} \phi |{\boldv}^{\sigma}_b|^2\, d\alpha^{\sigma} \right)^2  \label{6.2.18}\\
	\int_{\partial \Omega \times [0,\,\infty)} \phi |\boldg^{\sigma}_m|^2\, d\alpha^{\sigma} & \xrightarrow[m \to \infty]{} \int_{\partial \Omega\times [0,\,\infty)} \phi |{\boldv}^{\sigma}_b|^2\, d\alpha^{\sigma}. \label{6.2.19}
	\end{align}
	Therefore by substituting \eqref{6.2.18} and \eqref{6.2.19} into \eqref{6.2.17}, we obtain \eqref{6.2.16}. 
\end{proof}
Finally, considering all of the above arguments, we can prove Brakke's inequality \eqref{4.1.1.10} and Theorem \ref{thm3.9}.

\begin{proof}[Proof of Brakke's inequality]
	In Lemma \ref{thm.3.6} and \ref{thm3.7}, we can take the same subsequence $\{\varepsilon'_j\}_{j\in\mathbb{N}}$ such that $\mu^{\varepsilon'_j,\,\sigma}_t$ converges to $\mu^{\sigma}_t$ on $\overline{\Omega}$ for all $t\geq 0$ and $\alpha^{\varepsilon'_j,\,\sigma}$ converges to $\alpha^{\sigma}$ on $\partial \Omega \times [0,\,\infty)$, and thus it is sufficient to consider such a subsequence in the following. Let $\phi \in C^1_c(\overline{\Omega}\times[0,\,\infty))$ be such that $\phi \geq 0$ and $\nabla \phi(\cdot,\,t)\cdot\boldnu=0$ for any $t\geq0$. For any $0\leq t_1< t_2<\infty$, from \eqref{4.1.6} in Proposition \ref{prop.4.2} and the notation $f^{\varepsilon'_j,\,\sigma}\coloneqq-\varepsilon'_j\Delta u^{\varepsilon'_j,\,\sigma}+\frac{W'(u^{\varepsilon'_j,\,\sigma})}{\varepsilon'_j}$, we have that
	\begin{align}
	\mu^{\varepsilon'_j,\,\sigma}(\phi)\Big|_{t=t_1}^{t_2}&= \int_{t_1}^{t_2}\left(\int_{\Omega} -\frac{1}{\varepsilon'_j} \phi (f^{\varepsilon'_j,\,\sigma})^2  + f^{\varepsilon'_j,\,\sigma}\nabla \phi\cdot \nabla u^{\varepsilon'_j,\,\sigma} \,dx +\int_{\Omega} \partial_t \phi\,d\mu^{\varepsilon'_j,\,\sigma}_t \right)dt \nonumber\\
	&\quad \qquad - \frac{1}{\sigma}\int_{t_1}^{t_2}\int_{\partial\Omega} \phi \varepsilon'_j(\partial_t u^{\varepsilon'_j,\,\sigma})^2\,d\mathcal{H}^{n-1}\,dt \label{6.2.20} 
	\end{align}
	Since $\mu^{\varepsilon'_j,\,\sigma}_t$ converges to $\mu^{\sigma}_t=\|V_t\|$ on $\overline{\Omega}$ for all $t\geq 0$, the left hand side of \eqref{6.2.20} converges to that of \eqref{4.1.1.10} and so does the third term of the right hand side of \eqref{6.2.20}. Hence, combinig \eqref{6.2.14}, \eqref{6.2.15} and \eqref{6.2.16} with \eqref{6.2.20} and taking $j\to\infty$, we obtain
	\begin{align}
	\int_{\overline{\Omega}} \phi \,d\|V^{\sigma}_t\| \Big|_{t=t_1}^{t_2} &\leq - \liminf_{j\to\infty} \left(\int_{t_1}^{t_2}\int_{\Omega} \varepsilon'_j (\partial_t u^{\varepsilon'_j,\,\sigma})^2\,dx\,dt \right) + \int_{t_1}^{t_2}\int_{\overline{\Omega}} \nabla \phi \cdot \widetilde{\boldH}^{\sigma}_V\,d\|V^{\sigma}_t\|\,dt \nonumber\\
	&\quad \qquad +\int_{t_1}^{t_2}\int_{\overline{\Omega}} \partial_t \phi\,d\|V_t\|\,dt - \liminf_{j\to\infty} \frac{1}{\sigma}\int_{t_1}^{t_2}\int_{\partial \Omega} \varepsilon'_j\,\phi\,(\partial_t u^{\varepsilon'_j,\,\sigma'_j})^2\,d\mathcal{H}^{n-1}\,dt \nonumber\\
	&\leq \int_{t_1}^{t_2}\int_{\overline{\Omega}} (-\phi|\widetilde{\boldH}^{\sigma}_V|^2+ \nabla \phi \cdot \widetilde{\boldH}^{\sigma}_V+ \partial_t \phi)\,d\|V^{\sigma}_t\|\,dt-\frac{1}{\sigma}\int_{\partial \Omega\times[t_1,\,t_2]} \phi |{\boldv}^{\sigma}_b|^2\,d\alpha^{\sigma}. \label{6.2.21}
	\end{align}  
	This completes the proof of Brakke's inequality. Hence we obtain Theorem \ref{thm3.9}. 
\end{proof}

\section{Appendix}\label{appendix}
\subsection{Appendix A}\label{appendixMaximalPrinciple}
In this appendix, we show two results whose proofs are based on the maximal principle; one is that the absolute value of a solution of the Allen-Cahn equation \eqref{1.1.1} with $\sigma>0$ is bounded by 1 in $\Omega\times(0,\,T]$ for given $T>0$, namely, we show
\begin{equation}
	\sup_{\Omega\times(0,\,T]}|u^{\varepsilon,\,\sigma}|\leq 1
\end{equation}
for any $\varepsilon,\,\sigma>0$ (see Proposition \ref{appendAProp} for more detail). The other one is that the discrepancy measure $\xi_t^{\varepsilon,\,\sigma}$ associated with the solution is bounded from above in $\Omega \times (0,\,T]$ for each $T>0$ uniformly in $\varepsilon$ and $\sigma$ under several technical assumptions (see Proposition \ref{appendAProp2} for more detail). Note that the proof of the second result is inspired by \cite[Proposition 4.1]{MiTo}. Before stating and proving the propositions, we fix some notations as follows;
\begin{description}
	\item $n\in\mathbb{N}$ with $n\geq1$, $T>0$ 
	\item $\Omega\subset\mathbb{R}^n$ : a bounded domain with smooth boundary
	\item $Q_T\coloneqq\Omega\times(0,\,T)$, $Q_T^{*}\coloneqq\Omega\times(0,\,T]$, and $\partial_pQ_T\coloneqq\overline{Q_T}\setminus Q_T^{*}$
	\item $\kappa_1,\cdots,\,\kappa_{n-1}$ : the principal curvatures of $\partial \Omega$
	\item $\Pi_{\nu}$ : the second fundamental form of $\partial \Omega$ with the outer unit normal $\nu$ 
\end{description}
Now we are prepared for the claim. 
\begin{proposition}\label{appendAProp}
	Assume that the solution $u^{\varepsilon,\,\sigma}$ of the Allen-Cahn equation \eqref{1.1.1} is in $C^2(\overline{\Omega}\times(0,\,T])$ for each $T>0$ and $\sup_{\Omega}|u^{\varepsilon,\,\sigma}_0|\leq1$ for any $\varepsilon,\,\sigma>0$. Then it holds that, for any $\varepsilon,\,\sigma>0$, $\sup_{\Omega\times(0,\,\infty)}|u^{\varepsilon,\,\sigma}|\leq1$.
\end{proposition} 

\begin{proof}
	It is sufficient to prove the estimate for the time interval $(0,\,T]$ for any $T>0$. Defining a function $w$ by $w(x,\,t)\coloneqq((u^{\varepsilon,\,\sigma}(x,\,t))^2-1)\,e^{-\gamma t}$ for any $(x,\,t)\in Q_T$ where we set $\gamma\coloneqq\frac{4}{\varepsilon^2}$, we may have 
	\begin{align}
	\partial_t w-\Delta w&=-\gamma\,e^{-\gamma t}((u^{\varepsilon,\,\sigma})^2-1)+2u^{\varepsilon,\,\sigma}
	\,\partial_t u^{\varepsilon,\,\sigma}\,e^{-\gamma t}- 2u^{\varepsilon\,\sigma}\,\Delta u^{\varepsilon,,\sigma}\,e^{\gamma t}- 2e^{-\gamma t}|\nabla u^{\varepsilon,\,\sigma}|^2 \nonumber\\
	&=-\gamma w +2u^{\varepsilon,\,\sigma}\,e^{-\gamma t}(\partial_tu^{\varepsilon,\,\sigma}-\Delta u^{\varepsilon,\,\sigma})- 2e^{-\gamma t}|\nabla u^{\varepsilon,\,\sigma}|^2 \nonumber\\
	&=-\gamma w +2u^{\varepsilon,\,\sigma}\,e^{-\gamma t}\left(-\frac{1}{\varepsilon^2}W^{\prime}(u^{\varepsilon,\,\sigma})\right)- 2e^{-\gamma t}|\nabla u^{\varepsilon,\,\sigma}|^2\nonumber\\
	&=-\gamma w +\frac{4}{\varepsilon^2}e^{-\gamma t}(u^{\varepsilon,\,\sigma})^2(1-(u^{\varepsilon,\,\sigma})^2)- 2e^{-\gamma t}|\nabla u^{\varepsilon,\,\sigma}|^2\nonumber\\
	&=-\gamma w - \gamma w (1+e^{\gamma t}w)- 2e^{-\gamma t}|\nabla u^{\varepsilon,\,\sigma}|^2\nonumber\\
	&=-2\gamma w-\gamma e^{\gamma t}w^2- 2e^{-\gamma t}|\nabla u^{\varepsilon,\,\sigma}|^2\leq -2\gamma w\quad \text{in $\Omega\times(0,\,T]$.} \label{app.b.1}
	\end{align}
	Setting $Q_T^{+}\coloneqq\{(x,\,t)\in\overline{Q_T} \mid w(x,\,t)>0\}$, we have, from the assumption in the above, that $Q_T^{+}\subset\overline{\Omega}\times(0,\,T]$. Here, if either $u^{\varepsilon,,\sigma}\equiv+1$ or $u^{\varepsilon,,\sigma}\equiv-1$, then there is nothing to prove. Hence, we can assume, in the following, that $u^{\varepsilon,,\sigma}\not\equiv\pm1$ and this implies $Q_T^{+}\neq\emptyset$. Now, since $\Omega$ is bounded, there exists a point $(x_0,\,t_0)\in\overline{Q_T}$ such that $\max_{\overline{Q_T}}w=w(x_0,\,t_0)$. Assume, for a contradiction, that $(x_0,\,t_0)\in Q_T^{+}$. From the assumption that $|u^{\varepsilon,\,\sigma}_0|\leq1$, we have that $w(\cdot,\,0)\leq0$ and thus, $t_0\neq0$. If $x_0\in\Omega$, then, from the maximality of $w$ and $t_0\neq0$, it holds that $\partial_t w(x_0,\,t_0)\geq0$ and $\Delta w(x_0,\,t_0)\leq0$. Thus we have that
	\begin{equation}\label{app.b.2}
	0\leq\partial_t w(x_0,\,t_0)-\Delta w(x_0,\,t_0)\leq-2\gamma w <0\quad \text{at $(x_0,\,t_0)$},
	\end{equation}
	which is a contradiction. If $x_0\in\partial\Omega$, then,
	from the maximality of $w$ at $(x_0,\,t_0)$, we have that $\nabla w\cdot\nu\geq0$ at $(x_0,\,t_0)$, where $\nu$ is the outer unit normal of $\partial \Omega$. From the boundary condition in the equations \eqref{1.1.1}, we have that
	\begin{align}
	0\leq\partial_t w+ \sigma \nabla w \cdot \nu&= -\gamma e^{-\gamma t}((u^{\varepsilon,\,\sigma})^2-1)+ 2u^{\varepsilon,\,\sigma}\partial_t u^{\varepsilon,\,\sigma}\,e^{-\gamma t}+ 2\sigma u^{\varepsilon,\,\sigma}\,e^{-\gamma t}\nabla u^{\varepsilon,\,\sigma}\cdot\nu\nonumber\\
	&= -\gamma w + 2u^{\varepsilon,\,\sigma}\,e^{-\gamma t}(\partial_t u^{\varepsilon,\,\sigma}+\sigma\nabla u^{\varepsilon,\,\sigma}\cdot\nu)\nonumber\\
	&=-\gamma w<0\quad \text{at $(x_0,\,t_0)$,}\label{app.b.3}
	\end{align}
	which is also a contradiction. Therefore, both cases lead us to get a contradiction. Hence, for any $(x,\,t)\in \overline{Q_T}$, $w(x,\,t)\leq \max_{\overline{Q_T}}w= w(x_0,\,t_0)\leq 0$, in other words, $|u^{\varepsilon,\,\sigma}(x,\,t)|\leq1$ holds.
\end{proof}

Next, we will show the upper bound of the discrepancy function $\xi^{\varepsilon,\,\sigma}$ in $\Omega \times (0,\,T]$ for given $T>0$. Recall that $\xi^{\varepsilon,\,\sigma}(x,\,t)$ is defined by
\begin{equation}
	\frac{\varepsilon}{2}|\nabla u^{\varepsilon,\,\sigma}|^2(x,\,t) - \frac{1}{\varepsilon} W(u^{\varepsilon,\,\sigma})(x,\,t)
\end{equation}
for any $x \in \Omega$ and $t > 0$ (see also \ref{discrepancyFunc}). Precisely, we will prove 
\begin{proposition}\label{appendAProp2}
	Let $\Omega \subset \mathbb{R}^n$ be a convex set and let $T>0$ be any number. Assume that the solution $u^{\varepsilon,\,\sigma}$ is in $C^3(\Omega\times(0,\,T]) \cap C^2(\overline{\Omega} \times(0,\,T])$ is of the Allen-Cahn equation \eqref{1.1.1} and that there exists a constant $c_0>0$ independent of $\varepsilon$, $\sigma$, and $T$ such that $\sup_{\Omega} \xi^{\varepsilon,\,\sigma}_0 \leq c_0$ for any $\varepsilon,\,\sigma>0$. Then, if $u^{\varepsilon,\,\sigma}$ satisfies the condition 
	\begin{equation}\label{assumptionOnBoundary}
		\frac{\partial}{\partial t} (\nabla u^{\varepsilon,\,\sigma} \cdot \nu)^2(t) \leq  c_0\,\sigma\varepsilon^{-1}e^{\sigma(t-T)} \quad \text{on } \partial \Omega
	\end{equation}
	at $t\in(0,\,T]$ for any $\varepsilon,\,\sigma>0$, then it holds that
	\begin{equation}\label{boundednessDiscrepancy}
		\sup_{\Omega\times(0,\,T]}\xi^{\varepsilon,\,\sigma} \leq c_0
	\end{equation}
	for any $\varepsilon,\,\sigma>0$.
\end{proposition}

\begin{remark}[Comparison between the assumptions \eqref{3.1.9.2} and \eqref{assumptionOnBoundary}]
	Let us recall that, in Section \ref{exisdiri}, we impose the similar assumption to \eqref{assumptionOnBoundary} on the gradient of the solutions $u^{\varepsilon,\,\sigma}$ in direction to the outer unit normal of $\partial \Omega$ (see \eqref{3.1.9.2} for the precise assumption)and we can actually see by simple computation that the assumption \eqref{assumptionOnBoundary} implies the assumption \eqref{3.1.9.2} in Section \ref{exisdiri}. Indeed, from \eqref{assumptionOnBoundary} and solving the differential inequality, we obtain
	\begin{equation}\label{remarkAssumpBoundary}
		(\nabla u^{\varepsilon,\,\sigma} \cdot \nu)^2(t) \leq c_0\,\varepsilon^{-1}e^{\sigma(t-T)} + (\nabla u^{\varepsilon,\,\sigma}_0 \cdot \nu)^2 \quad \text{on } \partial \Omega
	\end{equation} 
	for any $t\in(0,\,T]$. Thus, integrating over $\partial \Omega \times [0,\,T]$ on the both sides of \eqref{remarkAssumpBoundary}, we have
	\begin{equation}\label{integrateAssumpBoundary}
		\int_{0}^{T}\int_{\partial\Omega}\frac{\varepsilon}{2}\left(\frac{\partial u^{\varepsilon,\,\sigma}}{\partial \nu}\right)^2\,d\mathcal{H}^{n-1}\,dt \leq \frac{c_0}{2\sigma}(1 - e^{-\sigma T})\mathcal{H}^{n-1}(\partial\Omega) + T\int_{\partial\Omega}\frac{\varepsilon}{2}\left(\frac{\partial u^{\varepsilon,\,\sigma}_0}{\partial \nu}\right)^2\,d\mathcal{H}^{n-1}.
	\end{equation}
	We observe that the right-hand side in \eqref{integrateAssumpBoundary} is uniformly finite with respect to both $\varepsilon>0$ and $\sigma>0$ since $\sigma^{-1}(1- e^{-\sigma\,T}) \rightarrow T$ as $\sigma \to 0$. On the other hand, we note that the condition \eqref{3.1.9.2} does not necessarily imply the condition \eqref{assumptionOnBoundary}.
\end{remark}

\begin{proof}
	Defining the function $\tilde{\xi}^{\varepsilon,\,\sigma}$ by $\tilde{\xi}^{\varepsilon,\,\sigma}(x,\,t) \coloneqq e^{-\sigma\,t}\,\xi^{\varepsilon,\,\sigma}(x,\,t)$ and employing the argument in \cite[Proposition 4.1]{MiTo}, we may compute as follows; for $(x,\,t) \in \Omega \times (0,\,T]$ with $|\nabla u^{\varepsilon,\,\sigma}(x,\,t)| \neq 0$, it holds that
	\begin{equation}\label{heatEquationDiscrepancy}
		\partial_t \tilde{\xi}^{\varepsilon,\,\sigma} - \Delta \tilde{\xi}^{\varepsilon,\,\sigma} \leq -\sigma\,\tilde{\xi}^{\varepsilon,\,\sigma} - \frac{2W'}{\varepsilon^3|\nabla u^{\varepsilon,\,\sigma}|^2} \nabla \tilde{\xi}^{\varepsilon,\,\sigma} \cdot \nabla u^{\varepsilon,\,\sigma}.
	\end{equation}
	Since $\Omega$ is bounded, there exists a maximum point $(x_0,\,t_0) \in \overline{\Omega} \times [0,\,T]$ of $\tilde{\xi}^{\varepsilon,\,\sigma}$. Suppose by contradiction that $\tilde{\xi}^{\varepsilon,\,\sigma}(x_0,\,t_0) > c_0\,e^{-\sigma\,T}$. If we are able to derive a contradiction, then we obtain $\max_{\overline{\Omega} \times [0,\,T]}\tilde{\xi}^{\varepsilon,\,\sigma} \leq c_0\,e^{-\sigma\,T}$ and thus it holds that
	\begin{equation}
		\max_{\overline{\Omega} \times [0,\,T]}\xi^{\varepsilon,\,\sigma} = \max_{\overline{\Omega} \times [0,\,T]} e^{\sigma\,t}\,e^{-\sigma\,t}\xi^{\varepsilon,\,\sigma} \leq e^{\sigma\,T} c_0\,e^{-\sigma\,T} = c_0. 
	\end{equation}
	Thus to derive a contradiction leads us to complete the proof of Proposition \ref{appendAProp2}. First of all, if $t_0 = 0$, then from the assumption on $\xi^{\varepsilon,\,\sigma}_0$ we have
	\begin{equation}
		c_0 = c_0\,e^{-\sigma\,t_0} < \max_{\overline{\Omega}} \tilde{\xi}^{\varepsilon,\,\sigma} = e^{-\sigma\,t_0}\,\xi^{\varepsilon,\,\sigma}_0(x_0) \leq c_0
	\end{equation}
	which is a contradiction. Thus we may assume that $t_0 \neq  0$ and $\partial_t \tilde{\xi}^{\varepsilon,\,\sigma} \geq 0$ at $(x_0,\,t_0)$. Moreover, if $|\nabla u^{\varepsilon,\,\sigma}|(x_0,\,t_0) = 0$, then from the definition of and the assumption on $\tilde{\xi}^{\varepsilon,\,\sigma}$, we obtain
	\begin{equation}
		0 < c_0\,e^{-\sigma \, t_0} \leq  \max_{\overline{\Omega}} \tilde{\xi}^{\varepsilon,\,\sigma} \leq - e^{-\sigma \, t_0}\frac{W(u^{\varepsilon,\,\sigma})}{\varepsilon} \leq 0
	\end{equation}
	which is a contradiction. Hence we may also assume that $|\nabla u^{\varepsilon,\,\sigma}|(x_0,\,t_0) \neq 0$ in the sequel.
	
	Now we consider the case that $x_0 \in \Omega$. In this case, by the maximality of $\tilde{\xi}^{\varepsilon,\,\sigma}$ we have
	\begin{equation}\label{maximalityInterior}
	\nabla \tilde{\xi}^{\varepsilon,\,\sigma}(x_0,\,t_0) = 0,\quad \Delta \tilde{\xi}^{\varepsilon,\,\sigma}(x_0,\,t_0) \leq 0
	\end{equation}
	and thus from \eqref{heatEquationDiscrepancy}, \eqref{maximalityInterior}, and the assumption that $\tilde{\xi}^{\varepsilon,\,\sigma}(x_0,\,t_0)> c_0\,e^{-\sigma\,T} >0$, it holds that
	\begin{align}
		0 \leq \partial_t \tilde{\xi}^{\varepsilon,\,\sigma} - \Delta \tilde{\xi}^{\varepsilon,\,\sigma}  \leq -\sigma\,\tilde{\xi}^{\varepsilon,\,\sigma}  < 0  \label{maxPointInOmega(1)}
	\end{align}
	at $x = x_0$ and $t = t_0$ and this is a contradiction.
	
	Next we consider the case that $x_0 \in \partial\Omega$. From the assumption \eqref{assumptionOnBoundary} and by recalling the boundary condition of the Allen-Cahn equation \eqref{1.1.1} and using the maximality of $\tilde{\xi}^{\varepsilon,\,\sigma}$, we obtain
	\begin{align}\label{maximalityBoundary(1)}
		0 \leq \partial_t \tilde{\xi}^{\varepsilon,\,\sigma} + \sigma  \nabla \tilde{\xi}^{\varepsilon,\,\sigma} \cdot \nu = -\sigma\, \tilde{\xi}^{\varepsilon,\,\sigma} +  \frac{\varepsilon}{2} e^{-\sigma\,t} \left( \frac{\partial}{\partial t}|\nabla u^{\varepsilon,\,\sigma}|^2 + \sigma\frac{\partial}{\partial \nu} |\nabla u^{\varepsilon,\,\sigma}|^2 \right)
	\end{align}
	at $x = x_0$ and $t = t_0$. Now let us recall the following identity:
	\begin{equation}\label{boundaryCurvatureEquality(1)}
		\frac{1}{2}\frac{\partial}{\partial t}|\nabla v|^2 + \frac{\sigma}{2}\frac{\partial}{\partial \nu} |\nabla v|^2 = \frac{1}{2}\frac{\partial}{\partial t}(\nabla v \cdot \nu)^2 + \sigma\,\Pi_{\nu}(\nabla_{\partial \Omega} v,\,\nabla_{\partial \Omega} v) \quad \text{on } \partial\Omega \times (0,\,T]
	\end{equation} 
	for any smooth function $v$ which satisfies $\partial_t v + \sigma \nabla v \cdot \nu = 0$ on $\partial \Omega \times (0,\,T]$. We prove \eqref{boundaryCurvatureEquality(1)} later. By applying \eqref{boundaryCurvatureEquality(1)} to $u^{\varepsilon,\,\sigma}$ and from the assumption \eqref{assumptionOnBoundary}, \eqref{maximalityBoundary(1)}, we have that
	\begin{align}
		0 &\leq -\sigma\,\tilde{\xi}^{\varepsilon,\,\sigma} + \frac{\varepsilon}{2} e^{-\sigma\,t_0} \left( \frac{\partial}{\partial t}(\nabla u^{\varepsilon,\,\sigma} \cdot \nu)^2 + 2 \sigma\,\Pi_{\nu}(\nabla_{\partial \Omega} u^{\varepsilon,\,\sigma},\,\nabla_{\partial \Omega} u^{\varepsilon,\,\sigma}) \right) \nonumber\\
		&\leq -\sigma\,\tilde{\xi}^{\varepsilon,\,\sigma} + \frac{\varepsilon}{2} e^{-\sigma\,t_0}\,c_0\,\,\sigma\varepsilon^{-1}e^{\sigma(t_0-T)}  + \sigma\,\varepsilon\,e^{-\sigma\,t_0} \Pi_{\nu}(\nabla_{\partial \Omega} u^{\varepsilon,\,\sigma},\,\nabla_{\partial \Omega} u^{\varepsilon,\,\sigma})  \nonumber\\
		&\leq -\sigma\,\tilde{\xi}^{\varepsilon,\,\sigma} + c_0\frac{\sigma}{2}e^{-\sigma\,T} + 0 
	\end{align}
	where we have used the convexity of $\Omega$, namely, the fact that if $\Omega$ is convex, then $\Pi_{\nu}(X,\,X) \leq 0$ for any tangential vector $X$ to $\partial \Omega$. Thus, by recalling the assumption that $\tilde{\xi}^{\varepsilon,\,\sigma}(x_0,\,t_0) > c_0\,e^{-\sigma\,T}$, we obtain
	\begin{equation}
		c_0\,\sigma\,e^{-\sigma\,T} < \sigma\,\tilde{\xi}^{\varepsilon,\,\sigma}(x_0,\,t_0) \leq c_0\frac{\sigma}{2}e^{-\sigma\,T}, \label{maximalityBoundaryContradiction(1)}
	\end{equation}
	and since $\sigma>0$ and $c_0>0$, this is a contradiction.
	
	Now we show the identity \eqref{boundaryCurvatureEquality(1)}. This sort of identity on the boundary is treated by many authors, for instance, \cite{CaHo}, \cite{MiTo}, or \cite{StZu}. Let $(x,\,t)\in\partial\Omega\times(0,\,T]$ be any point and $v$ be any smooth function satisfying the equation $\partial_t v + \sigma \nabla v \cdot \nu = 0$ on $\partial \Omega \times (0,\,T]$. Without loss of generality, by rotation and translation we may assume that $\partial\Omega$ is a graph of some function near $x=0\in\partial\Omega$, namely, there exist open sets $U'\subset\mathbb{R}^{n-1}$ and $U\subset \mathbb{R}^n$ with $0\in U'$ and $0\in U$ and a smooth function $\phi:U'\to \mathbb{R}$ such that $\partial \Omega \cap U \coloneqq \{(x',\,\phi(x'))\mid x'\in U' \}$ and 
	\begin{equation}\label{appendAProp2.8}
	\phi(0)=0,\quad \nabla'\phi(0)=0,\quad \frac{\partial^2 \phi}{\partial x_i \partial x_j}(0)= \kappa_i \delta_{ij}
	\end{equation}
	where $\nabla'=(\partial_1,\,\cdots,\,\partial_{n-1})$ and $\{\kappa_i\}_{i=1,\cdots,\,n-1}$ are the principal curvatures of $\partial\Omega$ at $x=0$. For convenience, we set $\kappa_n=0$. In this setting, the outer unit normal vector $\nu$ on the boundary $\partial\Omega \cap U$ is given by
	\begin{equation}\label{appendAProp2.9}
	\nu(x',\,\phi(x')) = \frac{1}{\sqrt{1+|\nabla' \phi(x')|^2}}(-\nabla' \phi(x'),\,1).
	\end{equation}
	Then from the boundary condition of $v$ and \eqref{appendAProp2.9}, we obtain
	\begin{equation}\label{appendAProp2.11}
	0= \partial_t v + \frac{\sigma}{\sqrt{1+|\nabla' \phi|^2}}\left(-\sum_{i=1}^{n-1}\frac{\partial \phi}{\partial x_i} \frac{\partial v}{\partial x_i} + \frac{\partial v}{\partial x_n}\right).
	\end{equation}
	Differentiating the both sides of \eqref{appendAProp2.11} at $x=0$ with respect to $x_j$ and using \eqref{appendAProp2.8}, we obtain
	\begin{equation}\label{appendAProp2.12}
		\frac{\partial}{\partial x_j}(\partial_t v) - \sigma \, \sum_{i=1}^{n-1}\frac{\partial^2 \phi}{\partial x_i \partial x_j}\frac{\partial v}{\partial x_i} + \frac{\partial}{\partial x_j}\left(\frac{\partial v}{\partial x_n}\right) = 0
	\end{equation}
	at $x'=0$. Then, multiplying $\frac{\partial v}{\partial x_j}$ to both sides of \eqref{appendAProp2.12} and summing over $j=1,\,\cdots,\,n-1$, we have from \eqref{appendAProp2.8} that
	\begin{align}
	0&= \sum_{j=1}^{n-1}\frac{\partial v}{\partial x_j}\,\frac{\partial}{\partial x_j}(\partial_t v) - \sigma \sum_{j=1}^{n-1}\frac{\partial v}{\partial x_j}\,\sum_{i=1}^{n-1}\frac{\partial^2 \phi}{\partial x_i \partial x_j} \frac{\partial v}{\partial x_i}  + \sigma \sum_{j=1}^{n-1}\frac{\partial v}{\partial x_j} \frac{\partial^2 v}{\partial x_n \partial x_j} \nonumber\\
	&= \sum_{j=1}^{n-1}\frac{\partial v}{\partial x_j}\,\frac{\partial}{\partial t}\left(\frac{\partial v}{\partial x_j}\right) - \sigma \sum_{j=1}^{n-1}\frac{\partial v}{\partial x_j} \,\sum_{i=1}^{n-1}\kappa_i \,\delta_{ij} \frac{\partial v}{\partial x_i}  + \sigma \sum_{j=1}^{n-1}\frac{\partial v}{\partial x_j}\frac{\partial}{\partial x_n}\left(\frac{\partial v}{\partial x_j}\right) \nonumber\\
	&= \frac{1}{2}\frac{\partial}{\partial t}|\nabla' v|^2 -\sigma \sum_{i=1}^{n-1} \kappa_i\, \frac{\partial v}{\partial x_i}\frac{\partial v}{\partial x_i} + \sigma \frac{1}{2}\nabla(|\nabla v|^2)\cdot \nu \nonumber\\
	&= \frac{1}{2}\frac{\partial}{\partial t}|\nabla' v|^2 -\sigma \sum_{i=1}^{n-1} \kappa_i\, \left(\frac{\partial v}{\partial x_i}\right)^2 + \sigma \frac{1}{2}\frac{\partial}{\partial \nu} |\nabla v|^2 \label{appendAProp2.13}
	\end{align}
	at $x' = 0$. Moreover from the choice of $\phi$, we have
	\begin{equation}\label{appendAProp2.14}
		\frac{\partial}{\partial t}\left( \frac{\partial v}{\partial x_n} \right)^2 = \frac{\partial}{\partial t}\left( \frac{\partial v}{\partial \nu} \right)^2
	\end{equation}
	at $x' = 0$. Therefore from the assumption, \eqref{appendAProp2.13}, and \eqref{appendAProp2.14} , we conclude that
	\begin{equation}\label{identityBoundary}
		\frac{1}{2}\frac{\partial}{\partial t}|\nabla v|^2 +  \frac{\sigma}{2}\frac{\partial}{\partial \nu} |\nabla v|^2 = \frac{1}{2}\frac{\partial}{\partial t}\left( \frac{\partial v}{\partial \nu} \right)^2 + \sigma \sum_{i=1}^{n-1} \kappa_i\, \left(\frac{\partial v}{\partial x_i}\right)^2
	\end{equation}
	holds at $x' = 0$. Therefore, recalling from the definition of $\Pi_{\nu}$ that 
	\begin{equation}
		\Pi_{\nu}(\nabla_{\partial \Omega} v,\,\nabla_{\partial \Omega} v) = \sum_{i=1}^{n-1} \kappa_i\, \left(\frac{\partial v}{\partial x_i}\right)^2,
	\end{equation}
	then we conclude that \eqref{boundaryCurvatureEquality(1)} is valid.
\end{proof} 

\begin{remark}[Necessity of the assumption \eqref{assumptionOnBoundary}]
	In the case of Neumann boundary conditions, we can show the upper bound of the discrepancy measure via the maximum principle. To see this, we refer to the work \cite{Ilmanen01}. Indeed, if we have the initial data $u^{\varepsilon,\,\sigma}_0$ satisfying $|u^{\varepsilon,\,\sigma}_0| <1$ in $\Omega$, then $|u^{\varepsilon,\,\sigma}| < 1$ in $\Omega \times (0,\,T]$ from the maximal principle (see also Proposition \ref{appendAProp}). Hence we may define a function $d^{\varepsilon,\sigma}$ by
	\begin{equation}
		d^{\varepsilon,\sigma}(x,\,t) \coloneqq \left(q^{\varepsilon}\right)^{-1}(u^{\varepsilon,\,\sigma}(x,\,t))
	\end{equation}
	for any $(x,\,t) \in \Omega \times (0,\,T]$ where $q^{\varepsilon}(s) \coloneqq \tanh(\varepsilon^{-1}s)$ for $s\in\mathbb{R}$. This function $d^{\varepsilon,\sigma}(\cdot,\,t)$ can be regarded as the approximation of the signed distance function from the zero level set of $u^{\varepsilon,\,\sigma}(\cdot,\,t)$. Then we have that 
	\begin{equation}
		\xi^{\varepsilon,\,\sigma}(x,t) = \frac{\varepsilon}{2} \left|(q^{\varepsilon})^{\prime}(d^{\varepsilon,\sigma})\right|^2 |\nabla d^{\varepsilon,\sigma}|^2 -\frac{1}{\varepsilon} W(q^{\varepsilon}(d^{\varepsilon,\sigma}))	= \frac{1}{\varepsilon} W(q^\varepsilon(d^{\varepsilon,\sigma}))(|\nabla d^{\varepsilon,\sigma}|^2 -1),
	\end{equation}
	where we used the fact that $\frac{\varepsilon}{2} |(q^{\varepsilon})^{\prime}(d^{\varepsilon,\sigma})|^2= \frac{1}{\varepsilon} W(q^\varepsilon(d^{\varepsilon,\sigma}))$. Thus it is sufficient to show that $w(x,t):=|\nabla d(x,t)|^2 -1 \leq 0$ for any $(x,\,t) \in \Omega \times (0,\,T]$. In order to use the maximum principle, we assume that $w(x,0) \leq 0$ for any $x\in \Omega$. Then in the same way as the computation shown in \cite[p. 432]{Ilmanen01}, we obtain
	\begin{equation}
		\partial_t w \leq  \Delta w 
		- \frac{2 q^{\varepsilon} (d^{\varepsilon,\sigma})} {\varepsilon} \nabla d^{\varepsilon,\sigma} \cdot \nabla w - \frac{2 (q^{\varepsilon})^{\prime}(d^{\varepsilon,\sigma})}{\varepsilon} |\nabla d^{\varepsilon,\sigma}|^2 w
	\end{equation}
	in $\Omega \times(0,\,T]$.
	
	In addition, the Neumann boundary condition for $d^{\varepsilon,\sigma}$, the convexity of $\Omega$, and \eqref{boundaryCurvatureEquality(1)} imply that $\nabla w \cdot \nu \leq 0$ on $\partial \Omega$ for $t>0$. Therefore, the maximum principle yields $w\leq 0$ for any $x\in \Omega$ and $t>0$.
	
	On the other hand, in the case of dynamic boundary conditions, we are unfortunately unable to obtain such an estimate as  
	\begin{equation}
		\frac{\partial w}{\partial t} + \sigma \frac{\partial w}{\partial \nu} \leq 0
	\end{equation}
	on $\partial \Omega \times (0,\,T]$. This is because we might fail to control the gradient in the direction of the unit normal vector of the boundary in time without any conditions. Hence so far we need such assumptions as \eqref{assumptionOnBoundary} present in Proposition \ref{appendAProp2} in our situation.
\end{remark}

\subsection{Appendix B}\label{vanishingDiscrepancy}
In this appendix, we prove the vanishing of the discrepancy measure only in the interior of a given domain under some constraints, namely, we show that there exists a subsequence $\{\varepsilon_j\}_{j\in\mathbb{N}}$, independent of $t\geq 0$, such that $|\xi^{\varepsilon_j,\,\sigma}_t| \rightarrow 0$ as $j\to\infty$ in measures on $\Omega$ for a.e. $t \in [0,\,T]$ and any $\sigma>0$. Before stating the claim, we give some constraints to prove the vanishing of the discrepancy measure. Precisely, we assume that the domain $\Omega$ is convex and the solution $u^{\varepsilon,\,\sigma}$ of the Allen-Cahn equation \eqref{1.1.1} is smooth and satisfies the following assumptions:
\begin{enumerate}
	\item $\sup_{\Omega}|u^{\varepsilon,\,\sigma}_0| \leq 1$ where $u^{\varepsilon,\,\sigma}_0(\cdot) \coloneqq u^{\varepsilon,\,\sigma}(\cdot,\,0)$ in $\Omega$.
	\item There exists a constant $c_0>0$, independent of $\varepsilon$, $\sigma$, and $T$, such that $\sup_{\Omega}\xi^{\varepsilon,\,\sigma}_0 \leq c_0$ where $\xi^{\varepsilon,\,\sigma}_0$ is defined by
	\begin{equation}
		\frac{\varepsilon}{2}|\nabla u^{\varepsilon,\,\sigma}_0|^2 - \frac{W(u^{\varepsilon,\,\sigma}_0)}{\varepsilon}.
	\end{equation}
	\item For any $T>0$, $\varepsilon,\,\sigma>0$, and $t \in (0,\,T]$, it holds that
	\begin{equation}
		\frac{\partial}{\partial t}(\nabla u^{\varepsilon,\,\sigma} \cdot \nu)^2(t) \leq c_0\,\sigma\varepsilon^{-1}e^{\sigma\,(t-T)} \quad \text{ on } \partial\Omega.
	\end{equation}
	Note that this condition is also present in Proposition \ref{appendAProp2} in Appendix A.
\end{enumerate}
As we prove Proposition \ref{appendAProp} and \ref{appendAProp2} in Appendix A of Section \ref{appendix}, from these assumptions in the above, the two estimates
\begin{equation}
	\sup_{\Omega\times(0,\,T]}|u^{\varepsilon,\,\sigma}| \leq 1, \quad \sup_{\Omega\times(0,\,T]}\xi^{\varepsilon,\,\sigma} \leq c_0
\end{equation}
holds where $c_0$ is as in the last assumption above.

Now we are prepared to show the main claim of this appendix. Before stating the claim, we give some notations and facts. Let $\xi^{\varepsilon,\,\sigma}_t$ be the discrepancy measure for any $t>0$, which is defined by $\xi^{\varepsilon,\,\sigma}(x,\,t)\mathcal{L}^n(x)$ for each $t>0$ where $\xi^{\varepsilon,\,\sigma}(x,\,t)$ is shown in \eqref{discrepancyFunc}. Recalling the estimate that $\xi^{\varepsilon,\,\sigma}_t \leq \mu^{\varepsilon,\,\sigma}_t$ on $\Omega \times (0,\,\infty)$, we have that $\sup_{\varepsilon,\,\sigma>0} \xi^{\varepsilon,\,\sigma}_t(\Omega) \leq D$ where $D$ is as in Proposition \ref{prop.4.1}. Then, setting $\xi^{\varepsilon,\,\sigma} \coloneqq |\xi^{\varepsilon,\,\sigma}_t| \otimes \mathcal{L}^1_t$ with a little abuse of notation, we can choose a subsequence $\{\xi^{\varepsilon_j,\,\sigma}\}_{j\in\mathbb{N}}$ converging to some Radon measure $\xi^{\sigma}$ in the sense of Radon measures in $\Omega \times (T_0,\,T)$ for any $0 < T_0 < T$. In this setting, we prove 
\begin{proposition}\label{vanishingProp}
	Let $\Omega\subset\mathbb{R}^n$ be a bounded and convex domain with a smooth boundary and $\sigma>0$ be given. Assume that $\{u^{\varepsilon,\,\sigma}\}_{\varepsilon>0}$ is a family of the solutions to the Allen-Cahn equation \eqref{1.1.1} such that $u^{\varepsilon,\,\sigma}$ satisfies the three conditions shown in the above. Then it holds that $\xi^{\sigma} \equiv 0$ in $\Omega \times (T_0,\,T)$.
\end{proposition}
\begin{remark}
	For simplicity, we will show Proposition \ref{vanishingProp} without indicating the parameter $\sigma>0$ because the parameter $\sigma>0$ does not affect the estimates in the interior $\Omega$ but only on the boundary $\partial \Omega$.
\end{remark}
The proof of Proposition \ref{vanishingProp} is based on the argument shown in \cite{MiTo}.

First of all, we will derive the monotonicity formula in the interior $\Omega$. The monotonicity formula is the key estimate to prove the vanishing of the discrepancy measure. The proof is based on the results by Ilmanen in \cite[Chapter 3]{Ilmanen01} and Mizuno and Tonegawa in \cite[Porposition 3.1]{MiTo}. Before stating the detail, we will give several notations; first we define $\Omega_m \subset \mathbb{R}^n$ for any $m\in\mathbb{N}$ by
\begin{equation}
	\Omega_m \coloneqq \left\{x\in\Omega \mid \dist(x,\,\partial\Omega) > \frac{1}{m}\right\}
\end{equation}
and define $d_0>0$ by
\begin{equation}
	d_0 \coloneqq \|\text{principal curvatures of }\partial \Omega\|_{L^\infty}^{-1}.
\end{equation}
Here if $\partial \Omega$ is flat and thus the principal curvatures are all zero, then we set $d_0$ as 1 instead of $\infty$. Secondly, for any $m\in\mathbb{N}$, we choose a smooth and radially symmetric cut-off function $\eta_m$ in such a way that $\eta_m$ satisfies
\begin{equation}
	0\leq \eta_m \leq 1,\quad \frac{\partial \eta_m}{\partial |x|}(x)\leq 0 \quad \text{for } x\in\mathbb{R}^n\setminus\{0\},\quad \spt\eta \subset B_{\frac{1}{m}}(0),\quad \eta_m \equiv 1 \text{ on } B_{\frac{1}{2m}}(0).
\end{equation}
Lastly, for any $y\in\mathbb{R}^n$ and $s>0$, we define the backward heat kernel $\rho_{(y,\,s)}$ by
\begin{equation}
	\rho_{(y,\,s)}(x,\,t) = \frac{1}{(4\pi(s-t))^{\frac{n-1}{2}}} \exp\left(-\frac{|x-y|^2}{4(s-t)}\right)
\end{equation}
for any $x\in\mathbb{R}^n$ and $s > t > 0$. 
	
Now we are going to prove the following proposition which gives us a version of the monotonicity formula shown in \cite{Ilmanen01, MiTo}: 
\begin{proposition}\label{monotonicityFormula}
	Let $m\in\mathbb{N}$ be such that $m > \max\{d_0,\,3\}$. Then there exists a constant $C_1>0$ depending on $m$ and $\Omega$ such that, for any $y\in\Omega_m$, $s>t>0$, and $\varepsilon>0$, it holds that
	\begin{align}\label{claimMonotoForm}
	\frac{d}{dt}\left( e^{2(s-t)^{\frac{1}{4}}} \int_{\Omega}\tilde{\rho}^m_{(y,\,s)}\,d\mu^{\varepsilon,\,\sigma}_t(x) \right) \leq e^{2(s-t)^{\frac{1}{4}}} \left( C_1 + \int_{\Omega} \frac{\tilde{\rho}^m_{(y,\,s)}}{2(s-t)} \,d\xi^{\varepsilon,\,\sigma}_t(x) \right) 
	\end{align}
	where $\tilde{\rho}^m_{(y,\,s)}$ is defined by
	\begin{equation}
	\tilde{\rho}^m_{(y,\,s)}(x,\,t) \coloneqq \eta_m(x-y)\,\rho_{(y,\,s)}(x,\,t) 
	\end{equation}
	for any $x,\,y\in\mathbb{R}^n$ and $s>t>0$.
\end{proposition}
\begin{proof}[Proof of Proposition \ref{monotonicityFormula}]
	The proof of this proposition is based on \cite[Proposition 3.1]{MiTo}. Let $m\in\mathbb{N}$ with $m > \max\{d_0,\,3\}$ and $\varepsilon>0$ be any and fixed. Moreover, we take any $y\in\Omega_m$ and $s>t>0$. Then by definition, we can easily see that $\eta_m(x-y) = 0$ for any $x\in\partial\Omega$ and $y\in\Omega_m$ and thus it holds that $\spt\tilde{\rho}^m_{(y,\,s)}(\cdot,\,t) \subset \Omega$ for any $s>t>0$ and $y \in \Omega_m$. Hence, by employing the computations and estimates except on the integrals over $\partial \Omega$ obtained in \cite[Proposition 3.1]{MiTo}, we may have
	\begin{align}
	\frac{d}{dt}\int_{\Omega}\tilde{\rho}^m_{(y,\,s)}\,d\mu^{\varepsilon}_t &\leq  \int_{\Omega}\left(\partial_t \tilde{\rho}^m_{(y,\,s)} + (I- a^{\varepsilon}\otimes a^{\varepsilon}): \nabla^2 \tilde{\rho}^m_{(y,\,s)} + \frac{(a^{\varepsilon}\cdot \nabla \tilde{\rho}^m_{(y,\,s)})^2}{\tilde{\rho}^m_{(y,\,s)}} \right) \,\varepsilon|\nabla u^{\varepsilon,\,\sigma}|^2dx \nonumber\\
	&\qquad -\int_{\Omega}\varepsilon\tilde{\rho}^m_{(y,\,s)}\,\left(\nu^{\varepsilon}+ \frac{\nabla u^{\varepsilon,\,\sigma} \cdot \nabla \tilde{\rho}^m_{(y,\,s)}}{\tilde{\rho}^m_{(y,\,s)}}\right)^2 \,dx \nonumber\\
	&\qquad \quad - \int_{\Omega} (\partial_t \tilde{\rho}^m_{(y,\,s)} + \Delta \tilde{\rho}^m_{(y,\,s)}) \,d\xi^{\varepsilon,\,\sigma}_t  \label{monotonicityForm01}
	\end{align}
	where $a^{\varepsilon}\coloneqq |\nabla u^{\varepsilon,\,\sigma}|^{-1} \nabla u^{\varepsilon,\,\sigma}$. Now we consider the first term of \eqref{monotonicityForm01}. By simple computation, we have the fact that
	\begin{equation} \label{monotonicityForm02}
		\partial_t \rho_{(y,\,s)}(x,\,t) + (I- a^{\varepsilon}\otimes a^{\varepsilon}): \nabla^2\rho_{(y,\,s)}(x,\,t) + \frac{(a^{\varepsilon}\cdot \nabla \rho_{(y,\,s)}(x,\,t) )^2}{\rho_{(y,\,s)}(x,\,t)} = 0 
	\end{equation}
	for any $x,\,y\in\mathbb{R}^n$ and $s>t>0$. From the definition of $\rho_{(y,\,s)}$, we have
	\begin{equation}
		\rho_{(y,\,s)}(x,\,t) \leq \frac{1}{(2\pi(s-t))^{\frac{n-1}{2}}}\exp\left(-\frac{1}{64m^2(s-t)}\right) < \tilde{c}\,m^{n-1}
	\end{equation}
	for any $x,\,y\in\mathbb{R}^n$ with $|x-y| \geq \frac{1}{2m}$ and $s>t>0$ where $\tilde{c}>0$ is a constant independent of $m$. Thus, from \eqref{monotonicityForm02}, we may compute as follows:
	\begin{align}
		&\int_{\Omega} \left(\partial_t \tilde{\rho}^m_{(y,\,s)} + (I- a^{\varepsilon}\otimes a^{\varepsilon}): \nabla^2 \tilde{\rho}^m_{(y,\,s)} + \frac{(a^{\varepsilon}\cdot \nabla \tilde{\rho}^m_{(y,\,s)})^2}{\tilde{\rho}^m_{(y,\,s)}} \right) \,\varepsilon|\nabla u^{\varepsilon}|^2\,dx \nonumber\\
		&= \int_{\Omega} \left( (I-a^{\varepsilon} \otimes a^{\varepsilon}):\nabla^2\eta_m\,\rho_{(y,\,s)} + 2\nabla \eta_m \cdot \nabla \rho_{(y,\,s)} + \frac{(a^{\varepsilon}\cdot\nabla \eta_m)^2}{\eta_m} \rho_{(y,\,s)}  \right) \,\varepsilon|\nabla u^{\varepsilon}|^2\,dx \nonumber  \\
		&\leq \int_{B_{\frac{1}{m}}\setminus B_{\frac{1}{2m}}(y)} \left(c_1\,\rho_{(y,\,s)} + c_2\frac{|x-y|}{2(s-t)}\rho_{(y,\,s)} + \sup_{\Omega}\|\nabla^2\eta_m\|_{C^2} \,\rho_{(y,\,s)} \right) \,2\,d\mu^{\varepsilon}_t(x) \nonumber\\
		&\leq \int_{B_{\frac{1}{m}}\setminus B_{\frac{1}{2m}}(y)}d\mu^{\varepsilon}_t(x) \,\frac{c_3}{(s-t)^{\frac{n+1}{2}}}\,\exp(-\frac{1}{64m^2(s-t)}) \nonumber\\
		&\leq c_4\,m^{n-1} \,\sup_{\varepsilon,\,t >0}\mu^{\varepsilon}_t(\Omega) \leq c_4\,m^{n-1}\,D \label{monotonicityForm03} 
	\end{align}
	where $c_i=c_{i}(m,\,\Omega)$ for $i\in\{1,\,2,\,3,\,4\}$ is a constant depending only on $m$ and $\Omega$. Next we consider the third term in \eqref{monotonicityForm01}. From the similar computations done in \eqref{monotonicityForm03}, we may also have
	\begin{equation}
		\int_{\Omega} \left(\partial_t \tilde{\rho}^m_{(y,\,s)} + \Delta \tilde{\rho}^m_{(y,\,s)} \right)\,d\xi^{\varepsilon}_t(x) \leq  \int_{\Omega} \left(\frac{|x-y|}{2(s-t)} \tilde{\rho}^m_{(y,\,s)} - \frac{\tilde{\rho}^m_{(y,\,s)}}{2(s-t)}\right) \,d\xi^{\varepsilon}_t(x) \label{monotonicityForm04}
	\end{equation}
	for any $y \in \Omega_m$ and $s>t>0$. As for the first term in \eqref{monotonicityForm04}, we have
	\begin{align}
		\int_{\Omega} \frac{|x-y|}{2(s-t)} \tilde{\rho}^m_{(y,\,s)} \,d\xi^{\varepsilon}_t(x) &\leq  \int_{\Omega\cap\{|x-y|<(s-t)^{\frac{1}{4}}\}} \frac{|x-y|}{2(s-t)} \tilde{\rho}^m_{(y,\,s)} \,d\mu^{\varepsilon}_t(x) \nonumber\\
		& \qquad \quad + \int_{\Omega \cap \{|x-y|\geq (s-t)^{\frac{1}{4}}\}} \frac{|x-y|}{2(s-t)} \tilde{\rho}^m_{(y,\,s)} \,d\mu^{\varepsilon}_t(x) \nonumber\\
		&\leq \frac{1}{2(s-t)^{\frac{3}{4}}} \int_{\Omega} \tilde{\rho}^m_{(y,\,s)}\,d\mu^{\varepsilon}_t(x) + d_1\,D \label{monotonicityForm05}
	\end{align}
	where $d_1>0$ is a constant depending only on $\Omega$. Thus from \eqref{monotonicityForm01}, \eqref{monotonicityForm04}, and \eqref{monotonicityForm05}, we obtain
	\begin{equation}
		\frac{d}{dt}\int_{\Omega}\tilde{\rho}^m_{(y,\,s)}\,d\mu^{\varepsilon}_t(x) \leq c_5 \,D +  \int_{\Omega}\frac{\tilde{\rho}^m_{(y,\,s)}}{2(s-t)} \,d\xi^{\varepsilon}_t(x) + \frac{1}{2(s-t)^{\frac{3}{4}}} \int_{\Omega} \tilde{\rho}^m_{(y,\,s)}\,d\mu^{\varepsilon}_t(x)
	\end{equation}
	for any $y\in\Omega_m$ and $s>t>0$ where $c_5$ is a constant depending on $m$ and $\Omega$. Therefore we obtain
	\begin{align}
		\frac{d}{dt}\left(e^{2(s-t)^{\frac{1}{4}}} \int_{\Omega}\tilde{\rho}^m_{(y,\,s)}\,d\mu^{\varepsilon}_t(x) \right) &= -\frac{e^{2(s-t)^{\frac{1}{4}}}}{2(s-t)^{\frac{3}{4}}} \int_{\Omega} \tilde{\rho}^m_{(y,\,s)}\,d\mu^{\varepsilon}_t(x) +  e^{2(s-t)^{\frac{1}{4}}} \frac{d}{dt}\int_{\Omega}\tilde{\rho}^m_{(y,\,s)}\,d\mu^{\varepsilon}_t(x) \nonumber\\
		&\leq e^{2(s-t)^{\frac{1}{4}}} \left(c_5 \,D +  \int_{\Omega}\frac{\tilde{\rho}^m_{(y,\,s)}}{2(s-t)} \,d\xi^{\varepsilon}_t(x)\right)
	\end{align}
	for any $y\in\Omega_m$ and $s>t>0$.
\end{proof}

By applying the monotonicity formula shown in Proposition \ref{monotonicityFormula}, we may prove the density estimate of the Radon measure $\mu_t$ only in the interior $\Omega$.
\begin{proposition}\label{densityEstimateLimitMu}
	Let $T>0$ be fixed and $m\in\mathbb{N}$ be any such that $m > \max\{d_0,\,3\}$. Then there exists a constant $D_0=D_0(T,\,m)$ independent of $\varepsilon$ and $\sigma$ such that
	\begin{equation}\label{densityEstimate}
		\mu^{\varepsilon,\,\sigma}_t(B_{r}(y) \cap \Omega) \leq D_0\,r^{n-1}
	\end{equation}
	for any $y \in \Omega_m$, $0<r<\frac{1}{m}$, $\varepsilon>0$, $\sigma>0$, and $t \geq T$.
\end{proposition}
\begin{proof}
	For $\hat{t} \geq T$ and $0<r<\frac{1}{m}$, we set $s \coloneqq \hat{t} + r^2$. Integrating the both sides of \eqref{claimMonotoForm} over $t \in [\hat{t}-\frac{T}{2},\,\hat{t}]$, we obtain
	\begin{align}
		e^{2r^{\frac{1}{2}}}\int_{\Omega}\tilde{\rho}^{m}_{(y,\,s)}(x,\,\hat{t}) \,d\mu^{\varepsilon,\,\sigma}_{\hat{t}}(x) &\leq e^{2(r^2+\frac{T}{2})^{\frac{1}{4}}}\int_{\Omega}\tilde{\rho}^{m}_{(y,\,s)}\left(x,\,\hat{t}-\frac{T}{2}\right) ,d\mu^{\varepsilon,\,\sigma}_{\hat{t}-\frac{T}{2}}(x) \nonumber\\
		&\qquad + \int_{\hat{t}-\frac{T}{2}}^{\hat{t}} e^{2(s-t)^\frac{1}{4}} \left(C_1 + \frac{\sqrt{4\pi}\,c_0}{(s-t)^{\frac{1}{2}}}  \right) \,dt  \nonumber\\
		&\leq e^{2(r^2+\frac{T}{2})^{\frac{1}{4}}} \int_{\Omega}\tilde{\rho}^{m}_{(y,\,s)}\left(x,\,\hat{t}-\frac{T}{2}\right) ,d\mu^{\varepsilon,\,\sigma}_{\hat{t}-\frac{T}{2}}(x) \nonumber\\
		&\qquad + e^{2(1+ \frac{T}{2})^{\frac{1}{4}}}\left(\frac{C_1\,T}{2} + \int_{\hat{t}-\frac{T}{2}}^{\hat{t}} \frac{\sqrt{4\pi}c_0}{(s-t)^{\frac{1}{2}}}\,dt \right) \label{densityEsti01}
	\end{align}
	where we have used \eqref{boundednessDiscrepancy} and the fact that $\int_{\mathbb{R}^n} (4\pi(s-t))^{-\frac{1}{2}}\rho\,dx \leq 1$. The last term on the right-hand side of \eqref{densityEsti01} can be estimated from above in terms of the constant which depends only on $T$ and $m$. Recalling the definition of $\eta_m$, we have that $\eta_m(x-y) = 1$ for any $x\in B_{\frac{1}{2m}}(y)$ and thus, from the fact that $s = \hat{t} + r^2$ and \eqref{densityEsti01}, we can compute as follows:
	\begin{align}
		\frac{e^{-\frac{1}{4}}}{(4\pi r^2)^{\frac{n-1}{2}}}\mu_{\hat{t}}^{\varepsilon,\,\sigma}(B_r(y) \cap \Omega) &\leq \int_{B_{r}(y) \cap \Omega} \frac{\eta_m(x-y)\,e^{-\frac{|x-y|^2}{4r^2}}}{(4\pi r^2)^{\frac{n-1}{2}}} \,d\mu_{\hat{t}}^{\varepsilon,\,\sigma}(x) \nonumber\\
		&\leq e^{2r^{\frac{1}{2}}} \int_{\Omega} \tilde{\rho}^{m}_{(y,\,s)}(x,\,\hat{t}) \,d\mu^{\varepsilon,\,\sigma}_{\hat{t}}(x). \label{densityEsti02}
	\end{align}
	On the other hand, we can compute the second term in \eqref{densityEsti01} as follows:
	\begin{align}
		e^{2(r^2+\frac{T}{2})^{\frac{1}{4}}} \int_{\Omega}\tilde{\rho}^{m}_{(y,\,s)}\left(x,\,\hat{t}-\frac{T}{2}\right) ,d\mu^{\varepsilon,\,\sigma}_{\hat{t}-\frac{T}{2}}(x) & \leq e^{2(r^2+\frac{T}{2})^{\frac{1}{4}}} \int_{\Omega} \frac{1}{(4\pi(r^2+\frac{T}{2}))^{\frac{n-1}{2}}} \,d\mu^{\varepsilon,\,\sigma}_{\hat{t}-\frac{T}{2}}(x) \nonumber\\
		&\leq e^{2(1+\frac{T}{2})^{\frac{1}{4}}} \int_{\Omega} \frac{1}{(2\pi\,T)^{\frac{n-1}{2}}} \,d\mu^{\varepsilon,\,\sigma}_{\hat{t}-\frac{T}{2}}(x) \leq   \frac{e^{2(1+\frac{T}{2})^{\frac{1}{4}}}}{(2\pi\,T)^{\frac{n-1}{2}}}\,D. \label{densityEsti03}
	\end{align}
	Therefore from \eqref{densityEsti01}, \eqref{densityEsti02}, and \eqref{densityEsti03}, we obtain
	\begin{equation}
		\frac{e^{-\frac{1}{4}}}{(4\pi r^2)^{\frac{n-1}{2}}}\mu_{\hat{t}}^{\varepsilon,\,\sigma}(B_r(y) \cap \Omega) \leq \frac{e^{2(1+\frac{T}{2})^{\frac{1}{4}}}}{(2\pi\,T)^{\frac{n-1}{2}}}\,D + e^{2(1+ \frac{T}{2})^{\frac{1}{4}}}\left(\frac{C_1(m)\,T}{2} + \sqrt{4\pi}c_0\,c_1(T) \right)
	\end{equation}
	where $c_1>0$ is a constant depending on $T$. Choosing an appropriate constant $D_0$ which depends on $T$, $m$, and $\Omega$, we may conclude that \eqref{densityEstimate} is valid.
\end{proof}

Next we consider the lemma stating that, for any point both in the domain and in the support of the measure $\mu \coloneqq \mu_t \otimes \mathcal{L}^1_t$ where $\mu_t$ is as in Lemma \ref{thm3.1} or Lemma \ref{thm.3.6}, the solution of the Allen-Cahn equation \eqref{1.1.1} should be small at that point. Note that in the sequel we fix $\sigma>0$.
\begin{lemma}\label{estiSolInSpt}
	Let $m\in\mathbb{N}$ be any such that $m > \max\{d_0,\,3\}$. Then for any $(x',\,t') \in \spt\mu$ with $x'\in\Omega_m$ and $t'>0$, there exists a subsequence $\{\varepsilon_i\}_{i\in\mathbb{N}}$ and a sequence $\{(x_i,\,t_i)\}_{i\in\mathbb{N}}$ such that $x_i \in \Omega_m$, $t_i>0$, and $|u^{\varepsilon_i,\,\sigma}(x_i,\,t_i)| < \frac{1}{2\sqrt{3}}$ for any $i\in\mathbb{N}$. 
\end{lemma}
We remark that this lemma can be proved in the same way as in \cite[Lemma 6.1]{MiTo} because we consider the estimates of the solution $u^{\varepsilon,\,\sigma}$ of \eqref{1.1.1} only in the interior $\Omega$ by using a proper cut-off function whose support is in $\Omega$. Thus we do not show this claim here again.

The next lemma can be derived from the fact that $\sup_{\Omega\times(0,\,T]}\xi^{\varepsilon,\,\sigma} \leq c_0$ for any $\varepsilon,\,\sigma>0$ and this is also stated in \cite[Lemma 4.4]{MiTo}.
\begin{lemma}\label{lemmaGradientEsti}
	There exists a constant $c_1>0$ such that
	\begin{equation}
	\sup_{\Omega\times(0,\,T]} \varepsilon\,|\nabla u^{\varepsilon,\,\sigma}| \leq c_1
	\end{equation}
	for any $\varepsilon\in(0,\,1)$ and $\sigma>0$.
\end{lemma}
\begin{proof}
	From Proposition \ref{appendAProp} and \ref{appendAProp2} in Appendix A, we have that 
	\begin{equation}
		\sup_{\Omega\times(0,\,T]} |u^{\varepsilon,\,\sigma}| \leq 1, \quad \sup_{\Omega\times(0,\,T]} \xi^{\varepsilon,\,\sigma} \leq c_0
	\end{equation}
	and thus we obtain
	\begin{equation}
	\sup_{\Omega\times(0,\,T]} \varepsilon^2|\nabla u^{\varepsilon,\,\sigma}|^2 \leq 2\varepsilon \sup_{\Omega\times(0,\,T]}\left( \frac{\varepsilon}{2}|\nabla u^{\varepsilon,\,\sigma}|^2 -\frac{W(u^{\varepsilon,\,\sigma})}{\varepsilon} \right) + 2 \sup_{\Omega\times(0,\,T]}W(u^{\varepsilon,\,\sigma}) \leq 2\varepsilon\,c_0 + 2.
	\end{equation}
	Hence setting $c_1 \coloneqq \sqrt{2c_0+2}$ which is independent of $\varepsilon$, $\sigma$, and $T$, we obtain
	\begin{equation}
	\sup_{\Omega\times(0,\,T]} \varepsilon|\nabla u^{\varepsilon,\,\sigma}| \leq c_1
	\end{equation}
	for any $\varepsilon\in(0,\,1)$ and $\sigma>0$.
\end{proof}

Next we prove, what we call, ``\textit{clearing-out lemma}" saying that, if there exists a point outside of the support of the measure $\mu = \mu_t \otimes \mathcal{L}^1_t$, then a proper neighbourhood of the point is not contained in the support of $\mu$ either. This lemma is based on \cite[6.1]{Ilmanen01} and \cite[Lemma 6.2]{MiTo}.
\begin{lemma}\label{clearingOutLemma}
	Let $T>T_0>0$ and $m\in\mathbb{N}$ be fixed such that $m > \max\{d_0,\,2\}$. Then there exist a constant  $\gamma_0(T,T_0,\Omega)>0$ independent of $m$ and constants $\delta_0(T,T_0,\Omega,m), r_0(T,T_0,\Omega,m)>0$ depending on $m$, with $\gamma_0\,r_0 < \frac{1}{m}$, and $m$ such that, if 
	\begin{equation}\label{assumpClearingLemma}
		\int_{\overline{\Omega}_m} \eta_m(x-y)\,\rho_{(y,\,s)}(x,\,t) \,d\mu_s(y) < \delta_0
	\end{equation}
	for some $s,\,t\in(T_0,\,T)$ with $t < s < t+\frac{r_0^2}{2} < T - \frac{r_0^2}{2}$ and $x \in \Omega_{[\frac{m}{2}]}$ where $[a]$ is the largest integer equal to or less than $a\in\mathbb{R}$, then it holds $(x',\,t') \not\in \spt\mu$ for any $x' \in B_{\gamma_0\,r}(x)$ and $t' = 2s - t$ with $r = \sqrt{2(s-t)}$. 
\end{lemma}

\begin{proof}[Proof of Lemma \ref{clearingOutLemma}]
	We assume that, for some $s,\,t\in(T_0,\,T)$ with $t < s < t+\frac{r_0^2}{2} < T - \frac{r_0^2}{2}$ and $x \in \Omega_{[\frac{m}{2}]}$,\eqref{assumpClearingLemma} holds where the constants $\delta_0$ and $r_0$ are determined later. Set $t' \coloneqq 2s -t < t + r^2_0 < T$. Suppose by contradiction that $(x',\,t') \in \spt \mu$ for some $x' \in B_{\gamma_0\,r}(x)$ and $t' \coloneqq 2s-t$ where $\gamma_0$ is determined later. Then from the choice that $x' \in B_{\gamma_0\,r}(x)$ and $x \in \Omega_{[\frac{m}{2}]}$, we have that $x'\in \Omega_m$. Thus from Lemma \ref{estiSolInSpt}, we can choose sequences $\{(x_j,\,t_j)\}_{j\in\mathbb{N}}$ and $\{\varepsilon_j\}_{j\in\mathbb{N}}$ such that $x_j \in \Omega_m$, $t_j <T$, $(x_j,\,t_j) \rightarrow (x',\,t')$ as $j\to \infty$, and $|u^{\varepsilon_j}(x_j,\,t_j)| < \frac{1}{2\sqrt{3}}$ for all $j$. Setting $r_j \coloneqq \gamma_0\,\varepsilon_j$ and $T_j \coloneqq t_j + r^2_j$, then we have
	\begin{align}
		\int_{B_{r_j}(x_j)}	&\eta_m(y-x_j)\,\rho_{(x_j,\,T_j)}(y,\,t_j)\,d\mu^{\varepsilon_j}_{t_j}(y) \nonumber\\
		&\geq \frac{1}{(4\pi)^{\frac{n-1}{2}} r^{n-1}_j } \int_{B_{r_j}(x_j)} \eta_m(y-x_j)\,\exp\left(-\frac{|y-x_j|^2}{4r^2_j} \right)\frac{W(u^{\varepsilon_j}(y,\,t_j))}{\varepsilon_j}\,dy. \label{clearOut01}
	\end{align}
	From Lemma \ref{lemmaGradientEsti}, it holds that for any $y\in B_{r_j}(x_j)$
	\begin{equation}\label{clearOut02}
		|u^{\varepsilon_j}(y,\,t_j)| \leq |x_j-y|\,\sup_{\Omega \times (T_0,\,T]}|\nabla u^{\varepsilon_j}| + |u^{\varepsilon_j}(x_j,\,t_j)| \leq c_1\,\gamma_0 + \frac{1}{2\sqrt{3}}
	\end{equation} 
	where $c_1>0$ is the constant in Lemma \ref{lemmaGradientEsti} independent of $m$ and thus by choosing $\gamma_0>0$ such that $c_1\,\gamma_0 + \frac{1}{2\sqrt{3}} < \frac{3}{4\sqrt{3}}$, we obtain from \eqref{clearOut01} that
	\begin{equation}\label{clearOut03}
		\int_{B_{r_j}(x_j)}	\eta_m(y-x_j)\,\rho_{(x_j,\,T_j)}(y,\,t_j)\,d\mu^{\varepsilon_j}_{t_j}(y) \geq \frac{c'_1\,\gamma_0}{(4\pi)^{\frac{n-1}{2}} r^{n}_j }\int_{B_{r_j}(x_j)}e^{-\frac{1}{4}}\,dy = \frac{c'_1\,\gamma_0\,e^{-\frac{1}{4}}\,\omega_n}{(4\pi)^{\frac{n-1}{2}} }
	\end{equation}	
	where $c^{\prime}_1>0$ is a constant independent of $m$. Now let $\tilde{c}_1$ denote the last term in \eqref{clearOut03}, which depends on $\gamma_0$. Then from the monotonicity formula \eqref{claimMonotoForm} in Proposition \ref{monotonicityFormula}, the choice of $x_j$, and the fact that $\sup_{\Omega\times(T_0,\,T]}\xi^{\varepsilon_j} \leq c_0$, we may compute as follows:
	\begin{align}
		\tilde{c}_1 & \leq \int_{\Omega} \eta_m(y-x_j)\,\rho_{(x_j,\,T_j)}(y,\,t_j)\,d\mu^{\varepsilon_j}_{t_j}(y) \nonumber\\
		&\leq e^{2(T_i - s)^{\frac{1}{4}}} \int_{\Omega} \eta_m(y-x_j)\,\rho_{(x_j,\,T_j)}(y,\,s)\,d\mu^{\varepsilon_j}_s(y) + \int_{s}^{t_j}e^{2(T_j-t)^{\frac{1}{4}}} \left(C_1(m) + \frac{\sqrt{4\pi}c_0}{(T_j-t)^{\frac{1}{2}}}\right)\,dt \label{clearOut04}
	\end{align}
	Letting $j \rightarrow \infty$ and from the convergence of the measure $\{\mu^{\varepsilon_j}\}_{j\in\mathbb{N}}$, we have
	\begin{align}
		\tilde{c}_1 &\leq e^{2(t' - s)^{\frac{1}{4}}} \int_{\Omega} \eta_m(y-x')\,\rho_{(x',\,t')}(y,\,s)\,d\mu_s(y) + \int_{s}^{t'}e^{2(t'-t)^{\frac{1}{4}}} \left(C_1(m) + \frac{\sqrt{4\pi}c_0}{(t'-t)^{\frac{1}{2}}}\right)\,dt \nonumber\\
		&\leq e^2 \int_{\Omega} \eta_m(y-x')\,\rho_{(x',\,t')}(y,\,s)\,d\mu_s(y) + e^2\,C_1(m)\,(t'-s) + e^2\,2\sqrt{4\pi}\,c_0\,(t'-s)^{\frac{1}{2}} \nonumber\\
		&\leq e^2 \int_{\Omega} \eta_m(y-x')\,\rho_{(x',\,t')}(y,\,s)\,d\mu_s(y) + e^2\left(C_1(m)r_0^2 + 2\sqrt{4\pi}\,c_0\,r_0 \right) . \label{clearOut05}
	\end{align}
	Choosing $r_0>0$ in such a way that $e^2\left(C_1(m)r_0^2 + 2\sqrt{4\pi}\,c_0\,r_0 \right) < \frac{1}{4}\tilde{c}_1$ and setting $\delta_0 \coloneqq \frac{3e^{-2}\tilde{c}_1}{16}$, we obtain 
	\begin{equation}\label{clearOut06}
		4\delta_0 \leq \int_{\Omega} \eta_m(y-x')\,\rho_{(x',\,t')}(y,\,s)\,d\mu_s(y).
	\end{equation}
	Let $\delta>0$ be any number given later. Recall Proposition \ref{densityEstimateLimitMu} and the assumption \eqref{assumpClearingLemma} with $\delta_0 = \frac{3e^{-2}\tilde{c}_1}{16}$. Then from the results in \cite[3.4. Lemma]{Ilmanen01} and \cite[Lemma A.1]{MiTo}, we can choose a constant $\gamma_1= \gamma_1(\delta)>0$ such that, by choosing $\gamma_0$ such that $\gamma_0 < \gamma_1$, we have
	\begin{align}
		\int_{\Omega} \eta_m(y-x')\,\rho_{(x',\,t')}(y,\,s)\,d\mu_s(y) &= \int_{\mathbb{R}^n} \eta_m(y-x') \rho^r_{x'}(y)\,d\mu_s(y) \nonumber\\
		&\leq (1+\delta)\int_{\mathbb{R}^n} \eta_m(y-x) \rho^r_{x}(y)\,d\mu_s(y) + \delta \, D_0 \nonumber\\
		&= (1+\delta) \int_{\Omega} \eta_m(y-x)\,\rho_{(x,\,t)}(y,\,s)\,d\mu_s(y) +  \delta \, D_0  \label{clearOut07}
	\end{align}
	where $D_0$ is the constant as in Proposition \ref{densityEstimateLimitMu} and $\rho^r_x(y)$ is defined by
	\begin{equation}
		\frac{1}{(2\pi r^2)^{\frac{n-1}{2}}} \exp\left( -\frac{|x-y|^2}{2r^2} \right), \quad r \coloneqq \sqrt{2(s-t)}.
	\end{equation}
	Since we have that $x\in \Omega_{[\frac{m}{2}]}$ and from the assumption \eqref{assumpClearingLemma}, we have $\eta_m(y-x) = 0$ for any $y \in \Omega\setminus\overline{\Omega_m}$ and thus it holds that
	\begin{equation}
		\int_{\Omega} \eta_m(y-x)\,\rho_{(x,\,t)}(y,\,s)\,d\mu_s(y) = \int_{\overline{\Omega}_m} \eta_m(y-x)\,\rho_{(x,\,t)}(y,\,s)\,d\mu_s(y) < \delta_0
	\end{equation} 
	and thus from \eqref{clearOut07} we can conclude
	\begin{equation}\label{clearOut07'}
		\int_{\Omega} \eta_m(y-x')\,\rho_{(x',\,t')}(y,\,s)\,d\mu_s(y) \leq (1 + \delta) \delta_0 + \delta\,D_0.
	\end{equation}
	If we choose $\delta>0$ such that $(1+\delta) \delta_0 + \delta \,D_0 < 2\delta_0$, then from \eqref{clearOut06} and \eqref{clearOut07'} we can derive the contradiction that
	\begin{equation}
		0< 4\delta_0 \leq \int_{\Omega} \eta_m(y-x')\,\rho_{(x',\,t')}(y,\,s)\,d\mu_s(y) < 2\delta_0.
	\end{equation}
	Notice that $\delta$ depends on $m$ because $D_0$ does. Thus $\gamma_0$ and $\gamma_1$ depend on $m$ and from the definition of $\delta_0$, $\delta_0$ also depends on $m$. 
\end{proof}

By employing the same argument as in \cite[Lemma 6.2]{MiTo} and combining Proposition \ref{monotonicityFormula}, Lemma \ref{estiSolInSpt}, and Lemma \ref{lemmaGradientEsti}, we can prove the next lemma (Lemma \ref{clearingOutLemma}) even in our case. This is because we restrict ourselves to only consider the interior $\Omega$ with a proper cut-off function $\eta_m$ for each $m\in\mathbb{N}$.
\begin{lemma}\label{forwardDensityLowerBound}
	For any $T_0,\,T>0$ and $m\in\mathbb{N}$ with $m > \max\{d_0,\,2\}$, we choose $\delta_0=\delta_0(T_0,T,m)>0$ as the constant as in Lemma \ref{clearingOutLemma}. Then it holds
	\begin{equation}
		\mu( Z_{m}(T_0,T) ) := (\mu_t \otimes \mathcal{L}^1_t) (Z_{m}(T_0,T)) = 0
	\end{equation}
	where we set
	\begin{align}
		&Z_{m}(T_0,T) \nonumber\\
		&\qquad \coloneqq \left\{ (x,\,t) \in \spt\mu \mid x \in \Omega_{[\frac{m}{2}]}, \,t\in(T_0,\,T), \, \overline{\lim_{s \downarrow t}} \int_{\overline{\Omega}_m}\eta_m(x-y)\,\rho_{(y,\,s)}(x,\,t) \,d\mu_{s}(y) < \delta_0  \right\}
	\end{align}  
\end{lemma}
\begin{proof}
	The proof is basically same as the one conducted in \cite[Lemma 6.3]{MiTo} because we restrict ourselves to deal with the interior estimates of the solution $u^{\varepsilon,\,\sigma}$ and thus there is no need to look at the effects from the boundary $\partial \Omega$. 
	
	We take any $T_0,\,T>0$ and $m \in \mathbb{N}$ with $m > \max\{d_0,\,2\}$. Corresponding to $T_0$ and $T$, let $\delta_0,\,r_0,\,\gamma_0>0$ be the constants as in Lemma \ref{clearingOutLemma}. For any $\tau\in(0,\,\frac{r_0^2}{2})$, we define $Z^{\tau}_{m}(T_0,T)$ by 
	\begin{equation}\label{fordensi01}
		\left\{ (x,t) \in \spt\mu \cap (\Omega_{[\frac{m}{2}]} \times (T_0,T)) \mid  \int_{\overline{\Omega}_m}\eta_m(x-y)\,\rho_{(y,\,s)}(x,\,t) \,d\mu_{s}(y) < \delta_0 \text{ for } s \in (t,t+\tau)  \right\}.
	\end{equation} 
	Then by choosing a sequence $\{\tau_j\}_{j\in\mathbb{N}}$ such that $\tau_j \to 0$ as $j\to\infty$, we have that $\cup_{j=1}^{\infty} Z^{\tau_j}_m(T_0,T) = Z_m(T_0,T)$ and thus we only need to show that $\mu(Z^{\tau_j}_m(T_0,T)) = 0$ for any $\tau_j$.
	
	Let $j\in\mathbb{N}$ be any and $(x,\,t) \in Z^{\tau_j}_m(T_0, T)$ with $\tau_j \ll T- T_0$ be a fixed point. We define the set $P_m(x,\,t)$ by 
	\begin{equation}\label{fordensi02}
		\left\{(x',\,t')\in\overline{\Omega}_m \times (T_0,\,T) \mid \gamma_0^{-2}\,|x'-x|^2 < |t'-t| < \tau_j \right\}.
	\end{equation}
	and in the following we claim $P(x,\,t) \cap Z^{\tau_j}_m(T_0, T) = \emptyset$. Suppose by contradiction that there exists a point $(x',\,t') \in P(x,\,t) \cap Z^{\tau_j}_m(T_0,T)$. Assume that $t'>t$ and set $s=\frac{1}{2}(t+t')$. Then we have that $t < s < t + \frac{\tau_j}{2}$, $|x' - x| < \gamma_0\sqrt{2(s-t)}$, and   
	\begin{equation}\label{fordensi03}
		\int_{\overline{\Omega}_m} \eta_m(x-y)\,\rho_{(y,\,s)}(x,\,t)\,d\mu_s(y) < \delta_0(m).
	\end{equation}
	From the choice of $x$ and Lemma \ref{clearingOutLemma}, we can derive that $(x',\,t') \not\in \spt\mu$, which contradicts $(x',\,t') \in Z^{\tau_j}_m(T_0, T)$.
	
	Next we set $Z^{j, x_0, t_0}_m(T_0, T)$ as the set 
	\begin{equation}
		Z^{\tau_j}_m(T_0, T) \cap \left( B_{\frac{\gamma_0}{2}\sqrt{\tau}}(x_0) \times \left(t_0 - \frac{\tau}{2},\,t_0 + \frac{\tau}{2}\right) \right)
	\end{equation}
	for any $(x_0,\,t_0) \in \overline{\Omega}_{[\frac{m}{2}]} \times (T_0,\,T)$. Then $Z^{\tau_j}_m(T_0, T)$ is the countable union of a family $\{Z^{j, x_{\ell}, t_{\ell}}_m(T_0, T)\}_{\ell\in\mathbb{N}}$ where $\{(x_{\ell},\,t_{\ell})\}_{\ell}$ is a suitable sequence distributed in $\overline{\Omega}_{[\frac{m}{2}]} \times (T_0,\,T)$. Therefore it is sufficient to prove that $\mu(Z^{j, x_{\ell}, t_{\ell}}_m(T_0, T)) = 0$ for each $\ell$. Let $Z^{j,\ell}_m$ denote $Z^{j, x_{\ell}, t_{\ell}}_m(T_0, T)$ for simplicity. Take any $\rho\in(0,\,\frac{1}{m})$. Then setting the projection $\pi_{\Omega}(Z^{j,\ell}_m) \coloneqq \{x \in \Omega_{[\frac{m}{2}]} \mid (x,\,t) \in Z^{j,\ell}_m \}$, we can choose a family of balls $\{B_{r_h}(x_h)\}_{h\in\mathbb{N}}$ and a sequence $\{t_h\}_{h\in\mathbb{N}}$ such that $r_h \leq \rho$, $(x_h,\,t_h) \in Z^{j,\ell}_m$, and the family covers $\pi_{\Omega}(Z^{j,\ell}_m)$. Thus we have
	\begin{equation}\label{fordensi04}
		\sum_{h = 1}^{\infty} \omega_n\,r_h^{n} \leq 2 \mathcal{L}^n \left( B_{\frac{\gamma_0}{2}\sqrt{\tau_j}}(x_{\ell}) \right).
	\end{equation}
	For such a family, we can obtain
	\begin{equation}\label{fordensi05}
		Z^{j,\ell}_m \subset \bigcup_{h=1}^{\infty} \left( B_{r_h}(x_h) \times (t_h - \gamma_0^{-2}r_h^2,\,t_h + \gamma_0^{-2}r_h^2)  \right).
	\end{equation}
	Indeed, for any $(x,\,t) \in Z^{j,\ell}_m$, we can choose $h \in \mathbb{N}$ such that $x \in B_{r_h}(x_h)$ and $(x_h,\,t_h) \in Z^{j,\ell}_m$. Then from the claim that $P(x,\,t) \cap Z^{\tau_j}_m(T_0, T) = \emptyset$, we have
	\begin{equation}
		|t - t_h| \leq \gamma_0^{-2} |x - x_h|^2 < \gamma_0^{-2}r_h^2.
	\end{equation}
	Therefore, from \eqref{fordensi04} and \eqref{fordensi05} and by applying Proposition \ref{densityEstimate}, we have
	\begin{align}
		\mu^{\varepsilon}(Z^{j,\ell}_m) \leq \sum_{h = 1}^{\infty} \int_{t_h - \gamma_0^{-2}r_h^2}^{t_h + \gamma_0^{-2}r_h^2} \mu^{\varepsilon}_t(B_{r_h}(x_h)) \,dt &\leq \sum_{h = 1}^{\infty} D_0\,r_h^{n-1} \, 2\gamma_0^{-2}r_h^2 \nonumber\\
		&\leq 4D_0 \,\gamma_0^{-2}\,\rho \,  \mathcal{L}^n \left( B_{\frac{\gamma_0}{2}\sqrt{\tau_j}}(x_{\ell}) \right)
	\end{align}
	for any $\varepsilon>0$ and thus letting $\varepsilon \to 0$ and using the lower semi-continuity for the Radon measure $\mu$, we obtain
	\begin{equation}
		\mu(Z^{j,\ell}_m) \leq 4D_0 \,\gamma_0^{-2}\,\rho \,  \mathcal{L}^n \left( B_{\frac{\gamma_0}{2}\sqrt{\tau_j}}(x_{\ell}) \right)
	\end{equation}
	and since $\rho\in(0,\,\frac{1}{m})$ is arbitrary, we can conclude $\mu(Z^{j,\ell}_m) = 0$.
\end{proof}

Now we prove the main claim of this appendix, namely, the vanishing of the discrepancy in the interior of the domain.
\begin{proof}[Proof of Proposition \ref{vanishingProp}]
	Now we are ready to prove the vanishing of the discrepancy measures. First of all, since we have the uniform estimate $\sup_{\varepsilon>0}\mu^{\varepsilon,\,\sigma}_t(\Omega) \leq D$ for all $t>0$ and $\sigma>0$ from Proposition \ref{prop.4.1}, we have that
	\begin{equation}
		\sup_{\varepsilon>0} \xi^{\varepsilon,\,\sigma}(\Omega \times (T_0,\,T))  \leq \sup_{\varepsilon>0} \mu^{\varepsilon,\,\sigma}(\Omega \times (T_0,\,T)) \leq D(T - T_0) <\infty
	\end{equation}
	for $0 < T_0 < T$. Thus from the compactness of Radon measures we may choose a convergent subsequence $\{\xi^{\varepsilon_j,\,\sigma}\}_{j\in\mathbb{N}}$ on $\Omega \times (T_0,\,T)$. Therefore we only need to show that the limit denoted by $\xi^{\sigma}$ is identically zero on $\Omega \times (T_0,\,T)$. However, from the construction of $\eta_m$ for each $m\in\mathbb{N}$, it is sufficient to prove that the discrepancy measure vanishes in $\overline{\Omega}_m \times (T_0,\,T)$ for each fixed $m \in \mathbb{N}$ with $m > \max\{d_0,\,3\}$.
	
	In the sequel, we omit the parameter $\sigma>0$ for simplicity. For any $m \in \mathbb{N}$, $y \in \overline{\Omega}_m$, and $s,\,t \in (0,\,T)$ with $T_0 < t < s$, we have from the monotonicity formula of Proposition \ref{monotonicityFormula} and the estimate of the discrepancy measure of Proposition \ref{appendAProp2},
	\begin{align}
		&\frac{d}{dt}\left( e^{2(s-t)^{\frac{1}{4}}} \int_{\Omega}\tilde{\rho}^m_{(y,\,s)}\,d\mu^{\varepsilon_j}_t(x) \right) + e^{2(s-t)^{\frac{1}{4}}} \int_{\Omega} \frac{\tilde{\rho}^m_{(y,\,s)}}{2(s-t)} \,d|\xi^{\varepsilon_j}_t|(x) \nonumber\\
		&\qquad \quad \leq e^{2(s-t)^{\frac{1}{4}}} \left( C_1 + \int_{\Omega} \frac{\tilde{\rho}^m_{(y,\,s)}}{2(s-t)} c_0\,dx \right) \leq e^{2(s-t)^{\frac{1}{4}}} \left( C_1 + \frac{\sqrt{4\pi}c_0}{\sqrt{s-t}}  \right)
	\end{align}
	Integrating over $t\in(T_0,\,s)$ and letting $j \to \infty$, we obtain
	\begin{align}
		 \int\int_{\Omega \times [T_0,\,s]} & \frac{\tilde{\rho}^m_{(y,\,s)}}{2(s-t)} \,d\xi(x,\,t) \nonumber\\
		 &\qquad  \leq e^{ 2 (s - T_0)^{\frac{1}{4}} } \int_{\Omega}\tilde{\rho}^m_{(y,\,s)}(x,\,T_0)\,d\mu_{T_0}(x) + \int_{T_0}^{s} e^{2(s-t)^{\frac{1}{4}}} \left( C_1 + \frac{\sqrt{4\pi}c_0}{\sqrt{s-t}}  \right)\,dt. \label{vanishingDiscrepancy01}
	\end{align}
	Note that we do not need to consider the effects from $\partial \Omega$ in the integrals in \eqref{vanishingDiscrepancy01} due to the construction of $\eta_m$. Since the right-hand side of \eqref{vanishingDiscrepancy01} is uniformly bounded in $y$ and $s$, integrating the both sides in \eqref{vanishingDiscrepancy01} over $(y,\,s) \in \overline{\Omega}_m \times (T_0,\,T)$ with respect to $d\mu_s(y) \otimes ds$ leads us to the fact that
	\begin{equation}\label{vanishingDiscrepancy02}
		\int_{T_0}^{T}\,ds\int_{\overline{\Omega}_m}\,d\mu_s(y)\int\int_{\Omega \times [T_0,\,s]} \frac{\tilde{\rho}^m_{(y,\,s)}}{2(s-t)} \,d\xi(x,\,t) < \infty.
	\end{equation}
	Thus, by Fubini's theorem \eqref{vanishingDiscrepancy02} turns out to be
	\begin{equation}
		\int_{T_0}^{T}\int_{\Omega}\,d\xi(x,\,t)\int_{t}^{T}\int_{\overline{\Omega}_m} \frac{\tilde{\rho}^m_{(y,\,s)}}{2(s-t)} \,d\mu_s(y)\,ds < \infty
	\end{equation}
	and we obtain
	\begin{equation}\label{vanishingDiscrepancy03}
		\int_{t}^{T}\int_{\overline{\Omega}_m} \frac{\tilde{\rho}^{m}_{(y,\,s)}}{2(s-t)}\,d\mu_s(y)\,ds < \infty
	\end{equation}
	for $\xi$-a.e. $(x,\,t) \in \Omega \times (T_0,\,T)$. 
	
	We next prove that, for each $m\in\mathbb{N}$ with $m > \max\{d_0,\,2\}$,
	\begin{equation}\label{keyestiProofDiscre}
		\lim_{s \downarrow t}\int_{\overline{\Omega}_m}\tilde{\rho}^{m}_{(y,\,s)}(x,\,t) \,d\mu_s(y) = 0
	\end{equation}
	for $\xi$-a.e. $(x,\,t) \in \Omega_k \times (T_0,\,T)$ and any $k\in\mathbb{N}$ such that $m \geq 2k$. For $t<s$, we define $\beta \coloneqq \log(s-t)$ and
	\begin{equation}\label{vanishingDiscrepancy04}
		h_m(s) \coloneqq \int_{\overline{\Omega}_m} \tilde{\rho}^{m}_{(y,\,s)}(x,\,t) \,d\mu_s(y).
	\end{equation}
	Then from \eqref{vanishingDiscrepancy04}, \eqref{vanishingDiscrepancy03} turns into
	\begin{equation}\label{vanishingDiscrepancy05}
		\int_{-\infty}^{\log(T-t)}h_m(t + e^{\beta}) \,d\beta < \infty.
	\end{equation}
	Let $\theta \in (0,\,1)$ be any number chosen later. From \eqref{vanishingDiscrepancy05} we may choose a sequence $\{\beta_i\}_{i\in\mathbb{N}}$ such that $0< \beta_{i} - \beta_{i+1} < \theta$, $\beta_i \to -\infty$ as $i \to \infty$, and $h_m(t+e^{\beta_i}) < \theta$ for all $i\in\mathbb{N}$. Then taking any $\beta \in (-\infty,\,\beta_0)$ and fixing it, we can choose $i_0\in\mathbb{N}$ with $i_0 \geq 1$ such that $\beta_{i} \leq \beta $ for any $i \geq i_0$ and $\beta_{i_0} \leq \beta < \beta_{i_0-1}$. By using the fact that $\rho_{(y,\,t+e^{\beta})}(x,\,t) = \rho_{(y,\,t+2e^{\beta})}(x,\,t+e^{\beta})$ and from Proposition \ref{appendAProp2} and  \ref{monotonicityFormula}, we have that  
	\begin{align}
		h_m(t+e^{\beta}) &= \int_{\overline{\Omega}_m} \tilde{\rho}^m_{(y,\,t+e^{\beta})}(x,\,t) \,d\mu_{t+e^{\beta}}(y) \nonumber\\
		&\leq e^{2e^{\frac{\beta}{4}}}\int_{\Omega} \tilde{\rho}^m_{(x,\,t+2e^{\beta})}(y,\,t+e^{\beta}) \,d\mu_{t+e^{\beta}}(y)  \nonumber\\
		&\leq e^{2(2e^{\beta} - e^{\beta_i})^{\frac{1}{4}}} \int_{\Omega} \tilde{\rho}^m_{(x,\,t+2e^{\beta_i})}(y,\,t+e^{\beta_i}) \,d\mu_{t+e^{\beta_i}}(y) \nonumber\\
		&\qquad + \int_{t+e^{\beta_i}}^{t+e^{\beta}} e^{2(t+2e^{\beta} - \tau)^{\frac{1}{4}}} \left( C_1 + \frac{\sqrt{4\pi}c_0}{\sqrt{t+2e^{\beta} - \tau}}  \right)\,d\tau  \label{vanishingDiscrepancy06}
	\end{align}
	 for any $i \geq i_0$ and $(x,\,t) \in \Omega_k \times (T_0,\,T)$ such that \eqref{vanishingDiscrepancy03} holds. Hence, setting $R^2_i \coloneqq 2e^{\beta} - e^{\beta_i}$ and 
	 \begin{equation}
	 	\rho^{R}_{y}(x) \coloneqq \frac{1}{(4\pi R^2)^{\frac{n-1}{2}}} \exp\left( -\frac{|x-y|^2}{4R^2}\right)
	 \end{equation}
	 for any $x,\,y \in \mathbb{R}^n$ and $R>0$, we obtain
	 \begin{align}
	 	h_m(t+e^{\beta}) &\leq e^{2\sqrt{R_i}} \int_{\Omega} \eta_m(x-y)\,\rho^{R_i}_{x}(y) \,d\mu_{t+e^{\beta_i}}(y) + c_i(\beta)  \label{vanishingDiscrepancy07}
	 \end{align}
	 where $c_i(\beta)$ is defined by the last term in \eqref{vanishingDiscrepancy06}. By the change of variables $\tau \mapsto t+e^{\beta}-\tau$, we have 
	 \begin{equation}
	 	c_i(\beta) = \int_{0}^{e^{\beta}-e^{\beta_i}} e^{2(\tau + e^{\beta})^{\frac{1}{4}}} \left( C_1 + \frac{\sqrt{4\pi}c_0}{\sqrt{\tau + e^{\beta}}}  \right)\,d\tau
	 \end{equation}
	 and thus we obtain
	\begin{equation}
		c(\beta) \coloneqq \lim_{i \to \infty}c_i(\beta) = \int_{0}^{e^{\beta}} e^{2(\tau + e^{\beta})^{\frac{1}{4}}} \left( C_1 + \frac{\sqrt{4\pi}c_0}{\sqrt{\tau + e^{\beta}}}  \right)\,d\tau.
	\end{equation}
	From the definition of $\beta$, we can see that $\lim_{\beta \to -\infty}c(\beta) = 0$. From the assumption of $\beta_i$,  
	\begin{align}
		\theta > h_m(t+e^{\beta_i})  &= \int_{\overline{\Omega}_m} \eta_m(x-y)\,\tilde{\rho}^{m}_{(y,\,t+e^{\beta_i})}(x,\,t) \,d\mu_{t+e^{\beta_i}}(y) \nonumber\\
		&= \int_{\overline{\Omega}_m} \eta_m(x-y)\,\rho^{r_i}_{x}(y) \,d\mu_{t+e^{\beta_i}}(y), \label{vanishingDiscrepancy08}
	\end{align}
	where $r^2_{i} = e^{\beta_i}$. Since $\beta \geq \beta_i$ and $\beta_{i-1} - \beta_{i} < \theta$ for any $i\geq i_0$, we have $R_i>r_i$, $\beta_{i_0} \leq \beta < \beta_{i_0-1}$, and $R^2_i\,r^{-2}_i \leq 2e^{\theta}-1$. By arbitrariness of $\theta$, we can make $R^2_i\,r^{-2}_i$ as close to 1 as possible if we let $\theta$ be small and $i$ be sufficiently large if necessary. From Proposition \ref{densityEstimateLimitMu} and by the result in \cite[3.4 Lemma]{Ilmanen01} or \cite[Lemma A.1]{MiTo}, for any $\delta\in(0,\,1)$ and $i\geq i_0$, there exists a constant $\gamma_1=\gamma_1(\delta,\,i_0)>0$ such that $1\leq R_i\,r^{-1}_i \leq 1+ \gamma_1$ and
	\begin{equation}
		\int_{\Omega} \eta_m(x-y)\,\rho^{R_i}_{x}(y) \,d\mu_{t+e^{\beta_i}}(y) \leq (1+\delta)\int_{\Omega} \eta_m(x-y)\,\rho^{r_i}_{x}(y) \,d\mu_{t+e^{\beta_i}}(y) + \delta\,D_0. \label{vanishingDiscrepancy09}
	\end{equation}
	Moreover, from the choice of $k$ and $x$, we have that $\Omega \cap B_{1/m}(x) \subset \overline{\Omega}_m$ and thus
	\begin{align}\label{vanishingDiscrepancy10}
		\int_{\Omega} \eta_m(x-y)\,\rho^{r_i}_{x}(y) \,d\mu_{t+e^{\beta_i}}(y) &= \int_{\Omega \cap B_{\frac{1}{m}}(x)} \eta_m(x-y)\,\rho^{r_i}_{x}(y) \,d\mu_{t+e^{\beta_i}}(y) \nonumber\\
		&\leq \int_{\overline{\Omega}_m} \eta_m(x-y)\,\rho^{r_i}_{x}(y) \,d\mu_{t+e^{\beta_i}}(y) = h_m(t+e^{\beta_i})
	\end{align}
	for $\xi$-a.e. $(x,\,t) \in \Omega_k \times (T_0,\,T)$. Hence from \eqref{vanishingDiscrepancy07}, \eqref{vanishingDiscrepancy08}, \eqref{vanishingDiscrepancy09}, and \eqref{vanishingDiscrepancy10}, we obtain
	\begin{align}
		h_m(t+e^{\beta}) &\leq e^{2\sqrt{R_i}}(1+\delta)\int_{\Omega} \eta_m(x-y)\,\rho^{r_i}_{x}(y) \,d\mu_{t+e^{\beta_i}}(y) + e^{2\sqrt{R_i}}\delta\,D_0  + c_i(\beta) \nonumber\\
		&\leq e^{2\sqrt{R_i}}(1+\delta)\,h_m(t+e^{\beta_i})  + e^{2\sqrt{R_i}}\delta\,D_0  + c_i(\beta) \nonumber\\
		&\leq e^{2\sqrt{R_i}}(1+\delta)\,\theta + e^{2\sqrt{R_i}}\delta\,D_0  + c_i(\beta)
	\end{align}
	and thus, letting first $i \to \infty$ and $\delta \downarrow 0$ and then $\beta \to -\infty$, we conclude that $\lim_{s \downarrow t} h_m(s) = 0$ for $\xi$-a.e. $(x,\,t) \in \Omega_k \times (T_0,\,T)$ and any $k \in \mathbb{N}$ such that $m \geq 2k$.
	
	Therefore, setting $k \coloneqq [\frac{m}{2}]$ for each $m$, we obtain $\xi(\overline{\Omega}_k \times (T_0,\,T) \setminus Z_m(T_0,T)) = 0$; otherwise from \eqref{keyestiProofDiscre} we can choose a point $(x,\,t) \in \Omega_k \times (T_0,\,T) \setminus Z_m(T_0,T)$ such that
	\begin{equation}
		0= \lim_{s \downarrow t}h_m(s) = \overline{\lim_{s \downarrow t}}\int_{\overline{\Omega}_m} \rho^m_{(y,\,s)}(x,\,t) \,d\mu_s(y) \geq \delta_0(T_0,T,m) >0,
	\end{equation}
	which is a contradiction. Combining Lemma \ref{forwardDensityLowerBound} with the fact that $\xi \leq \mu$ on $\Omega \times [0,\,\infty)$, we have that
	\begin{align}
		\xi(\Omega_{[\frac{m}{2}]} \times (T_0,\,T)) &= \xi(\Omega_{[\frac{m}{2}]} \times (T_0,\,T) \setminus Z_m(T_0,T)) + \xi(Z_m(T_0,T)) \nonumber\\
		&\leq 0 + \mu(Z_m(T_0,T)) = 0
	\end{align}
	for any $m\in\mathbb{N}$ with $m > \max\{d_0,\,2\}$ and thus letting $m \to \infty$ and from the fact that $\cup_{m>\max\{d_0,\,2\}}\Omega_{[\frac{m}{2}]} = \Omega$, we can conclude that $\xi(\Omega \times (T_0,\,T)) = 0$ for any $0<T_0<T$.
\end{proof}

\subsection{Appendix C}\label{appendixPoincareInequality}
In this appendix, we prove Poincar\'e-Wirtinger inequality on hypersurfaces. This claim was applied to prove the positivity of the limit measure $\alpha$ on $\Gamma_2\times [t_1,\,t_2]$ for some non-empty connected component $\Gamma_2$ of $\partial\Omega$ and some $0<t_1<t_2<\infty$ in Subsection \ref{chara.2}. 

\begin{lemma}\label{app.lem}
	Let $N\geq3$ and $M\subset \mathbb{R}^N$ be a smooth, bounded, and connected hypersurface embedded in $\mathbb{R}^N$ without boundaries. Then there exists a constant $C(N,\,M)>0$ such that, for any $u \in W^{1,\,1}(M)$,
	\begin{equation}\label{7.1.1}
	\|u-u_M\|_{L^{1}(M)}\leq C(N,\,M)\|\nabla_{M}u\|_{L^1(M)}
	\end{equation}
	and $u_M\coloneqq(\mathcal{H}^{N-1}(M))^{-1}\int_{M}u\,d\mathcal{H}^{N-1}$.
\end{lemma}

\begin{proof}
	Setting $X\coloneqq\{u\in W^{1,\,1}(M) \mid u_M=0 \}$, we have that $X$ is a closed subspace in $W^{1,\,1}(M)$. Then, it is sufficient to prove that there exists $C(N,\,M)>0$ such that $\|u\|_{L^{1}}\leq C(N,\,M)\|\nabla_{M}u\|_{L^1}$ for each $u\in X$. This is because we can easily have that, if $u\in W^{1,\,1}(M)$, then $u-u_{M}\in X$ and $\nabla_{M}(u-u_{M})=\nabla_{M}u$. 
	
	We assume, for a contradiction, that, for each $n\in\mathbb{N}$, there exists $u_n\in X$ such that $\|u_n\|_{L^{1}} > n \|\nabla_{M}u_n\|_{L^1}$ holds. From the assumption, we may consider that $\{u_n\}_{n\in\mathbb{N}}$ is a bounded sequence in $W^{1,\,1}$ and $\|u_n\|_{W^{1,\,1}(M)}\geq1$ for any $n\in\mathbb{N}$. Since $N\geq3$ and thus $X\subset W^{1,\,1}(M) \hookrightarrow L^1(M)$ is the compact embedding (see, for instance, \cite{Aubin}) and $X$ is closed in $W^{1,\,1}(M)$, there exist a subsequence $\{u_{n_i}\}_{i\in\mathbb{N}}$ and $u_{\infty}\in X$ such that $u_{n_i} \xrightarrow[i\to\infty]{} u_{\infty}$ in $L^1(M)$. Then, from the assumption, it follows that
	\begin{equation}\label{7.1.2}
	1\leq\|u_{n_i}\|_{W^{1,\,1}}=\|u_{n_i}\|_{L^1}+\|\nabla_{M}u_{n_i}\|_{L^1} \xrightarrow[i\to\infty]{} \|u_{\infty}\|_{L^1}+0,
	\end{equation}
	and thus $\|u_{\infty}\|_{L^1}\geq1$. On the other hand, for any $\phi \in C^1_c(M)$, from the divergence theorem on hypersurfaces and recalling the fact that $u_{n_i}\xrightarrow[i\to\infty]{} u_{\infty}$ in $L^{1}(M)$, we have that
	\begin{align}
	0 \xleftarrow[i\to\infty]{} \|\phi\|_{C^1(M)}\,\|\nabla_{M} u_{n_i}\|_{L^1(M)} &\geq  \left|\int_{M} \phi\nabla_{M}u_{n_i}\,d\mathcal{H}^{N-1} \right|\nonumber\\
	&=\left| \int_{M} u_{n_i} (\nabla_{M}\phi-\phi\boldH_M)\,d\mathcal{H}^{N-1} \right| \nonumber\\
	&\xrightarrow[i\to\infty]{} \left|\int_{M} u_{\infty} (\nabla_{M}\phi-\phi\boldH_M)\,d\mathcal{H}^{N-1}\right| \nonumber\\
	&=\left|\int_{M} \phi \nabla_{M} u_{\infty}\,d\mathcal{H}^{N-1} \right|, \label{7.1.3}
	\end{align}
	where $\boldH_M$ is the mean curvature vector of $M$ which is bounded because $M$ is smooth and compact. This implies $\nabla_{M}u_{\infty}=0$ $\mathcal{H}^{N-1}$-a.e. on $M$ and thus it follows that $\Delta_{M}u_{\infty}=0$ on $M$ in distribution. Then, from Weyl's lemma, we have $u_{\infty}\in C^{\infty}(M)$. Since $M$ is connected, $u_{\infty}$ is in $X$, and $\nabla_M u_{\infty}=0$ $\mathcal{H}^{N-1}$-a.e. on $M$, we may conclude that $u_{\infty}=0$ on $M$ and this contradicts the fact that $\|u_{\infty}\|_{L^1}\geq1$. 
\end{proof}

\section*{\footnotesize{Acknowledgements}}
\footnotesize{
	The first author is partly supported by JSPS KAKENHI Kiban A (No. 19H00639), Kaitaku (No. 18H05323), Kiban S (No. 26220702), Kiban A (No. 17H01091), and Kiban B (No. 16H03948). The second author is partly supported by Leading Graduate Course for Frontiers of Mathematical Science and Physics (FMSP), which is a program for Leading Graduate Schools, MEXT Japan. The third author is partly supported by JSPS KAKENHI WAKATE B (No. 16K17622), Wakate (No. 20K14343), Kiban A (No. 18H03670),
	and JSPS Leading Initiative for Excellent Young Researchers (LEADER) operated by Funds
	for the Development of Human Resources in Science and Technology.
}

\end{document}